\documentclass[11pt]{article}
\usepackage{graphicx}
\usepackage{amssymb}
\usepackage{amsfonts}
\usepackage{amsmath}
\usepackage{hyperref}

\newtheorem{theorem}{Theorem}[section]
\newtheorem{proposition}[theorem]{Proposition}
\newtheorem{lemma}[theorem]{Lemma}
\newtheorem{corollary}[theorem]{Corollary}
\newtheorem{definition}[theorem]{Definition}
\newtheorem{example}[theorem]{Example}
\newtheorem{remark}[theorem]{Remark}

\newenvironment{proof}[1][Proof]{\textbf{#1.} }{\ \rule{0.5em}{0.5em}}

\newenvironment{notes}[1][Notes]{\medskip\footnotesize\textbf{#1}}{\smallskip}

\input{tcilatex}

\def\RM{\rm}

\def\endofdocument{\end{document}}

\def\limfunc#1{\mathop{\mathrm{#1}}}
\def\func#1{\mathop{\mathrm{#1}}\nolimits}

\def\dsum{\displaystyle\sum}

\def\enddoc{\end{document}}

\def\Qlb#1{#1}

\def\Qcb#1{#1} 

\def\FRAME#1#2#3#4#5#6#7#8
{
 \begin{figure}[ptbh]
 \begin{center}
 \includegraphics[height=#3]{#7}
 \caption{#5}
 \label{#6}
 \end{center}
 \end{figure}
}

\begin{document}

\author{Alexander Grigor'yan \thanks{%
Partially supported by SFB 701 of German Research Council and by Visiting
Grants of Harvard University and MSC, Tsinghua University} \\
%EndAName
Department of Mathematics\\
University of Bielefeld\\
33613 Bielefeld, Germany \and Yong Lin \thanks{%
Supported by the Fundamental Research Funds for the Central Universities,
the Research Funds of Renmin University of China(11XNI004), and a Visiting
Grant of Harvard University} \\
%EndAName
Department of Mathematics\\
Information School\\
Renmin University of China\\
Beijing 100872, China \and Yuri Muranov\thanks{%
Partially supported by the CONACyT Grant 98697, SFB 701 of German Research
Council, and a Visiting Grant of Harvard University} \\
%EndAName
Department of Mathematics\\
Grodno State University\\
Grodno 230023, Republic of Belarus \and Shing-Tung Yau \\
%EndAName
Department of Mathematics\\
Harvard University\\
Cambridge MA 02138, USA}
\title{Homologies of path complexes and digraphs }
\date{May 2013}
\maketitle
\tableofcontents

%TCIMACRO{\TeXButton{eject}{\vfil\eject}}%
%BeginExpansion
\vfil\eject%
%EndExpansion

\section{Introduction}

\label{Sec1}In this paper we introduce a new notion -- a path complex that
can be regarded as a generalization of the notion of a simplicial complex.
In short, a path complex $P$ on a finite set $V$ is a collection of paths
(=sequences of points) on $V$ such that if a path $v$ belongs to $P$ then a
truncated path that is obtained from $v$ by removing either the first or the
last point, is also in $P$. Given a path complex $P$, all the paths in $P$
are called \emph{allowed}, while the paths outside $P$ are called \emph{%
non-allowed}.

Any simplicial complex $S$ determines naturally a path complex by
associating with any simplex from $S$ the sequence of its vertices (see
Section \ref{Secpath} for details).

However, the main motivation for considering path complexes comes from
directed graphs (digraphs). A digraph $G$ is a pair $\left( V,E\right) $
where $V$ is a set as above and $E$ is a binary relation on $V$ that is, $E$
is a subset of $V\times V$. If $\left( a,b\right) \in E$ then the pair $%
\left( a,b\right) $ is called a directed edge; this fact is also denoted by $%
a\rightarrow b$. Any digraph naturally gives rise to a path complex where
allowed paths go along the arrows of the digraph.

One of our key observations is that any path complex $P$ gives rise to a
chain complex with an appropriate boundary operator $\partial $ that leads
to the notion of homology groups of $P$. We refer to this notion as a \emph{%
path homology}.

In the case when $P$ arises from a simplicial complex \thinspace $S$, the
path homology of $P$ coincides with the simplicial homology of $S$. If $P$
arises from a digraph $G$ then we obtain a new notion: the path homology of
a digraph. The path complexes of digraphs are the central objects of this
paper. Although most of the results are presented for arbitrary path
complexes, we always have in mind applications for digraphs. On the other
hand, the notion of a path complex provides an alternative approach to the
classical results about simplicial complexes.

There has been a number of attempts to define the notions of homology and
cohomology for graphs. At a trivial level, any graph can be regarded as an
one-dimensional simplicial complex, so that its simplicial homologies are
defined. However, all homology groups of order $2$ and higher are trivial,
which makes this approach uninteresting.

Another way to make a graph into a simplicial complex is to consider all its
cliques (=complete subgraphs) as simplexes of the corresponding dimensions
(cf. \cite{ChenYauYeh}, \cite{Ivash}). Then higher dimensional homologies
may be non-trivial, but in this approach the notion of graph looses its
identity and becomes a particular case of the notion of a simplicial
complex. Besides, some desirable functorial properties of homologies fail,
for example, the K\"{u}nneth formula is not true for Cartesian product of
graphs.

Yet another approach to homologies of digraphs can be realized via
Hochschild homology. Indeed, allowed paths on a digraph have a natural
operation of product, which allows to define the notion of a \emph{path
algebra} of a digraph. The Hochschild homology of the path algebra is a
natural object to consider. However, it was shown in \cite{Happel} that
Hochschild homologies of order $\geq 2$ are trivial, which makes this
approach not so attractive.

The path homologies of digraphs that we introduce in this paper have many
advantages in comparison with the previously studied notions of graph
homologies.

Firstly, the homologies of all dimensions could be non-trivial. Also, the
chain complex associated with a path complex has a richer structure that
simplicial chain complexes. It contains not only cliques but also binary
hypercubes and many other subgraphs. By the way, the dimensions of the chain
spaces are themselves non-trivial invariants of the digraphs.

Secondly, this notion is well linked to graph-theoretical operations. For
example, the K\"{u}nneth formula is true for join of two digraphs as well as
for Cartesian product of two digraphs.

Thirdly, there is a dual cohomology theory with the coboundary operator $d$
that arises independently and naturally as an exterior derivative on the
algebra of functions on the vertex set of the graph. The latter approach to
the cohomology of digraphs, that is based on the classification of exterior
derivations on algebras (cf. \cite[III, \S 10.2]{Bourbaki}), was developed
by Dimakis and M\"{u}ller-Hoissen in \cite{DimMuDiff} and \cite{DimMuGauge}.
In the present paper we introduce the notion of cohomology of path complexes
independently, using the duality with homologies. The reader is referred to 
\cite{GrigMuralg} where the equivalence of the two approaches is explained.

We feel that the notion of path homology of digraphs has a rich mathematical
content and hope that it will become a useful tool in various areas of pure
and applied mathematics. For example, this notion was employed in \cite%
{GrigMurYauhh} to give a new elementary proof of a theorem of Gerstenhaber
and Schack \cite{GShh} that represents simplicial homology as a Hochschild
homology. A link between path homologies of digraphs and cubical homologies
was revealed in \cite{GrigMurYaucubic}. On the other hand, it is conceivable
that the notion of path homology can be used in practical applications such
as hole detection in graphs or coverage verification in sensor networks (cf. 
\cite{Ali}).

Let us briefly describe the structure of the paper and the main results. In
Section \ref{Sec2}, we define the notions of $p$-paths and $p$-forms on a
finite set $V$, as well as the dual operators $\partial $ and $d$. We also
define the notions of join of paths and concatenation of forms and prove
that they satisfy the product rule.

In Section \ref{Sec4} we define the notions of a path complex, a $\partial $%
-invariant path (element of a chain space), and a path homology. Then we
define also the dual notions of $d$-invariant form and path cohomology.

In Section \ref{Sec5} we apply the aforementioned notions to digraphs and
give numerous examples of $\partial $-invariant paths on digraphs. We prove
some basic results about path homologies of digraphs. For example, we
describe chain spaces and homologies of digraphs without squares (Theorem %
\ref{Tsq}) and prove the Poincar\'{e} lemma for star-shaped digraphs
(Theorem \ref{Tstar}).

In Section in \ref{Sec7} we prove some relations between path homologies of
a digraph and its subgraphs (Theorems \ref{Tend}, \ref{Tabac}, and \ref{Tcab}%
).

In Section \ref{SecJoin} we introduce the operation \emph{join} of two path
complexes and prove the K\"{u}nneth formula for homologies of join (Theorem %
\ref{Tjoin}). Particular cases of join are operation of cone and suspension
of a digraph that behave homologically in the same way as those in the
classical algebraic topology.

In Section \ref{SecProduct} we introduce the notions of cross product of
paths and Cartesian product of path complexes. The latter matches the notion
of Cartesian product of digraph. The main result of this section is the K%
\"{u}nneth formula for Cartesian product (Theorem \ref{TK} that is based on
Theorem \ref{Tuivi}).

In Section \ref{Sec8} we sketch an approach to hole detection on digraphs.
In Section \ref{Sec3} (Appendix) we revise for the convenience of the reader
some elementary results of homological algebra that are used in the main
body of the paper.

\section{Paths and forms a finite set}

\setcounter{equation}{0}\label{Sec2}

\subsection{Space of paths and the boundary operator}

\label{Secpchain}Let $V$ be an arbitrary non-empty finite set whose elements
are called vertices. Fix a field $\mathbb{K}$ whose elements are called
scalars.

\begin{definition}
\RM For any non-negative integer $p$, an \emph{elementary} $p$-\emph{path}
on a set $V$ is any sequence $\left\{ i_{k}\right\} _{k=0}^{p}$ of $p+1$
vertices of $V$ (a priori the vertices in the path do not have to be
distinct). For $p=-1$, an elementary $p$-path is the empty set $\emptyset $.
\end{definition}

The $p$-path $\left\{ i_{k}\right\} _{k=0}^{p}$ will also be denoted simply
by $i_{0}...i_{p}$, without delimiters between the vertices.

Denote by $\Lambda _{p}=\Lambda _{p}\left( V,\mathbb{K}\right) $ the $%
\mathbb{K}$-linear space that consists of all formal linear combinations of
all elementary $p$-paths with the coefficients from $\mathbb{K}$.

\begin{definition}
\RM The elements of $\Lambda _{p}$ are called $p$-\emph{paths} on $V$.
\end{definition}

An elementary $p$-path $i_{0}...i_{p}$ as an element of $\Lambda _{p}$ will
be denoted by $e_{i_{0}...i_{p}}$. The empty set as an element of $\Lambda
_{-1}$ will be denoted by $e$.

By definition, the family $\left\{ e_{i_{0}...i_{p}}:i_{0},...,i_{p}\in
V\right\} $ is a basis in $\Lambda _{p}.$ Each $p$-path $v$ has a unique
representation in the form%
\begin{equation}
v=\sum_{i_{0},...,i_{p}\in V}v^{_{i_{0}...i_{p}}}e_{i_{0}...i_{p}},
\label{ve}
\end{equation}%
where $v^{_{i_{0}...i_{p}}}\in \mathbb{K}.$ For example, the space $\Lambda
_{0}$ consists of all linear combinations of the elements $e_{i}$ that are
just the vertices of $V$, the space $\Lambda _{1}$ consists of all linear
combinations of the elements $e_{ij}$ that are pairs of vertices, etc. Note
that $\Lambda _{-1}$ consists of all multiples of $e$, so that $\Lambda
_{-1}\cong \mathbb{K}$.

\begin{definition}
\RM For any $p\geq 0$, the \emph{boundary operator} $\partial :\Lambda
_{p}\rightarrow \Lambda _{p-1}$ is a linear operator that is defined on the
elementary paths by%
\begin{equation}
\partial e_{i_{0}...i_{p}}=\dsum\limits_{q=0}^{p}\left( -1\right)
^{q}e_{i_{0}...\widehat{i_{q}}...i_{p}},  \label{dev}
\end{equation}%
where the hat $\widehat{i_{q}}$ means omission of the index $i_{q}$.
\end{definition}

For example, we have 
\begin{equation}
\partial e_{i}=e\text{,\ \ \ \ }\partial e_{ij}=e_{j}-e_{i},\ \ \ \partial
e_{ijk}=e_{jk}-e_{ik}+e_{ij}.  \label{deij}
\end{equation}%
For an arbitrary $p$-path (\ref{ve}) with $p\geq 0$, we have 
\begin{equation*}
\partial v=\sum_{i_{0},...,i_{p}}v^{_{i_{0}...i_{p}}}\partial
e_{i_{0}...i_{p}}=\sum_{i_{0},...,i_{p}}\dsum\limits_{q=0}^{p}\left(
-1\right) ^{q}v^{_{i_{0}...i_{p}}}e_{i_{0}...\widehat{i_{q}}...i_{p}}
\end{equation*}%
whence%
\begin{eqnarray}
\left( \partial v\right) ^{j_{0}...j_{p-1}}
&=&\sum_{i_{0},...,i_{p}}\dsum\limits_{q=0}^{p}\left( -1\right)
^{q}v^{_{i_{0}...i_{p}}}(e_{i_{0}...\widehat{i_{q}}%
...i_{p}})^{j_{0}...j_{p-1}}  \notag \\
&=&\sum_{k\in V}\sum_{q=0}^{p}\left( -1\right)
^{q}v^{j_{0}...j_{q-1}k\,j_{q}...j_{p-1}}  \label{dv}
\end{eqnarray}%
where the index $k$ is inserted in the path $j_{0}...j_{p-1}$ between $%
j_{q-1}$ and $j_{q}$ if $1\leq q<p$, before $j_{0}$ if $q=0$, and after $%
j_{p-1}$ if $q=p$.

For example, or any $v\in \Lambda _{0}$, we have%
\begin{equation}
\partial v=\sum_{k\in V}v^{k},  \label{dv0}
\end{equation}%
for $v\in \Lambda _{1}$ we have%
\begin{equation*}
\left( \partial v\right) ^{i}=\sum_{k}\left( v^{ki}-v^{ik}\right)
\end{equation*}%
and for $v\in \Lambda _{2}$ we have 
\begin{equation*}
\left( \partial v\right) ^{ij}=\sum_{k}\left( v^{kij}-v^{ikj}+v^{ijk}\right)
.
\end{equation*}

Set also $\Lambda _{-2}=\left\{ 0\right\} $ and define $\partial :\Lambda
_{-1}\rightarrow \Lambda _{-2}$ to be zero.

\begin{lemma}
We have $\partial ^{2}=0.$
\end{lemma}

\begin{proof}
The operator $\partial ^{2}$ acts from $\Lambda _{p}$ to $\Lambda _{p-2},$
so that the identity $\partial ^{2}=0$ makes sense for all $p\geq 0$. In the
case $p=0$ the identity $\partial ^{2}=0$ is trivial. For $p\geq 1$, we have
by (\ref{dev})%
\begin{eqnarray*}
\partial ^{2}e_{i_{0}...i_{p}} &=&\dsum\limits_{q=0}^{p}\left( -1\right)
^{q}\partial e_{i_{0}...\widehat{i_{q}}...i_{p}} \\
&=&\dsum\limits_{q=0}^{p}\left( -1\right) ^{q}\left( \sum_{r=0}^{q-1}\left(
-1\right) ^{r}e_{i_{0}...\widehat{i_{r}}...\widehat{i_{q}}%
...i_{p}}+\sum_{r=q+1}^{p}\left( -1\right) ^{r-1}e_{i_{0}...\widehat{i_{q}}%
...\widehat{i_{r}}...i_{p}}\right) \\
&=&\sum_{0\leq r<q\leq p}\left( -1\right) ^{q+r}e_{i_{0}...\widehat{i_{r}}...%
\widehat{i_{q}}...i_{p}}-\sum_{0\leq q<r\leq p}\left( -1\right)
^{q+r}e_{i_{0}...\widehat{i_{q}}...\widehat{i_{r}}...i_{p}}.
\end{eqnarray*}%
After switching $q$ and $r$ in the last sum we see that the two sums cancel
out, whence $\partial ^{2}e_{i_{0}...i_{p}}=0$. This implies $\partial
^{2}v=0$ for all $v\in \Lambda _{p}$.
\end{proof}

Consequently, we have the following chain complex of the set $V$:%
\begin{equation}
0\leftarrow \mathbb{K}\leftarrow \Lambda _{0}\leftarrow ...\leftarrow
\Lambda _{p-1}\leftarrow \Lambda _{p}\leftarrow ...  \label{cLa}
\end{equation}%
where the arrows are given by the operator $\partial $.

\subsection{Join of paths}

\label{SecProductPath}

\begin{definition}
\RM For all $p,q\geq -1$ and for any two paths $u\in \Lambda _{p}$ and $v\in
\Lambda _{q}$ define their\emph{\ join} $uv\in \Lambda _{p+q+1}$ as follows:%
\begin{equation}
\left( uv\right)
^{i_{0}...i_{p}j_{0}...j_{q}}=u^{i_{0}...i_{p}}v^{j_{0}...j_{q}}.
\label{uvdef}
\end{equation}
\end{definition}

Clearly, join of paths is a bilinear operation that satisfies the
associative law (but is not commutative). For $u=e_{i_{0}...i_{p}}$ and $%
v=e_{j_{0}...j_{q}},$ we obtain from (\ref{uvdef}) 
\begin{equation}
e_{i_{0}...i_{p}}e_{j_{0}...j_{q}}=e_{i_{0}...i_{p}j_{0}...j_{q}}.
\label{epeq}
\end{equation}%
If $p=-2$ and $q\geq -1$ then set $uv=0\in \Lambda _{q-1}.$ A similar rule
applies if $q=-2$ and $p\geq -1.$

Now we can state and prove the product rule for the operation of joining the
paths.

\begin{lemma}
\label{Lemuv}\emph{(Product rule)} For all $p,q\geq -1$ and $u\in \Lambda
_{p}$, $v\in \Lambda _{q}$ we have%
\begin{equation}
\partial \left( uv\right) =(\partial u)v+\left( -1\right) ^{p+1}u\partial v.
\label{duv}
\end{equation}
\end{lemma}

\begin{proof}
It suffices to prove (\ref{duv}) for $u=e_{i_{0}...i_{p}}$ and $%
v=e_{j_{0}...j_{q}}.$ We have 
\begin{eqnarray*}
\partial \left( uv\right) &=&\partial
e_{i_{0}...i_{p}j_{0}...j_{q}}=e_{i_{1}...i_{p}j_{0}...j_{q}}-e_{i_{0}i_{2}...i_{p}j_{0}...j_{q}}+...
\\
&&+\left( -1\right) ^{p+1}\left(
e_{i_{0}...i_{p}j_{1}...j_{q}}-e_{i_{0}...i_{p}j_{0}j_{2}...j_{q}}+...\right)
\\
&=&\left( \partial e_{i_{0}...i_{p}}\right) e_{j_{0}...j_{q}}+\left(
-1\right) ^{p+1}e_{i_{0}...i_{p}}\partial e_{j_{0}...j_{q}},
\end{eqnarray*}%
whence (\ref{duv}) follows.
\end{proof}

\subsection{Regular paths}

\label{SecRegpath}

\begin{definition}
\RM We say that an elementary path $i_{0}...i_{p}$ is \emph{non}-\emph{%
regular} if $i_{k-1}=i_{k}$ for some $k=1,...,p$, and \emph{regular}
otherwise.
\end{definition}

For example, a $1$-path $ii$ is non-regular, while a $2$-path $iji$ is
regular provided $i\neq j$.

For any $p\geq -1$, consider the following subspace of $\Lambda _{p}$
spanned by the regular elementary paths: 
\begin{equation*}
\mathcal{R}_{p}=\mathcal{R}_{p}\left( V\right) :=\limfunc{span}\left\{
e_{i_{0}...i_{p}}:i_{0}...i_{p}\text{ is regular}\right\} .
\end{equation*}%
Note that $\mathcal{R}_{p}=\Lambda _{p}$ for $p\leq 0$, but $\mathcal{R}_{p}$
is strictly smaller than $\Lambda _{p}$ for $p\geq 1$. Similarly, consider
the subspace of $\Lambda _{p}$ spanned by non-regular elementary paths: 
\begin{equation*}
I_{p}=I_{p}\left( V\right) =\limfunc{span}\left\{
e_{i_{0}...i_{p}}:i_{0}...i_{p}\text{ is non-regular}\right\} .
\end{equation*}%
For example, we have $I_{0}=\left\{ 0\right\} $ and $I_{1}=\limfunc{span}%
\left\{ e_{ii}\right\} $.

For $p=-2$ we set $\mathcal{R}_{-2}=I_{-2}=\left\{ 0\right\} $. Then we have 
$\Lambda _{p}=\mathcal{R}_{p}\tbigoplus I_{p}$ for all $p\geq -2$, which
implies that the quotient space%
\begin{equation*}
\widetilde{\mathcal{R}}_{p}=\widetilde{\mathcal{R}}_{p}\left( V\right)
:=\Lambda _{p}/I_{p}
\end{equation*}%
is isomorphic to $\mathcal{R}_{p}$.

In what follows, the elements of a quotient space $U/W$ will be denoted by $u%
\func{mod}W$ where $u\in U.$ Also, we write $u_{1}=u_{2}\func{mod}W$ if $%
u_{1}-u_{2}\in W.$

\begin{definition}
\RM The elements of $\mathcal{R}_{p}$ are called \emph{regular} $p$-paths.
The elements of $\widetilde{\mathcal{R}}_{p}$, that is, the equivalence
classes $v\func{mod}I_{p}$ (where $v\in \Lambda _{p}$), are called \emph{%
regularized} $p$-paths.
\end{definition}

Obviously, any regularized $p$-path has exactly one representative in
regular $p$-paths.

We would like to consider the operator $\partial $ on the spaces $\mathcal{R}%
_{p}$. However, $\partial $ is not invariant on spaces of regular paths. For
example, $e_{iji}\in \mathcal{R}_{2}$ for $i\neq j$ while its boundary $%
\partial e_{iji}=e_{ji}-e_{ii}+e_{ij}$ is not in $\mathcal{R}_{1}$ as it has
a non-regular component $e_{ii}$. The same applies to the notion of join of
paths: the join of two regular path does not have to be regular, for
example, $e_{i}e_{i}=e_{ii}$.

Below we will modify the definitions of $\partial $ and join to make them
invariant on the spaces $\mathcal{R}_{p}$. The basic idea is that when
applying $\partial $ or join on regular paths, one should drop from the
result all the non-regular components, so that it becomes regular.
Technically it is easier (and cleaner) to define $\partial $ and join first
on the quotient spaces $\widetilde{\mathcal{R}}_{p}$, and then pass to $%
\mathcal{R}_{p}$ by isomorphism.

\begin{lemma}
\label{LemdIp}Let $p,q\geq -1.$

\begin{itemize}
\item[$\left( a\right) $] If $v_{1},v_{2}\in \Lambda _{p}$ and $v_{1}=v_{2}%
\func{mod}I_{p}$ then $\partial v_{1}=\partial v_{2}\func{mod}I_{p-1}$.

\item[$\left( b\right) $] Let $u_{1},u_{2}\in \Lambda _{p}$, $v_{1},v_{2}\in
\Lambda _{q}$. If $u_{1}=u_{2}\func{mod}I_{p}$ and $v_{1}=v_{2}\func{mod}%
I_{q}$ then $u_{1}v_{1}=u_{2}v_{2}\func{mod}I_{p+q+1}.$
\end{itemize}
\end{lemma}

\begin{proof}
$\left( a\right) $ If $p\leq 0$ there is nothing to prove since $%
I_{p}=\left\{ 0\right\} $. In the case $p\geq 1$, it suffices to prove that
if $v=0\func{mod}I_{p}$ then $\partial v=0\func{mod}I_{p-1}$. Since $v$ is a
linear combination of elementary non-regular paths $e_{i_{0}...i_{p}}$, it
suffices to prove that if $e_{i_{0}...i_{p}}$ is non-regular then $\partial
e_{i_{0}...i_{p}}$ is non-regular, too. Indeed, for a non-regular path $%
i_{0}...i_{p}$ there exists an index $k$ such that $i_{k}=i_{k+1}.$ Then we
have%
\begin{eqnarray}
\partial e_{i_{0}...i_{p}} &=&e_{i_{1}...i_{p}}-e_{i_{0}i_{2}...i_{p}}+... 
\notag \\
&&+\left( -1\right) ^{k}e_{i_{0}...i_{k-1}i_{k+1}i_{k+2}...i_{p}}+\left(
-1\right) ^{k+1}e_{i_{0}...i_{k-1}i_{k}i_{k+2}...i_{p}}  \label{e-e} \\
&&+...+\left( -1\right) ^{p}e_{i_{0}...i_{p-1}}.  \notag
\end{eqnarray}%
By $i_{k}=i_{k+1}$ the two terms in the middle line of (\ref{e-e}) cancel
out, whereas all other terms are non-regular, whence $\partial
e_{i_{0}...i_{p}}\in I_{p-1}.$

$\left( b\right) $ Let us first verify that if $u=0\func{mod}I_{p}$ or $v=0%
\func{mod}I_{q}$ then $uv=0\func{mod}I_{p+q+1}$ . Let, for example, $u=0%
\func{mod}I_{p}$. If $p\leq 0$ then this implies $u=0$, and the claim is
trivially satisfied. If $p\geq 1$ then $u$ is a linear combination of
non-regular paths $e_{i_{0}...i_{p}}$. Since the join of an non-regular path
with any path is obviously non-regular, we obtain that $uv$ is non-regular,
which proves the claim.

Since 
\begin{equation*}
u_{1}v_{1}-u_{2}v_{2}=\left( u_{1}-u_{2}\right) v_{1}+u_{2}\left(
v_{1}-v_{2}\right)
\end{equation*}%
and by hypothesis 
\begin{equation*}
u_{1}-u_{2}=0\func{mod}I_{p},\ \ \ v_{1}-v_{2}=0\func{mod}I_{q},
\end{equation*}%
we conclude that%
\begin{equation*}
u_{1}v_{1}=u_{2}v_{2}\func{mod}I_{p+q+1}.
\end{equation*}
\end{proof}

Lemma \ref{LemdIp} shows that the boundary operator $\partial $ and the join
are well-defined on the quotient spaces $\Lambda _{p}/I_{p}=\widetilde{%
\mathcal{R}}_{p}$ through the operations with the representatives of the
classes. In particular, the identity $\partial ^{2}=0$ and the product rule
are satisfied in the spaces $\widetilde{\mathcal{R}}_{p}$.

Now we define the operations $\partial $ and join on the spaces $\mathcal{R}%
_{p}$ simply as pullbacks from $\widetilde{\mathcal{R}}_{p}$ using the
natural linear isomorphism $\mathcal{R}_{p}\rightarrow \widetilde{\mathcal{R}%
}_{p}$ .

\begin{definition}
\RM The operator $\partial $ on $\mathcal{R}_{p}$ will be called a \emph{%
regular} boundary operator, and the join on the spaces $\mathcal{R}_{p}$
will be called a \emph{regular} join. To distinguish them from the
operations $\partial $ and join on the spaces $\Lambda _{p}$, the latter
operations will be referred to as \emph{non-regular}.
\end{definition}

When applying the formulas for the regular boundary operator $\partial $ and
join, one should make the following adjustments:

\begin{enumerate}
\item[$\left( I\right) $] all the components $v^{i_{0}...i_{p}}$ of $v\in 
\mathcal{R}_{p}$ for non-regular paths $i_{0}...i_{p}$ are equal to $0$ by
definition;

\item[$\left( II\right) $] all non-regular paths $e_{i_{0}...i_{p}}$, should
they arise as a result of an operation, are treated as zeros (because after
applying the operation on representatives, we should pass in the end to a
regular representative).
\end{enumerate}

Thus, the formula (\ref{dv}) for the component $\left( \partial v\right)
^{j_{0}...j_{p-1}}$ is valid only for regular paths $j_{0}...j_{p-1}$, while
for non-regular $j_{0}...j_{p-1}$ we have by definition $\left( \partial
v\right) ^{j_{0}...j_{p-1}}=0$. Similarly, the formula (\ref{uvdef}) for $%
\left( uv\right) ^{i_{0}...i_{p}j_{0}...j_{q}}$ is valid only for regular
paths $i_{0}...i_{p}j_{0}...j_{q}$.

On the other hand, the formula (\ref{dev}) for $\partial e_{i_{0}...i_{p}}$
and the formula (\ref{epeq}) for $e_{i_{0}...i_{p}}e_{j_{0}...j_{q}}$ remain
valid for all sequences of indices as it follows from Lemma \ref{LemdIp},
provided one applies adjustment $\left( II\right) $.

For example, we have for the non-regular operator $\partial $%
\begin{equation*}
\partial e_{iji}=e_{ji}-e_{ii}+e_{ij},
\end{equation*}%
whereas for the regular operator $\partial $%
\begin{equation*}
\partial e_{iji}=e_{ji}+e_{ij}
\end{equation*}%
since $e_{ii}$ is non-regular and, hence, is replaced by $0$. For
non-regular join we have%
\begin{equation*}
e_{ij}e_{ji}=e_{ijji}
\end{equation*}%
whereas for the regular join $e_{ij}e_{ji}=0$ since $e_{ijji}$ is
non-regular.

Consequently, we obtain the \emph{regular }chain complex of the set $V$:%
\begin{equation}
0\leftarrow \mathbb{K}\leftarrow \mathcal{R}_{0}\leftarrow ...\leftarrow 
\mathcal{R}_{p-1}\leftarrow \mathcal{R}_{p}\leftarrow ...  \label{cR}
\end{equation}%
where all the arrows are given by regular operator $\partial $.

Let $V^{\prime }$ be a subset of $V$. Clearly, every elementary regular $p$%
-path $e_{i_{0}...i_{p}}$ on $V^{\prime }$ is also a regular $p$-path on $V$%
, so that we have a natural inclusion 
\begin{equation}
\mathcal{R}_{p}\left( V^{\prime }\right) \subset \mathcal{R}_{p}\left(
V\right) .  \label{V'inV}
\end{equation}%
By (\ref{dev}), $\partial e_{i_{0}...i_{p}}$ has the same expression in the
both spaces $\mathcal{R}_{p}\left( V^{\prime }\right) $,$\ \mathcal{R}%
_{p}\left( V\right) $ so that $\partial $ commutes with the inclusion (\ref%
{V'inV}).

\subsection{Form and exterior differential}

\label{Secd}For any integer $p\geq -1$, denote by $\Lambda ^{p}=\Lambda
^{p}\left( V\right) $ the linear space of all $\mathbb{K}$-valued functions
on $V^{p+1}=\underset{p+1\ \text{terms}}{\underbrace{V\times ...\times V}}$.
\ In particular, $\Lambda ^{0}$ is the linear space of all $\mathbb{K}$%
-valued functions on $V$, and $\Lambda ^{-1}$ is the space of all $\mathbb{K}
$-values functions on $V^{0}:=\left\{ 0\right\} $, that is, $\Lambda ^{-1}$
can (and will) be identified with $\mathbb{K}$. Set also $\Lambda
^{-2}=\left\{ 0\right\} .$

\begin{definition}
\RM The elements of $\Lambda ^{p}$ are called $p$-\emph{forms} on $V$.
\end{definition}

The value of a $p$-form $\omega $ at a point $\left(
i_{0},i_{1},...,i_{p}\right) \in V^{p+1}$ will be denoted by $\omega
_{i_{0}i_{1}...i_{p}}$. In particular, the value of a function $\omega \in
\Lambda ^{0}\left( V\right) $ at $i\in V$ will be denoted by $\omega _{i}$.
Each element $\omega \in \Lambda ^{-1}$ is determined by its value at $0$
that will be denoted by the same letter $\omega $.

Denote by $e^{j_{0}...j_{p}}$ a $p$-form that takes value $1_{\mathbb{K}}$
at the point $\left( j_{0},j_{1},...,j_{p}\right) $ and $0$ at all other
points. For example, $e^{j}$ is a function on $V$ that is equal to $1$ at $j$
and $0$ away from $j$. Also, $e$ stands for the function on $\left\{
0\right\} $ taking value $1_{\mathbb{K}}.$ Let us refer to $%
e^{j_{0}...j_{p}} $ as an \emph{elementary} $p$-form. Clearly, the family $%
\left\{ e^{j_{0}...j_{p}}\right\} $ of all elementary $p$-forms is a basis
in the linear space $\Lambda ^{p}$ and, for any $\omega \in \Lambda ^{p},$ 
\begin{equation*}
\omega =\sum_{j_{0},...,j_{p}\in V}\omega _{j_{0}...j_{p}}e^{j_{0}...j_{p}}.
\end{equation*}%
We have a natural pairing of $p$-forms and $p$-paths as follows:%
\begin{equation}
\left( \omega ,v\right) :=\sum_{i_{0},...,i_{p}\in V}\omega
_{i_{0}...i_{p}}v^{i_{0}...i_{p}}  \label{omv}
\end{equation}%
for all $\omega \in \Lambda ^{p}$ and $v\in \Lambda _{p}$. Obviously, the
spaces $\Lambda ^{p}$ and $\Lambda _{p}$ are dual with respect to this
pairing.

\begin{definition}
\RM Define \emph{the exterior differential }$d:\Lambda ^{p}\rightarrow
\Lambda ^{p+1}$ by 
\begin{equation}
\left( d\omega \right) _{i_{0}...i_{p+1}}=\dsum\limits_{q=0}^{p+1}\left(
-1\right) ^{q}\omega _{i_{0}...\widehat{i_{q}}...i_{p+1}},  \label{dom}
\end{equation}%
for any $\omega \in \Lambda ^{p}$.
\end{definition}

For example, for any $\omega \in \Lambda ^{-1}$ we have%
\begin{equation*}
\left( d\omega \right) _{i}=\omega ,
\end{equation*}%
for any function $\omega \in \Lambda ^{0}$ we have%
\begin{equation*}
\left( d\omega \right) _{ij}=\omega _{j}-\omega _{i},
\end{equation*}%
for a $1$-form $\omega $%
\begin{equation*}
\left( d\omega \right) _{ijk}=\omega _{jk}-\omega _{ik}+\omega _{ij}.
\end{equation*}%
It follows from (\ref{dom}) that%
\begin{equation}
de^{i_{0}...i_{p}}=\dsum\limits_{k\in V}\dsum\limits_{q=0}^{p+1}\left(
-1\right) ^{q}e^{i_{0}...i_{q-1}ki_{q}...i_{p}}  \label{de}
\end{equation}%
(cf. (\ref{dv})). For example, we have%
\begin{equation*}
de=\sum_{k}e^{k}=1
\end{equation*}%
where $1$ stands for the function on $V$ with constant value $1_{\mathbb{K}}$%
. Also, we have%
\begin{equation*}
de^{i}=\sum_{k}\left( e^{ki}-e^{ik}\right)
\end{equation*}%
and%
\begin{equation*}
de^{ij}=\sum_{k}\left( e^{kij}-e^{ikj}+e^{ijk}\right) .
\end{equation*}

\begin{lemma}
\label{LemStokes}Let $p\geq -2$. For any $p$-form $\omega $ and any $\left(
p+1\right) $-path $v$ the following identity holds%
\begin{equation*}
\left( d\omega ,v\right) =\left( \omega ,\partial v\right) .
\end{equation*}%
Consequently, the operators $d:\Lambda ^{p}\rightarrow \Lambda ^{p+1}$ and $%
\partial :\Lambda _{p+1}\rightarrow \Lambda _{p}$ are dual.
\end{lemma}

\begin{proof}
For $p=-2$ the both sides are $0$. For $p\geq -1$ it suffices to prove this
identity for $v=e_{i_{0}...i_{p+1}}$. Using (\ref{dom}) and (\ref{dev}), we
obtain 
\begin{equation*}
\left( d\omega ,v\right) =\left( d\omega \right)
_{i_{0}...i_{p+1}}=\dsum\limits_{q=0}^{p+1}\left( -1\right) ^{q}\omega
_{i_{0}...\widehat{i_{q}}...i_{p+1}}
\end{equation*}%
and%
\begin{equation*}
\left( \omega ,\partial v\right) =\left( \omega
,\dsum\limits_{q=0}^{p+1}\left( -1\right) ^{q}e_{i_{0}...\widehat{i_{q}}%
...i_{p+1}}\right) =\dsum\limits_{q=0}^{p+1}\left( -1\right) ^{q}\omega
_{i_{0}...\widehat{i_{q}}...i_{p+1}},
\end{equation*}%
whence the required identity follows.
\end{proof}

\begin{corollary}
We have $d^{2}=0.$
\end{corollary}

Hence, we obtain a cochain complex%
\begin{equation}
0\rightarrow \mathbb{K}\rightarrow \Lambda _{0}\rightarrow ...\rightarrow
\Lambda _{p}\rightarrow \Lambda _{p+1}\rightarrow ...  \label{coLa}
\end{equation}%
where all arrows are given by $d$, and this cochain complex is dual to the
chain complex (\ref{cLa}).

\subsection{Concatenation of forms}

\label{SecCon}

\begin{definition}
\RM For $p,q\geq 0$ and for any two forms $\varphi \in \Lambda ^{p}$ and $%
\psi \in \Lambda ^{q}$, define their \emph{concatenation} $\varphi \psi \in
\Lambda ^{p+q}$ by%
\begin{equation}
\left( \varphi \psi \right) _{i_{0}...i_{p+q}}=\varphi _{i_{0}...i_{p}}\psi
_{i_{p}i_{p+1}...i_{p+q}}.  \label{concat}
\end{equation}
\end{definition}

Clearly, concatenation is a bilinear operation that satisfies the
associative law. It is obviously non-commutative. For example, if $\varphi $
is a function, that is, $p=0$, then $\varphi \psi \in \Lambda ^{q}$ and%
\begin{equation*}
\left( \varphi \psi \right) _{i_{0}...i_{q}}=\varphi _{i_{0}}\psi
_{i_{0}...i_{q}},
\end{equation*}%
while%
\begin{equation*}
\left( \psi \varphi \right) _{i_{0}...i_{q}}=\psi _{i_{0}...i_{q}}\varphi
_{i_{q}}.
\end{equation*}%
For the elementary forms $e^{i_{0}...i_{p}}$ and $e^{j_{0}...j_{q}}$ we have%
\begin{equation}
e^{i_{0}...i_{p}}e^{j_{0}...j_{q}}=\left\{ 
\begin{array}{ll}
0, & i_{p}\neq j_{0}, \\ 
e^{i_{0}...i_{p}j_{1}...i_{q}}, & i_{p}=j_{0}.%
\end{array}%
\right.  \label{econ}
\end{equation}%
The operation of concatenation is reminiscent of the operation of cup
product in algebraic topology. Let us emphasize that concatenation of forms
is essentially different from join of paths, which can be seen from
comparison of (\ref{uvdef}) and (\ref{concat}): in the former the index $%
i_{p}$ is used twice in the right hand side whereas in the latter -- only
once. Consequently, concatenation acts from $\Lambda ^{p}\times \Lambda ^{q}$
to $\Lambda ^{p+q}$, whereas join acts from $\Lambda _{p}\times \Lambda _{q}$
to $\Lambda _{p+q+1}$. Despite the differences, the both operations do
satisfy the product rules with respect to the operators $d$ and $\partial $,
respectively.

\begin{lemma}
\label{Lemdff}For all $p,q\geq 0$ and $\varphi \in \Lambda ^{p}$, $\psi \in
\Lambda ^{q}$, we have%
\begin{equation}
d\left( \varphi \psi \right) =\left( d\varphi \right) \psi +\left( -1\right)
^{p}\varphi d\psi .  \label{Leib}
\end{equation}
\end{lemma}

\begin{proof}
Denoting $\omega =\varphi \psi ,$ we have%
\begin{eqnarray*}
\left( d\omega \right) _{i_{0}...i_{p+q+1}}
&=&\dsum\limits_{r=0}^{p+q+1}\left( -1\right) ^{r}\omega _{i_{0}...\widehat{%
i_{r}}...i_{p+q+1}} \\
&=&\dsum\limits_{r=0}^{p}\left( -1\right) ^{r}\omega _{i_{0}...\widehat{i_{r}%
}...i_{p+1}...i_{p+q+1}}+\dsum\limits_{r=p+1}^{p+q+1}\left( -1\right)
^{r}\omega _{i_{0}...i_{p}...\widehat{i_{r}}...i_{p+q+1}} \\
&=&\dsum\limits_{r=0}^{p}\left( -1\right) ^{r}\varphi _{i_{0}...\widehat{%
i_{r}}...i_{p+1}}\psi
_{i_{p+1}...i_{p+q+1}}+\dsum\limits_{r=p+1}^{p+q+1}\left( -1\right)
^{r}\varphi _{i_{0}...i_{p}}\psi _{i_{p}...\widehat{i_{r}}...i_{p+q+1}}.
\end{eqnarray*}%
Noticing that%
\begin{equation*}
\left( d\varphi \right) _{i_{0}...i_{p+1}}=\dsum\limits_{r=0}^{p+1}\left(
-1\right) ^{r}\varphi _{i_{0}...\widehat{i_{r}}...i_{p+1}}
\end{equation*}%
and 
\begin{equation*}
\left( d\psi \right) _{i_{p}...i_{p+q+1}}=\dsum\limits_{r=p}^{p+q+1}\left(
-1\right) ^{r-p}\psi _{i_{p}...\widehat{i_{r}}...i_{p+q+1}},
\end{equation*}%
we obtain 
\begin{eqnarray*}
\left( d\omega \right) _{i_{0}...i_{p+q+1}} &=&\left[ \left( d\varphi
\right) _{i_{0}...i_{p+1}}-\left( -1\right) ^{p+1}\varphi _{i_{0}...i_{p}}%
\right] \psi _{i_{p+1}...i_{p+q+1}} \\
&&+\left( -1\right) ^{p}\varphi _{i_{0}...i_{p}}\left[ \left( d\psi \right)
_{i_{p}...i_{p+q+1}}-\psi _{i_{p+1}...i_{p+q+1}}\right] \\
&=&\left( \left( d\varphi \right) \psi \right) _{i_{0}...i_{p+q+1}}+\left(
-1\right) ^{p}\left( \varphi d\psi \right) _{i_{0}...i_{p+q+1}}
\end{eqnarray*}%
which was to be proved.
\end{proof}

\begin{remark}
\RM\label{RemDGAL}The direct sum of vector spaces 
\begin{equation*}
\Lambda =\bigoplus_{p\geq 0}\Lambda ^{p}
\end{equation*}%
with the additional operation concatenation is a graded algebra over field $%
\mathbb{K}$. It is easy to see that the constant function $1$ on $V$ is a
unity of this algebra. As it follows from Lemma \ref{Lemdff}, the couple $%
\left( \Lambda ,d\right) $ is a \emph{differential} graded algebra.

Note that $\left( \Lambda ,d\right) $ does not satisfy the \emph{minimality}
condition. The latter condition says that the minimal left $\Lambda ^{0}$%
-module generated by $d\Lambda ^{p}$ must coincide with $\Lambda ^{p+1}$,
which is not the case here. Indeed, each element of the left $\Lambda ^{0}$%
-module generated by $d\Lambda ^{0}$ is a finite some of the terms like $fdg$
where $f,g\in \Lambda ^{0}.$ For each $i\in V$, we have%
\begin{equation*}
\left( fdg\right) _{ii}=f_{i}\left( dg\right) _{ii}=f_{i}\left(
g_{i}-g_{i}\right) =0.
\end{equation*}%
Hence, the sum of such terms cannot be equal to $e^{ii}\in \Lambda ^{1}.$ In
Section \ref{SecRegforms} we will consider the spaces of regularized forms
that do satisfy the minimality condition.
\end{remark}

\subsection{Regular forms}

\label{SecRegforms}For any integer $p\geq -2$, consider the following
subspace of $\Lambda ^{p}$: 
\begin{eqnarray*}
\mathcal{R}^{p} &=&\mathcal{R}^{p}\left( V\right) =\limfunc{span}\left\{
e^{i_{0}...i_{p}}:i_{0}...i_{p}\text{ is regular}\right\} \\
&=&\left\{ \omega \in \Lambda ^{p}:\omega _{i_{0}...i_{p}}=0\text{\ if }%
i_{0}...i_{p}\text{ is non-regular}\right\} .
\end{eqnarray*}

\begin{definition}
\RM The elements of $\mathcal{R}^{p}$ are called \emph{regular} $p$-forms.
\end{definition}

For example, a $1$-form $\omega $ belongs to $\mathcal{R}^{1}$ if $\omega
_{ii}\equiv 0$, and a $2$-form $\omega $ belongs to $\mathcal{R}^{2}$ if $%
\omega _{iij}\equiv \omega _{jii}\equiv 0.$ For $p\leq 0$ the condition $%
f\in \mathcal{R}^{p}$ has no additional restriction so that $\mathcal{R}%
^{p}=\Lambda ^{p}$.

The next lemma shows that the operations of exterior differentiation,
concatenation and pairing can be restricted to regular forms.

\begin{lemma}
\label{LemE0}

\begin{itemize}
\item[$\left( a\right) $] If $\omega \in \mathcal{R}^{p}$ then $d\omega \in 
\mathcal{R}^{p+1}$.

\item[$\left( b\right) $] If $\varphi \in \mathcal{R}^{p}$ and $\psi \in 
\mathcal{R}^{q}$ then $\varphi \psi \in \mathcal{R}^{p+q}$, assuming that $%
p,q\geq 0$.

\item[$\left( c\right) $] If $\omega \in \mathcal{R}^{p}$, $v_{1},v_{2}\in
\Lambda _{p}$ and $v_{1}=v_{2}\func{mod}I_{p}$ then $\left( \omega
,v_{1}\right) =\left( \omega ,v_{2}\right) $.
\end{itemize}
\end{lemma}

\begin{proof}
$\left( a\right) $ To prove that $d\omega \in \mathcal{R}^{p+1}$, we must
show that 
\begin{equation}
\left( d\omega \right) _{i_{0}...i_{p+1}}=0  \label{dw=0i}
\end{equation}%
whenever $i_{0}...i_{p+1}$ is non-regular, say $i_{k}=i_{k+1}.$ We have by (%
\ref{dom})%
\begin{equation*}
\left( d\omega \right) _{i_{0}...i_{p+1}}=\dsum\limits_{q=0}^{p+1}\left(
-1\right) ^{q}\omega _{i_{0}...\widehat{i_{q}}...i_{p+1}}.
\end{equation*}%
If $q\neq k,k+1$ then both $i_{k},i_{k+1}$ are present in $\omega _{i_{0}...%
\widehat{i_{q}}...i_{p+1}}$ which makes this term equal to $0$ since $\omega 
$ is regular. In the remaining two cases $q=k$ and $q=k+1$ the term $\omega
_{i_{0}...\widehat{i_{q}}...i_{p+1}}$ has the same values (because the
sequences $i_{0}...\widehat{i_{q}}...i_{p+1}$ are the same) but the signs $%
\left( -1\right) ^{q}$ are opposite. Hence, they cancel out, which proves (%
\ref{dw=0i}).

$\left( b\right) $ By (\ref{concat}), we have%
\begin{equation*}
\left( \varphi \psi \right) _{i_{0}...i_{p+q}}=\varphi _{i_{0}...i_{p}}\psi
_{i_{p}...i_{p+q}}.
\end{equation*}%
If the sequence $i_{0}...i_{p+q}$ is non-regular, say $i_{k}=i_{k+1}$ then
the both indices $i_{k},i_{k+1}$ are present either in the sequence $%
i_{0}...i_{p}$ or in $i_{p}...i_{p+q}$, which implies that one of the terms $%
\varphi _{i_{0}...i_{p}}$,$\ \psi _{i_{p}...i_{p+q}}$ vanishes. It follows
that $\left( \varphi \psi \right) _{i_{0}...i_{p+q}}=0$ and, hence, $\varphi
\psi \in \mathcal{R}^{p+q}.$

$\left( c\right) $ Indeed, $v_{1}-v_{2}\in I_{p}$ is a linear combination of
non-regular paths $e_{i_{0}...i_{p}}.$ Since $\left( \omega
,e_{i_{0}...i_{p}}\right) =0$ for non-regular paths, it follows that $\left(
\omega ,v_{1}-v_{2}\right) =0$ and $\left( \omega ,v_{1}\right) =\left(
\omega ,v_{2}\right) .$
\end{proof}

It follows from \ref{LemE0}$\left( c\right) $ that the pairing $\left(
\omega ,v\right) $ is well defined for $\omega \in \mathcal{R}^{p}$ and $%
v\in \widetilde{\mathcal{R}}_{p}$. In particular, the spaces $\mathcal{R}%
^{p} $ and $\widetilde{\mathcal{R}}_{p}$ are dual. It follows from Lemmas %
\ref{LemdIp} and \ref{LemStokes} that, for all $\omega \in \mathcal{R}^{p}$
and $v\in \widetilde{\mathcal{R}}_{p+1}$,%
\begin{equation}
\left( d\omega ,v\right) =\left( \omega ,\partial v\right) .  \label{Stokes}
\end{equation}%
In particular, the operators $d:\mathcal{R}^{p}\rightarrow \mathcal{R}^{p+1}$
and $\partial :\widetilde{\mathcal{R}}_{p+1}\rightarrow \widetilde{\mathcal{R%
}}_{p}$ are dual. Replacing the regularized paths by their regular
representatives, we obtain that the spaces $\mathcal{R}^{p}$ and $\mathcal{R}%
_{p}$ are dual (which is obvious directly from their definitions, though)
and that the operator $d:\mathcal{R}^{p}\rightarrow \mathcal{R}^{p+1}$ is
dual to the regular operator $\partial :\mathcal{R}_{p+1}\rightarrow 
\mathcal{R}_{p}.$ In particular, we obtain the \emph{regular} cochain
complex of $V$%
\begin{equation*}
0\rightarrow \mathbb{K}\rightarrow \mathcal{R}^{0}\rightarrow ...\rightarrow 
\mathcal{R}^{p}\rightarrow \mathcal{R}^{p+1}\rightarrow ...
\end{equation*}%
that is dual to the regular chain complex (\ref{cR}).

\begin{remark}
\RM\label{RemDGAR}Similarly to Remark \ref{RemDGAL}, the direct sum of
vector spaces%
\begin{equation*}
\mathcal{R}=\bigoplus_{p\geq 0}\mathcal{R}^{p}
\end{equation*}%
with the operations $d$ and concatenation is a graded differential algebra
over $\mathbb{K}$. This algebra does satisfy the minimality condition: the
minimal left $\mathcal{R}^{0}$-module generated by $d\mathcal{R}^{p}$
coincides with $\mathcal{R}^{p+1}$; that is, any element of $\mathcal{R}%
^{p+1}$ is a finite sum of the terms like $fd\omega $ with $f\in \mathcal{R}%
^{0}$ and $\omega \in \mathcal{R}^{p}$.

Recall that there is a standard procedure of construction of the \emph{%
universal }differential graded algebra starting with any associative unital
algebra. If one starts with the algebra of all $\mathbb{K}$-valued functions
on $V$ (that is $\mathcal{R}^{0}$), then one obtains in this way exactly $%
\left( \mathcal{R},d\right) .$ The universality of $\left( \mathcal{R}%
,d\right) $ means that any minimal differential graded algebra over $%
\mathcal{R}^{0}$ is a certain quotient of $\left( \mathcal{R},d\right) $.
The details can be found in \cite{DimMuGauge} and \cite{GrigMuralg}.
\end{remark}

\section{Path complexes}

\label{Sec4}\setcounter{equation}{0}

\subsection{Path complexes, simplicial complexes, digraphs}

\label{Secpath}

\begin{definition}
\RM A \emph{path complex} over a set $V$ is a non-empty collection $P$ of
elementary paths on $V$ with the following property: for any $n\geq 0$, 
\begin{equation}
\text{if\ \ }i_{0}...i_{n}\in P\ \text{then\ also\ the truncated\ paths\ }%
i_{0}...i_{n-1}\ \text{and\ }i_{1}...i_{n}\ \text{belong to }P.  \label{t}
\end{equation}
\end{definition}

The set of $n$-paths from $P$ is denoted by $P_{n}$. Then a path complex $P$
can be regarded as a collection $\left\{ P_{n}\right\} _{n=-1}^{\infty }$
satisfying (\ref{t}). When a path complex $P$ is fixed, all the paths from $%
P $ are called \emph{allowed}, whereas all the elementary paths that are not
in $P$ are called \emph{non-allowed}.

The set $P_{-1}$ consists of a single empty path $e$. The elements of $P_{0}$
(that is, allowed $0$-paths) are called the \emph{vertices} of $P$. Clearly, 
$P_{0}$ is a subset of $V$. By the property (\ref{t}), if $i_{0}...i_{n}\in
P $ then all $i_{k}$ are vertices. Hence, we can (and will) remove from the
set $V$ all non-vertices so that $V=P_{0}.$ The elements of $P_{1}$ (that
is, allowed $1$-paths) are called (directed) \emph{edges} of $P$. By (\ref{t}%
), if $i_{0}...i_{n}\in P$ then all $\,1$-paths $i_{k-1}i_{k}$ are edges.

\begin{example}
\RM By definition, an abstract finite simplicial complex $S$ is a collection
of subsets of a finite vertex set $V$ that satisfies the following property: 
\begin{equation*}
\text{if\ }\sigma \in S\ \text{then any subset of }\sigma \ \text{is also in 
}S.
\end{equation*}%
Let us enumerate the elements of $V$ by distinct reals and identify any
subset $s$ of $V$ with the elementary path that consists of the elements of $%
s$ put in the (strictly) increasing order. Hence, we can regard $S$ as a
collection of elementary paths on $V$. Then the defining property of a
simplex can be restated the following: 
\begin{equation}
\text{if an elementary path belongs to }S\text{ then its any subsequence
also belongs to }S\text{.}  \label{s}
\end{equation}%
Consequently, the family $S$ satisfies the property (\ref{t}) so that $S$ is
a path complex. The allowed $n$-paths in $S$ are exactly the $n$-simplexes.

For example, a simplicial complex on Fig. \ref{pic32} has the following
allowed paths (=simplexes):

\begin{enumerate}
\item[$0$-paths:] $0,...,8$

\item[$1$-paths:] $01,02,03,04,05,06,07,08,12,34,35,45,67,68,78$

\item[$2$-paths:] $012,678,034,035,045,678$

\item[$3$-paths:] $0345$ \FRAME{ftbhFU}{6.365in}{1.996in}{0pt}{\Qcb{A
simplicial complex}}{\Qlb{pic32}}{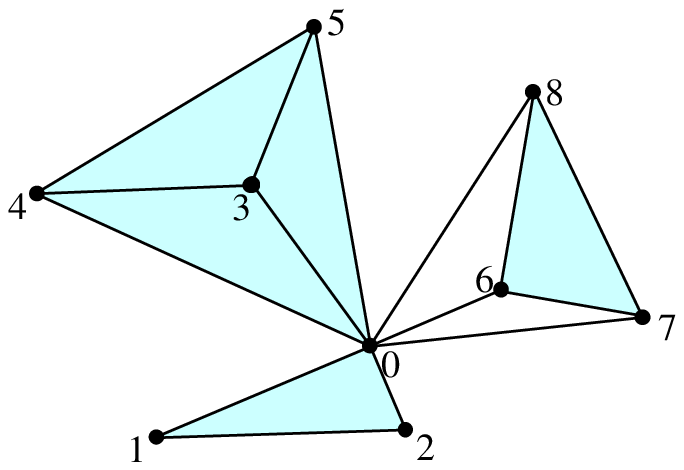}{\special{language "Scientific
Word";type "GRAPHIC";maintain-aspect-ratio TRUE;display "USEDEF";valid_file
"F";width 6.365in;height 1.996in;depth 0pt;original-width
6.3027in;original-height 1.9571in;cropleft "0";croptop "1";cropright
"1";cropbottom "0";filename 'pic32.eps';file-properties "XNPEU";}}
\end{enumerate}
\end{example}

\begin{example}
\RM Let $G=\left( V,E\right) $ be a finite digraph, where $V$ is a finite
set of vertices and $E$ is the set of directed edges, that is, $E\subset
V\times V$. Equivalently, one can say that a digraph is a set $V$ endowed
with a binary relation $E$. The fact that $\left( i,j\right) \in E$ will
also be denoted by $i\rightarrow j$. \ 

An elementary $n$-path $i_{0}...i_{n}$ on $V$ is called allowed if $%
i_{k-1}\rightarrow i_{k}$ for any $k=1,...,n$. Denote by $P_{n}=P_{n}\left(
G\right) $ the set of all allowed $n$-paths. In particular, we have $P_{0}=V$
and $P_{1}=E$. Clearly, the family $\left\{ P_{n}\right\} $ of all allowed
paths satisfies the condition (\ref{t}) so that $\left\{ P_{n}\right\} $ is
a path complex. This path complex is naturally associated with the digraph $%
G $ and will be denoted by $P\left( G\right) $.

For example, a digraph on Fig. \ref{pic33} has the following path complex:

\begin{enumerate}
\item[$0$-paths:] $0,...,8$

\item[$1$-paths:] $01,02,03,04,05,06,07,08,12,34,35,45,67,68,78$

\item[$2$-paths:] $012,678,034,035,045,067,068,678$

\item[$3$-paths:] $0345,0678$\FRAME{ftbhFU}{6.365in}{1.996in}{0pt}{\Qcb{A
digraph}}{\Qlb{pic33}}{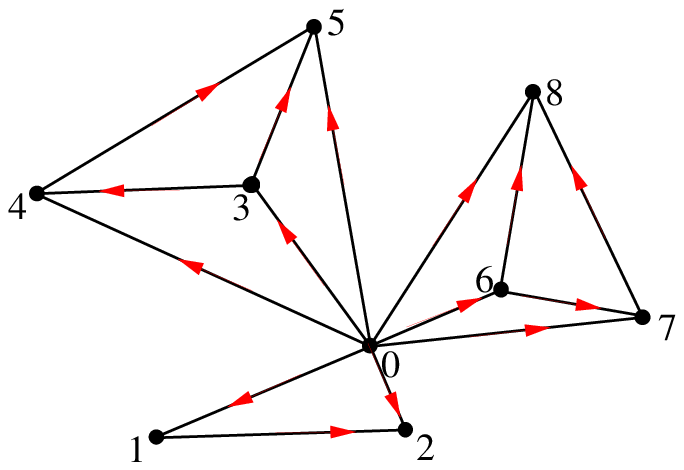}{\special{language "Scientific Word";type
"GRAPHIC";maintain-aspect-ratio TRUE;display "USEDEF";valid_file "F";width
6.365in;height 1.996in;depth 0pt;original-width 6.3027in;original-height
1.9571in;cropleft "0";croptop "1";cropright "1";cropbottom "0";filename
'pic33.eps';file-properties "XNPEU";}}
\end{enumerate}
\end{example}

The path complexes of digraphs are the central objects of this paper.
Although most of the results are proved for arbitrary path complexes, we
always have in mind possible applications to digraphs. On the other hand,
the notion of a path complex provides an alternative approach to the
classical results about simplicial complexes.

It is easy to see that a path complex arises from a digraph if and only if
it satisfies the following additional condition: if in a path $i_{0}...i_{n}$
all pairs $i_{k-1}i_{k}$ are allowed then the whole path $i_{0}...i_{n}$ is
allowed.

Let us describe explicitly those path complexes that arise from simplicial
complexes.

\begin{definition}
\RM We say that a path complex $P$ is \emph{perfect}, if any subsequence of
any allowed elementary path of $P$ is also an allowed path.
\end{definition}

\begin{definition}
\RM We say that a path complex $P$ is \emph{monotone}, if there is an
injective real-valued function on the vertex set of $P$ that is strictly
monotone increasing along any path from $P$.
\end{definition}

\begin{proposition}
\label{LemPM}\label{PropPM}A path complex $P$ is the path complex of a
simplicial complex if and only if it is perfect and monotone.
\end{proposition}

\begin{proof}
The path complex of a simplicial complex is both perfect and monotone by
definition. Let us prove the converse. By the monotonicity condition, the
vertices in any path $i_{0}...i_{n}\in P$ are all distinct. Hence, with any
path $i_{0}...i_{n}\in P$ we can associate a simplex $[i_{0},...,i_{n}]$;
denote by $S$ the collection of all such simplexes. Then the perfectness of $%
P$ implies that $S$ is a simplicial complex. Ordering the vertex set of $S$
using the monotone function from the monotonicity condition, we see that
each simplex from $S$ gives back a path from $P$.
\end{proof}

Observe that the path complex of a digraph $G=(V,E)$ is perfect if and only
if the edge relation $\rightarrow $ is transitive, that is, if

\begin{equation}
\qquad i\rightarrow j\rightarrow k\Rightarrow i\rightarrow k.  \label{ijk}
\end{equation}%
In particular, this condition holds for \emph{posets} (=partially ordered
sets). Indeed, by definition a poset is a digraph $G$ where the edge
relation \U{2192} is reflexive, antisymmetric and transitive. Hence, the
path complex of a poset is perfect, but satisfies also the following
additional properties: all $1$-paths $ii$ are allowed, while 2-paths $iji$
with $i\neq j$ are non-allowed.

It is easy to see that the path complex of a digraph $G$ is monotone if
there is a function $\Phi :V\rightarrow \mathbb{R}$ such that%
\begin{equation*}
i\rightarrow j\Rightarrow \Phi \left( i\right) <\Phi \left( j\right) .
\end{equation*}%
For example, the path complex of the digraph on Fig. \ref{pic33} is both
monotone and perfect.

The path complex of a poset is not monotone as it has allowed $1$-paths $ii$%
. However, if we reduce the set of edges on a poset by removing all loops $%
i\rightarrow i$, then the resulting digraph is perfect and monotone.

\subsection{Allowed paths}

Given an arbitrary path complex $P=\left\{ P_{n}\right\} _{n=0}^{\infty }$
with a finite vertex set $V$, consider for any integer $n\geq -1$ the $%
\mathbb{K}$-linear space $\mathcal{A}_{n}$ that is spanned by all the
elementary $n$-paths from $P$, that is%
\begin{equation*}
\mathcal{A}_{n}=\mathcal{A}_{n}\left( P\right) =\left\{
\sum_{i_{0},...,i_{n}\in
V}v^{i_{0}...i_{n}}e_{i_{0}...i_{n}}:i_{0}...i_{n}\in P_{n},\
v^{i_{0}...i_{n}}\in \mathbb{K}\right\} .
\end{equation*}%
By construction, $\mathcal{A}_{n}$ is a subspace of the space $\Lambda _{n}$
defined in Section \ref{Secpchain}. For example, $\mathcal{A}_{0}$ is
spanned by all vertexes of $P$ so that $\mathcal{A}_{0}=\Lambda _{0}$. The
space $\mathcal{A}_{1}$ is spanned by all edges of $P$ and can be smaller
than $\Lambda _{1}.$ It is clear that $\mathcal{A}_{-1}\cong \mathbb{K}$.
Set also $\mathcal{A}_{-2}=\left\{ 0\right\} .$

\begin{definition}
\RM The elements of $\mathcal{A}_{n}$ are called \emph{allowed} $n$-paths.
\end{definition}

We would like to restrict the boundary operator $\partial $ on the spaces $%
\Lambda _{n}$ to the spaces $\mathcal{A}_{n}$. For some path complexes it
can happen that 
\begin{equation}
\partial \mathcal{A}_{n}\subset \mathcal{A}_{n-1},  \label{AA}
\end{equation}%
so that the restriction is straightforward. If it is not the case then an
additional construction is needed as will be explained below. Let us
describe first the setting when the inclusion (\ref{AA}) takes place.
Namely, let us show that for a perfect path complex the inclusion (\ref{AA})
takes places. For $n\leq 1$ this is obvious as $\mathcal{A}_{n-1}=\Lambda
_{n-1}.$ Assuming $n\geq 2$, let us show that if $e_{i_{0}...i_{n}}$ is
allowed then $\partial e_{i_{0}...i_{n}}$ is also allowed. Indeed, we have
by (\ref{dev})%
\begin{equation}
\partial e_{i_{0}...i_{n}}=\sum_{q=0}^{n}\left( -1\right) ^{q}e_{i_{0}...%
\widehat{i_{q}}...i_{n}},  \label{dei}
\end{equation}%
where all the terms $e_{i_{0}...\widehat{i_{q}}...i_{n}}$ in the right hand
side are allowed because of the perfectness. Hence, (\ref{AA}) follows.
Consequently, we obtain a chain complex%
\begin{equation}
0\leftarrow \mathbb{K}\leftarrow \mathcal{A}_{0}\leftarrow ...\leftarrow 
\mathcal{A}_{n-1}\leftarrow \mathcal{A}_{n}\leftarrow ...  \label{cAe}
\end{equation}%
Its homology groups are denoted by $\widetilde{H}_{\bullet }\left( P\right) $
are referred to as the \emph{reduced path homologies} of $P$. Consider also
the truncated complex 
\begin{equation}
0\leftarrow \mathcal{A}_{0}\leftarrow ...\leftarrow \mathcal{A}%
_{n-1}\leftarrow \mathcal{A}_{n}\leftarrow ...  \label{cA}
\end{equation}%
whose homology groups are denoted by $H_{\bullet }\left( P\right) $ and are
referred to as the \emph{path homologies} of $P$ (in the latter case the
operator $\partial $ on $\mathcal{A}_{0}$ is set to be zero).

If $P$ is the path complex of a simplicial complex $S$, then the boundary
operator (\ref{dei}) matches the classical boundary operator on simplexes:%
\begin{equation*}
\partial \left[ i_{0},...,i_{n}\right] =\sum_{q=0}^{n}\left( -1\right)
^{q}[i_{0},...,\widehat{i_{q}},...i_{n}].
\end{equation*}%
In this case, (\ref{cA}) coincides with the classical chain complex of a
simplicial complex, and the path homologies $H_{\bullet }\left( P\right) $
are identical to the simplicial homologies $H_{\bullet }\left( S\right) $.

If $P$ is a path complex of a digraph $G$, satisfying the transitivity
condition (\ref{ijk}) then $P$ is perfect and, hence, its homology groups
are defined as above. In this case we denote them also by $H_{\bullet
}\left( G\right) .$

\begin{example}
\RM Let $S$ be a finite simplicial complex. Consider the digraph $G$ whose
set of vertices is $S$, while the edges are defined as follows: if $s,t$ are
simplexes from $\ S$, then 
\begin{equation*}
s\rightarrow t\Leftrightarrow s\supset t\ \text{and\ \ }s\neq t.
\end{equation*}%
Clearly, the graph $G$ satisfies the transitivity condition (\ref{ijk}) so
that the path complex $P=P\left( G\right) $ of the digraph $G$ is perfect.
Let us prove that $H_{n}\left( G\right) \cong H_{n}\left( S\right) $ for all 
$n\geq 0.$

Let us first show that the path complex $P$ is monotone. For that, enumerate
the vertices of $S$ by numbers $1,2,2^{2},2^{3},...$ and assign to each
simplicial complex $s\in S$ (that is, to a vertex of $G$) the sum of the
numbers of all its vertices. The resulting function on $G$ is injective and
strictly monotone decreasing along each edge and, hence, along any allowed
path. By Lemma \ref{LemPM} the path complex $P$ arises from a simplicial
complex.

It is easy to see that the latter simplicial complex is nothing else but the
barycentric subdivision $B\left( S\right) $ of $S$ (cf. Fig. \ref{pic34}).%
\FRAME{ftbphFU}{5.7891in}{1.8014in}{0pt}{\Qcb{Simplicial complex $S$ and the
digraph $G$}}{\Qlb{pic34}}{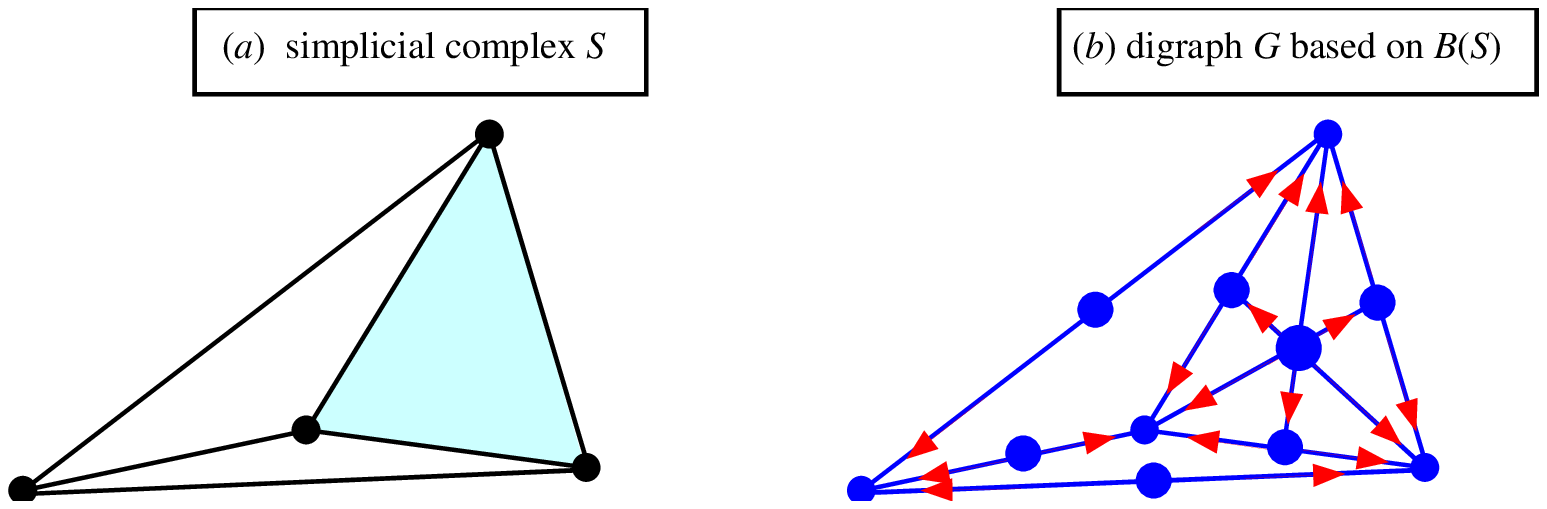}{\special{language "Scientific
Word";type "GRAPHIC";maintain-aspect-ratio TRUE;display "USEDEF";valid_file
"F";width 5.7891in;height 1.8014in;depth 0pt;original-width
6.4013in;original-height 1.97in;cropleft "0";croptop "1";cropright
"1";cropbottom "0";filename 'pic34.eps';file-properties "XNPEU";}}

Indeed, any allowed path $s_{0}...s_{n}$ on the digraph $G$ consists of a
sequence of simplexes from $S$ such that $s_{k}$ is a face of $s_{k-1}$. If
each simplex $s_{k}$ in this path is replaced by its barycenter $b\left(
s_{k}\right) $ (assuming that the simplicial complex $S$ is geometrically
realized in a higher dimensional space $\mathbb{R}^{N}$) then the sequence $%
\left\{ b\left( s_{k}\right) \right\} _{k=0}^{n}$ forms a $n$-simplex of the
barycentric subdivision $B\left( S\right) $ of $S$. Converse is obviously
also true. Hence, the path complex $P$ of the digraph $G$ coincides with the
path complex of the simplicial complex $B\left( S\right) .$ It follows that%
\begin{equation*}
H_{\bullet }\left( G\right) =H_{\bullet }\left( P\right) =H_{\bullet }\left(
B\left( S\right) \right) .
\end{equation*}%
The proof is finished by citing a classical result that $H_{\bullet }\left(
B\left( S\right) \right) \cong H_{\bullet }\left( S\right) .$
\end{example}

\subsection{$\partial $-invariant paths}

\label{SecOmega}Now consider a general case when $\partial \mathcal{A}_{n}$
does not have to be a subspace of $\mathcal{A}_{n-1}.$ Consider first a
simple example.

\begin{example}
\RM Consider the digraph $G=\left( V,E\right) $ as on the diagram%
\begin{equation*}
\begin{array}{c}
_{_{\nearrow }}\overset{1}{\bullet }_{_{\searrow }} \\ 
^{0}\bullet \ \ \text{\ \ \ \ \ }\ \bullet ^{2}%
\end{array}%
\end{equation*}%
that is $V=\left\{ 0,1,2\right\} $ and $E=\left\{ 01,12\right\} .$ Then the $%
2$-path $e_{012}$ is allowed, while%
\begin{equation*}
\partial e_{012}=e_{12}-e_{02}+e_{01}
\end{equation*}%
is non-allowed because $e_{02}$ is non-allowed.
\end{example}

For any $n\geq -1$ consider the following subspaces of $\mathcal{A}_{n}$:%
\begin{equation}
\Omega _{n}=\Omega _{n}\left( P\right) =\left\{ v\in \mathcal{A}%
_{n}:\partial v\in \mathcal{A}_{n-1}\right\} .  \label{Omdef}
\end{equation}%
For example, we have:

\begin{itemize}
\item $\Omega _{-1}=\mathcal{A}_{-1}=\mathbb{K}$;

\item $\Omega _{0}=\mathcal{A}_{0}$ is the space of linear combinations of
all the vertices of $P$;

\item $\Omega _{1}=\mathcal{A}_{1}$ is the space of linear combinations of
all the edges of $P$ (indeed, $\partial e_{ij}=e_{j}-e_{i}$ is always in $%
\mathcal{A}_{0}$).
\end{itemize}

The spaces $\Omega _{n}$ with $n\geq 2$ can actually be smaller than $%
\mathcal{A}_{n}$ as will be seen from many examples in the subsequent
sections. We claim that always%
\begin{equation*}
\partial \Omega _{n}\subset \Omega _{n-1}.
\end{equation*}%
Indeed, if $v\in \Omega _{n}$ then $\partial v\in \mathcal{A}_{n-1}$ and $%
\partial \left( \partial v\right) =0\in \mathcal{A}_{n-2}$ whence it follows
that $\partial v\in \Omega _{n-1}$, which was to be proved.

\begin{definition}
\RM The elements of $\Omega _{n}$ are called $\partial $-\emph{invariant} $n$%
-paths.
\end{definition}

Thus, we obtain the chain complex of $\partial $-invariant paths:%
\begin{equation}
0\leftarrow \mathbb{K}\leftarrow \Omega _{0}\leftarrow ...\leftarrow \Omega
_{n-1}\leftarrow \Omega _{n}\leftarrow \Omega _{n+1}\leftarrow ...\ 
\label{cOmr}
\end{equation}%
where all arrows are given by $\partial $. We consider also its \emph{%
truncated} version%
\begin{equation}
0\leftarrow \Omega _{0}\leftarrow ...\leftarrow \Omega _{n-1}\leftarrow
\Omega _{n}\leftarrow \Omega _{n+1}\leftarrow ...  \label{cOm}
\end{equation}%
where the definition of the boundary operator $\partial $ on $\Omega _{0}$
is modified by setting $\partial \equiv 0.$ We refer to this modification of 
$\partial $ as a \emph{truncated} boundary operator. Note that this
modification does not affect $\partial $ on $\Omega _{n}$ with $n\geq 1.$

\label{Secreg}There is a different kind of modification of the above
procedure as follows.

\begin{definition}
\RM A path complex $P$ is called \emph{regular} if it contains no $1$-path
of the form $ii.$
\end{definition}

Equivalently, $P$ is regular if all the paths $i_{0}...i_{n}\in P$ are
regular\footnote{%
Recall that a path $i_{0}...i_{n}$ is called regular if $i_{k-1}\neq i_{k}$
for all $k=1,...,n.$%
\par
{}}. For example, the path complex of a simplicial complex is always regular
as all the vertices in any allowed elementary path are distinct. The path
complex of a digraph is regular if and only if the digraph is loopless, that
is, if the $1$-paths $ii$ are not edges.

For a regular path complex the above construction of the spaces $\Omega _{n}$
allows the following variation. As the space $\mathcal{A}_{n}$ of allowed $n$%
-path is in this case a subspace of the space $\mathcal{R}_{n}$ of regular $%
n $-paths, we can replace in (\ref{Omdef}) a non-regular boundary operator $%
\partial $ on $\Lambda _{n}$ by a regular boundary operator on $\mathcal{R}%
_{n}$ as described in Section \ref{SecRegpath}. The resulting space $\Omega
_{n}$ will be referred to as a \emph{regular} space of $\partial $-invariant
paths. Hence, if the path complex $P$ is regular then we can consider also
regular versions of the chain complexes (\ref{cOmr}) and (\ref{cOm}).

The both chain complexes (\ref{cOmr}) and (\ref{cOm}) (regular and
non-regular versions) are denoted shortly by $\Omega _{\bullet }\left(
P\right) $ and are referred to as the chain complex of the path complex $P$.
In order not to overload the notation, we do not reflect the variations in
definition in the notation $\Omega _{\bullet }\left( P\right) $. However,
whenever using it, one should specify which of four possible versions
(regular versus non-regular and truncated versus full) is being considered.

\begin{definition}
\RM The homology groups of the truncated complex (\ref{cOm}) are referred to
as the \emph{path homology groups} of the path complex $P$ and are denoted
by $H_{n}\left( P\right) ,n\geq 0$. The homology groups of the complex (\ref%
{cOmr}) are called the \emph{reduced path homology groups} of $P$ and are
denoted by $\widetilde{H}_{n}\left( P\right) ,n\geq -1$.
\end{definition}

Note that for a regular path complex $P$ the both homologies $H_{\bullet
}\left( P\right) $ and $\widetilde{H}_{\bullet }\left( P\right) $ admit
regular and non-regular versions.

Hence, by definition we have for any $n\geq 0$%
\begin{equation}
H_{n}\left( P\right) =H_{n}\left( \Omega _{\bullet }\left( P\right) \right)
=\ker \partial |_{\Omega _{n}}\left/ \func{Im}\partial |_{\Omega
_{n+1}}\right. ,  \label{Hndef}
\end{equation}%
so that $H_{n}\left( P\right) $ are linear spaces over $\mathbb{K}$. Recall
that the paths of $\ker \partial |_{\Omega _{n}}$ are called \emph{closed},
and the paths from $\func{Im}\partial |_{\Omega _{n+1}}$ -- exact.

The reduced homologies $\widetilde{H}_{n}\left( P\right) $ are defined
similarly for all $n\geq -1$. We clearly have%
\begin{equation*}
\widetilde{H}_{n}\left( P\right) =H_{n}\left( P\right) \ \ \text{for all }%
n\geq 1
\end{equation*}%
and $\widetilde{H}_{-1}\left( P\right) =\left\{ 0\right\} $. To describe $%
\widetilde{H}_{0}\left( P\right) $ observe that the pairing $\left(
1,v\right) $ between $0$-form $1$ (=the constant function on $V$ that is
equal to $1_{\mathbb{K}}$ at all vertices) and $0$-path $v$ is extended to $%
\left( 1,h\right) $ for any homology class $h\in H_{0}\left( P\right) $
because for any exact path $v=\partial u$ we have $\left( 1,v\right) =\left(
1,\partial u\right) =\left( d1,u\right) =0.$ Then we have%
\begin{equation}
\widetilde{H}_{0}\left( P\right) =\left\{ h\in H_{0}\left( P\right) :\left(
1,h\right) =0\right\}  \label{H0v}
\end{equation}%
and, in particular,%
\begin{equation*}
\dim \widetilde{H}_{0}\left( P\right) =\dim H_{0}\left( P\right) -1.
\end{equation*}

If the path complex $P$ is perfect then we obtain $\Omega _{n}\left(
P\right) =\mathcal{A}_{n}\left( P\right) $ for all $n$ (in this case there
is no difference between regular and non-regular versions). Hence, in this
case the chain complex (\ref{cOmr}) is identical to (\ref{cAe}), and (\ref%
{cOm}) is identical to (\ref{cA}).

If $P\left( G\right) $ is the path complex of a digraph $G$ then we use the
notation 
\begin{equation*}
\Omega _{n}\left( G\right) :=\Omega _{n}\left( P\left( G\right) \right) .
\end{equation*}%
The corresponding homology groups will be denoted by $H_{n}\left( G\right) ,%
\widetilde{H}\left( G\right) $ and are referred to as the\emph{\ path
homologies of the digraph} $G$.

Since the chain complex (\ref{cOm}) is finite dimensional, the following
identity always takes places:%
\begin{equation}
\dim H_{n}=\dim \Omega _{n}-\dim \partial \Omega _{n}-\dim \partial \Omega
_{n+1}  \label{dimHn}
\end{equation}%
as it follows from (\ref{Hndef}) and the rank-nullity theorem (a similar
identity holds also for reduced homologies). The Euler characteristic of the
path complex is defined by%
\begin{equation}
\chi \left( P\right) =\sum_{p=0}^{n}\left( -1\right) ^{p}\dim H_{p}\left(
P\right)  \label{Edef}
\end{equation}%
provided $n$ is so big that 
\begin{equation}
\dim H_{p}\left( P\right) =0\ \text{for\ all\ }p>n.  \label{dimH=0}
\end{equation}%
There are examples showing that the condition (\ref{dimH=0}) is not always
fulfilled. In the latter case $\chi \left( P\right) $ is not defined. For a
regular path complex $P$ there is a regular and non-regular versions of $%
\chi \left( P\right) $ that do not have to match.

If $\dim \Omega _{p}=0$ for $p>n$, then it follows from (\ref{dimHn}) that 
\begin{equation}
\chi \left( P\right) =\sum_{p=0}^{n}\left( -1\right) ^{p}\dim \Omega
_{p}\left( P\right) .  \label{EulerP0}
\end{equation}%
The definition (\ref{Edef}) has an advantage that it may work even when $%
\dim \Omega _{p}>0$ for all $p.$\label{rem: generating function of dimOm_n}

Sometimes it is useful to be able to determine the homology groups $H_{n}$
directly via the spaces $\mathcal{A}_{n}$, without $\Omega _{n}$, as in the
next statement.

\begin{proposition}
We have 
\begin{equation}
H_{n}=\left. \ker \partial |_{\mathcal{A}_{n}}\right/ \left( \mathcal{A}%
_{n}\cap \partial \mathcal{A}_{n+1}\right)  \label{HnAn}
\end{equation}%
and%
\begin{equation}
\dim H_{n}=\dim \mathcal{A}_{n}-\dim \partial \mathcal{A}_{n}-\dim \left( 
\mathcal{A}_{n}\cap \partial \mathcal{A}_{n+1}\right) .  \label{dimHnAn}
\end{equation}
\end{proposition}

\begin{proof}
Observe first that%
\begin{equation*}
\ker \partial |_{\mathcal{A}_{n}}=\ker \partial |_{\Omega _{n}}
\end{equation*}%
because $v\in \mathcal{A}_{n}$ and $\partial v=0$ imply $v\in \Omega _{n}$.
Next, it follows from the definition of $\Omega _{n+1}$ that 
\begin{equation*}
u\in \partial \Omega _{n+1}\Leftrightarrow u\in \mathcal{A}_{n}\text{\ \
and\ \ }u=\partial v\text{ for some }v\in \mathcal{A}_{n+1},
\end{equation*}%
which is equivalent to 
\begin{equation*}
\partial \Omega _{n+1}=\mathcal{A}_{n}\cap \partial \mathcal{A}_{n+1}.
\end{equation*}%
Then (\ref{HnAn}) follows from (\ref{Hndef}). Finally, (\ref{dimHnAn})
follows from (\ref{HnAn}) and the rank-nullity theorem.
\end{proof}

Let us present a simple example showing a distinction between regular and
non-regular versions of $\Omega _{\bullet }\left( P\right) $ and $H_{\bullet
}\left( P\right) .$ For the rest of this section we use the superscript $%
^{reg}$ to refer to all regular notions. For example, $\partial ^{reg}$ will
denote the regular boundary operator on the spaces $\mathcal{R}_{n}$, and $%
\Omega _{n}^{reg}$ will denote the space of regular $\partial $-invariant
paths, that is,%
\begin{equation*}
\Omega _{n}^{reg}=\left\{ v\in \mathcal{A}_{n}:\partial ^{reg}v\in \mathcal{A%
}_{n-1}\right\} .
\end{equation*}

\begin{example}
\RM Consider the digraph $^{0}\bullet \longleftrightarrow \bullet ^{1}$ with 
$V=\left\{ 0,1\right\} $ and $E=\left\{ 01,10\right\} ,$ and let $P$ be its
path complex, that is,%
\begin{equation}
P=\left\{ 0,1,01,10,010,101,...\right\} .  \label{P01}
\end{equation}%
Clearly, $P$ is regular. The spaces $\left\{ \mathcal{A}_{n}\right\} $ of
allowed paths are as follows:%
\begin{eqnarray*}
\mathcal{A}_{0} &=&\limfunc{span}\left\{ e_{0},e_{1}\right\} \\
\mathcal{A}_{1} &=&\limfunc{span}\left\{ e_{01},e_{10}\right\} \\
\mathcal{A}_{2} &=&\limfunc{span}\left\{ e_{010},e_{101}\right\} \\
\mathcal{A}_{3} &=&\limfunc{span}\left\{ e_{0101},e_{1010}\right\}
\end{eqnarray*}%
etc. Then we have $\Omega _{0}=\mathcal{A}_{0}$, $\Omega _{1}=\mathcal{A}%
_{1} $ whereas all non-regular spaces $\Omega _{n}$ with $n\geq 2$ are
trivial. Indeed, consider, for example,%
\begin{equation*}
\Omega _{2}=\left\{ v\in \mathcal{A}_{2}:\partial v\in \mathcal{A}%
_{1}\right\} .
\end{equation*}%
We have%
\begin{equation}
\begin{array}{l}
\partial e_{010}=e_{10}-e_{00}+e_{01}\ , \\ 
\partial e_{101}=e_{01}-e_{11}+e_{10}\ .%
\end{array}
\label{e101}
\end{equation}%
Since $e_{00}$ and $e_{11}$ are non-allowed, the only allowed linear
combination of $e_{010}$ and $e_{101}$ is zero. Hence, $\Omega _{2}=\left\{
0\right\} .$ In the same way also $\Omega _{n}=\left\{ 0\right\} $ for all $%
n\geq 2$.

It is easy to see that 
\begin{equation*}
\partial \Omega _{1}=\limfunc{span}\left\{ e_{1}-e_{0}\right\} ,
\end{equation*}%
while $\partial \Omega _{n}=\left\{ 0\right\} $ for all $n\neq 1.$ Hence, we
obtain by (\ref{dimHn})%
\begin{eqnarray*}
\dim H_{0} &=&\dim \Omega _{0}-\dim \partial \Omega _{0}-\dim \partial
\Omega _{1}=2-0-1=1 \\
\dim H_{1} &=&\dim \Omega _{1}-\dim \partial \Omega _{1}-\dim \partial
\Omega _{2}=2-1-0=1
\end{eqnarray*}%
and $\dim H_{n}=0$ for all $n\geq 2.$ Note that $\dim H_{n}$ can also be
computed using (\ref{dimHnAn}). One can also show that%
\begin{equation*}
H_{0}\cong \limfunc{span}\left\{ e_{0}+e_{1}\right\} \ \ \text{and\ \ \ }%
H_{1}\cong \limfunc{span}\left\{ e_{01}+e_{10}\right\} .
\end{equation*}%
The spanning $1$-path $e_{01}+e_{10}$ of $H_{1}$ can be regarded as a kind
of \textquotedblleft hole\textquotedblright\ in the digraph. The non-regular
Euler characteristic is%
\begin{equation*}
\chi =\dim H_{0}-\dim H_{1}=0.
\end{equation*}

Consider now regular spaces $\Omega _{n}^{reg}$. The spaces $\Omega
_{0}^{reg}$ and $\Omega _{1}^{reg}$ are the same as $\Omega _{0}$ and $%
\Omega _{1}$, respectively. However, the formulas (\ref{e101}) for the case
of the regular operator $\partial ^{reg}$ should be modified as follows:%
\begin{equation*}
\begin{array}{l}
\partial ^{reg}e_{010}=e_{10}+e_{01} \\ 
\partial ^{reg}e_{101}=e_{01}+e_{10}%
\end{array}%
\end{equation*}%
where we have replaced the non-regular $1$-paths $e_{00}$ and $e_{11}$ by $0$%
. It follows that both $\partial e_{010}$ and $\partial e_{101}$ belong to $%
\mathcal{A}_{1}$ whence both $e_{010}$ and $e_{101}$ belong to $\Omega
_{2}^{reg}.$ Hence, in this case $\Omega _{2}^{reg}=\mathcal{A}_{2}$.
Similarly, one can verify that $\Omega _{n}^{reg}=\mathcal{A}_{n}$ for all $%
n\geq 2$.

It is easy to see that $\partial \Omega _{1}^{reg}=\partial \Omega _{1}$,
while 
\begin{eqnarray*}
\partial \Omega _{2}^{reg} &=&\limfunc{span}\left\{ e_{01}+e_{10}\right\} \\
\partial \Omega _{3}^{reg} &=&\limfunc{span}\left\{ e_{010}-e_{101}\right\}
\end{eqnarray*}%
etc. We obtain $\dim H_{0}^{reg}=1$ as in non-regular case, while%
\begin{equation*}
\dim H_{1}^{reg}=\dim \Omega _{1}^{reg}-\dim \partial \Omega _{1}^{reg}-\dim
\partial \Omega _{2}^{reg}=2-1-1=0
\end{equation*}%
and in the same way $\dim H_{n}^{reg}=0$ for all $n\geq 1.$ Hence, the
regular homologies do not see the \textquotedblleft hole\textquotedblright\ $%
e_{01}+e_{10}$. The regular Euler characteristic is $\chi ^{reg}=1.$
\end{example}

Of course, it is a matter of convention whether a two-way path $%
e_{01}+e_{10} $ should qualify as a \textquotedblleft hole\textquotedblright
. We are inclined to think that it should not. For this and for other
reasons, in the subsequent sections of this paper we deal mostly with
regular homologies unless otherwise mentioned.

To finish the discussion \textquotedblleft regular versus
non-regular\textquotedblright , let us provide a condition ensuring the
identity $\Omega _{n}=\Omega _{n}^{reg}$.

\begin{definition}
\RM We say that a path complex $P$ is \emph{strictly regular }if it is
regular and contains no path of the form $iji.$
\end{definition}

Note that the path complex of a simplicial complex is always strictly
regular because the sequences of indices in allowed paths are strictly
increasing. The path complex of a digraph is strictly regular if and only if
the digraph is loopless (that is, $ii$ is never an edge) and contains no
two-way edges (that is, $i\rightarrow j$ implies $j\not\rightarrow i$).
Clearly, the path complex (\ref{P01}) is regular but not strictly regular.

\begin{proposition}
Let $P$ be a regular path complex.

\begin{itemize}
\item[$\left( a\right) $] For all $n\geq -1$ we have $\Omega _{n}\subset
\Omega _{n}^{reg}.$

\item[$\left( b\right) $] If $P$ is strictly regular then $\Omega
_{n}=\Omega _{n}^{reg}$ for all $n\geq -1$.

\item[$\left( c\right) $] If $\Omega _{2}=\Omega _{2}^{reg}$ then $P$ is
strictly regular.
\end{itemize}
\end{proposition}

\begin{proof}
$\left( a\right) $ By (\ref{Omdef}) if $v\in \Omega _{n}$ then $v\in 
\mathcal{A}_{n}$ and $\partial v\in \mathcal{A}_{n-1}.$ Recall that $%
\partial ^{reg}v$ is obtained from $\partial v$ by removing all the
components $e_{i_{0}...i_{n-1}}$ with non-regular $i_{0}...i_{n-1}.$
However, if $\partial v\in \mathcal{A}_{n-1}$ then $\partial v$ is allowed
and, hence, has no non-regular component. Therefore, $\partial
^{reg}v=\partial v$ whence $v\in \Omega _{n}^{reg},$ which proves the
inclusion $\Omega _{n}\subset \Omega _{n}^{reg}$.

$\left( b\right) $ Let us prove the opposite inclusion for the case of
strictly regular $P$. It suffices to show that if $v\in \Omega _{n}^{reg}$
then $\partial ^{reg}v=\partial v.$ Suppose this is not the case, that is, $%
\partial v$ contains a non-regular component. All the components of $%
\partial v$ comes from differentiating of the components of $v$. If $%
e_{i_{0}...i_{n}}$ is one of the components of $v$ then $\partial
e_{i_{0}...i_{n}}$ consists of the terms of the form $e_{i_{0}...\widehat{%
i_{k}}...i_{n}}$ with an omitted index $i_{k}$. Since $i_{0}...i_{n}$ is
allowed and, hence, regular, the only way $e_{i_{0}...\widehat{i_{k}}%
...i_{n}}$ can be non-regular if $i_{k-1}=i_{k+1}.$ However, the path $%
i_{k-1}i_{k}i_{k+1}$ is allowed, and by the strict regularity the identity $%
i_{k-1}=i_{k+1}$ is not possible. Hence, $\partial v$ cannot contain
non-regular terms, which proves that $\partial v=\partial ^{reg}v$ and,
hence, $\Omega _{n}=\Omega _{n}^{reg}$.

$\left( c\right) $ Assume from the contrary that $P$ is not strictly
regular, that is, $P$ contains a path $iji$ for some $i,j\in V$. Since%
\begin{eqnarray*}
\partial e_{iji} &=&e_{ji}-e_{ii}+e_{ij} \\
\partial ^{reg}e_{iji} &=&e_{ji}+e_{ij}
\end{eqnarray*}%
we see that $\partial ^{reg}e_{iji}\in \mathcal{A}_{2}$ whereas $\partial
e_{iji}\notin \mathcal{A}_{2}$. It follows that $e_{iji}\in \Omega
_{2}^{reg}\setminus \Omega _{2},$ which contradict the hypothesis.
\end{proof}

\subsection{$d$-invariant forms}

Given a (regular or non-regular) chain complex $\Omega _{\bullet }$ of a
path complex $P$, we define here the dual cochain complex $\Omega ^{\bullet
} $ of forms and the exterior differential $d$ on forms.

Denote by $\mathcal{N}^{n}$ the subspace of $\Lambda ^{n},$ spanned by the
non-allowed elementary $n$-forms, that is,%
\begin{eqnarray*}
\mathcal{N}^{n} &=&\limfunc{span}\left\{ e^{i_{0}...i_{n}}:i_{0}...i_{n}%
\text{ is non-allowed}\right\} \\
&=&\left\{ \omega \in \Lambda ^{n}:\omega _{i_{0}...i_{n}}=0\text{ for all
allowed }i_{0}...i_{n}\right\} .
\end{eqnarray*}%
The elements of $\mathcal{N}^{n}$ are referred to as non-allowed $n$-forms.
Then set 
\begin{equation}
\mathcal{J}^{n}=\mathcal{N}^{n}+d\mathcal{N}^{n-1},  \label{Omnn}
\end{equation}%
and 
\begin{equation}
\Omega ^{n}=\Omega ^{n}\left( P\right) =\Lambda ^{n}\left/ \mathcal{J}%
^{n}\right. .  \label{OmL}
\end{equation}%
Denoting by $\mathcal{A}^{n}$ the subspace of $\Lambda ^{n}$ spanned by
allowed elementary $n$-forms and noticing that $\Lambda ^{n}=\mathcal{A}%
^{n}\oplus \mathcal{N}^{n}$, we obtain that%
\begin{equation}
\Omega ^{n}\cong \mathcal{A}^{n}\left/ \left( \mathcal{A}^{n}\cap \mathcal{J}%
^{n}\right) \right. .  \label{A/J}
\end{equation}

\begin{lemma}
\label{LemJ}$\left( a\right) \ $If $\omega \in \mathcal{J}^{n}$ then $%
d\omega \in \mathcal{J}^{n+1}.$ Consequently, $d$ is well defined on spaces $%
\Omega ^{n}$.

$\left( b\right) $ If $\omega \in \mathcal{J}^{n}$ then $\left( \omega
,v\right) =0$ for all $v\in \Omega _{n}.$
\end{lemma}

\begin{proof}
$\left( a\right) $ Since $d^{2}=0$, it follows from (\ref{Omnn}) that 
\begin{equation*}
d\mathcal{J}^{n}\subset d\mathcal{N}^{n}+d^{2}\mathcal{N}^{n-1}=d\mathcal{N}%
^{n}\subset \mathcal{J}^{n+1}.
\end{equation*}%
Hence, if $\omega _{1}=\omega _{2}\func{mod}\mathcal{J}^{n}$ then $d\omega
_{1}=d\omega _{2}\func{mod}\mathcal{J}^{n}$ so that $d$ is well-defined on
the cosets $\omega \func{mod}\mathcal{J}^{n}$ that are the elements of the
quotient space $\Lambda ^{n}\left/ \mathcal{J}^{n}\right. =\Omega ^{n}$.
\end{proof}

\begin{proof}
$\left( b\right) $ By (\ref{Omnn}) $\omega =\varphi +d\psi $ where $\varphi
\in \mathcal{N}^{n}$ and $\psi \in \mathcal{N}^{n-1}$. Note that $\varphi
\in \mathcal{N}^{n}$ and $v\in \mathcal{A}_{n}$ imply that 
\begin{equation*}
\left( \varphi ,v\right) =\sum \varphi _{i_{0}...i_{n}}v^{i_{0}...i_{n}}=0
\end{equation*}%
because if $i_{0}...i_{n}$ is allowed then $\varphi _{i_{0}...i_{n}}=0$
while for non-allowed $i_{0}...i_{n}$ we have $v^{i_{0}...i_{n}}=0.$ Next,
we have 
\begin{equation*}
\left( d\psi ,v\right) =\left( \psi ,dv\right) =0
\end{equation*}%
because $\psi \in \mathcal{N}^{n-1}$ and $\partial v\in \mathcal{A}_{n-1}$.
Hence, combining these two lines, we conclude $\left( \omega ,v\right) =0.$
\end{proof}

\begin{definition}
\RM The elements of the dual space $\Omega ^{n}$ are called $d$-\emph{%
invariant} $n$-forms of the path complex $P$, and the operator $d:\Omega
^{n}\rightarrow \Omega ^{n+1}$ is called the \emph{exterior differential}.
\end{definition}

\label{here}By Lemma \ref{LemJ}$\left( b\right) $, any element $\omega \func{%
mod}\mathcal{J}^{n}$ of $\Omega ^{n}$ determines a linear functional on $%
\Omega _{n}$ by%
\begin{equation*}
\left( \omega \func{mod}J^{n},v\right) =\left( \omega ,v\right) .
\end{equation*}%
Hence, we obtain a mapping 
\begin{equation}
\Omega ^{n}\rightarrow \left( \Omega _{n}\right) ^{\prime }  \label{nn}
\end{equation}%
where $\left( \Omega _{n}\right) ^{\prime }$ is the dual space to $\Omega
_{n}$.

\begin{lemma}
The mapping \emph{(\ref{nn})} is a linear isomorphism, that is, $\Omega ^{n}$
can be identified as a dual space of $\Omega _{n}$.
\end{lemma}

\begin{proof}
Every linear functions on $\Omega _{n}$ can be extended to that on $\Lambda
_{n}$ and, hence, is given by $v\mapsto \left( \omega ,v\right) $ for some $%
\omega \in \Lambda ^{n}$. Therefore, it is determined also by $\omega \func{%
mod}J^{n}\in \Omega ^{n}$, which means that the mapping (\ref{nn}) is
surjective. To prove the injectivity of (\ref{nn}) is suffices to show that $%
\dim \Omega ^{n}=\dim \Omega _{n}$. For that, let us first show that%
\begin{equation}
\Omega _{n}=\left( \mathcal{J}^{n}\right) ^{\bot },  \label{OmJ}
\end{equation}%
where $\left( \mathcal{J}^{n}\right) ^{\bot }$ denotes the annihilator in $%
\Lambda _{n}$ of $\mathcal{J}^{n}$ as a subspace of $\Lambda ^{n}$. Indeed,
for $v\in \Lambda _{n}$ the condition $v\in \left( \mathcal{J}^{n}\right)
^{\bot }$ means that%
\begin{equation*}
v\bot \mathcal{N}^{n}\ \ \text{and\ \ }v\bot d\mathcal{N}^{n-1}.
\end{equation*}%
The first condition here is equivalent to $v\in \left( \mathcal{N}%
^{n}\right) ^{\bot }=\mathcal{A}_{n}$ while the second condition is
equivalent to%
\begin{equation*}
\left( d\omega ,v\right) =0\ \ \forall \omega \in \mathcal{N}%
^{n-1}\Leftrightarrow \left( \omega ,\partial v\right) =0\ \forall \omega
\in \mathcal{N}^{n-1}\Leftrightarrow \partial v\bot \mathcal{N}%
^{n-1}\Leftrightarrow \partial v\in \left( \mathcal{N}^{n-1}\right) ^{\bot },
\end{equation*}%
that is, to $\partial v\in \mathcal{A}_{n-1}$. We are left to recall that $%
v\in \mathcal{A}_{n}$ and $\partial v\in \mathcal{A}_{n-1}$ is equivalent to 
$v\in \Omega _{n}$, which proves (\ref{OmJ}).

Finally, we obtain%
\begin{equation*}
\dim \Omega _{n}=\dim \left( \mathcal{J}^{n}\right) ^{\bot }=\dim \Lambda
^{n}-\dim \mathcal{J}^{n}=\dim \Lambda ^{n}\left/ \mathcal{J}^{n}\right.
=\dim \Omega ^{n},
\end{equation*}%
which finishes the proof.
\end{proof}

It follows that the operators $\partial :\Omega _{n+1}\rightarrow \Omega
_{n} $ and $d:\Omega ^{n}\rightarrow \Omega ^{n+1}$ are dual, that is, for
any $v\in \Omega _{n+1}$ and $\omega \in \Omega ^{n}$%
\begin{equation*}
\left( d\omega ,v\right) =\left( \omega ,\partial v\right) .
\end{equation*}%
because this identity is true for any representative of $\omega $ in $%
\Lambda ^{n}$.

We obtain a cochain complex $\Omega ^{\bullet }\left( P\right) $ of $P$,
that is,%
\begin{equation}
0\rightarrow \mathbb{K}\rightarrow \Omega ^{0}\rightarrow \dots \rightarrow
\Omega ^{n}\rightarrow \Omega ^{n+1}\rightarrow \dots  \label{cochain}
\end{equation}%
where all arrows are given by $d$. Its cohomologies are referred to as \emph{%
reduced path cohomologies} of $P$ and are denoted by $\widetilde{H}^{\bullet
}\left( P\right) ,$ that is 
\begin{equation*}
\widetilde{H}^{n}\left( P\right) =H^{n}\left( \Omega ^{\bullet }\left(
P\right) \right) =\ker d|_{\Omega ^{n}}\left/ \func{Im}d|_{\Omega
^{n-1}}\right. ,
\end{equation*}%
for any $n\geq -1$. It follows from the construction that $\widetilde{H}%
_{n}\left( P\right) $ and $\widetilde{H}^{n}\left( P\right) $ are dual
vector spaces over $\mathbb{K}$, in particular, their dimensions are the
same.

The cohomologies of the truncated cochain complex%
\begin{equation}
0\rightarrow \Omega ^{0}\rightarrow \dots \rightarrow \Omega ^{n}\rightarrow
\Omega ^{n+1}\rightarrow \dots  \label{cochaintrun}
\end{equation}%
are called \emph{path cohomologies} of $P$ and are defined by $H^{n}\left(
P\right) $, $n\geq 0$. Clearly, $H^{n}\left( P\right) $ is a dual space to $%
H_{n}\left( P\right) .$ Similarly to (\ref{dimHn}), we have%
\begin{equation*}
\dim H^{n}=\dim \Omega ^{n}-\dim d\Omega ^{n}-\dim d\Omega ^{n-1},
\end{equation*}%
and an analogous identity holds for reduced cohomologies $\widetilde{H}^{n}$.

Let now $P$ be a \emph{regular} path complex. Then a similar construction
works using the regular spaces $\Omega _{n}$ that are subspaces of $\mathcal{%
R}_{n}.$ Setting 
\begin{eqnarray*}
\mathcal{N}^{n} &=&\limfunc{span}\left\{ e^{i_{0}...i_{n}}:i_{0}...i_{n}%
\text{ is regular and non-allowed}\right\} \\
&=&\left\{ \omega \in \mathcal{R}^{n}:\omega _{i_{0}...i_{n}}=0\text{ for
all allowed }i_{0}...i_{n}\right\}
\end{eqnarray*}%
and defining $\mathcal{J}^{n}$ as before by (\ref{Omnn}), we set 
\begin{equation}
\Omega ^{n}=\mathcal{R}^{n}\left/ \mathcal{J}^{n}\right.  \label{OmE}
\end{equation}%
and show as above that $d$ is well-defined on $\Omega ^{n}$, that $\Omega
^{n}$ can be identified with $\left( \Omega _{n}\right) ^{\prime }$ and that
operators $d$ and $\partial $ are dual. Since $\mathcal{R}^{n}=\mathcal{A}%
^{n}\oplus \mathcal{N}^{n}$, the formula (\ref{A/J}) holds in the regular
case, too.

Let us prove that the concatenation is well-defined on the (regular and
non-regular) spaces $\Omega ^{n}$.

\begin{lemma}
\label{Lemcon}Let $\varphi $ be a $p$-form and $\psi $ be a $q$-form. If $%
\varphi \in \mathcal{J}^{p}$ or $\psi \in \mathcal{J}^{q}$ then $\varphi
\psi \in \mathcal{J}^{p+q},$ that is, $\left\{ \mathcal{J}^{p}\right\} $ is
a graded ideal for the concatenation. Consequently, the concatenation of two
forms is well-defined on the spaces $\mathcal{J}^{p}$ as well as on $\Omega
^{p}$, and it satisfies the product rule \emph{(\ref{Leib})}.
\end{lemma}

\begin{proof}
Observe first that if $\varphi \in \mathcal{N}^{p}$ then $\varphi \psi \in 
\mathcal{N}^{p+q}$. Indeed, it suffices to prove this for elementary forms $%
\varphi =e^{i_{0}...i_{p}}$ and $\psi =e^{j_{0}...j_{q}}$ where the claim is
obvious: if the $p$-path $i_{0}...i_{p}$ is non-allowed then so is the
concatenated $\left( p+q\right) $-path $i_{0}...i_{p}j_{1}...j_{q}$, by the
definition of a path complex (in the regular case we use in addition the
fact that concatenation of regular paths is regular).

If $\varphi \in \mathcal{J}^{p}$ then $\varphi =\varphi _{0}+d\varphi _{1}$
where $\varphi _{0}\in \mathcal{N}^{p}$ and $\varphi _{1}\in \mathcal{N}%
^{p-1}$. Then we have%
\begin{eqnarray*}
\varphi \psi &=&\varphi _{0}\psi +\left( d\varphi _{1}\right) \psi \\
&=&\varphi _{0}\psi +d\left( \varphi _{1}\psi \right) -\left( -1\right)
^{p-1}\varphi _{1}d\psi .
\end{eqnarray*}%
By the above observation, all the forms $\varphi _{0}\psi $,$\ \varphi
_{1}\psi $, $\varphi _{1}d\psi $ are in $\mathcal{N}^{\bullet }.$ It follows
that $d\left( \varphi _{1}\psi \right) \in \mathcal{J}^{p+q}$ and, hence, $%
\varphi \psi \in \mathcal{J}^{p+q}.$ In the same way one handles the case $%
\psi \in \mathcal{J}^{q}.$

To prove that concatenation is well defined on $\Omega ^{p}$, we need to
verify that if $\varphi =\varphi ^{\prime }\func{mod}\mathcal{J}^{p}$ and $%
\psi =\psi ^{\prime }\func{mod}\mathcal{J}^{q}$ then $\varphi \psi =\varphi
\psi \func{mod}\mathcal{J}^{p+q}.$ Indeed, we have%
\begin{equation*}
\varphi \psi -\varphi ^{\prime }\psi ^{\prime }=\varphi \left( \psi -\psi
^{\prime }\right) +\left( \varphi -\varphi ^{\prime }\right) \psi ^{\prime },
\end{equation*}%
and each of the terms in the right hand side belong to $J^{p+q}$ by the
first part. Finally, the product rule for equivalence classes follows from
that for their representatives.
\end{proof}

Lemma \ref{Lemcon} and the product rule allow to extend concatenation to an
operation of homology classes.

\begin{proposition}
If $\varphi \in \Omega ^{p}$ and $\psi \in \Omega ^{q}$ are closed forms and
one of the forms $\varphi ,\psi $ is exact then $\varphi \psi $ is also
exact. Consequently, concatenation is well defined as an operation from $%
H^{p}\times H^{q}$ to $H^{p+q}.$
\end{proposition}

\begin{proof}
If $\varphi =d\omega $ then 
\begin{equation*}
d\left( \omega \psi \right) =\left( d\omega \right) \psi +\left( -1\right)
^{p}\omega d\psi =\varphi \psi ,
\end{equation*}%
so that $\varphi \psi $ is exact. Hence, if $\varphi _{1},\varphi _{2}$ and $%
\psi _{1},\psi _{2}$ are closed forms such that $\varphi _{1}=\varphi _{2}%
\func{mod}\func{Im}d$ and $\psi _{1}=\psi _{2}\func{mod}\func{Im}d$ then 
\begin{equation*}
\varphi _{1}\psi _{1}-\varphi _{2}\psi _{2}=\varphi _{1}\left( \psi
_{1}-\psi _{2}\right) +\left( \varphi _{1}-\varphi _{2}\right) \psi _{2}=0%
\func{mod}\func{Im}d,
\end{equation*}%
that is, $\varphi _{1}\psi _{1}$ and $\varphi _{2}\psi _{2}$ represent the
same homology class, which was to be proved.
\end{proof}

Hence, concatenation is analogous to operations of cup product for
simplicial complexes and wedge product for differential forms on manifolds.

\begin{remark}
\RM We see that the direct sum%
\begin{equation*}
\Omega =\bigoplus_{n\geq 0}\Omega ^{n}
\end{equation*}%
with the operations $d$ and concatenation is a graded differential algebra.
Note that $\Omega ^{0}$ coincides with the space $\mathcal{R}^{0}$ of all $%
\mathbb{K}$-valued functions on $V$. As it was mentioned in Remark \ref%
{RemDGAR}, all minimal graded differential algebras over $\mathcal{R}^{0}$
are the quotients of $\mathcal{R}.$ In particular, in the case of a regular
path complex, $\Omega $ is explicitly given by (\ref{OmE}) as a quotient of $%
\mathcal{R}$.
\end{remark}

\subsection{A condition for $\dim \Omega ^{p}=0$}

\label{SecLow}Let us consider the following equivalence relation. For two $n$%
-forms $\varphi ,\psi $ (from $\Lambda ^{n}$ or $\mathcal{R}^{n}$) we write%
\begin{equation*}
\varphi \simeq \psi \ \ \text{if\ \ }\varphi =\psi \func{mod}\mathcal{J}^{n},
\end{equation*}%
where $\mathcal{J}^{n}$ is defined by (\ref{Omnn}). Then we already know
that 
\begin{equation*}
\varphi \simeq 0\Rightarrow d\varphi \simeq 0,
\end{equation*}%
and 
\begin{equation*}
\varphi \simeq 0\ \ \text{or\ \ }\psi \simeq 0\Rightarrow \varphi \psi
\simeq 0
\end{equation*}%
(cf. \ref{Lemcon}). By (\ref{OmL}) or (\ref{OmE}), the equivalence classes
of $\simeq $ can be identified with the elements of $\Omega ^{n}$.

\begin{proposition}
\label{Pdimn}

\begin{itemize}
\item[$\left( a\right) $] If $\dim \Omega ^{n}=0$ then $\dim \Omega ^{p}=0$
for all $p>n.$

\item[$\left( b\right) $] If the spaces $\Omega ^{\bullet }$ are regular
then $\dim \Omega ^{n}\leq 1$ implies that $\dim \Omega ^{p}=0$ for all $%
p>n. $
\end{itemize}
\end{proposition}

\begin{proof}
In the both cases, we need to show that, for any $p$-path $i_{0}...i_{p}$
with $p>n$, 
\begin{equation*}
e^{i_{0}...i_{p}}\simeq 0.
\end{equation*}

$\left( a\right) $ By hypothesis $\Omega ^{n}=\left\{ 0\right\} $ we have $%
e^{i_{0}...i_{n}}\simeq 0.$ Since%
\begin{equation*}
e^{i_{0}...i_{p}}=e^{i_{0}...i_{n}}e^{i_{n}...i_{p}},
\end{equation*}%
it follows by Lemma \ref{Lemcon} that also $e^{i_{0}...i_{p}}\simeq 0.$

$\left( b\right) $ It suffices to treat the case $\dim \Omega ^{n}=1$. Since
for a non-allowed path $i_{0}...i_{p}$ the relation $e^{i_{0}...i_{p}}\simeq
0$ holds by definition, we can assume that $i_{0}...i_{p}$ is allowed. We
have%
\begin{equation}
e^{i_{0}...i_{p}}=e^{i_{0}...i_{n}}e^{i_{n}...i_{p}}=e^{i_{0}i_{1}}e^{i_{1}...i_{n+1}}e^{i_{n+1}...i_{p}}
\label{e=e}
\end{equation}%
If 
\begin{equation}
\text{either\ \ \ }e^{i_{0}...i_{n}}\simeq 0\ \ \text{or\ \ \ }%
e^{i_{1}...i_{n+1}}\simeq 0,  \label{eo}
\end{equation}%
then we obtain $e^{i_{0}...i_{p}}\simeq 0.$ If (\ref{eo}) fails then the
both forms $e^{i_{0}...i_{n}}$ and $e^{i_{1}...i_{n+1}}$ represent non-zero
elements of $\Omega ^{n}$. Since the latter space has dimension $1$, it
follows that for some constant $\alpha \in \mathbb{K}$,%
\begin{equation*}
e^{i_{1}...i_{n+1}}\simeq \alpha e^{i_{0}...i_{n}}.
\end{equation*}%
Substituting into (\ref{e=e}), we obtain%
\begin{equation*}
e^{i_{0}...i_{p}}\simeq \alpha
e^{i_{0}i_{1}}e^{i_{0}...i_{n}}e^{i_{n+1}...i_{p}}.
\end{equation*}%
Since the path $i_{0}...i_{p}$ is allowed and the path complex in question
is regular, this path is regular and, hence, $i_{0}\neq i_{1}$. It follows
that $e^{i_{0}i_{1}}e^{i_{0}...i_{n}}=0,$ whence $e^{i_{0}...i_{p}}\simeq 0$%
, which finishes the proof.
\end{proof}

\subsection{Connected components and $H^{0}$}

\label{SecH0}Given a path complex $P$ with a vertex set $V$, by a \emph{%
connected component} of $P$ we mean any minimal\footnote{%
The minimality of $U$ means that no proper subset of $U$ satisfies the same
property.} subset $U$ of $V$ that if $i\in U$ then $U$ contains any vertex $%
j\in V$ such that $ij$ or $ji$ is an allowed $1$-path. Clearly, any two
connected components are either disjoint or identical, and the vertex set $V$
is a disjoint union of the connected components. If $V$ itself is a
connected component then the path complex $P$ is called connected.

For example, if $P$ is the path complex of a digraph then the connected
components of $P$ coincide with those of the underlying undirected graph.

\begin{proposition}
\label{PH0}For any path complex $P$ we have%
\begin{equation}
\dim H^{0}\left( P\right) =C,  \label{dimH0}
\end{equation}%
where $C$ is the number of connected components of $P$. In particular, if $P$
is connected then $\dim H^{0}\left( P\right) =1$ and, hence, $\dim 
\widetilde{H}^{0}\left( P\right) =0.$
\end{proposition}

\begin{proof}
By definition, we have 
\begin{equation*}
H^{0}\left( \Omega \right) =\ker d|_{\Omega ^{0}}=\left\{ f\in \Omega
^{0}:df\simeq 0\right\} .
\end{equation*}%
The condition $df\simeq 0$ means that $df\in \mathcal{J}^{0}=\mathcal{N}^{0}$%
, that is, $\left( df\right) _{ij}=0$ for all allowed $1$-paths $ij$.
Therefore, we have $f_{i}=f_{j}$ for all allowed $1$-paths $ij$. The latter
is equivalent to the fact that $f=\func{const}$ on any connected component
of $P$. Hence, the dimension of the space of such functions is equal to $C$.
\end{proof}

\subsection{Disjoint union and connected sum}

\label{Secdis}For any two path complexes $P^{\prime }$ and $P^{\prime \prime
}$ with the vertex sets $V^{\prime }$ and $V^{\prime \prime }$,
respectively, their union $P^{\prime }\cup P^{\prime \prime }$ is obviously
also a path complex with the vertex set $V^{\prime }\cup V^{\prime \prime }.$
We say that $P^{\prime }$ and $P^{\prime \prime }$ are disjoint if their
vertex sets are disjoint.

\begin{proposition}
\label{Pdisjoint}If $P^{\prime }$ and $P^{\prime \prime }$ are disjoint path
complexes then, for their union $P=P^{\prime }\cup P^{\prime \prime }$ we
have%
\begin{equation*}
\Omega ^{n}\left( P\right) =\Omega ^{n}\left( P^{\prime }\right) \oplus
\Omega ^{n}\left( P^{\prime \prime }\right)
\end{equation*}%
and, hence,%
\begin{equation*}
H^{n}\left( P\right) \cong H^{n}\left( P^{\prime }\right) \oplus H^{n}\left(
P^{\prime \prime }\right)
\end{equation*}%
for all $n\geq 0$.
\end{proposition}

\begin{proof}
This follows from the obvious identities 
\begin{equation*}
S^{n}\left( P\right) =S^{n}\left( P^{\prime }\right) \oplus S^{n}\left(
P^{\prime \prime }\right)
\end{equation*}%
for each space $S=\Lambda ,\mathcal{R},\mathcal{N},\mathcal{J}$, and the
fact that $d$ on $\Lambda ^{n}\left( P\right) $ splits into the direct sum
of the operators $d$ on $\Lambda ^{n}\left( P^{\prime }\right) $ and $%
\Lambda ^{n}\left( P^{\prime \prime }\right) $.
\end{proof}

\section{$\partial $-invariant paths on digraphs}

\setcounter{equation}{0}\label{Sec5}In this section, we fix a digraph $%
G=\left( V,E\right) $ without loops. Then its path complex $P\left( G\right) 
$ is regular. We study here the regular spaces $\Omega _{n}\left( G\right) =$
$\Omega _{n}\left( P\left( G\right) \right) $ of $\partial $-invariant paths
and the associated homology groups $H_{n}\left( G\right) =H_{n}\left(
P\left( G\right) \right) $ and $\widetilde{H}_{n}\left( G\right) =\widetilde{%
H}_{n}\left( P\left( G\right) \right) .$

\subsection{Semi-edges and $\partial $-invariant paths}

\label{Secsemi}Let us describe more explicitly the notion of $\partial $%
-invariant paths on a digraph $G$. Let us say that a pair $ij$ of vertices
is a \emph{semi-edge} if it is not an edge but there is a vertex $k$ (not
necessarily unique) such that $ik$ and $kj$ are edges. The $2$-path $ikj$ is
called a \emph{bridge} of the semi-edge $ij$. The semi-edge $ij$ will be
denoted by $i\rightharpoonup j$ as on the diagram: 
\begin{equation*}
\begin{array}{c}
_{_{\nearrow }}\overset{k}{\bullet }_{_{\searrow }} \\ 
^{i}\bullet \ \rightharpoonup \ \bullet ^{j}%
\end{array}%
\end{equation*}

Let us say that an elementary path $i_{0}...i_{p}$ is \emph{semi-allowed} if
among the pairs $i_{q-1}i_{q}$, $q=1,...,p$, there is exactly one semi-edge,
while all others are edges, as on the diagram:%
\begin{equation*}
\begin{array}{ccc}
& \overset{k}{\bullet } &  \\ 
\cdots \rightarrow \bullet \rightarrow _{i_{q-1}}\bullet _{\ }^{^{\nearrow }}
& \rightharpoonup & _{\ }^{^{\searrow }}\bullet _{i_{q}}\rightarrow \
\bullet \rightarrow \cdots%
\end{array}%
\end{equation*}%
A path $i_{0}...i_{q-1}ki_{q}...i_{p}$ that is obtained by replacing in $%
i_{0}...i_{q-1}i_{q}...i_{p}$ the semi-edge $i_{q-1}i_{q}$ by the bridge $%
i_{q-1}ki_{q}$, is obviously allowed and will be called an allowed extension
of $i_{0}...i_{p}$.

Let us use the following notation: if $i_{0}...i_{p}$ is semi-allowed with
the semi-edge $i_{q-1}i_{q}$ then, for any $p$-path $v$, define its \emph{%
deficiency} $\left[ v\right] ^{i_{0}...i_{p}}$ along the path $i_{0}...i_{p}$
by%
\begin{equation}
\left[ v\right] ^{i_{0}...i_{p}}:=\sum_{k\in
V}v^{i_{0}...i_{q-1}ki_{q}...i_{p}}.  \label{dd}
\end{equation}%
Clearly, it suffices to restrict the summation to those $k$ forming a bridge 
$i_{q-1}ki_{q}$. Alternatively, one can say that the summation in (\ref{dd})
is performed across all allowed extensions of the path $i_{0}...i_{p}$.

\begin{lemma}
\label{Lemsemialowed}Let $p\geq 1$. A path $v\in \mathcal{A}_{p+1}$ belongs
to $\Omega _{p+1}$ if and only if for all semi-allowed paths $i_{0}...i_{p}$%
, 
\begin{equation*}
\left[ v\right] ^{i_{0}...i_{p}}=0.
\end{equation*}
\end{lemma}

\begin{proof}
The condition $v\in \Omega _{p+1}$ is equivalent to $\partial v\in \mathcal{A%
}_{p}$, while the latter is equivalent to 
\begin{equation}
\left( \partial v\right) ^{i_{0}...i_{p}}=0  \label{dvk0}
\end{equation}%
for all non-allowed regular paths $i_{0}...i_{p-1}$. By (\ref{dv}) we have%
\begin{equation}
\left( \partial v\right)
^{i_{0}...i_{p}}=\dsum\limits_{q=0}^{p+1}\dsum\limits_{k}\left( -1\right)
^{q}v^{i_{0}...i_{q-1}ki_{q}...i_{p}}.  \label{dvk}
\end{equation}%
If $i_{0}...i_{p}$ is not semi-allowed then all the paths $%
i_{0}...i_{q-1}ki_{q}...i_{p}$ are not allowed, because by inserting $k$ one
can eliminate only one non-edge. Hence, for such $i_{0}...i_{p}$ the
condition (\ref{dvk0}) is satisfied automatically, so that (\ref{dvk0}) is
non-void only for semi-allowed paths. If the only semi-edge in $%
i_{0}...i_{p} $ is $i_{q-1}i_{q}$ then (\ref{dvk0}) amounts to 
\begin{equation*}
\sum_{k}v^{i_{0}...i_{q-1}ki_{q}...i_{p}}=0,
\end{equation*}%
which was to be proved.
\end{proof}

\subsection{Triangles, squares and $\dim \Omega _{p}$}

\label{Secsq}Recall that $\dim \Omega _{0}=\dim \mathcal{A}_{0}=\left\vert
V\right\vert $ and $\dim \Omega _{1}=\dim \mathcal{A}_{1}=\left\vert
E\right\vert $. Here we give an explicit formula for $\dim \Omega _{2}$
using the set $P_{2}$ of allowed $2$-paths and the set $\mathcal{S}$ of
semi-edges of the digraph $G$ (see Section \ref{Secsemi} for the definition
of a semi-edge).

\begin{proposition}
\label{PropS}We have%
\begin{equation}
\dim \Omega _{2}=\dim \mathcal{A}_{2}-\left\vert \mathcal{S}\right\vert
=\left\vert P_{2}\right\vert -\left\vert \mathcal{S}\right\vert .
\label{Om2}
\end{equation}
\end{proposition}

\begin{proof}
Recall that%
\begin{equation*}
\mathcal{A}_{2}=\limfunc{span}\left\{ e_{abc}:abc\text{ is allowed}\right\}
,\ \ \ \dim \mathcal{A}_{2}=\left\vert P_{2}\right\vert ,
\end{equation*}%
and%
\begin{equation*}
\Omega _{2}=\left\{ v\in \mathcal{A}_{2}:\partial v\in \mathcal{A}%
_{1}\right\} =\left\{ v\in \mathcal{A}_{2}:\partial v=0\func{mod}\mathcal{A}%
_{1}\right\} .
\end{equation*}%
If $abc$ is allowed then $ab$ and $bc$ are edges, whence%
\begin{equation*}
\partial e_{abc}=e_{bc}-e_{ac}+e_{ab}=-e_{ac}\func{mod}\mathcal{A}_{1}.
\end{equation*}%
If $ac$ is an edge then $e_{ac}=0\func{mod}\mathcal{A}_{1}$. If $ac$ is not
an edge then $ac$ is a semi-edge, and in this case 
\begin{equation*}
\partial e_{abc}\neq 0~\func{mod}\mathcal{A}_{1}.
\end{equation*}%
For any $v\in \Omega _{2}$, we have%
\begin{equation*}
v=\sum_{\left\{ abc\text{ is allowed}\right\} }v^{abc}e_{abc}
\end{equation*}%
hence it follows that%
\begin{equation*}
\partial v=-\sum_{\left\{ abc:\ ac\text{ is semi-edge}\right\}
}v^{abc}e_{ac}~\func{mod}\mathcal{A}_{1}.
\end{equation*}%
The condition $\partial v=0\func{mod}\mathcal{A}_{1}$ is equivalent to%
\begin{equation*}
\sum_{\left\{ abc:\ ac\text{ is semi-edge}\right\} }v^{abc}e_{ac}=0\func{mod}%
\mathcal{A}_{1},
\end{equation*}%
which is equivalent to $\sum_{b}v^{abc}=0$ for all semi-edges $ac.$ The
number of these conditions is exactly $\left\vert \mathcal{S}\right\vert $,
and they all are independent for different semi-edges, because a triple $abc$
determines at most one semi-edge. Hence, $\Omega _{2}$ is obtained from $%
\mathcal{A}_{2}$ by imposing $\left\vert \mathcal{S}\right\vert $ linearly
independent conditions, which implies (\ref{Om2}).
\end{proof}

Let us call by a \emph{triangle} a sequence of three distinct vertices $%
a,b,c\in V$ such that $a\rightarrow b,b\rightarrow c,a\rightarrow c$:%
\begin{equation*}
\begin{array}{ccc}
& \overset{b}{\bullet } &  \\ 
_{a}\bullet _{\ }^{^{\nearrow }} & \rightarrow & _{\ }^{^{\searrow }}\bullet
_{c}\ 
\end{array}%
\end{equation*}%
Note that a triangle determines a $2$-path $e_{abc}\in \Omega _{2}$ as $%
e_{abc}\in \mathcal{A}_{2}$ and 
\begin{equation*}
\partial e_{abc}=e_{bc}-e_{ac}+e_{ab}\in \mathcal{A}_{1}.
\end{equation*}

Let us called by a \emph{square} a sequence of four distinct vertices $%
a,b,b^{\prime },c\in V$ such that $a\rightarrow b,b\rightarrow
c,a\rightarrow b^{\prime },b^{\prime }\rightarrow c$:%
\begin{equation*}
\begin{array}{ccc}
_{b}\bullet & \longrightarrow & \bullet _{c} \\ 
\ \uparrow &  & \uparrow \  \\ 
_{a}\bullet & \longrightarrow & \bullet _{b^{\prime }}%
\end{array}%
\end{equation*}%
Note that a square determines a $2$-path $v:=e_{abc}-e_{ab^{\prime }c}\in
\Omega _{2}$ as $v\in \mathcal{A}_{2}$ and%
\begin{eqnarray*}
\partial v &=&\left( e_{bc}-e_{ac}+e_{ab}\right) -\left( e_{b^{\prime
}c}-e_{ac}+e_{ab^{\prime }}\right) \\
&=&e_{ab}+e_{bc}-e_{ab^{\prime }}-e_{b^{\prime }c}\in \mathcal{A}_{1}.
\end{eqnarray*}

\begin{theorem}
\label{Tsq}\label{CorS}Assume that a digraph $G=\left( V,E\right) $ contains
no squares (as subgraphs). Then $\dim \Omega _{2}\left( G\right) $ is equal
to the number of distinct triangles in $G$, and $\dim \Omega _{p}\left(
G\right) =0$ for all $p>2$.

In particular, if $G$ contains neither triangle nor square then $\dim \Omega
_{p}\left( G\right) =0$ for all $p\geq 2$. Consequently, $\dim H_{p}\left(
G\right) =0$ for all $p\geq 2.$
\end{theorem}

\begin{proof}
Let us split the family $P_{2}$ of allowed $2$-paths into two subsets: an
allowed path $abc$ is of the first kind if $ac$ is an edge and of the second
kind otherwise:%
\begin{equation*}
1^{st}\text{ kind:\ }%
\begin{array}{ccc}
& \overset{b}{\bullet } &  \\ 
_{a}\bullet _{\ }^{^{\nearrow }} & \rightarrow & _{\ }^{^{\searrow }}\bullet
_{c}\ 
\end{array}%
,\ \ \ \ \ 2^{nd}\text{ kind:\ }%
\begin{array}{ccc}
& \overset{b}{\bullet } &  \\ 
_{a}\bullet _{\ }^{^{\nearrow }} &  & _{\ }^{^{\searrow }}\bullet _{c}\ 
\end{array}%
\end{equation*}%
Clearly, the paths of the first kind are in one-to-one correspondence with
triangles. Each path $abc$ of the second kind determines a semi-edge $ac.$
The mapping of $abc\mapsto ac$ from the paths of second kind to semi-edges
is also one-to-one: if $abc\mapsto ac$ and $ab^{\prime }c\mapsto ac$ then we
obtain a square $a,b,b^{\prime },c$ which contradicts the hypotheses. Hence,
the number of the path of the second kind is equal to $\left\vert \mathcal{S}%
\right\vert ,$ which implies that the number of the paths of the first kind
is equal to $\left\vert P_{2}\right\vert -\left\vert \mathcal{S}\right\vert $%
, and so is the number of triangles. Comparing with (\ref{Om2}) we obtain
that $\dim \Omega _{2}$ is equal to the number of triangles.

Let us prove that $\Omega _{3}=\left\{ 0\right\} ,$ that is, any $v\in
\Omega _{3}$ is identical $0$. It suffices to prove that $v^{ijkl}=0$ for
any allowed path $ijkl$ on $G.$ Fix an allowed path $ijkl$ and assume first $%
jl$ is a semi-edge. Then $ijl$ is semi-allowed, and by Lemma \ref%
{Lemsemialowed} we obtain $\left[ v\right] ^{ijl}=0,$ that is,%
\begin{equation*}
\sum_{k^{\prime }}v^{ijk^{\prime }l}=0.
\end{equation*}%
However, the only allowed path of the form $ijk^{\prime }l$ is $ijkl$
because of the absence of squares: 
\begin{equation*}
\begin{array}{ccc}
& \overset{k}{\bullet } &  \\ 
_{i}\bullet \rightarrow _{j}\bullet _{\searrow }^{\nearrow } & 
\rightharpoonup & _{\nearrow }^{\searrow }\bullet _{l} \\ 
& \underset{k^{\prime }}{\bullet } & 
\end{array}%
\end{equation*}%
We conclude that $v^{ijkl}=0$, provided $jl$ is a semi-edge. In the same way 
$v^{ijkl}=0$ provided $ik$ is a semi-edge.

Now we claim that, for any allowed path $ijkl,$ either $ik$ or $jl$ is a
semi-edge. Indeed, if neither of them is a semi-edge then both $ik$ and $jl$
must be edges, which implies that the sequence $i,j,k,l$ forms a square:%
\begin{equation*}
\begin{array}{ccc}
& \overset{k}{\bullet } &  \\ 
i\bullet _{\searrow }^{\nearrow } & \uparrow & _{\nearrow }^{\searrow
}\bullet _{l} \\ 
& \underset{j}{\bullet } & 
\end{array}%
,
\end{equation*}%
which contradicts the hypothesis. It follows that $v^{ijkl}=0$ for any
allowed path $ijkl$, which proves that $\Omega _{3}=\left\{ 0\right\} .$ By
Proposition \ref{Pdimn} we conclude that $\Omega _{p}=\left\{ 0\right\} $
for all $p\geq 3$.
\end{proof}

In the presence of squares one cannot relate directly $\dim \Omega _{2}$ to
the number of squares and triangles since there may be a linear dependence
between them as in the next example.

\begin{example}
\RM In the following digraph 
\begin{equation*}
\begin{array}{ccc}
& \overset{1}{\bullet } &  \\ 
0\overset{\nearrow }{\underset{\searrow }{\bullet \rightarrow }} & \overset{2%
}{\bullet } & \underset{\nearrow }{\overset{\searrow }{\rightarrow \bullet }}%
4\  \\ 
& \underset{3}{\bullet } & 
\end{array}%
\end{equation*}%
there are three squares $0,1,2,4$, $0,1,3,4$, and $0,2,3,4$, which determine
three $\partial $-invariant paths 
\begin{equation*}
e_{014}-e_{024},\ \ \ e_{024}-e_{034},\ \ e_{034}-e_{014}.
\end{equation*}%
These paths are linearly dependent as their sum is equal to $0$. It is easy
to see that $\dim \Omega _{2}=2$ as $\left\vert P_{2}\right\vert =3$ and $%
\left\vert \mathcal{S}\right\vert =1$ as $\mathcal{S}=\left\{ 04\right\} $.
For this digraph all homologies are trivial.
\end{example}

Also, in the presence of squares one may have non-trivial $\Omega _{p}$ for
arbitrary $p$ as one can see from numerous examples in the subsequent
sections.

\subsection{Snakes and simplexes}

\label{Secsim}A \emph{snake} of length $p$ is a digraph with $p+1$ vertices,
say $0,1,...,p$, and with the edges $i\left( i+1\right) $ and $i\left(
i+2\right) $ (see Fig. \ref{pic10}). In particular, any triple $i\left(
i+1\right) \left( i+2\right) $ is a triangle.

\FRAME{ftbphFU}{6.3304in}{1.414in}{0pt}{\Qcb{A snake}}{\Qlb{pic10}}{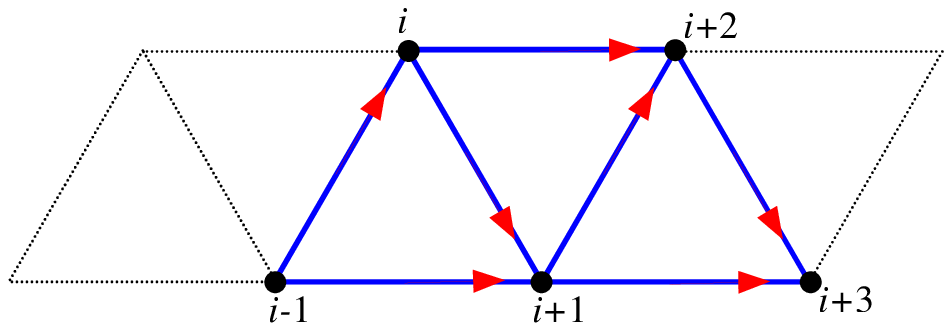%
}{\special{language "Scientific Word";type "GRAPHIC";maintain-aspect-ratio
TRUE;display "USEDEF";valid_file "F";width 6.3304in;height 1.414in;depth
0pt;original-width 6.4411in;original-height 1.3863in;cropleft "0";croptop
"1";cropright "1";cropbottom "0";filename 'pic10.eps';file-properties
"XNPEU";}}

A snake of length $p$ contains a $\partial $-invariant $p$-path $%
v=e_{01...p}.$ Indeed, this path is obviously allowed, its boundary%
\begin{equation*}
\partial v=\sum_{k=0}^{p}\left( -1\right) ^{k}e_{0...\widehat{k}...p}
\end{equation*}%
is also allowed (because $\left( k-1\right) \left( k+1\right) $ is an edge),
whence $v\in \Omega _{p}.$

\label{Exsimplex}\RM Let us define for any $n\geq 0$ a \emph{simplex-digraph}
$\func{Sm}_{n}$ as follows: its set of vertices is $\left\{
0,1,...,n\right\} $ and the edges are $i\rightarrow j$ for all $i<j$. For
example, we have 
\begin{equation*}
\func{Sm}_{1}=\ ^{0}\bullet \ \rightarrow \ \bullet ^{1},\ \ \ \ \func{Sm}%
_{2}=%
\begin{array}{c}
_{_{\nearrow }}\overset{2}{\bullet }_{_{\nwarrow }} \\ 
^{0}\bullet \ \rightarrow \ \bullet ^{1}%
\end{array}%
,
\end{equation*}%
and $\func{Sm}_{3}$ is shown on Fig. \ref{pic11}. \FRAME{ftbphFU}{6.9613in}{%
1.5221in}{0pt}{\Qcb{A $3$-simplex digraph $\func{Sm}_{3}$}}{\Qlb{pic11}}{%
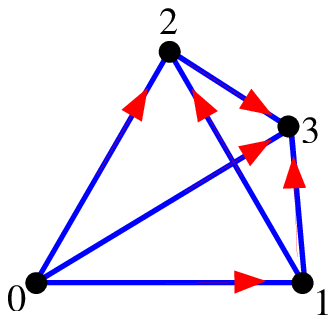}{\special{language "Scientific Word";type
"GRAPHIC";maintain-aspect-ratio TRUE;display "USEDEF";valid_file "F";width
6.9613in;height 1.5221in;depth 0pt;original-width 6.3027in;original-height
1.3586in;cropleft "0";croptop "1";cropright "1";cropbottom "0";filename
'pic11.eps';file-properties "XNPEU";}}

Since a simplex contains a snake as a subgraph, the $n$-path $v=e_{01...n}$
is $\partial $-invariant on $\func{Sm}_{n}.$

\subsection{Star-shaped digraphs and Poincar\'{e} lemma}

\label{SecStar}

\begin{definition}
\RM We say that a digraph $G$ is \emph{star-shaped} if there is a vertex $a$
(called a star center) such that $a\rightarrow b$ for all $b\neq a.$
Similarly, a digraph $G$ is called inverse star-shaped if if there is a
vertex $a$ (called a star center) such that $b\rightarrow a$ for all $b\neq
a $
\end{definition}

For example, a digraph $%
\begin{array}{c}
_{_{\nearrow }}\overset{1}{\bullet }_{_{\searrow }} \\ 
^{0}\bullet \ \leftrightarrows \ \bullet ^{2}%
\end{array}%
$ is star-shaped with the star center $0.$

\begin{theorem}
\label{Tstar}\emph{(A Poincar\'{e} lemma)} If $G$ is a (inverse) star-shaped
digraph, then all reduced homologies $\widetilde{H}_{n}\left( G\right) $ are
trivial.
\end{theorem}

\begin{proof}
To prove that $\widetilde{H}_{n}\left( G\right) =\left\{ 0\right\} $, we
need to show that if $v\in \Omega _{n}$ and $\partial v=0$ then $v=\partial
u $ for some $u\in \Omega _{n+1}.$ Set $u=e_{a}v.$ We claim that $u\in 
\mathcal{A}_{n+1}$. Since $v$ is a linear combination of allowed paths $%
e_{i_{0}...i_{n}},$ it suffices to show that $e_{ai_{0}...i_{n}}\in \mathcal{%
A}_{n+1}$ for any allowed path $e_{i_{0}...i_{n}}.$ Indeed, if $i_{0}=a$
then $e_{ai_{0}...i_{n}}=0\in \mathcal{A}_{n+1}.$ If $i_{0}\neq a$ then $%
e_{ai_{0}...i_{n}}$ is allowed by the star condition. Hence, we have $u\in 
\mathcal{A}_{n+1}.$

By the product rule (\ref{duv}) we have 
\begin{equation*}
\partial u=\partial \left( e_{a}v\right) =v-e_{a}\partial v=v,
\end{equation*}%
where we have used $\partial v=0.$ It follows that $\partial u\in \mathcal{A}%
_{n}$ and, hence, $u\in \Omega _{n+1}$, which finishes the proof.

In a similar manner one handles the inverse star-shaped graphs.
\end{proof}

For example, the simplex-digraph $\func{Sm}_{n}$ is star-shaped (and inverse
star-shaped), we obtain by Theorem \ref{Tstar} that all reduced homologies
of $\func{Sm}_{n}$ are trivial.

\subsection{Cycle-graphs}

\label{SecCycle}We say that a digraph $G=\left( V,E\right) $ is a \emph{%
cycle-graph} if it is connected (as an undirected graph) and every vertex
had the degree $2$ (see Fig. \ref{pic16}).

\FRAME{ftbphFU}{6.9612in}{0.9689in}{0pt}{\Qcb{A cycle-graph (directions of
edges are not shown)}}{\Qlb{pic16}}{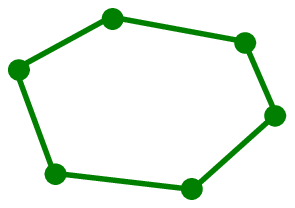}{\special{language "Scientific
Word";type "GRAPHIC";maintain-aspect-ratio TRUE;display "USEDEF";valid_file
"F";width 6.9612in;height 0.9689in;depth 0pt;original-width
6.4643in;original-height 0.9792in;cropleft "0";croptop "1";cropright
"1";cropbottom "0";filename 'pic16.eps';file-properties "XNPEU";}}

For a cycle-graph we have $\dim H_{0}\left( G\right) =1$ and 
\begin{equation}
\dim \Omega _{0}\left( G\right) =\left\vert V\right\vert =\left\vert
E\right\vert =\dim \Omega _{1}\left( G\right) .  \label{V=E}
\end{equation}

\begin{proposition}
\label{Pcycle}\label{Propcycle}Let $G$ be a cycle-graph. Then 
\begin{eqnarray*}
\dim \Omega _{p}\left( G\right) &=&0\ \ \text{for all }p\geq 3 \\
\dim H_{p}\left( G\right) &=&0\text{ for all }p\geq 2.
\end{eqnarray*}%
If $G$ is a triangle or a square then 
\begin{equation*}
\dim \Omega _{2}\left( G\right) =1,\ \dim H_{1}\left( G\right) =0,\ \ \chi
=1\ 
\end{equation*}%
whereas otherwise 
\begin{equation*}
\dim \Omega _{2}\left( G\right) =0,\ \ \ \dim H_{1}\left( G\right) =1,\ \
\chi =0.
\end{equation*}
\end{proposition}

\begin{proof}
Observe first that $\dim \Omega _{2}\leq 1$ will imply $\dim \Omega _{p}=0$
for all $p\geq 3$ by Proposition \ref{Pdimn}, whence $\dim H_{p}=0$ for $%
p\geq 3.$ Hence, we need only to handle the cases $p=1,2.$

Using two equivalent definition of the Euler characteristic, we have%
\begin{eqnarray*}
\chi &=&\dim H_{0}-\dim H_{1}+\dim H_{2} \\
&=&\dim \Omega _{0}-\dim \Omega _{1}+\dim \Omega _{2}
\end{eqnarray*}%
whence%
\begin{equation}
\chi =\dim \Omega _{2}=1-\dim H_{1}+\dim H_{2}.  \label{ksiOmH}
\end{equation}

Assume first that $G$ is neither triangle nor square. Then $G$ contains
neither triangle nor square. By Theorem \ref{Tsq} $\dim \Omega _{2}=0$
whence $\dim H_{2}=0$ and by (\ref{ksiOmH}) $\chi =0$ and $\dim H_{1}=1$.

Let us construct an $1$-path spanning $H_{1}$. For that let us identify $G$
with $\mathbb{Z}_{N}$ where $N=\left\vert V\right\vert $ so that in the
undirected graph based on $G$ the edges are $i\left( i+1\right) $. Hence, in
the digraph $G$ either $i\left( i+1\right) $ or $\left( i+1\right) i$ is an
edge. Consider an allowed $1$-path $\sigma $ with components%
\begin{equation}
\sigma ^{i\left( i+1\right) }=\left\{ 
\begin{array}{ll}
1, & \text{if }i\left( i+1\right) \text{ is an edge} \\ 
-1, & \text{if }\left( i+1\right) i\text{ is an edge,}%
\end{array}%
\right.  \label{vii+1}
\end{equation}%
and all other components of $\sigma $ vanish (see Fig. \ref{pic16a}).\FRAME{%
ftbphFU}{6.9596in}{1.1399in}{0pt}{\Qcb{The $1$-path $\protect\sigma %
=-e_{01}-e_{12}+e_{23}+e_{34}-e_{45}+e_{50}$ spans $H_{1}.$ }}{\Qlb{pic16a}}{%
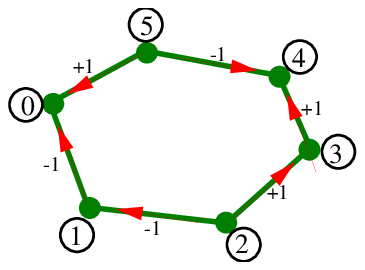}{\special{language "Scientific Word";type
"GRAPHIC";maintain-aspect-ratio TRUE;display "USEDEF";valid_file "F";width
6.9596in;height 1.1399in;depth 0pt;original-width 6.4643in;original-height
1.0777in;cropleft "0";croptop "1";cropright "1";cropbottom "0";filename
'pic16a.eps';file-properties "XNPEU";}}

Since $\sigma \neq 0$, $\sigma $ is not in $\func{Im}\partial |_{\Omega
_{2}}.$ However, $\sigma \in \ker \partial _{\Omega _{1}}$ because by
construction $\sigma ^{i\left( i+1\right) }-\sigma ^{\left( i+1\right)
i}\equiv 1$ whence for any $i$ 
\begin{equation*}
\left( \partial \sigma \right) ^{i}=\sum_{j\in V}\left( \sigma ^{ji}-\sigma
^{ij}\right) =\sigma ^{\left( i-1\right) i}+\sigma ^{\left( i+1\right)
i}-\sigma ^{i\left( i-1\right) }-\sigma ^{i\left( i+1\right) }=1-1=0.
\end{equation*}

Let $G$ be a triangle%
\begin{equation*}
\begin{array}{ccc}
& \overset{b}{\bullet } &  \\ 
_{a}\bullet _{\ }^{^{\nearrow }} & \rightarrow & _{\ }^{^{\searrow }}\bullet
_{c}\ 
\end{array}%
\end{equation*}
Then $\dim \mathcal{A}_{2}=1$, $\mathcal{S}=\emptyset $ whence $\dim \Omega
_{2}=1$ and $\chi =1.$ Clearly, we have $\Omega _{2}=\limfunc{span}\left\{
e_{abc}\right\} $. Since $\partial e_{abc}\neq 0$, we see that $\ker
\partial |_{\Omega _{2}}=0$ and, hence, $\dim H_{2}=0$. Then by (\ref{ksiOmH}%
) $\dim H_{1}=0$.

Let $G$ be a square, say $a,b,b^{\prime },c$: 
\begin{equation*}
\begin{array}{ccc}
_{b}\bullet & \longrightarrow & \bullet _{c} \\ 
\ \uparrow &  & \uparrow \  \\ 
_{a}\bullet & \longrightarrow & \bullet _{b^{\prime }}%
\end{array}%
\end{equation*}%
Then 
\begin{equation*}
\mathcal{A}_{2}=\limfunc{span}\left\{ e_{abc},e_{ab^{\prime }c}\right\} ,\ \ 
\mathcal{S}=\left\{ ac\right\}
\end{equation*}%
whence $\dim \Omega _{2}=2-1=1$ and $\chi =1.$ Note that in this case 
\begin{equation*}
\Omega _{2}=\limfunc{span}\left\{ e_{abc}-e_{ab^{\prime }c}\right\} .
\end{equation*}%
As in the case of a triangle, we obtain $\ker \partial |_{\Omega _{2}}=0$, $%
\dim H_{2}=0$ and $\dim H_{1}=0$.
\end{proof}

\subsection{An example of direct computation of $\dim H_{p}$}

Consider the digraph $G=\left( V,E\right) $ with $V=\left\{
0,1,2,3,5\right\} $ and $E=\left\{ 01,02,13,14,23,24,53,54\right\} ,$ see
Fig. \ref{pic14}.

\FRAME{ftbphFU}{5.0704in}{2.0487in}{0pt}{\Qcb{A digraph with $6$ vertices
and $8$ edges}}{\Qlb{pic14}}{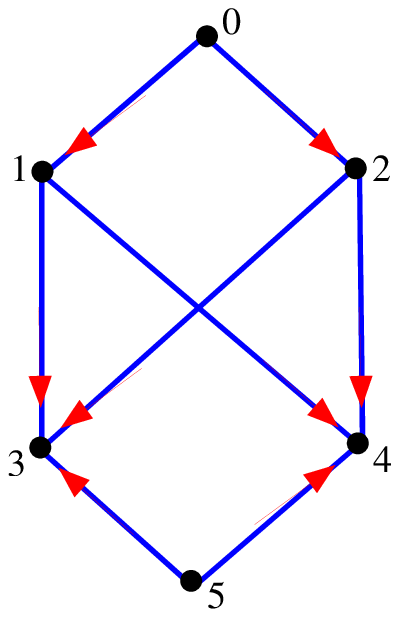}{\special{language "Scientific
Word";type "GRAPHIC";maintain-aspect-ratio TRUE;display "USEDEF";valid_file
"F";width 5.0704in;height 2.0487in;depth 0pt;original-width
6.3027in;original-height 2.5261in;cropleft "0";croptop "1";cropright
"1";cropbottom "0";filename 'pic14.eps';file-properties "XNPEU";}}

Let us compute the (regular) spaces $\Omega _{p}$ and their homologies $%
H_{p}.$ We have%
\begin{eqnarray*}
\Omega _{0} &=&\mathcal{A}_{0}=\limfunc{span}\left\{
e_{0},e_{1},e_{2},e_{3},e_{4},e_{5}\right\} ,\ \ \ \dim \Omega _{0}=6 \\
\Omega _{1} &=&\mathcal{A}_{1}=\limfunc{span}\left\{
e_{01},e_{02},e_{13},e_{14},e_{23},e_{24},e_{53},e_{54}\right\} ,\ \ \ \dim
\Omega _{1}=8 \\
\mathcal{A}_{2} &=&\limfunc{span}\left\{
e_{013},e_{014},e_{023},e_{024}\right\} ,\ \ \dim \mathcal{A}_{2}=4.
\end{eqnarray*}%
The set of semi-edges is $\mathcal{S}=\left\{ e_{03},e_{04}\right\} $ so
that $\dim \Omega _{2}=\dim \mathcal{A}_{2}-\left\vert S\right\vert =2.$ The
basis in $\Omega _{2}$ can be easily spotted as each of two squares $0,1,2,3$
and $0,1,2,4$ determine a $\partial $-invariant $2$-paths, whence%
\begin{equation*}
\Omega _{2}=\limfunc{span}\left\{ e_{013}-e_{023},e_{014}-e_{024}\right\} .
\end{equation*}%
Since there are no allowed $3$-paths, we see that $\mathcal{A}_{3}=\Omega
_{3}=\left\{ 0\right\} .$ It follows that 
\begin{equation*}
\chi =\dim \Omega _{0}-\dim \Omega _{1}+\dim \Omega _{2}=6-8+2=0.
\end{equation*}%
By (\ref{dimHn}) we obtain 
\begin{equation*}
\dim H_{2}=\dim \Omega _{2}-\dim \partial \Omega _{2}-\dim \partial \Omega
_{3}=2-\dim \partial \Omega _{2}.
\end{equation*}%
The image $\partial \Omega _{2}$ is spanned by two $1$-paths%
\begin{eqnarray*}
\partial \left( e_{013}-e_{023}\right) &=&e_{13}-e_{03}+e_{01}-\left(
e_{23}-e_{03}+e_{02}\right) =e_{13}+e_{01}-e_{23}-e_{02} \\
\partial \left( e_{014}-e_{024}\right) &=&e_{14}-e_{04}+e_{01}-\left(
e_{24}-e_{04}+e_{02}\right) =e_{14}+e_{01}-e_{24}-e_{02}
\end{eqnarray*}%
that are clearly linearly independent. Hence, $\dim \partial \Omega _{2}=2$
whence $\dim H_{2}=0.$ The dimension of $H_{1}$ can be computed similarly,
but we can do easier using the Euler characteristic: since 
\begin{equation*}
\dim H_{0}-\dim H_{1}+\dim H_{2}=\chi =0,
\end{equation*}%
it follows that $\dim H_{1}=1.$

In fact, $H_{1}$ is spanned by $1$-path 
\begin{equation*}
v=e_{13}-e_{14}-e_{53}+e_{54}.
\end{equation*}%
Indeed, by a direct computation $\partial v=0,$ so that $v\in \ker \partial
|_{\Omega _{1}}$, while $v\notin \func{Im}\partial |_{\Omega _{2}}$ because
any element of $\func{Im}\partial |_{\Omega _{2}}$ is a linear combination
of $\partial \left( e_{013}-e_{023}\right) $ and $\partial \left(
e_{014}-e_{024}\right) \ $that does not contain the term $e_{54}.$

\subsection{Triangulation as a closed path}

\label{Sectrian}Given a closed oriented $n$-dimensional manifold $M$, let $T$
be its triangulation, that is, a partition into $n$-dimensional simplexes.
Denote by $V$ the set of all vertices of the simplexes from $T$ and by $E$
-- the set of all edges, so that $\left( V,E\right) $ is a graph embedded on 
$M$. We would like to make $\left( V,E\right) $ into a digraph and define on
that digraph a closed $n$-path as a certain alternating sum of elementary $n$%
-paths arising from the simplexes from $T.$

Let us enumerate the set of vertices $V$ by distinct integers. For any
simplex from $T$ with the vertices $i_{0}...i_{n}$ define the quantity $%
\sigma ^{i_{0}...i_{n}}$ to be equal to $1$ if the orientation of the
simplex $i_{0}...i_{n}$ matches the orientation of the manifold $M$, and $-1$
otherwise. Note that each simplex gives rise to $n!$ different ordered
sequences of its vertices, each of them defining the quantity $\sigma
^{i_{0}...i_{n}}.$

Let us introduce the orientation on the set of edges $E$ by choosing on each
edge the direction from the vertex with a smaller number to the vertex with
a larger number. Then each simplex from $T$ becomes a simplex-digraph as
defined in Section \ref{Exsimplex}. Denote by $\overrightarrow{T}$ the set
of all digraph simplexes constructed in this way. That is, $i_{0}...i_{n}\in 
\overrightarrow{T}$ if $i_{0}...i_{n}$ is a monotone increasing sequence
that determines a simplex from $T$.

Then consider the following $n$-path on the digraph $G=\left( V,E\right) $:%
\begin{equation}
\sigma =\sum_{i_{0}...i_{n}\in \overrightarrow{T}}\sigma
^{i_{0}...i_{n}}e_{i_{0}...i_{n}}.  \label{sigma}
\end{equation}%
This path is allowed on $G$ by the definition of the orientation of the
edges.

We claim that the path $\sigma $ is closed, that is, $\partial \sigma =0$,
which, in particular, implies that $\sigma $ is $\partial $-invariant.
Observe that $\partial \sigma $ is the a linear combination with
coefficients $\pm 1$ of the terms $e_{j_{0}...j_{n-1}}$ where the sequence $%
j_{0},...,j_{n-1}$ is monotone increasing and forms an $\left( n-1\right) $%
-dimensional face of one of the $n$-simplexes from $T$. In fact, every $%
\left( n-1\right) $-face arises from \emph{two} $n$-simplexes, say $%
A=j_{0}...j_{k-1}aj_{k}...j_{n-1}$ and $B=j_{0}...j_{l-1}bj_{l}...j_{n-1}$
(cf. Fig. \ref{pic21}).

\FRAME{ftbphFU}{5.847in}{2.606in}{0pt}{\Qcb{Two $n$-simplexes $A,B$ having a
common $\left( n-1\right) $-dimensional face $j_{0}...j_{n-1}$}}{\Qlb{pic21}%
}{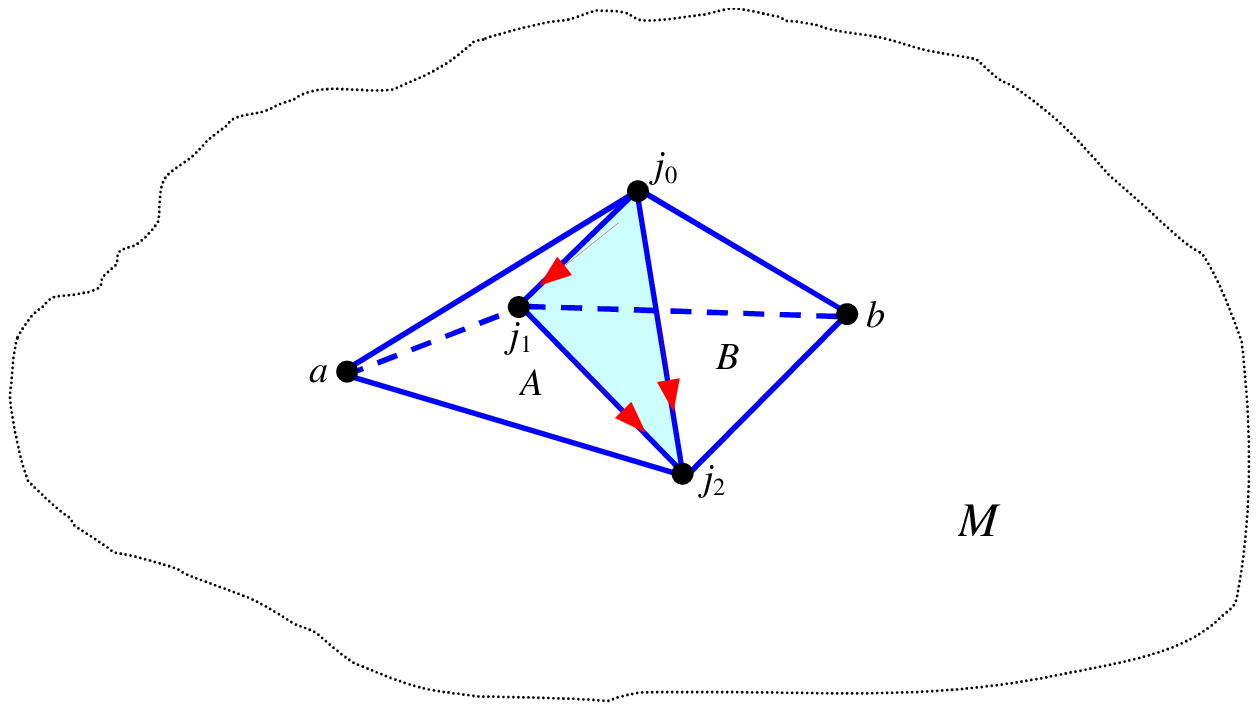}{\special{language "Scientific Word";type
"GRAPHIC";maintain-aspect-ratio TRUE;display "USEDEF";valid_file "F";width
5.847in;height 2.606in;depth 0pt;original-width 6.4643in;original-height
2.8632in;cropleft "0";croptop "1";cropright "1";cropbottom "0";filename
'pic21.eps';file-properties "XNPEU";}}

We have by (\ref{dev})%
\begin{equation*}
\partial e_{j_{0}...j_{k-1}aj_{k}...j_{n-1}}=...+\left( -1\right)
^{k}e_{j_{0}...j_{k-1}j_{k}...j_{n-1}}+...~.
\end{equation*}%
Since interchanging the order of two neighboring vertices in an $n$-simplex
changes its orientation, we have 
\begin{equation*}
\sigma ^{j_{0}...j_{k-1}aj_{k}...j_{n-1}}=\left( -1\right) ^{k}\sigma
^{aj_{0}...j_{k-1}j_{k}...j_{n-1}}.
\end{equation*}%
Multiplying the above lines, we obtain%
\begin{equation*}
\partial \left( \sigma ^{A}e_{A}\right) =...+\sigma
^{aj_{0}...j_{n-1}}e_{j_{0}...j_{n-1}}+...\ ,
\end{equation*}%
and in the same way%
\begin{equation*}
\partial \left( \sigma ^{B}e_{B}\right) =...+\sigma
^{bj_{0}...j_{n-1}}e_{j_{0}...j_{n-1}}+...
\end{equation*}%
However, the vertices $a$ and $b$ are located on the opposite sides of the
face $j_{0}...j_{n-1}$, which implies that the simplexes $aj_{0}...j_{n-1}$
and $bj_{0}...j_{n-1}$ have the opposite orientations relative to that of $M$%
. Hence,%
\begin{equation*}
\sigma ^{aj_{0}...j_{n-1}}+\sigma ^{bj_{0}...j_{n-1}}=0,
\end{equation*}%
which means that the term $e_{j_{0}...j_{n-1}}$ cancels out in the sum $%
\partial \left( \sigma ^{A}e_{A}+\sigma ^{B}e_{B}\right) $ and, hence, in $%
\partial \sigma .$ This proves that $\partial \sigma =0.$

The closed paths $\sigma $ defined by (\ref{sigma}) is called a \emph{%
surface path} of $M$ (or $T$).

There is a number of triangulations when a surface path $\sigma $ happens to
be exact, that is, $\sigma =\partial v$ for some $\left( n+1\right) $-path $%
v $. If this is the case then $v$ is called a \emph{solid path} as in this
case $v$ represents a \textquotedblleft solid\textquotedblright\ shape whose
boundary is given by $M$ (or $T$).

\begin{example}
\RM If $M=\mathbb{S}^{1}$ then $T$ is a cycle graph, and a surface path $%
\sigma $ in this case was constructed in the proof of Proposition \ref%
{Pcycle}. We have seen there that $\sigma $ is exact if the digraph $G$ is a
triangle or square, and non-exact otherwise. In the former case a solid
paths $v$ represents a triangle or a square, respectively, in the latter
case a solid path does not exist.
\end{example}

\begin{example}
\RM Let $M=\mathbb{S}^{n}$ and let the faces of a triangulation $T$ of $M$
form a $\left( n+1\right) $-simplex, so that the digraph $G$ is a $\left(
n+1\right) $-simplex digraph. Denoting the vertices by $0,1,...,n+1$, we
obtain $\partial e_{0...n+1}=\sigma $ so that $e_{i_{0}...i_{n+1}}$ is a
solid path representing a solid $\left( n+1\right) $-simplex.
\end{example}

There are also higher dimensional examples when a surface path is not exact,
see Example \ref{Expic13} below. Further examples of surface and solid paths
will be given in Section \ref{SecProduct}.

\subsection{Lemma of Sperner revisited}

Consider a triangle $ABC$ on the plane $\mathbb{R}^{2}$ and its
triangulation $T$. The set of vertices of $T$ is colored with three colors $%
1,2,3$ in such a way that

\begin{itemize}
\item the vertices $A,B,C$ are colored with $1,2,3$ respectively;

\item each vertex on any edge of $ABC$ is colored with one of the two colors
of the endpoints of the edge (see Fig. \ref{pic12}).
\end{itemize}

\FRAME{ftbphFU}{5.0704in}{2.2278in}{0pt}{\Qcb{A Sperner coloring}}{\Qlb{pic12%
}}{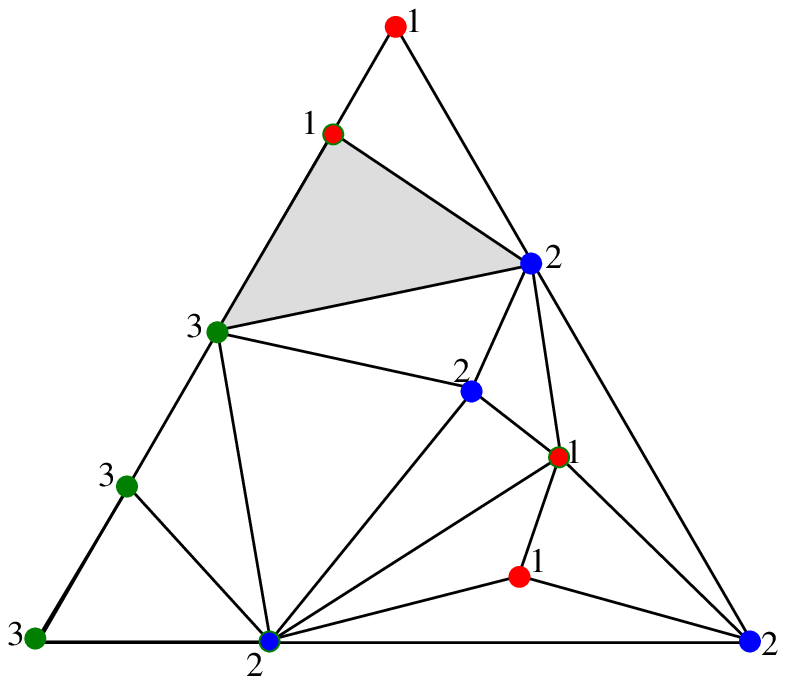}{\special{language "Scientific Word";type
"GRAPHIC";maintain-aspect-ratio TRUE;display "USEDEF";valid_file "F";width
5.0704in;height 2.2278in;depth 0pt;original-width 6.3027in;original-height
2.5417in;cropleft "0";croptop "1";cropright "1";cropbottom "0";filename
'pic12.eps';file-properties "XNPEU";}}

The classical lemma of Sperner says that then there exists in $T$ a $3$%
-color triangle, that is, a triangle, whose vertices are colored with the
three different colors. Moreover, the number of such triangles is odd.

We give here a new proof using the boundary operator $\partial $ for $1$%
-paths. Although this proof is no shorter that the classical proof based on
a double counting argument, it still provides a new insight into the
subject, that $3$-color triangles appear as sources and sinks of some
\textquotedblleft vector field\textquotedblright\ on a digraph.

Let us first do some reduction. Firstly, let us modify the triangulation $T$
so that there are no vertices on the edges $AB,AC,BC$ except for $A,B,C.$
Indeed, if $X$ is a vertex on $AB$ then we move $X$ a bit inside the
triangle $ABC.$ This gives rise to a new triangle in the triangulation $T$
that is formed by $X$ and its former neighbors, say $Y$ and $Z$, on the edge 
$AB$ (while keeping all old triangles). However, since all $X,Y,Z$ are
colored with two colors, no $3$-color triangle emerges after that move. By
induction, we remove all the vertices from the edges of $ABC.$

Secondly, we project the triangle $ABC$ and the triangulation $T$ onto the
sphere $\mathbb{S}^{2}$ and add to the set $T$ the triangle $ABC$ itself
from the other side of the sphere. Then we obtain a triangulation of $%
\mathbb{S}^{2},$ denote it again by $T$, and we need to prove that the
number of 3-color triangles is \emph{even}. Indeed, since we know that one
of the triangles, namely, $ABC$ is 3-color, this would imply that the number
of $3$-color triangles in the original triangulation is odd.

Let us regard $T$ as a graph on $\mathbb{S}^{2}$ and construct a dual graph $%
V$. Chose at each face of $T$ a point and regard them as vertices of the
dual graph $V$. The vertices in $V$ are connected if the corresponding
triangles in $T$ have a common edge (see Fig. \ref{pic17}). Then the faces
of $V$ are in one-to-one correspondence to the vertices of $T.$

\FRAME{ftbphFU}{5.0704in}{1.6942in}{0pt}{\Qcb{Construction of a dual graph}}{%
\Qlb{pic17}}{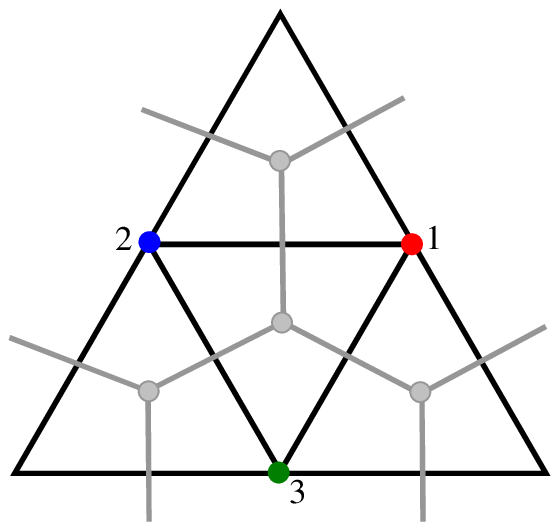}{\special{language "Scientific Word";type
"GRAPHIC";maintain-aspect-ratio TRUE;display "USEDEF";valid_file "F";width
5.0704in;height 1.6942in;depth 0pt;original-width 6.4643in;original-height
2.2485in;cropleft "0";croptop "1";cropright "1";cropbottom "0";filename
'pic17.eps';file-properties "XNPEU";}}

Hence, given a graph $V$ on $\mathbb{S}^{2}$ such that each vertex has
degree $3$ and each face is colored with one of the colors $1,2,3,$ we need
to prove that the number of $3$-color vertices (that is, the vertices, whose
adjacent faces have all three colors) is even.

Let us make $V$ into a digraph as follows. Each edge $\xi $ in $V$ has two
adjacent faces. Choose the orientation on $\xi $ so that the color from the
left hand side and that from the right hand side of $\xi $ form one of the
following pairs: $\left( 1,2\right) ,\left( 2,3\right) ,\left( 3,1\right) $
(see Fig. \ref{pic18}), while if the colors are the same then allow both
orientations of $\xi $.

\FRAME{ftbphFU}{3.9678in}{0.6175in}{0pt}{\Qcb{The orientation of an edge
depends on the colors of adjacent faces}}{\Qlb{pic18}}{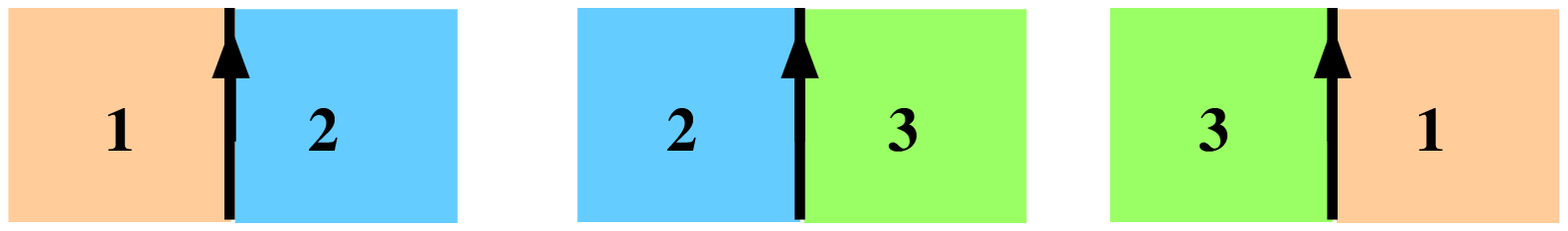}{\special%
{language "Scientific Word";type "GRAPHIC";maintain-aspect-ratio
TRUE;display "USEDEF";valid_file "F";width 3.9678in;height 0.6175in;depth
0pt;original-width 6.7348in;original-height 1.2205in;cropleft "0";croptop
"1";cropright "1";cropbottom "0";filename 'pic18.eps';file-properties
"XNPEU";}}

Examples of such orientations are shown on Fig. \ref{pic17b}.

\FRAME{ftbphFU}{5.0704in}{1.7815in}{0pt}{\Qcb{Oriented edges in the dual
graph $\left( V,E\right) $}}{\Qlb{pic17b}}{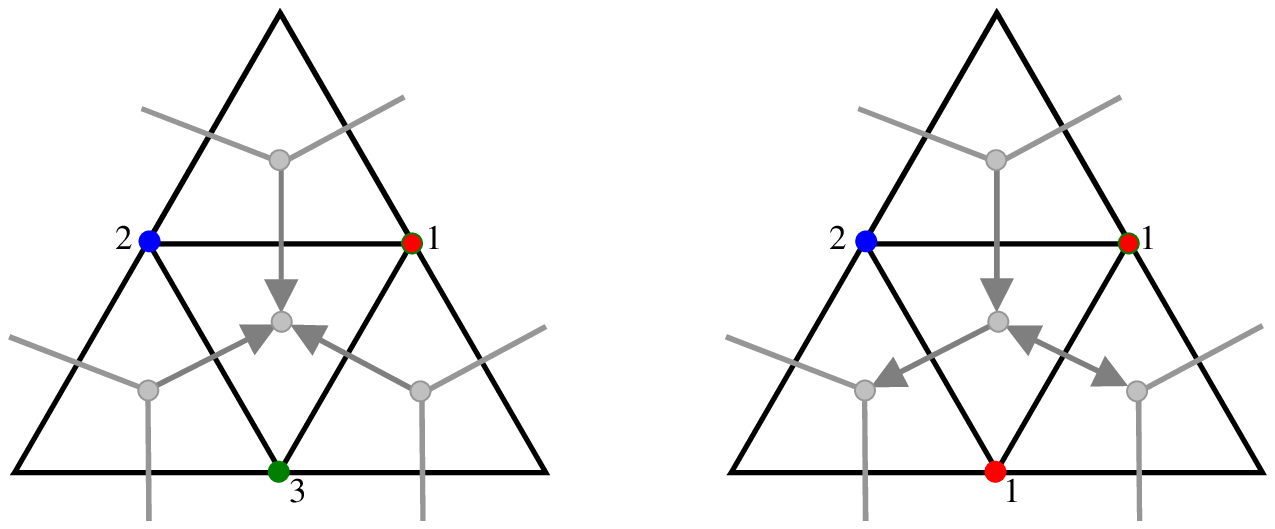}{\special{language
"Scientific Word";type "GRAPHIC";maintain-aspect-ratio TRUE;display
"USEDEF";valid_file "F";width 5.0704in;height 1.7815in;depth
0pt;original-width 6.4643in;original-height 2.2485in;cropleft "0";croptop
"1";cropright "1";cropbottom "0";filename 'pic17b.eps';file-properties
"XNPEU";}}

Denote by $E$ the set of the oriented edges and set $v=\sum_{\left\{ ab\in
E\right\} }e_{ab}.$ We have for any $a\in V$ 
\begin{equation*}
\left( \partial v\right) _{a}=\sum_{b}v^{ba}-\sum_{c}v^{ac}=\#\{\text{%
incoming edges}\}-\#\{\text{outcoming edges}\},
\end{equation*}%
where $\#A$ denotes the number of elements in the set $A$. If $a$ is $3$%
-color, then either all three edges at $a$ are incoming or all are
outcoming, whence $\left( \partial v\right) _{a}=3$ or $-3$, respectively.
If $a$ is not $3$-color then $\left( \partial v\right) _{a}=0$ (cf. Fig. \ref%
{pic17b}). Denoting by $n_{1}$ the number of $3$-color edges with incoming
orientation and by $n_{2}$ that with outcoming orientation, we obtain that $%
\left( \partial v,1\right) =3\left( n_{1}-n_{2}\right) $. On the other hand, 
$\left( \partial v,1\right) =\left( v,d1\right) =0$ whence we conclude that $%
n_{1}=n_{2}$. In particular, the total number of $3$-color vertices is $%
2n_{1}$, that is, even, which was to be proved. In fact, we have proved a
bit more: in a triangulation of a sphere, the numbers of $3$-color triangles
of the opposite orientations are the same.

\section{Homologies of subgraphs}

\setcounter{equation}{0}\label{Sec7}

\subsection{Chain complex of a subgraph}

\label{SecV'E'}Let $G^{\prime }=\left( V^{\prime },E^{\prime }\right) $ and $%
G=\left( V,E\right) $ be two digraph. We say that $G^{\prime }$ is a
subgraph of $G$ if $V^{\prime }\subset V$ and $E^{\prime }\subset E$. Let us
mark by the dash "$^{\prime }$" all the notation related to the graph $%
G^{\prime }$ rather than to $G$, for example, $\mathcal{R}_{p}^{\prime
}\equiv \mathcal{R}_{p}\left( G^{\prime }\right) $ while $\mathcal{R}%
_{p}\equiv \mathcal{R}_{p}\left( G\right) .$

As it was already observed (cf. (\ref{V'inV})), $\mathcal{R}_{p}^{\prime
}\subset \mathcal{R}_{p}$ and $\partial $ commutes with this inclusion. It
is also obvious that if $e_{i_{0}...i_{p}}$ is an allowed path in $G^{\prime
}$ then it is also allowed in $G$, whence $\mathcal{A}_{p}^{\prime }\subset 
\mathcal{A}_{p}.$ Moreover, this argument shows that%
\begin{equation}
\mathcal{A}_{p}^{\prime }=\mathcal{A}_{p}\cap \mathcal{R}_{p}^{\prime }.
\label{A'}
\end{equation}

By the definition (\ref{Omdef}) of $\Omega _{p}$, we obtain that $\Omega
_{p}^{\prime }\subset \Omega _{p}$ and $\partial $ commutes with this
inclusion. Consequently, the chain complex 
\begin{equation*}
\begin{array}{cccccccccc}
0 & \leftarrow & \Omega _{0}^{\prime } & \overset{\partial }{\leftarrow } & 
\Omega _{1}^{\prime } & \overset{\partial }{\leftarrow } & \Omega
_{2}^{\prime } & \overset{\partial }{\leftarrow } & \Omega _{3}^{\prime } & 
\overset{\partial }{\leftarrow }\dots%
\end{array}%
\end{equation*}%
is a sub-complex of 
\begin{equation*}
\begin{array}{cccccccccc}
0 & \leftarrow & \Omega _{0} & \overset{\partial }{\leftarrow } & \Omega _{1}
& \overset{\partial }{\leftarrow } & \Omega _{2} & \overset{\partial }{%
\leftarrow } & \Omega _{3} & \overset{\partial }{\leftarrow }\dots%
\end{array}%
\end{equation*}%
By Proposition \ref{Plong} (cf. (\ref{longXJi})) we obtain that the
following long sequence is exact:

\begin{notes}[\ ]
\begin{equation}
0\leftarrow H_{0}(\Omega /\Omega ^{\prime })\leftarrow H_{0}(\Omega
)\leftarrow H_{0}(\Omega ^{\prime })\leftarrow \dots \leftarrow H_{p}(\Omega
/\Omega ^{\prime })\leftarrow H_{p}(\Omega )\leftarrow H_{p}(\Omega ^{\prime
})\leftarrow H_{p+1}(\Omega /\Omega ^{\prime })\leftarrow \dots
\label{Om/Om'}
\end{equation}
\end{notes}

It is also worth mentioning that%
\begin{equation}
\Omega _{p}^{\prime }=\Omega _{p}\cap \mathcal{A}_{p}^{\prime }=\Omega
_{p}\cap \mathcal{R}_{p}^{\prime }.  \label{Om'}
\end{equation}%
Indeed, the inclusions $\Omega _{p}^{\prime }\subset \Omega _{p}\cap 
\mathcal{A}_{p}^{\prime }\subset \Omega _{p}\cap \mathcal{R}_{p}^{\prime }$
are obvious. To prove the opposite inclusions, observe that $v\in \Omega
_{p}\cap \mathcal{R}_{p}^{\prime }$ implies by (\ref{Omdef}) and (\ref{A'}) 
\begin{equation*}
v\in \mathcal{A}_{p}\cap \mathcal{R}_{p}^{\prime }=\mathcal{A}_{p}^{\prime
}\ \ \ \text{and\ \ \ \ }\partial v\in \mathcal{A}_{p-1}\cap \mathcal{R}%
_{p-1}^{\prime }=\mathcal{A}_{p-1}^{\prime },
\end{equation*}%
whence $v\in \Omega _{p}^{\prime }.$

\subsection{Removing a vertex of degree $1$}

\label{SecLoose}

\begin{theorem}
\label{Tend}Suppose that a graph $G$ has a vertex $a$ such that there is
only one outcoming edge $a\rightarrow b$ from $a$ and no incoming edges to $%
a $. Let $G^{\prime }=\left( V^{\prime },E^{\prime }\right) $ be the digraph
with $V^{\prime }=V\setminus \left\{ a\right\} $ and $E^{\prime }=E\setminus
\left\{ ab\right\} $.%
\begin{equation*}
\fbox{$\ \ \ a~\bullet \longrightarrow $\fbox{$\bullet ~b\ \ \ \ 
\begin{array}{c}
\  \\ 
\ 
\end{array}%
\ G^{\prime }\ $}\ \ $\ G$}
\end{equation*}%
Then $H_{p}\left( G\right) \cong H_{p}(G^{\prime })$ for all $p\geq 0.$
\end{theorem}

\begin{remark}
\RM The same is true if the vertex $a$ has one incoming edge $b\rightarrow a$
and no outcoming edges.

Theorem \ref{Tend} is a particular case of a more general Theorem \ref{Tabac}
from the next section. We give here an independent proof based on the
identity (\ref{Omp'=}) below that may be interesting on its own right.
\end{remark}

\begin{proof}
Let us first prove that 
\begin{equation}
\Omega _{p}=\Omega _{p}^{\prime }\ \text{for all\ \ }p\geq 2,  \label{Omp'=}
\end{equation}%
which will imply that, for all $p\geq 2$, 
\begin{equation}
\dim H_{p}\left( \Omega ^{\prime }\right) =\dim H_{p}\left( \Omega \right) .
\label{Hp'}
\end{equation}%
In the view of (\ref{Om'}), to prove (\ref{Omp'=}) it suffices to show that,
for all $p\geq 2$,%
\begin{equation}
\Omega _{p}\subset \mathcal{A}_{p}^{\prime },  \label{OmA}
\end{equation}%
that is%
\begin{equation*}
v\in \mathcal{A}_{p}\ \ \text{and\ \ }\partial v\in \mathcal{A}_{p-1}\
\Rightarrow v\in \mathcal{A}_{p}^{\prime }.
\end{equation*}%
Every elementary allowed $p$-path on $G$ either is allowed on $G^{\prime }$
or starts with $ab$, which implies that $v$ can be represented in the form%
\begin{equation*}
v=e_{ab}u+v^{\prime },
\end{equation*}%
where $v^{\prime }\in \mathcal{A}_{p}^{\prime }$, while $u\in \mathcal{A}%
_{p-2}^{\prime }$ is a linear combination of the paths $e_{i_{0}...i_{p-2}}%
\in \mathcal{A}_{p-2}^{\prime }$ with $i_{0}\neq b$. It follows that%
\begin{equation}
\partial v=\left( e_{b}-e_{a}\right) u+e_{ab}\partial u+\partial v^{\prime }.
\label{veb-ea}
\end{equation}%
Note that $e_{a}u$ is a linear combination of the elementary paths $%
e_{ai_{0}...i_{p-2}}$ where $i_{0},...,i_{p-2}\in V^{\prime }$ and $%
i_{0}\neq b$. Since $ai_{0}$ is not an edge, those elementary paths are not
allowed in $G$. No other terms in the right hand side of (\ref{veb-ea}) has $%
e_{ai_{0}...i_{p-2}}$-component. Since $\partial v$ is allows in $G$, its $%
e_{ai_{0}...i_{p-2}}$-component is $0$, which is only possible if $e_{a}u=0$%
, that is, $u=0.$ It follows that $v=v^{\prime }\in \mathcal{A}_{p}^{\prime
} $, which finishes the proof of (\ref{OmA}).

Hence, we have the identity (\ref{Hp'}) for $p\geq 2$. For $p=0$ this
identity also true as the number of connected components of $G$ and $%
G^{\prime }$ is the same.

We are left to treat the case $p=1.$ Observe that 
\begin{equation}
\Omega _{0}=\Omega _{0}^{\prime }+\limfunc{span}\left\{ e_{a}\right\} \ \ 
\text{and \ }\Omega _{1}=\Omega _{1}^{\prime }+\limfunc{span}\left\{
e_{ab}\right\} .  \label{Om12=}
\end{equation}%
By (\ref{Omp'=}) and (\ref{Om12=}) the cochain complex $\Omega /\Omega
^{\prime }$ has the form 
\begin{equation*}
0\longleftarrow \limfunc{span}\left\{ e_{a}\right\} \overset{\partial }{%
\longleftarrow }\limfunc{span}\left\{ e_{ab}\right\} \longleftarrow 0=\Omega
_{2}/\Omega _{2}^{\prime }.
\end{equation*}%
Since%
\begin{equation*}
\partial e_{ab}=e_{b}-e_{a}=-e_{a}\func{mod}\Omega _{0}^{\prime },
\end{equation*}%
it follows that $\func{Im}\partial |_{\Omega _{1}/\Omega _{1}^{\prime }}=%
\limfunc{span}\left\{ e_{a}\right\} $, while $\ker \partial |_{\Omega
_{1}/\Omega _{1}^{\prime }}=0,$ whence 
\begin{equation*}
\dim H_{0}\left( \Omega /\Omega ^{\prime }\right) =\dim H_{1}\left( \Omega
/\Omega ^{\prime }\right) =0.\text{\ }
\end{equation*}%
By (\ref{Om/Om'}) we have a long exact sequence%
\begin{equation*}
H_{0}\left( \Omega /\Omega ^{\prime }\right) =0\longleftarrow H_{1}\left(
\Omega \right) \longleftarrow H_{1}\left( \Omega ^{\prime }\right)
\longleftarrow 0=H_{1}\left( \Omega /\Omega ^{\prime }\right)
\end{equation*}%
which implies that%
\begin{equation*}
\dim H_{1}\left( \Omega \right) =\dim H_{1}\left( \Omega ^{\prime }\right) ,
\end{equation*}%
thus finishing the proof.
\end{proof}

\begin{corollary}
Let a digraph $G$ be a tree \emph{(that is, the underlying undirected graph
is a tree)}. Then $H_{p}\left( G\right) =0$ for all $p\geq 1.$
\end{corollary}

\begin{proof}
Induction in the number of edges $\left\vert E\right\vert .$ If $\left\vert
E\right\vert =0$ then the claim is obvious. If $\left\vert E\right\vert >0$
then there is a vertex $a\in V$ of degree $1$ (indeed, if this is not the
case then moving along undirected edges allows to produce a cycle). Removing
this vertex and the adjacent edge, we obtain a tree $G^{\prime }$ with $%
\left\vert E^{\prime }\right\vert <\left\vert E\right\vert $. By the
inductive hypothesis $H_{p}\left( G^{\prime }\right) =0$ for $p\geq 1$,
whence by Theorem \ref{Tend} also $H_{p}\left( G\right) =0$.

That $H_{p}\left( G\right) =0$ for $p\geq 2$ follows also from Theorem \ref%
{Tsq}.
\end{proof}

\subsection{Removing a vertex of degree $n$}

\label{Sectriangle2}

\begin{theorem}
\label{Tabac}Suppose that a digraph $G=\left( V,E\right) $ has a vertex $a$
with $n$ outcoming edges $a\rightarrow b_{0},a\rightarrow
b_{1},...,a\rightarrow b_{n-1}$ and no incoming edges. Assume also that $%
b_{0}\rightarrow b_{i}$ for all $i\geq 1$:%
\begin{equation*}
\fbox{$%
\begin{array}{c}
\  \\ 
a\overset{\nearrow }{\underset{\searrow }{\bullet \rightarrow }} \\ 
\ 
\end{array}%
$\fbox{$%
\begin{array}{c}
\bullet \ b_{1}\  \\ 
\overset{\uparrow }{\underset{\downarrow }{\bullet }}\ b_{0}\  \\ 
\bullet \ b_{2}\ 
\end{array}%
$\ $\cdots $\ \ \ $G^{\prime }$}\ \ \ $G$}
\end{equation*}%
Denote by $G^{\prime }=\left( V^{\prime },E^{\prime }\right) $ the digraph
that is obtained from $G$ by removing a vertex $a$ with all adjacent edges,
that is, $V^{\prime }=V\setminus \left\{ a\right\} $ and $E^{\prime
}=E\setminus \left\{ ab_{i}\right\} _{i=0}^{n-1}$. Then, for any $p\geq 0$, 
\begin{equation}
\dim H_{p}\left( G\right) =\dim H_{p}\left( G^{\prime }\right) .
\label{Hpp'}
\end{equation}
\end{theorem}

\begin{remark}
\RM The same is true if a vertex $a$ has $n$ incoming edges $%
b_{0}\rightarrow a,b_{1}\rightarrow a,...,b_{n-1}\rightarrow a$ and no
outcoming edges, while $b_{i}\rightarrow b_{0}$ for all $i\geq 1$:%
\begin{equation*}
\fbox{$%
\begin{array}{c}
\  \\ 
a\overset{\swarrow }{\underset{\nwarrow }{\bullet \leftarrow }} \\ 
\ 
\end{array}%
$\fbox{$%
\begin{array}{c}
\bullet \ b_{1}\  \\ 
\overset{\downarrow }{\underset{\uparrow }{\bullet }}\ b_{0}\  \\ 
\bullet \ b_{2}\ 
\end{array}%
$\ $\cdots $\ \ \ $G^{\prime }$}\ \ \ $G$}
\end{equation*}%
Theorem \ref{Tabac} can be regarded as a discrete analogous of the classical
result of homotopy invariance of homologies on manifolds.
\end{remark}

\begin{example}
\RM Consider a digraph $G$ as shown in Fig. \ref{pic5b}.\FRAME{ftbphFU}{%
6.3304in}{1.567in}{0pt}{\Qcb{A digraph with many triangles and squares}}{%
\Qlb{pic5b}}{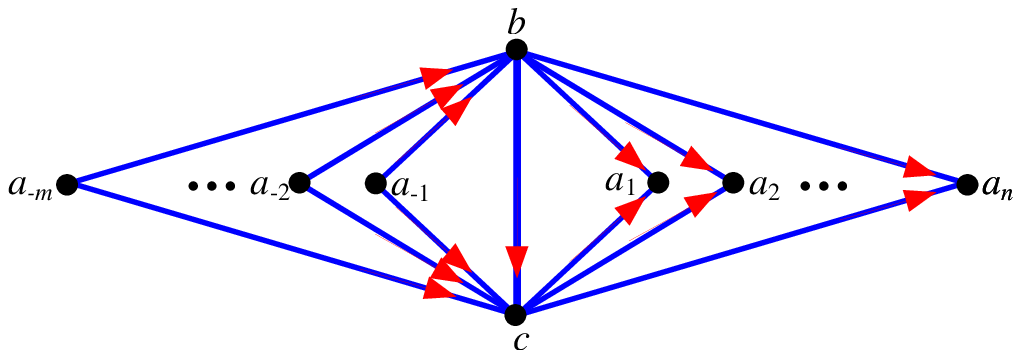}{\special{language "Scientific Word";type
"GRAPHIC";maintain-aspect-ratio TRUE;display "USEDEF";valid_file "F";width
6.3304in;height 1.567in;depth 0pt;original-width 6.3027in;original-height
1.5939in;cropleft "0";croptop "1";cropright "1";cropbottom "0";filename
'pic5b.eps';file-properties "XNPEU";}}

Each of the vertices $a_{k}$ satisfies the hypotheses of Theorem \ref{Tabac}
with $n=2$ (either with incoming or outcoming edges). Removing successively
the vertices $a_{k}$, we see that all the homologies of $G$ are the same as
those of the remaining graph $^{b}\bullet \rightarrow \bullet ^{c}$. Since
it is a star-shaped graph, we obtain $\dim H_{0}=1$ and $\dim H_{p}=0$ for
all $p\geq 1.$ In particular, $\chi =1$.

Note that for this digraph 
\begin{equation*}
\dim \Omega _{0}=m+n+2,\ \ \dim \Omega _{1}=2m+2n+1.
\end{equation*}%
Using Proposition \ref{PropS} and observing that the number of semi-edges is 
$mn$, we obtain 
\begin{equation*}
\dim \Omega _{2}=m+n+mn,
\end{equation*}
where the basis in $\Omega _{2}$ is given by the triangles $e_{a_{-k}bc}$,\ $%
e_{bca_{l}}$ and squares $e_{a_{-k}ba_{l}}-e_{a_{-k}ca_{l}}$. Finally, $\,$%
we have%
\begin{equation*}
\dim \Omega _{3}=\dim \mathcal{A}_{3}=mn
\end{equation*}
where the basis in $\Omega _{3}$ is determined by the snakes $%
e_{a_{-k}bca_{l}}$, and $\dim \Omega _{p}=0$ for all $p>3.$
\end{example}

\begin{proof}[Proof of Theorem \protect\ref{Tabac}]
Since the number of connected components of the graphs $G$ and $G^{\prime }$
is obviously the same, the identity (\ref{Hpp'}) for $p=0$ follows from
Proposition \ref{PH0}.

For $p\geq 1$ consider the long exact sequence (\ref{Om/Om'}), that is, 
\begin{equation*}
...\leftarrow H_{p}\left( \Omega /\Omega ^{\prime }\right) \leftarrow
H_{p}\left( \Omega \right) \leftarrow H_{p}\left( \Omega ^{\prime }\right)
\leftarrow H_{p+1}\left( \Omega /\Omega ^{\prime }\right) \leftarrow ...,
\end{equation*}%
that implies the identity 
\begin{equation*}
\dim H_{p}\left( \Omega \right) =\dim H_{p}\left( \Omega ^{\prime }\right) \
\ \text{for }p\geq 1,
\end{equation*}%
if we prove that%
\begin{equation}
\dim H_{p}\left( \Omega /\Omega ^{\prime }\right) =0\ \ \text{for }p\geq 1.
\label{p/p'}
\end{equation}%
The condition (\ref{p/p'}) means that 
\begin{equation*}
\ker \partial |_{\Omega _{p}/\Omega _{p}^{\prime }}\subset \func{Im}\partial
|_{\Omega _{p+1}/\Omega _{p+1}^{\prime }}
\end{equation*}%
that is, if 
\begin{equation}
v\in \Omega _{p}\ \ \text{and\ \ }\partial v=0\func{mod}\Omega
_{p-1}^{\prime }  \label{vdv}
\end{equation}%
then there exists $\omega \in \Omega _{p+1}$ such that 
\begin{equation}
\partial \omega =v\func{mod}\Omega _{p}^{\prime }.  \label{wv}
\end{equation}%
In fact, it suffices to prove the existence of $\omega \in \mathcal{A}_{p+1}$
such that 
\begin{equation}
\partial \omega =v\func{mod}\mathcal{A}_{p}^{\prime }.  \label{wva}
\end{equation}%
Indeed, (\ref{wva}) implies $\partial \omega \in \mathcal{A}_{p}$ and,
hence, $\omega \in \Omega _{p+1}.$ Since $\partial \omega -v\in \mathcal{A}%
_{p}^{\prime }$ and%
\begin{equation*}
\partial \left( \partial \omega -v\right) =-\partial v\in \mathcal{A}%
_{p-1}^{\prime },
\end{equation*}%
it follows that $\partial \omega -v\in \Omega _{p}^{\prime }$ which proves (%
\ref{wv}).

To prove the existence of $\omega $ as above, observe that $v$ can be
represented in the form%
\begin{equation}
v=e_{a}u~\func{mod}\mathcal{A}_{p}^{\prime },  \label{u+wA}
\end{equation}%
where $u\in \mathcal{A}_{p-1}^{\prime }$ and $e_{a}u\in \mathcal{A}_{p}.$ We
have then%
\begin{equation*}
\partial v=u-e_{a}\partial u~\func{mod}\mathcal{R}_{p-1}^{\prime }.
\end{equation*}%
Since $\partial v,u\in \mathcal{A}_{p-1}^{\prime },$ it follows that 
\begin{equation}
e_{a}\partial u=0\ \func{mod}\mathcal{R}_{p-1}^{\prime }.  \label{eadu}
\end{equation}%
However, all components of the path $e_{a}\partial u$ start with $e_{a...}$,
whereas the condition (\ref{eadu}) means that the path $e_{a}\partial u$ has
no such component. Hence, (\ref{eadu}) is only possible if $\partial u=0.$

Since $e_{a}u\in \mathcal{A}_{p}$, any component of $u$ has the form $%
e_{b_{i}...}$, which, together with the hypothesis that $b_{0}b_{i}$ is an
edge, implies that $e_{b_{0}}u\in \mathcal{A}_{p}^{\prime }$ and $%
e_{ab_{0}}u\in \mathcal{A}_{p+1}$. Using $\partial u=0$ and (\ref{u+wA}), we
obtain%
\begin{equation*}
\partial \left( e_{ab_{0}}u\right) =\left( e_{b_{0}}-e_{a}\right)
u+e_{ab_{0}}\partial u=e_{b_{0}}u-e_{a}u=-v~\func{mod}\mathcal{A}%
_{p}^{\prime },
\end{equation*}%
so that (\ref{wva}) holds with $\omega =-e_{ab_{0}}u.$
\end{proof}

\subsection{Removing a vertex of degree $1+1$}

\label{Sectriangles}Recall that a pair $cb$ of distinct vertices on a graph
is a \emph{semi-edge} if $cb$ is not an edge but there is a vertex $j$ such
that $cj$ and $jb$ are edges:%
\begin{equation*}
\begin{array}{c}
\bullet b \\ 
\ \upharpoonleft \  \\ 
\bullet c%
\end{array}%
\ _{\nearrow }^{\nwarrow }\bullet j
\end{equation*}%
In the next theorem the field $K$ has characteristic $0.$

\begin{theorem}
\label{Tloostri}\label{Tcab}Suppose that a graph $G=\left( V,E\right) $ has
a vertex $a$ such that there is only one outcoming edge $a\rightarrow b$
from $a$ and only one incoming edge $c\rightarrow a$, where $b\neq c.$
Consider the digraph $G^{\prime }=\left( V^{\prime },E^{\prime }\right) $
where $V^{\prime }=V\setminus \left\{ a\right\} $ and $E^{\prime
}=E\setminus \left\{ ab,ca\right\} $. 
\begin{equation*}
\fbox{$\ \ \ a~\bullet _{\nwarrow }^{\nearrow }$\fbox{$%
\begin{array}{c}
\bullet b \\ 
\vdots \  \\ 
\bullet c%
\end{array}%
\ \ ~\ \ \ \ \ G^{\prime }\ $}\ \ $\ G$}
\end{equation*}%
Then the following is true.

\begin{itemize}
\item[$\left( a\right) $] For any $p\geq 2$,%
\begin{equation}
\dim H_{p}\left( G\right) =\dim H_{p}(G^{\prime }).  \label{p=p}
\end{equation}

\item[$\left( b\right) $] If $cb$ is an edge or a semi-edge in $G^{\prime }$
then \emph{(\ref{p=p})} is satisfied also for $p=0,1$, that is, for all $%
p\geq 0.$

\item[$\left( c\right) $] If $cb$ is neither edge nor semi-edge in $%
G^{\prime }$, but $b,c$ belong to the same connected component of $G^{\prime
}$ then 
\begin{equation*}
\dim H_{1}\left( G\right) =\dim H_{1}(G^{\prime })+1\ \ 
\end{equation*}%
and $\dim H_{0}\left( G\right) =\dim H_{0}\left( G^{\prime }\right) .$

\item[$\left( d\right) $] If $b,c$ belong to different connected components
of $G^{\prime }$ then 
\begin{equation*}
\dim H_{1}\left( G\right) =\dim H_{1}(G^{\prime })
\end{equation*}%
and $\dim H_{0}\left( G\right) =\dim H_{0}(G^{\prime })-1.$\ 
\end{itemize}

Consequently, in the case $\left( b\right) ,$ $\chi \left( G\right) =\chi
\left( G^{\prime }\right) ,$ whereas in the cases $\left( c\right) $ and $%
\left( d\right) ,$ $\chi \left( G\right) =\chi \left( G^{\prime }\right) -1.$
\end{theorem}

\begin{example}
\RM Consider the graphs 
\begin{equation*}
G=%
\begin{array}{ccc}
& \overset{b}{\bullet } &  \\ 
^{a}\bullet _{\nwarrow }^{\nearrow } & \downarrow & _{\nearrow }^{\nwarrow
}\bullet ^{d} \\ 
& \underset{c}{\bullet } & 
\end{array}%
\ \ \text{and\ \ }G^{\prime }=%
\begin{array}{cc}
\overset{b}{\bullet } &  \\ 
\downarrow & _{\nearrow }^{\nwarrow }\bullet ^{d} \\ 
\underset{c}{\bullet } & 
\end{array}%
\ 
\end{equation*}%
Since $cb$ is semi-edge in $G^{\prime }$ we have case $\left( b\right) $ so
that all homologies of $G$ and $G^{\prime }$ are the same. Removing further
vertex $d$ we obtain a digraph $\ $%
\begin{equation*}
G^{\prime \prime }=\left. ^{b}\bullet \rightarrow \bullet ^{c}\right.
\end{equation*}%
It is a star-shaped graph whence $\dim H_{p}\left( G^{\prime \prime }\right)
=0$ for $p\geq 1.$ Since $cb$ is neither edge nor semi-edge in $G^{\prime
\prime }$, but the graph $G^{\prime \prime }$ is connected, we conclude by
case $\left( c\right) $ that 
\begin{equation*}
\dim H_{1}\left( G^{\prime }\right) =\dim H_{1}\left( G^{\prime \prime
}\right) +1=1.
\end{equation*}%
and $H_{p}\left( G^{\prime }\right) =H_{p}\left( G^{\prime \prime }\right) $
for $p\geq 2$. It follows that $\dim H_{p}\left( G\right) =0$ for $p\geq 2$
and $\dim H_{1}\left( G\right) =1$.
\end{example}

\begin{example}
\label{Exantisnake}\RM Consider a digraph as on Fig. \ref{pic10a} (a kind of
anti-snake).\FRAME{ftbphFU}{5.7009in}{1.2756in}{0pt}{\Qcb{An anti-snake}}{%
\Qlb{pic10a}}{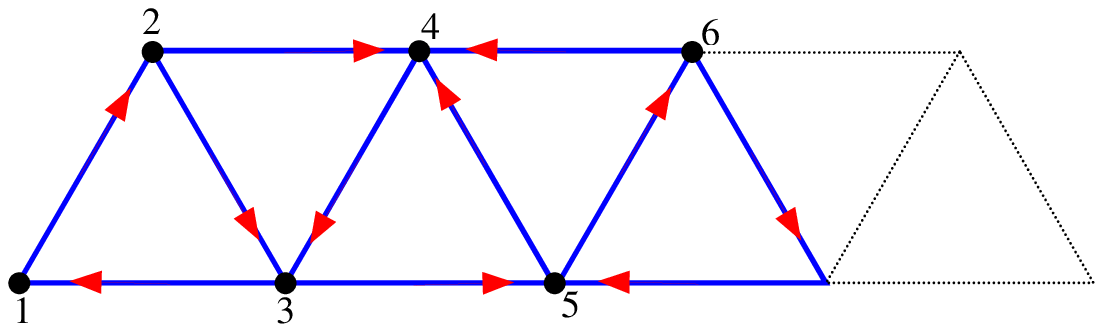}{\special{language "Scientific Word";type
"GRAPHIC";maintain-aspect-ratio TRUE;display "USEDEF";valid_file "F";width
5.7009in;height 1.2756in;depth 0pt;original-width 6.4643in;original-height
1.4218in;cropleft "0";croptop "1";cropright "1";cropbottom "0";filename
'pic10a.eps';file-properties "XNPEU";}}

We start building this graph with $1\rightarrow 2.$ Since $21$ is neither
edge nor semi-edge, adding a path $2\rightarrow 3\rightarrow 1$ increases $%
\dim H_{1}$ by $1$ and preserves other homologies. Since $23$ is an edge,
adding a path $2\rightarrow 4\rightarrow 3$ preserves all homologies. Since $%
34$ is neither edge nor semi-edge, adding a path $3\rightarrow 5\rightarrow
4 $ increases $\dim H_{1}$ by $1$ and preserves other homologies. Similarly,
adding a path $5\rightarrow 6\rightarrow 4$ preserves all homologies.

One can repeat this pattern arbitrarily many times. By doing so we construct
a digraph with a prescribed positive value of $\dim H_{1}$ while keeping $%
\dim H_{p}=0$ for all $p\geq 2$. Consequently, the Euler characteristic $%
\chi $ can take arbitrary negative values.
\end{example}

\begin{example}
\RM Consider a digraph on Fig. \ref{pic33}. By Theorem \ref{Tabac}, we can
remove the vertices $5$ and $8$ (and their adjacent edges) without change of
homologies. Then by the same theorem we can remove $4$ and $7$. By Theorem %
\ref{Tcab} we can remove the vertex $1.$ The resulting graph with the
vertices $0,2,3,6$ is star-shaped, so that by Theorem \ref{Tstar} the
homology groups $H_{p}$ are trivial for all $p\geq 1$, while $\dim H_{0}=1.$
\end{example}

\begin{proof}[Proof of Theorem \protect\ref{Tcab}]
\emph{Proof of }$\left( a\right) .$ The identity (\ref{p=p}) for $p\geq 2$
will follow if if prove that 
\begin{equation}
\dim H_{p}\left( \Omega /\Omega ^{\prime }\right) =0\ \text{for }p\geq 2.
\label{HpO}
\end{equation}%
In order to prove (\ref{HpO}) it suffices to show that%
\begin{equation*}
\ker \partial |_{\Omega _{p}/\Omega _{p}^{\prime }}=0,
\end{equation*}%
which is equivalent to 
\begin{equation}
v\in \Omega _{p},\ \ \partial v=0\func{mod}\Omega _{p-1}^{\prime
}\Rightarrow v=0\func{mod}\Omega _{p}^{\prime }.  \label{vpp-1}
\end{equation}%
By the definition (\ref{Omdef}) of $\Omega _{p}$, (\ref{vpp-1}) is
equivalent to 
\begin{equation}
v\in \mathcal{A}_{p}\text{\ \ and}\ \ \partial v\in \mathcal{A}%
_{p-1}^{\prime }\Rightarrow v\in \mathcal{A}_{p}^{\prime }.  \label{App-1}
\end{equation}%
Hence, let us prove (\ref{App-1}) for all $p\geq 2.$

Every elementary allowed $p$-path on $G$ either contains one of the edges $%
ab $, $ca$ or is allowed in $G^{\prime }$. Let us show that, for any $v$ as
in (\ref{App-1}), its components $v^{...ab...}$ and $v^{...ca...}$ vanish,
which will imply that $v\in \mathcal{A}_{p}^{\prime }$. Any such component
can be written in the form $v^{\alpha ab\beta }$ or $v^{\gamma ca\beta }$
where $\alpha ,\beta ,\gamma $ are some paths. Consider the following cases.
For further applications, in the Cases 1,2 we assume only that $v\in \Omega
_{p}$ (whereas in the Case 3 $v$ is as in (\ref{App-1})).

\emph{Case 1.} Let us consider first the component $v^{\alpha ab\beta }$
where $\beta $ is non-empty. If $\alpha ab\beta $ is not allowed in $G$ then 
$v^{\alpha ab\beta }=0$ by definition. Let $\alpha ab\beta $ be allowed in $%
G $. The path $\alpha a\beta $ is not allowed because the only outcoming
edge from $a$ is $ab$. Since $\partial v\in \mathcal{A}_{p-1}$, we have 
\begin{equation*}
\left( \partial v\right) ^{\alpha a\beta }=0.
\end{equation*}%
Let us show that%
\begin{equation}
\left( \partial v\right) ^{\alpha a\beta }=\pm v^{\alpha ab\beta },
\label{vab}
\end{equation}%
which will imply $v^{\alpha ab\beta }=0.$ Indeed, by (\ref{dv}) $\left(
\partial v\right) ^{\alpha a\beta }$ is the sum of the terms $\pm v^{\omega
} $ where $\omega $ is a $p$-path that is obtained from $\alpha a\beta $ by
inserting one vertex. Since there is no edge from $a$ to $\beta $, the only
way $\omega $ can be allowed is when $\omega =\alpha ab\beta $. Since for
any other $\omega $ we have $v^{\omega }=0$, we obtain (\ref{vab}), which
implies that $v^{\alpha ab\beta }=0.$

\emph{Case 2.} In the same way one proves that $v^{\gamma ca\beta }=0$
provided $\gamma $ is non-empty, using the fact that the only incoming edge
in $a$ is $ca.$

\emph{Case 3.} Consider now an arbitrary component $v^{\alpha ab\beta }.$ If 
$\beta $ is non-empty then $v^{\alpha ab\beta }=0$ by Case 1. Let $\beta $
be empty. Then $\alpha $ must have the form $\alpha =\gamma c$ so that $%
v^{\alpha ab\beta }=v^{\gamma cab}.$ If $\gamma $ is non-empty then $%
v^{\gamma cab}=0$ by Case 2. Finally, let $\gamma $ be also empty so that $%
v^{\alpha ab\beta }=v^{cab}$ (which is only possible if $p=2$). Since $%
\partial v\in \mathcal{A}_{1}^{\prime }$, we have%
\begin{equation*}
\left( \partial v\right) ^{ab}=0.
\end{equation*}%
On the other hand, 
\begin{equation*}
\left( \partial v\right) ^{ab}=\sum_{i\in V}v^{iab}-v^{aib}+v^{abi}.
\end{equation*}%
Here all the terms of the form $v^{iab}$ vanish, except possibly for $%
v^{cab} $, because $ia$ is not an edge unless $i=c$. All the terms $v^{aib}$
vanish because $ai$ is not an edge. All the terms $v^{abi}$ vanish by Case
1. Hence, we obtain 
\begin{equation*}
\left( \partial v\right) ^{ab}=v^{cab}
\end{equation*}%
whence $v^{cab}=0$ follows, thus finishing the proof of the part $\left(
a\right) .$

\emph{Proof of }$\left( b\right) ,\left( c\right) ,\left( d\right) .$ If $%
b,c $ belong to the same connected component of $G^{\prime }$ then the
number of connected components of $G$ and that of $G^{\prime }$ are the
same, so that%
\begin{equation}
\dim H_{0}\left( \Omega \right) =\dim H_{0}\left( \Omega ^{\prime }\right) ,
\label{0'=0}
\end{equation}%
whereas if $b,c$ belong to different connected components of $G^{\prime }$
then after joining them by $a$ the number of connected components reduces by 
$1$, so that%
\begin{equation}
\dim H_{0}\left( \Omega \right) =\dim H_{0}\left( \Omega ^{\prime }\right)
-1.  \label{0'=0+1}
\end{equation}

To handle $H_{1}$ we use the long exact sequence (\ref{Om/Om'}) that by (\ref%
{HpO}) has the form%
\begin{equation}
0\leftarrow H_{0}\left( \Omega /\Omega ^{\prime }\right) \leftarrow
H_{0}\left( \Omega \right) \leftarrow H_{0}\left( \Omega ^{\prime }\right)
\leftarrow H_{1}\left( \Omega /\Omega ^{\prime }\right) \leftarrow
H_{1}\left( \Omega \right) \leftarrow H_{1}\left( \Omega ^{\prime }\right)
\leftarrow 0.  \label{Om/Om'1}
\end{equation}%
Since we know already the relation between $H_{0}\left( \Omega ^{\prime
}\right) $ and $H_{0}\left( \Omega \right) $, to obtain the relation between 
$H_{1}\left( \Omega ^{\prime }\right) $ and $H_{1}\left( \Omega \right) $ we
need to compute $\dim H_{0}\left( \Omega /\Omega ^{\prime }\right) $ and $%
\dim H_{1}\left( \Omega /\Omega ^{\prime }\right) $ from the quotient
complex $\Omega /\Omega ^{\prime }$. Observe that%
\begin{equation}
\Omega _{0}=\Omega _{0}^{\prime }+\limfunc{span}\left\{ e_{a}\right\} ,\ \ 
\text{\ }\Omega _{1}=\Omega _{1}^{\prime }+\limfunc{span}\left\{
e_{ab},e_{ca}\right\}  \label{Om01'}
\end{equation}%
so that the quotient complex $\Omega /\Omega ^{\prime }$ has the form 
\begin{equation*}
0\longleftarrow \limfunc{span}\left\{ e_{a}\right\} \overset{\partial }{%
\longleftarrow }\limfunc{span}\left\{ e_{ab},e_{ca}\right\} \overset{%
\partial }{\longleftarrow }\Omega _{2}/\Omega _{2}^{\prime }\overset{%
\partial }{\longleftarrow }...
\end{equation*}%
We need to determine $\func{Im}\partial |_{\Omega _{1}/\Omega _{1}^{\prime
}} $, $\ker \partial |_{\Omega _{1}/\Omega _{1}^{\prime }}$, $\func{Im}%
\partial |_{\Omega _{2}/\Omega _{2}^{\prime }}$. Since%
\begin{equation*}
\partial e_{ab}=e_{b}-e_{a}=-e_{a}\func{mod}\Omega _{0}^{\prime },
\end{equation*}%
it follows that 
\begin{equation*}
\func{Im}\partial |_{\Omega _{1}/\Omega _{1}^{\prime }}=\Omega _{0}/\Omega
_{0}^{\prime },
\end{equation*}%
whence 
\begin{equation}
\dim H_{0}\left( \Omega /\Omega ^{\prime }\right) =0.\text{\ }  \label{00}
\end{equation}%
For any scalars $k,l\in \mathbb{K}$, we have 
\begin{equation*}
\partial (ke_{ab}+le_{ca})=(l-k)e_{a}\func{mod}\Omega _{0}^{\prime },
\end{equation*}%
so that $\partial (ke_{ab}+le_{ca})=0$ if and only if $k=l$, that is 
\begin{equation}
\ker \partial |_{\Omega _{1}/\Omega _{1}^{\prime }}=\limfunc{span}%
(e_{ab}+e_{ca})\func{mod}\Omega _{1}^{\prime }.  \label{ker1/1}
\end{equation}

Let us now compute $\func{Im}\partial |_{\Omega _{2}/\Omega _{2}^{\prime }}$%
. For any $v\in \Omega _{2}$ we have by the above Cases 1,2 that $%
v^{abi}=v^{jca}=0$, which implies that $v$ has the form%
\begin{equation}
v=v^{\prime }+v^{cab}e_{cab},  \label{vv'}
\end{equation}%
where $v^{\prime }\in \mathcal{A}_{2}^{\prime }$. It follows that%
\begin{equation}
\partial v=\partial v^{\prime }+v^{cab}\left( e_{ab}-e_{cb}+e_{ca}\right) .
\label{dvv'}
\end{equation}%
Since all $1$-paths $\partial v$, $e_{ab}$ and $e_{ca}$ belong to $\mathcal{A%
}_{1}$, it follows that $\partial v^{\prime }-v^{cab}e_{cb}\in \mathcal{A}%
_{1}$ whence also $\partial v^{\prime }-v^{cab}e_{cb}\in \mathcal{A}%
_{1}^{\prime }.$ Therefore,%
\begin{equation}
\partial v=v^{cab}\left( e_{ab}+e_{ca}\right) \func{mod}\Omega _{1}^{\prime
}.  \label{ab+ca}
\end{equation}%
Next consider two cases.

$\left( i\right) $ Let $\Omega _{2}$ contain an element $v$ with $%
v^{cab}\neq 0.$ Then by (\ref{ab+ca}) 
\begin{equation}
\func{Im}\partial |_{\Omega _{2}/\Omega _{2}^{\prime }}=\limfunc{span}\left(
e_{ab}+e_{ca}\right) \func{mod}\Omega _{1}^{\prime },  \label{Im2/2}
\end{equation}%
which together with (\ref{ker1/1}) implies 
\begin{equation}
\dim H_{1}\left( \Omega /\Omega ^{\prime }\right) =0.\text{\ }  \label{10}
\end{equation}%
Substituting (\ref{00}) and (\ref{10}) into the exact sequence (\ref{Om/Om'1}%
), we obtain that the identity 
\begin{equation*}
\dim H_{p}\left( \Omega ^{\prime }\right) =\dim H_{p}\left( \Omega \right)
\end{equation*}%
holds for all $p\geq 0$.

$\left( ii\right) $ Assume that $v^{cab}=0$ for all $v\in \Omega _{2}$. Then
by (\ref{ab+ca})%
\begin{equation*}
\func{Im}\partial |_{\Omega _{2}/\Omega _{2}^{\prime }}=0,
\end{equation*}%
which together with (\ref{ker1/1}) implies 
\begin{equation}
\dim H_{1}\left( \Omega /\Omega ^{\prime }\right) =1.  \label{H1=1}
\end{equation}%
Using again the exact sequence (\ref{Om/Om'1}), that is, 
\begin{equation*}
0\leftarrow H_{0}\left( \Omega \right) \leftarrow H_{0}\left( \Omega
^{\prime }\right) \leftarrow H_{1}\left( \Omega /\Omega ^{\prime }\right)
\leftarrow H_{1}\left( \Omega \right) \leftarrow H_{1}\left( \Omega ^{\prime
}\right) \leftarrow 0,
\end{equation*}%
we obtain by (\ref{sum=0}) and (\ref{H1=1}) 
\begin{equation}
\dim H_{1}\left( \Omega ^{\prime }\right) -\dim H_{1}(\Omega )+1-\dim
H^{0}\left( \Omega ^{\prime }\right) +\dim H^{0}(\Omega )=0  \label{1=1+}
\end{equation}%
Let us now specify when $\left( i\right) $ or $\left( ii\right) $ occur.
Assume first that $cb$ is an edge:%
\begin{equation*}
\fbox{$\ \ \ a~\bullet _{\nwarrow }^{\nearrow }$\fbox{$%
\begin{array}{c}
\bullet b \\ 
\ \uparrow \ \  \\ 
\bullet c%
\end{array}%
\ \ ~\ \ \ \ \ G^{\prime }\ $}\ \ $\ G$}
\end{equation*}%
Then 
\begin{equation*}
\partial e_{cab}=e_{ab}-e_{cb}+e_{ca}\in \mathcal{A}_{1},
\end{equation*}%
whence it follows that $e_{cab}\in \Omega _{2}$. Hence, we have the case $%
\left( i\right) $ with $v=e_{cab}$.

Assume now that $cb$ is not an edge. Denote by $J$ the set of vertices $j\in
V^{\prime }$ such that the $2$-path $cjb$ is allowed in $G^{\prime }$: 
\begin{equation*}
a~\bullet _{\nwarrow }^{\nearrow }%
\begin{array}{c}
\bullet b \\ 
\  \\ 
\bullet c%
\end{array}%
\ _{\nearrow }^{\nwarrow }\fbox{$\overset{\ }{\underset{\ }{\bullet j\ ...}\
J}$}
\end{equation*}%
Assume first that $J$ is non-empty, that is, $cb$ is a semi-edge, and set 
\begin{equation*}
v=e_{cab}-\frac{1}{\left\vert J\right\vert }\sum_{j\in J}e_{cjb},
\end{equation*}%
where $\left\vert J\right\vert $ is the number of elements in $J$. It is
clear that $v\in \mathcal{A}_{2}$. We have%
\begin{eqnarray}
\partial v &=&\left( e_{ab}-e_{cb}+e_{ca}\right) -\frac{1}{\left\vert
J\right\vert }\sum_{j\in J}\left( e_{jb}-e_{cb}+e_{cj}\right)  \notag \\
&=&\left( e_{ab}+e_{ca}\right) -\frac{1}{\left\vert J\right\vert }\sum_{j\in
J}\left( e_{jb}+e_{cj}\right) ,  \label{ll3}
\end{eqnarray}%
where the term $e_{cb}$ has cancelled out. It follows from (\ref{ll3}) that $%
\partial v\in \mathcal{A}_{1}$ whence $v\in \Omega _{2}$, and we obtain
again the case $\left( i\right) $. This finishes the proof of $\left(
b\right) $.

Let us show that if $J=\emptyset $ (that is, when $cb$ is neither edge nor
semi-edge) then we have the case $\left( ii\right) .$ Any $2$-path $v\in
\Omega _{2}$ has the form (\ref{vv'}) and $\partial v$ is given by (\ref%
{dvv'}). It follows that%
\begin{equation*}
\left( \partial v\right) ^{cb}=\left( \partial v^{\prime }\right)
^{cb}-v^{cab}.
\end{equation*}%
Since $\partial v\in \mathcal{A}_{1}$ and $cb$ is not an edge, we have $%
\left( \partial v\right) ^{cb}=0.$ We have by (\ref{dv})%
\begin{equation*}
\left( \partial v^{\prime }\right) ^{cb}=\sum_{j\in V^{\prime }}\left(
v^{\prime }\right) ^{jcb}-\left( v^{\prime }\right) ^{cjb}+\left( v^{\prime
}\right) ^{cbj},
\end{equation*}%
which implies that $\left( \partial v^{\prime }\right) ^{cb}=0$ as no
elementary $2$-path of the form $jcb,cjb,cbj$ is allowed in $G^{\prime },$
whereas $v^{\prime }\in \mathcal{A}_{2}^{\prime }$. It follows that $%
v^{cab}=0$ so that we have the case $\left( ii\right) $.

If in addition $b,c$ belong to the same connected component of $G^{\prime }$
then we have (\ref{0'=0}), that is,%
\begin{equation*}
\dim H^{0}\left( \Omega \right) =\dim H^{0}(\Omega ^{\prime }).
\end{equation*}%
Substituting into (\ref{1=1+}), we obtain 
\begin{equation*}
\dim H_{1}\left( \Omega \right) =\dim H_{1}(\Omega ^{\prime })+1.
\end{equation*}%
which proves part $\left( c\right) $.

If $b,c$ belong to different components of $G^{\prime }$ then we have by (%
\ref{0'=0+1})%
\begin{equation*}
\dim H^{0}\left( \Omega \right) =\dim H^{0}(\Omega ^{\prime })-1,
\end{equation*}%
whence by (\ref{1=1+}) 
\begin{equation*}
\dim H_{1}\left( \Omega \right) =\dim H_{1}(\Omega ^{\prime }),
\end{equation*}%
which finishes the proof of part $\left( d\right) .$

Finally, the identities for the Euler characteristic follow easily from the
relations between $\dim H_{p}\left( \Omega \right) $ and $\dim H_{p}\left(
\Omega ^{\prime }\right) .$
\end{proof}

\section{Join of path complexes}

\setcounter{equation}{0}\label{SecJoin}In this and next sections we use
slightly different way of denoting the path/form spaces associated with a
given path complex as we will have to consider path complexes on more than
one set. Given a finite set $X$, denote by $P\left( X\right) $ a path
complex on $X$. The space $\mathcal{A}_{p}\left( P\left( X\right) \right) $
of all allowed $p$-paths (= finite $\mathbb{K}$-linear combinations of
elements from $P_{p}\left( X\right) $) will be denoted shortly by $\mathcal{A%
}_{p}\left( X\right) $. Similarly, the space $\Omega _{p}\left( P\left(
X\right) \right) $ of all $\partial $-invariant $p$-paths will be denoted by 
$\Omega _{p}\left( X\right) $. Similar notation will apply to all other
relevant notions including homologies $H_{p}\left( X\right) $, etc.

\subsection{Definition and examples of a join}

\label{SecJoinEx}

\begin{definition}
\RM Given two disjoint finite sets $X,Y$ and their path complexes $P\left(
X\right) ,P\left( Y\right) $, set $Z=X\sqcup Y$ and define a path complex $%
P\left( Z\right) $ as follows: $P\left( Z\right) $ consists of all paths of
the form $\,uv$ where $u\in P\left( X\right) $ and $v\in P\left( Y\right) $.
The path complex $P\left( Z\right) $ is called a \emph{join} of $P\left(
X\right) ,P\left( Y\right) $ and is denoted by $P\left( Z\right) =P\left(
X\right) \ast P\left( Z\right) .$
\end{definition}

An example of the path $uv\in P\left( Z\right) $ is given on Fig. \ref{pic36}%
. Note that each of $u,v$ can be empty so that all allowed paths on $X$ and $%
Y$ will also be allowed on $Z$.

\FRAME{ftbhFU}{4.1633in}{2.0326in}{0pt}{\Qcb{A path $uv$}}{\Qlb{pic36}}{%
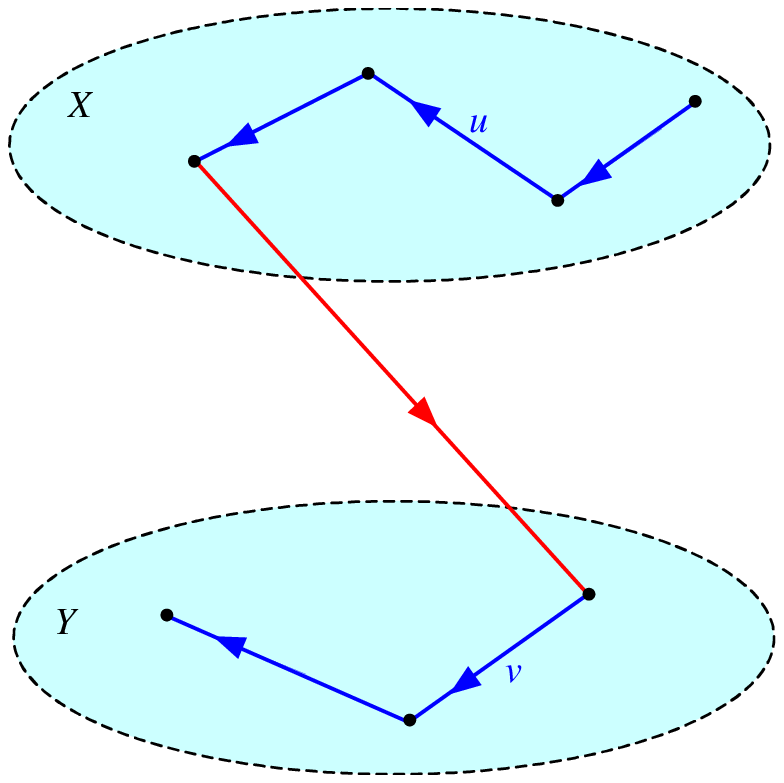}{\special{language "Scientific Word";type
"GRAPHIC";maintain-aspect-ratio TRUE;display "USEDEF";valid_file "F";width
4.1633in;height 2.0326in;depth 0pt;original-width 5.9088in;original-height
2.8644in;cropleft "0";croptop "1";cropright "1";cropbottom "0";filename
'pic36.eps';file-properties "XNPEU";}}

The operation $\ast $ on the path complexes is obviously non-commutative but
associative.

\begin{example}
\RM Let $X,Y$ be digraphs and $P\left( X\right) $ and $P\left( Y\right) $
with the path complexes arising from their digraph structures. Consider the
digraph $Z$ whose the set of vertices is $X\sqcup Y$, while the set of edges
of $Z$ consists of all the edges of $X$ and $Y$, as well as of all the edges 
$x\rightarrow y$ for all $x\in X$ and $y\in Y$. The digraph $Z$ is called a 
\emph{\ join} of $X$ and $Y$ and is denoted by $\,X\ast Y.$

An example of a join of two digraphs is shown on Fig. \ref{pic35}.\FRAME{%
ftbhFU}{4.1635in}{2.0324in}{0pt}{\Qcb{A digraph $Z=X\ast Y$}}{\Qlb{pic35}}{%
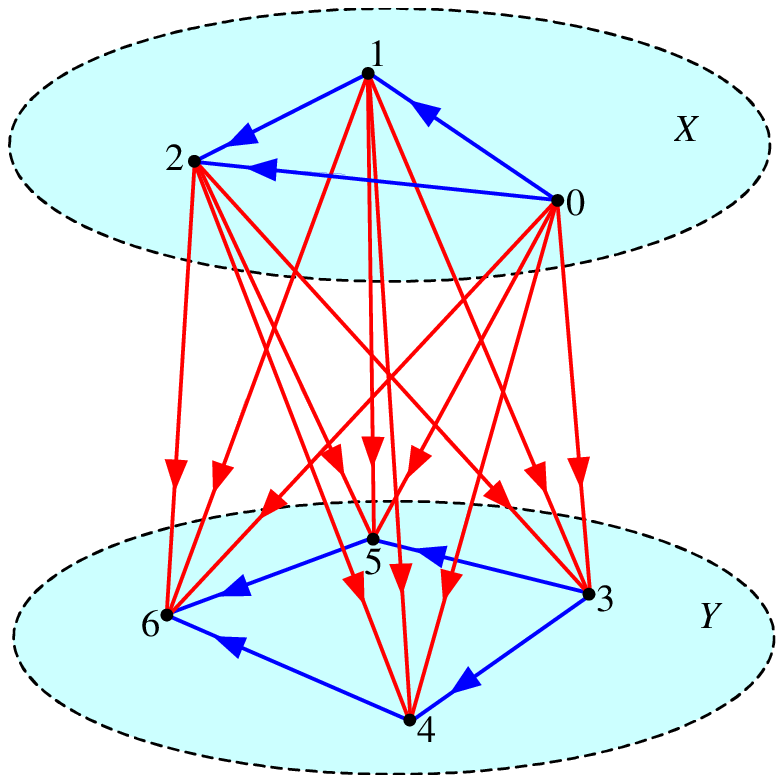}{\special{language "Scientific Word";type
"GRAPHIC";maintain-aspect-ratio TRUE;display "USEDEF";valid_file "F";width
4.1635in;height 2.0324in;depth 0pt;original-width 5.9088in;original-height
2.8644in;cropleft "0";croptop "1";cropright "1";cropbottom "0";filename
'pic35.eps';file-properties "XNPEU";}}

Let $P\left( Z\right) $ be the path complex arising from the digraph
structure of $Z$. Then it is obvious from the definition that $P\left(
Z\right) $ is the join of $P\left( X\right) $ and $P\left( Y\right) $ so that%
\begin{equation*}
P\left( X\ast Y\right) =P\left( X\right) \ast P\left( Y\right) .
\end{equation*}%
Hence, the operation of joining of digraphs is a particular case of the
operation of joining path complexes.

If $Y$ consists of a single vertex $a$ then $Z=X\ast Y$ is called a \emph{%
cone} over $X$, and if $Y$ consists of two vertices $a,b$ and no edges then $%
Z=X\ast Y$ is called a \emph{suspension} over $X$ (see Sections \ref{SecCone}
and \ref{SecSus} for more details).
\end{example}

\begin{example}
\RM Let $X$ and $Y$ be the vertex sets of finite simplicial complexes $%
S\left( X\right) $ and $S\left( Y\right) .$ Let us construct a simplicial
complex $S\left( Z\right) $ with the vertex set $Z=X\sqcup Y$ as follows.
Assuming that $\left\vert X\right\vert =n$ and $\left\vert Y\right\vert =m$,
embed the set $X$ (together with all simplexes from $S\left( X\right) $)
into a hyperplane $h^{n-1}\subset \mathbb{R}^{n+m-1}$ and $Y$ -- into a
hyperplane $h^{m-1}\subset R^{n+m-1}$, where the hyperplanes $%
h^{n-1},h^{m-1} $ are orthogonal and non-intersecting. For any two simplexes 
$\sigma _{1}\in S\left( X\right) $ and $\sigma _{2}\in S\left( Y\right) $,
define their join $\sigma _{1}\ast \sigma _{2}$ as the convex hull of $%
\sigma _{1}$ and $\sigma _{2}$ embedded in $\mathbb{R}^{n+m-1}$ as above
(see Fig. \ref{pic39}).\FRAME{ftbphFU}{3.3009in}{2.2001in}{0pt}{\Qcb{A join $%
\protect\sigma _{1}\ast \protect\sigma _{2}$ of two one-dimensional
simplexes $\protect\sigma _{1},\protect\sigma _{2}$ (case $n=m=2$)}}{\Qlb{%
pic39}}{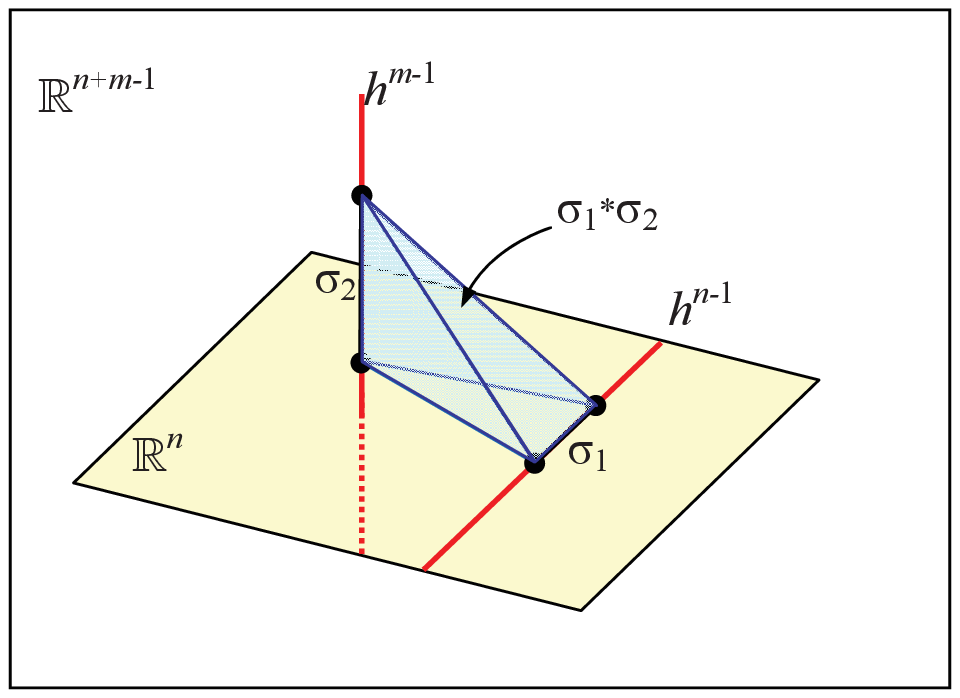}{\special{language "Scientific Word";type
"GRAPHIC";maintain-aspect-ratio TRUE;display "ICON";valid_file "F";width
3.3009in;height 2.2001in;depth 0pt;original-width 0pt;original-height
0pt;cropleft "0";croptop "1";cropright "1";cropbottom "0";filename
'pic39.eps';file-properties "XNPEU";}}

Due to a general position of $\sigma _{1}$ and $\sigma _{2}$, the join $%
\sigma _{1}\ast \sigma _{2}$ is also a simplex. Then $S\left( Z\right) $ is
a collection of all simplexes $\sigma _{1}\ast \sigma _{2}$ with $\sigma
_{1}\in S\left( X\right) $ and $\sigma _{2}\in S\left( Y\right) $. We refer
to $S\left( Z\right) $ as a join of simplicial complexes $S\left( X\right)
,S\left( Y\right) $ and denote it by $S\left( X\right) \ast S\left( Y\right) 
$.

Equivalently, one can define $S\left( Z\right) $ in an abstract way without
embedding into a Euclidean space. Indeed, considering simplexes as sequences
of vertices, we can say that $S\left( Z\right) $ consists of all simplexes
of the form $\left[ x_{0},...,x_{p},y_{0},...,y_{q}\right] $ where $\left[
x_{0},...,x_{p}\right] \in S\left( X\right) $ and $\left[ y_{0},...,y_{q}%
\right] \in S\left( Y\right) $. It is clear that $S\left( Z\right) $
satisfies the defining property (\ref{s}) of a simplicial complex, so that $%
S\left( Z\right) $ is a simplicial complex. It is also obvious that the path
complexes $P\left( X\right) ,P\left( Y\right) ,P\left( Z\right) $ of the
simplicial complexes $S\left( X\right) ,S\left( Y\right) ,S\left( Z\right) $%
, respectively, satisfy $P\left( Z\right) =P\left( X\right) \ast P\left(
Y\right) .$ Hence, the operation of joining of simplicial complexes is a
particular case of the operation of joining path complexes.

Alternatively one can see the latter using Proposition \ref{PropPM}: if both
path complexes $P\left( X\right) $ and $P\left( Y\right) $ arise from
simplicial complexes, that is, if they are monotone and perfect, then the
path complex $P\left( X\right) \ast P\left( Y\right) $ is also monotone and
perfect. Hence, it arises from a simplicial complex with the vertex set $Z$.
\end{example}

\begin{proposition}
\label{Puv}Let $P\left( X\right) $ and $P\left( Y\right) $ be two path
complexes and let $P\left( Z\right) =P\left( X\right) \ast P\left( Y\right) $%
. If $u\in \Omega _{p}\left( X\right) $ and $v\in \Omega _{q}\left( Y\right) 
$ then $uv\in \Omega _{p+q+1}\left( Z\right) .$ Moreover, the operation $%
u,v\mapsto uv$ of join extends to that for the homology classes $u\in 
\widetilde{H}_{p}\left( X\right) $ and $v\in \widetilde{H}_{q}\left(
Y\right) $ so that $uv\in \widetilde{H}_{p+q+1}\left( Z\right) .$
\end{proposition}

\begin{proof}
If $u$ and $v$ are allowed then $uv$ is allowed on $Z$ by definition. In
particular, if $u\in \Omega _{p}\left( X\right) $ and $v\in \Omega
_{q}\left( Y\right) $ then $uv\in \mathcal{A}_{p+q+1}\left( Z\right) .$ Let
us show that $\partial \left( uv\right) \in \mathcal{A}_{p+q}\left( Z\right) 
$, which would imply $uv\in \Omega _{p+q+1}\left( Z\right) $. Indeed, we have%
\begin{equation}
\partial \left( uv\right) =\left( \partial u\right) v+\left( -1\right)
^{p+1}u\left( \partial v\right) .  \label{duv1}
\end{equation}%
Since $\partial u$ and $\partial v$ are also allowed, we obtain that the
right hand side here is allowed, whence the claim follows.

If $u,v$ are representatives of homology classes, that is, closed paths,
then by (\ref{duv1}) the join $uv$ is also closed, so that $uv$ represents a
homology class of $Z$. We are left to verify that the class of $uv$ depends
only on the classes of $u$ and $v$. For that it suffices to prove that if
either $u$ or $v$ is exact then so is $uv$. Indeed, if $u=dw$ then%
\begin{equation*}
\partial \left( wv\right) =\left( \partial w\right) v+\left( -1\right)
^{p}w\left( \partial v\right) =uv
\end{equation*}%
so that $uv$ is exact.
\end{proof}

\subsection{Chain complex of a join}

Given a regular path complex $P\left( X\right) $ on a finite set $X$, we
consider as before the paths spaces $\mathcal{R}_{n}\left( X\right) ,%
\mathcal{A}_{n}\left( X\right) $ and $\Omega _{n}\left( X\right) $, where $%
n\geq -1$. If $W_{n}\left( X\right) $ is one of these space then set 
\begin{equation*}
W_{n}^{\prime }\left( X\right) \equiv W_{n-1}\left( X\right) .
\end{equation*}%
The graded linear space $W_{\bullet }^{\prime }$ and $W_{\bullet }$ coincide
as linear spaces but have slightly different graded structure. Recall that $%
\mathcal{R}_{\bullet }\left( X\right) $, $\mathcal{R}_{\bullet }^{\prime
}\left( X\right) $, $\Omega _{\bullet }\left( X\right) $,$\ \Omega _{\bullet
}^{\prime }\left( X\right) $ carry in addition the structure of chain
complexes with respect to the regular boundary operator $\partial $.

In the next statement we use the notion of the tensor product of chain
complexes (see Section \ref{SecTensor} for definition).

\begin{theorem}
\label{Tjoin}Let $X,Y$ be two finite non-empty sets and $P\left( X\right) $
and $P\left( Y\right) $ be regular path complexes on $X$ and $Y$
respectively. Set $Z=X\sqcup Y$ and consider the join path complex 
\begin{equation*}
P\left( Z\right) =P\left( X\right) \ast P\left( Y\right) .
\end{equation*}%
We have the following isomorphism of the chain complexes:%
\begin{equation}
\Omega _{\bullet }\left( Z\right) \cong \Omega _{\bullet }^{\prime }\left(
X\right) \otimes \Omega _{\bullet }\left( Y\right) ,  \label{rpq}
\end{equation}%
where the mapping $\Omega _{\bullet }^{\prime }\left( X\right) \otimes
\Omega _{\bullet }\left( Y\right) \rightarrow \Omega _{\bullet }\left(
Z\right) $ is given by $u\otimes v\mapsto uv.$

In particular, for any $r\geq -1$,%
\begin{equation}
\Omega _{r}\left( Z\right) \cong \tbigoplus_{\left\{ p,q\geq
-1:p+q=r-1\right\} }\left( \Omega _{p}\left( X\right) \otimes \Omega
_{q}\left( Y\right) \right)  \label{Omrpq-1}
\end{equation}%
and, for any $r\geq 0$,%
\begin{equation}
\widetilde{H}_{r}\left( Z\right) \cong \tbigoplus_{\left\{ p,q\geq
0:p+q=r-1\right\} }\left( \widetilde{H}_{p}\left( X\right) \otimes 
\widetilde{H}_{q}\left( Y\right) \right)  \label{HZ}
\end{equation}%
(the K\"{u}nneth formula for join).
\end{theorem}

It follows from (\ref{HZ}) that%
\begin{equation}
\dim \widetilde{H}_{r}\left( Z\right) =\sum_{\left\{ p,q\geq
0:p+q=r-1\right\} }\dim \widetilde{H}_{p}\left( X\right) \dim \widetilde{H}%
_{q}\left( Y\right) .  \label{dimHZ}
\end{equation}%
If the both path complexes $P\left( X\right) $ and $P\left( Y\right) $ are
connected then $\widetilde{H}_{0}=\left\{ 0\right\} $ for the both complexes
(cf. Proposition \ref{PH0}), and (\ref{HZ}) can be restated as follows: for
any $r\geq 1$%
\begin{equation*}
H_{r}\left( Z\right) \cong \tbigoplus_{\left\{ p,q\geq 1:p+q=r-1\right\}
}\left( H_{p}\left( X\right) \otimes H_{q}\left( Y\right) \right) .
\end{equation*}

\begin{proof}
Let us first show how (\ref{Omrpq-1}) and (\ref{HZ}) follow from (\ref{rpq}%
). By definition (\ref{rpq}) means that%
\begin{equation*}
\Omega _{r}\left( Z\right) \cong \tbigoplus_{\left\{ p\geq 0,q\geq
-1:p+q=r\right\} }\left( \Omega _{p}^{\prime }\left( X\right) \otimes \Omega
_{q}\left( Y\right) \right) ,
\end{equation*}%
whence (\ref{Omrpq-1}) follows by changing $p-1$ to $p$. The isomorphism (%
\ref{rpq}) of the chain complexes $\Omega _{\bullet }\left( Z\right) $ and $%
\Omega _{\bullet }^{\prime }\left( X\right) \otimes \Omega _{\bullet }\left(
Y\right) $ implies that their homologies are also isomorphic. On the other
hand, by the K\"{u}nneth theorem (\ref{KAB}), we have%
\begin{equation*}
H_{\bullet }\left( \Omega _{\bullet }^{\prime }\left( X\right) \otimes
\Omega _{\bullet }\left( Y\right) \right) \cong H_{\bullet }\left( \Omega
_{\bullet }^{\prime }\left( X\right) \right) \otimes H_{\bullet }\left(
\Omega _{\bullet }\left( Y\right) \right)
\end{equation*}%
whence%
\begin{equation*}
H_{\bullet }\left( \Omega _{\bullet }\left( Z\right) \right) \cong
H_{\bullet }\left( \Omega _{\bullet }^{\prime }\left( X\right) \right)
\otimes H_{\bullet }\left( \Omega _{\bullet }\left( Y\right) \right) .
\end{equation*}%
More explicitly this means that, for any $r\geq -1$,%
\begin{eqnarray*}
H_{r}\left( \Omega _{\bullet }\left( Z\right) \right) &\cong
&\tbigoplus_{\left\{ p^{\prime }\geq 0,q\geq -1:p^{\prime }+q=r\right\}
}\left( H_{p^{\prime }}\left( \Omega _{\bullet }^{\prime }\left( X\right)
\right) \otimes H_{q}\left( \Omega _{\bullet }\left( Y\right) \right) \right)
\\
&=&\tbigoplus_{\left\{ p,q\geq -1:p+q=r-1\right\} }\left( H_{p}\left( \Omega
_{\bullet }\left( X\right) \right) \otimes H_{q}\left( \Omega _{\bullet
}\left( Y\right) \right) \right) .
\end{eqnarray*}%
Since the homology group $H_{-1}\left( \Omega _{\bullet }\right) $ is always
trivial, the condition $p,q\geq -1$ can be replaced here by $p,q\geq 0$.
Finally, observing that $H_{p}\left( \Omega _{\bullet }\left( X\right)
\right) =\widetilde{H}_{p}\left( X\right) $ and $H_{q}\left( \Omega
_{\bullet }\left( Y\right) \right) =\widetilde{H}_{q}\left( Y\right) $ are
the reduced homologies, we obtain (\ref{HZ}).

Now we concentrate on the proof of (\ref{rpq}). We will consider here the
path spaces $\left\{ \mathcal{R}_{p}\right\} $,$\ \left\{ \mathcal{A}%
_{p}\right\} $, $\left\{ \Omega _{p}\right\} $ associated with the path
complexes $P\left( X\right) ,P\left( Y\right) $ and $P\left( Z\right) $. If $%
\left\{ W_{p}\right\} $ is one of these families then set 
\begin{equation*}
W_{\bullet }\left( X,Y\right) =W_{\bullet }^{\prime }\left( X\right) \otimes
W_{\bullet }\left( Y\right) .
\end{equation*}%
Then (\ref{rpq}) can be restated as follows: 
\begin{equation*}
\Omega _{\bullet }\left( Z\right) \cong \Omega _{\bullet }\left( X,Y\right) .
\end{equation*}%
To prove this, we will construct explicitly a mapping 
\begin{equation*}
\Phi :\Omega _{r}\left( X,Y\right) \rightarrow \Omega _{r}\left( Z\right)
\end{equation*}%
that will be isomorphism of linear spaces and will commute with the boundary
operator $\partial $.

Consider first a larger the chain complex 
\begin{equation*}
\mathcal{R}_{\bullet }\left( X,Y\right) =\mathcal{R}_{\bullet }^{\prime
}\left( X\right) \otimes \mathcal{R}_{\bullet }\left( Y\right)
\end{equation*}
and define for any $r\geq -1$ the linear mapping 
\begin{equation*}
\Phi :\mathcal{R}_{r}\left( X,Y\right) \rightarrow \mathcal{R}_{r}\left(
Z\right)
\end{equation*}%
as follows: for all $u\in \mathcal{R}_{p}^{\prime }\left( X\right) $ and $%
v\in \mathcal{R}_{q}\left( Y\right) $ with $p+q=r$, set%
\begin{equation}
\Phi \left( u\otimes v\right) =uv,  \label{Ruv}
\end{equation}
where $uv$ is the join of $u$ and $v$ on $Z$ (note that $X$ and $Y$ are
subsets of $Z$).

It follows from Lemma \ref{Lemuv} that, for $u,v$ as above, 
\begin{equation}
\partial \left( uv\right) =\left( \partial u\right) v+\left( -1\right)
^{p}u\partial v.  \label{duvp}
\end{equation}%
Here the operator $\partial $ is the boundary operator on $\mathcal{R}%
_{\bullet }\left( Z\right) $, but in the expressions $\partial u$ and $%
\partial v$ it coincides with the boundary operators on $\mathcal{R}%
_{\bullet }\left( X\right) $ and $\mathcal{R}_{\bullet }\left( Y\right) $,
respectively. By (\ref{utsv}) we have for the operator $\partial $ on $%
\mathcal{R}_{\bullet }\left( X,Y\right) $%
\begin{equation*}
\partial \left( u\otimes v\right) =\left( \partial u\right) \otimes v+\left(
-1\right) ^{p}u\otimes \partial v.
\end{equation*}%
The comparison with (\ref{duvp}) shows that the following diagram is
commutative:%
\begin{equation*}
\begin{array}{ccc}
\mathcal{R}_{r-1}\left( X,Y\right) & \overset{\partial }{\leftarrow } & 
\mathcal{R}_{r}\left( X,Y\right) \\ 
\ \downarrow ^{\Phi } &  & \ \downarrow ^{\Phi } \\ 
\mathcal{R}_{r-1}\left( Z\right) & \overset{\partial }{\leftarrow } & 
\mathcal{R}_{r}\left( Z\right)%
\end{array}%
\end{equation*}%
Hence, the mapping $\Phi $ is a homomorphism of chain complexes $\mathcal{R}%
_{\bullet }\left( X,Y\right) $ and $\mathcal{R}_{\bullet }\left( Z\right) $.

Let us verify that $\Phi $ is in fact a monomorphism. Indeed, the basis in $%
\mathcal{R}_{r}\left( X,Y\right) $ consists of all elements of the form $%
e_{x}\otimes e_{y}$ where $x\in R_{p}\left( X\right) ,y\in R_{q}\left(
Y\right) $ with $p+q=r.$ Since $\Phi \left( e_{x}\otimes e_{y}\right)
=e_{xy} $ and all such paths $e_{xy}$ are linearly independent in $\mathcal{R%
}_{r}\left( Z\right) $, we see that $\Phi $ is injective.

Next observe that%
\begin{equation}
\Phi \left( \mathcal{A}_{r}\left( X,Y\right) \right) =\mathcal{A}_{r}\left(
Z\right) .  \label{FAA}
\end{equation}%
Indeed, the basis in $\mathcal{A}_{r}\left( X,Y\right) $ consists of all
elements of the form $e_{x}\otimes e_{y}$ where $x\in R_{p}\left( X\right)
,y\in R_{q}\left( Y\right) $ with $p+q=r,$ while the basis in $\mathcal{A}%
_{r}\left( Z\right) $ consists of the paths $e_{xy}$ with the same set of $%
x,y$, whence the claim follows. In particular, the linear spaces $\mathcal{A}%
_{r}\left( X,Y\right) $ and $\mathcal{A}_{r}\left( Z\right) $ are isomorphic.

Now we aim at restricting $\Phi $ to $\Omega _{\bullet }\left( X,Y\right) $.
In fact, we will prove that 
\begin{equation}
\Phi \left( \Omega _{r}\left( X,Y\right) \right) =\Omega _{r}\left( Z\right)
,  \label{FFOM=}
\end{equation}%
for all $r\geq -1$, which will finish the proof of (\ref{rpq}). The inclusion%
\begin{equation}
\Phi \left( \Omega _{r}\left( X,Y\right) \right) \subset \Omega _{r}\left(
Z\right)  \label{OMOMin}
\end{equation}%
is trivial because by Proposition \ref{Puv} $u\in \Omega _{p}^{\prime
}\left( X\right) $ and $v\in \Omega _{q}\left( Y\right) $ with $p+q=r$ imply 
$uv\in \Omega _{r}\left( Z\right) .$

Since $\Phi $ is injective, in order to prove the opposite inclusion in (\ref%
{OMOMin}), it suffices to show that%
\begin{equation}
\dim \left( \Omega _{r}\left( X,Y\right) \right) \geq \dim \Omega _{r}\left(
Z\right) .  \label{ddOm}
\end{equation}%
\label{rem: rewrite using only paths and orthogonal complement]}To that end,
we consider the dual spaces $\left\{ \mathcal{R}^{p}\right\} $, $\left\{ 
\mathcal{A}^{p}\right\} $, $\left\{ \Omega ^{p}\right\} $ of forms on $X,Y,Z$%
. We use the same notation as in the case of path spaces: $W^{\prime
p}=W^{p-1}$ and 
\begin{equation*}
W^{\bullet }\left( X,Y\right) =W^{\prime \bullet }\left( X\right) \otimes
W^{\bullet }\left( Y\right) .
\end{equation*}%
In particular, $\Omega ^{r}\left( X,Y\right) $ is dual to $\Omega _{r}\left(
X,Y\right) $. We will prove that%
\begin{equation}
\dim \left( \Omega ^{r}\left( X,Y\right) \right) \geq \dim \Omega ^{r}\left(
Z\right) ,  \label{ddOmm}
\end{equation}%
which will settle (\ref{ddOm}) and, hence, (\ref{FFOM=}).

Recall that%
\begin{equation*}
\Omega ^{r}=\left. \mathcal{A}^{r}\right/ \mathcal{K}^{r}
\end{equation*}%
where%
\begin{equation*}
\mathcal{K}^{r}=\left( \mathcal{N}^{r}+d\mathcal{N}^{r-1}\right) \cap 
\mathcal{A}^{r}.
\end{equation*}%
Therefore, we have%
\begin{eqnarray*}
\Omega ^{r}\left( X,Y\right) &=&\bigoplus_{\left\{ p\geq 0,q\geq
-1:p+q=r\right\} }\Omega ^{\prime p}\left( X\right) \otimes \Omega
^{q}\left( Y\right) \\
&&\bigoplus_{\left\{ p\geq 0,q\geq -1:p+q=r\right\} }\left( \left. \mathcal{A%
}^{\prime p}\left( X\right) \right/ \mathcal{K}^{\prime p}\left( X\right)
\right) \otimes \left( \left. \mathcal{A}^{q}\left( Y\right) \right/ 
\mathcal{K}^{q}\left( Y\right) \right) \\
&\cong &\bigoplus_{\left\{ p\geq 0,q\geq -1:p+q=r\right\} }\left. \left( 
\mathcal{A}^{\prime p}\left( X\right) \otimes \mathcal{A}^{q}\left( Y\right)
\right) \right/ \left( \mathcal{K}^{\prime p}\left( X\right) \otimes 
\mathcal{A}^{q}\left( Y\right) +\mathcal{A}^{\prime p}\left( X\right)
\otimes \mathcal{K}^{q}\left( Y\right) \right) ,
\end{eqnarray*}%
which implies that%
\begin{eqnarray*}
\dim \Omega ^{r}\left( X,Y\right) &=&\sum_{\left\{ p\geq 0,q\geq
-1:p+q=r\right\} }\dim \left( \mathcal{A}^{\prime p}\left( X\right) \otimes 
\mathcal{A}^{q}\left( Y\right) \right) \\
&&-\sum_{\left\{ p\geq 0,q\geq -1:p+q=r\right\} }\dim \left( \mathcal{K}%
^{\prime p}\left( X\right) \otimes \mathcal{A}^{q}\left( Y\right) +\mathcal{A%
}^{\prime p}\left( X\right) \otimes \mathcal{K}^{q}\left( Y\right) \right)
\end{eqnarray*}%
It follows from (\ref{FAA}) that 
\begin{equation*}
\sum_{\left\{ p\geq 0,q\geq -1:p+q=r\right\} }\dim \left( \mathcal{A}%
^{\prime p}\left( X\right) \otimes \mathcal{A}^{q}\left( Y\right) \right)
=\dim \mathcal{A}^{r}\left( Z\right) .
\end{equation*}%
Since%
\begin{equation*}
\dim \Omega ^{r}\left( Z\right) =\dim \mathcal{A}^{r}\left( Z\right) -\dim 
\mathcal{K}^{r}\left( Z\right) ,
\end{equation*}%
the inequality (\ref{ddOmm}) will follow if we prove that%
\begin{equation}
\sum_{\left\{ p\geq 0,q\geq -1:p+q=r\right\} }\dim \left( \mathcal{K}%
^{\prime p}\left( X\right) \otimes \mathcal{A}^{q}\left( Y\right) +\mathcal{A%
}^{\prime p}\left( X\right) \otimes \mathcal{K}^{q}\left( Y\right) \right)
\leq \dim \mathcal{K}^{r}\left( Z\right) .  \label{KAK}
\end{equation}

In the next part of the proof we need an operation of joining of the forms
on $X$ and $Y$. For any forms $\varphi \in \mathcal{A}^{\bullet }\left(
X\right) $ and $\psi \in \mathcal{A}^{\bullet }\left( Y\right) $, define the
form $\varphi \ast \psi \in \mathcal{A}^{\bullet }\left( Z\right) $ first
for the elementary forms by%
\begin{equation*}
e^{i_{0}...i_{p}}\ast e^{j_{0}...j_{q}}=e^{i_{0}...i_{p}j_{0}...j_{q}}
\end{equation*}%
(clearly, if the paths $i_{0}...i_{p}$ and $j_{0}...j_{q}$ are allowed on $X$
and $Y$ respectively, then their join path is also allowed on $Z$), and then
extend to all $\varphi $ and $\psi $ by bilinearity.

The operation $\ast $ allows us to define a linear mapping%
\begin{equation*}
\Psi :\mathcal{A}^{r}\left( X,Y\right) \rightarrow \mathcal{A}^{r}\left(
Z\right)
\end{equation*}%
by%
\begin{equation}
\Psi \left( \varphi \otimes \psi \right) =\varphi \ast \psi .  \label{Psi}
\end{equation}%
The mapping $\Psi $ is obviously bijective as all allowed paths on $Z$ are
obtained by joining the allowed paths on $X$ and $Y$ in a unique way.

Let us show that the operation $\ast $ satisfies the following version of
the product rule: for all $\varphi \in \mathcal{A}^{p}\left( X\right) $, $%
\psi \in \mathcal{A}^{q}\left( Y\right) $ 
\begin{equation}
d\left( \varphi \ast \psi \right) =\left( d\varphi \right) \ast \psi +\left(
-1\right) ^{p+1}\varphi \ast d\psi \ \func{mod}\mathcal{N}^{p+q+2}\left(
Z\right) .  \label{uvmod}
\end{equation}%
It suffices to prove (\ref{uvmod}) for $\varphi =e^{i_{0}...i_{p}}$ and $%
\psi =e^{j_{0}...j_{q}}=e^{i_{p+1}...i_{p+q+1}}.$ We have by (\ref{de})%
\begin{eqnarray*}
d\left( \varphi \ast \psi \right) &=&de^{i_{0}...i_{p+q+1}}=\sum_{k\in
Z}\sum_{r=0}^{p+q+2}\left( -1\right) ^{r}e^{i_{0}...i_{r-1}ki_{r}..i_{p+q+1}}
\\
&=&\sum_{k\in Z}\sum_{r=0}^{p}\left( -1\right)
^{r}e^{i_{0}...i_{r-1}ki_{r}...i_{p}i_{p+1}...i_{p+q+1}} \\
&&+\sum_{k\in X\sqcup Y}\left( -1\right)
^{p+1}e^{i_{0}...i_{p}ki_{p+1}...i_{p+q+1}} \\
&&+\sum_{k\in Z}\sum_{r=p+2}^{p+q+2}\left( -1\right)
^{r}e^{i_{0}...i_{p}i_{p+1}...i_{r-1}ki_{r}...i_{p+q+1}}.
\end{eqnarray*}%
Observe that the form $e^{...i_{r-1}ki_{r}...i_{p}...}$ in the first sum is
non-allowed if $k\in Y,$ so that the range in the first sum can be
restricted to $k\in X$. Similarly, the form $%
e^{...i_{p+1}...i_{r-1}ki_{r}...}$ in the third sum is non-allowed if $k\in
X $, so that its range can be restricted to $k\in Y$. Splitting the range of
the second sum into two parts $k\in X$ and $k\in Y$ and combining them with
the first and third sums, respectively, we obtain the following identity $%
\func{mod}\mathcal{N}^{\bullet }\left( Z\right) $:%
\begin{eqnarray*}
d\left( \varphi \ast \psi \right) &=&\sum_{k\in X}\sum_{r=0}^{p+1}\left(
-1\right) ^{r}e^{i_{0}...i_{r-1}ki_{r}...i_{p}}\ast e^{i_{p+1}...i_{p+q+1}}
\\
&&+\sum_{k\in Y}\sum_{r=p+1}^{p+q+2}\left( -1\right)
^{r}e^{i_{0}...i_{p}}\ast e^{i_{p+1}...i_{r-1}ki_{r}...i_{p+q+1}} \\
&=&\left( d\varphi \right) \ast \psi +\left( -1\right) ^{p+1}\varphi \ast
d\psi ,
\end{eqnarray*}%
which proves (\ref{uvmod}).

Let us now verify that%
\begin{equation}
\Psi \left( \mathcal{K}^{\prime p}\left( X\right) \otimes \mathcal{A}%
^{q}\left( Y\right) +\mathcal{A}^{\prime p}\left( X\right) \otimes \mathcal{K%
}^{q}\left( Y\right) \right) \subset \mathcal{K}^{r}\left( Z\right) ,
\label{Psin}
\end{equation}%
provided $p+q=r$. We first prove that 
\begin{equation}
\Psi \left( \mathcal{K}^{\prime p}\left( X\right) \otimes \mathcal{A}%
^{q}\left( Y\right) \right) \subset \mathcal{K}^{r}\left( Z\right) .
\label{PsiAK}
\end{equation}%
For that we need to verify that 
\begin{equation*}
u\in \mathcal{K}^{\prime p}\left( X\right) \ \ \text{and\ }v\in \mathcal{A}%
^{q}\left( Y\right) \ \Rightarrow \ u\ast v\in \mathcal{K}^{r}\left(
Z\right) .
\end{equation*}%
Since $u\in \mathcal{A}^{\prime p}\left( X\right) $, we have $u\ast v\in 
\mathcal{A}^{r}\left( Z\right) .$ By definition of $\mathcal{K}^{\prime
p}\left( X\right) $, we have $u=d\varphi +\psi $ where $\varphi \in \mathcal{%
N}^{p-2}\left( X\right) $ and $\psi \in \mathcal{N}^{p-1}\left( X\right) $.
Using (\ref{uvmod}), we obtain%
\begin{equation*}
u\ast v=\left( d\varphi +\psi \right) \ast v=d\left( \varphi \ast v\right)
-\left( -1\right) ^{p-1}\varphi \ast dv+\psi \ast v\ \func{mod}\mathcal{N}%
^{\bullet }\left( Z\right) .
\end{equation*}%
Clearly, all terms starting with $\varphi \ast $ and $\psi \ast $ belong to $%
\mathcal{N}^{\bullet }\left( Z\right) $, which implies that the right hand
side belongs to $d\mathcal{N}^{\bullet }\left( Z\right) +\mathcal{N}%
^{\bullet }\left( Z\right) $, whence%
\begin{equation*}
u\ast v\in d\mathcal{N}^{\bullet }\left( Z\right) +\mathcal{N}^{\bullet
}\left( Z\right) .
\end{equation*}%
It follows that $u\ast v\in \mathcal{K}^{\bullet }\left( Z\right) ,$ which
proves (\ref{PsiAK}). In the same way one proves that%
\begin{equation*}
\Psi \left( \mathcal{A}^{\prime p}\left( X\right) \otimes \mathcal{K}%
^{q}\left( Y\right) \right) \subset \mathcal{K}^{r}\left( Z\right) ,
\end{equation*}%
whence (\ref{Psin}) follows.\FRAME{ftbphFU}{4.8038in}{2.0788in}{0pt}{\Qcb{%
The $\Psi $-images of the spaces $\mathcal{K}^{\prime p}\left( X\right)
\otimes \mathcal{A}^{q}\left( Y\right) +\mathcal{A}^{\prime p}\left(
X\right) \otimes \mathcal{K}^{q}\left( Y\right) $}}{\Qlb{pic37}}{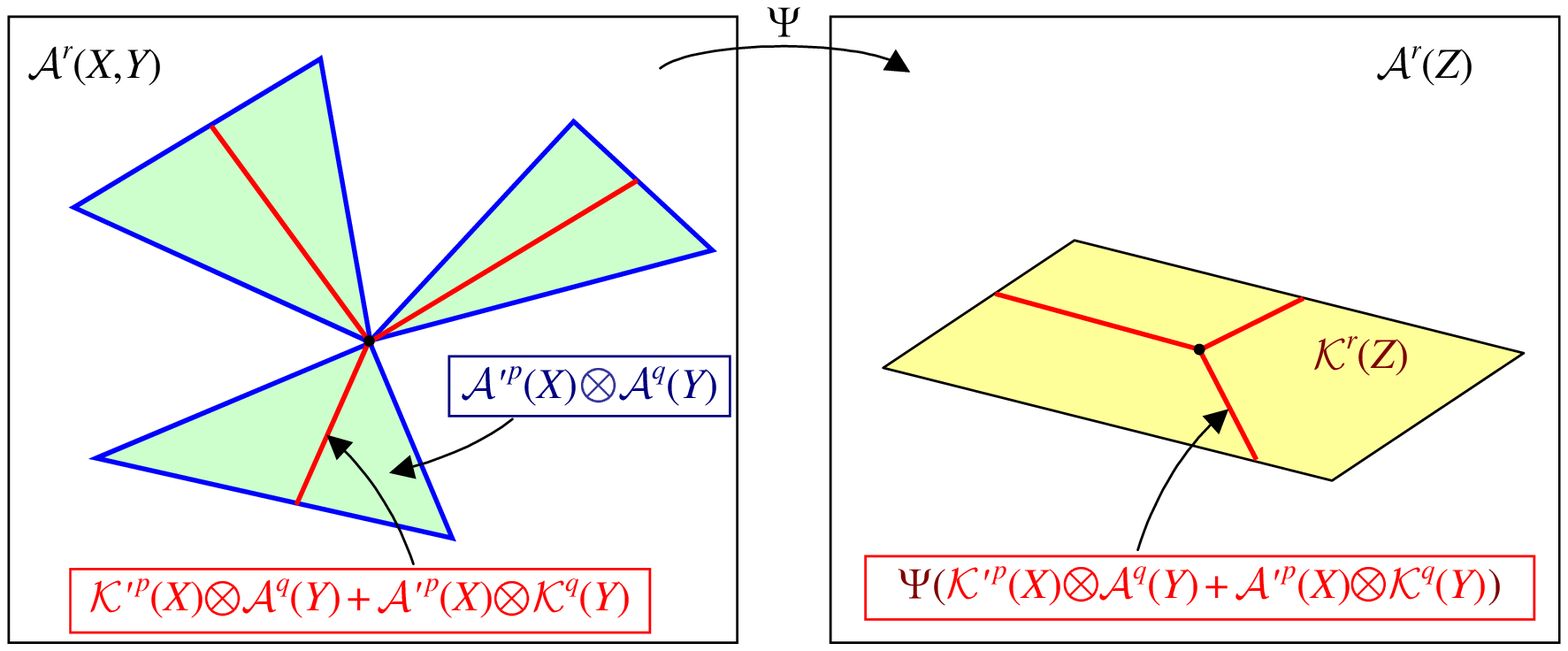}{%
\special{language "Scientific Word";type "GRAPHIC";maintain-aspect-ratio
TRUE;display "USEDEF";valid_file "F";width 4.8038in;height 2.0788in;depth
0pt;original-width 6.8226in;original-height 2.9301in;cropleft "0";croptop
"1";cropright "1";cropbottom "0";filename 'pic37.eps';file-properties
"XNPEU";}}

Observe that the spaces $\mathcal{K}^{\prime p}\left( X\right) \otimes 
\mathcal{A}^{q}\left( Y\right) +\mathcal{A}^{\prime p}\left( X\right)
\otimes \mathcal{K}^{q}\left( Y\right) $ have trivial intersections across
different pairs $p,q$ as they are subspaces of $\mathcal{A}^{\prime p}\left(
X\right) \otimes \mathcal{A}^{q}\left( Y\right) $ (cf. Fig. \ref{pic37}).
Since $\Psi $ is a monomorphism, the same applies to the $\Psi $-images of
those spaces. Since by (\ref{Psin}) all the $\Psi $-images lie in $\mathcal{K%
}^{r}\left( Z\right) $, we obtain (\ref{KAK}).
\end{proof}

\begin{remark}
\label{RemBasis}\RM It follows from (\ref{FFOM=}) that $\Omega _{r}\left(
Z\right) $ has a basis 
\begin{equation}
\tbigsqcup_{\left\{ p,q\geq -1:p+q=r-1\right\} }\left\{ u_{i}^{\left(
p\right) }v_{j}^{\left( q\right) }\right\} ,  \label{uvbasis}
\end{equation}%
where $\left\{ u_{i}^{\left( p\right) }\right\} $ is a basis in $\Omega
_{p}\left( X\right) $ and $\left\{ v_{j}^{\left( q\right) }\right\} $ is a
basis in $\Omega _{q}\left( Y\right) $. In the same way one expresses a
basis in $\widetilde{H}_{r}\left( Z\right) $ via the basis in $\widetilde{H}%
_{p}\left( X\right) $ and $\widetilde{H}_{q}\left( Y\right) $.
\end{remark}

\begin{example}
\RM\label{Expic35}Consider the graph $Z=X\ast Y$ as on Fig. \ref{pic35}. We
have%
\begin{equation*}
\begin{array}{ll}
\Omega _{0}\left( X\right) =\limfunc{span}\left\{ e_{0},e_{1},e_{2}\right\}
& \Omega _{0}\left( Y\right) =\limfunc{span}\left\{
e_{3},e_{4},e_{5},e_{6}\right\} \\ 
\Omega _{1}\left( X\right) =\limfunc{span}\left\{
e_{01},e_{02},e_{12}\right\} & \Omega _{1}\left( Y\right) =\limfunc{span}%
\left\{ e_{34},e_{35},e_{46},e_{56}\right\} \\ 
\Omega _{2}\left( X\right) =\limfunc{span}\left\{ e_{012}\right\} & \Omega
_{2}\left( Y\right) =\limfunc{span}\left\{ e_{346}-e_{356}\right\} \\ 
\Omega _{p}\left( X\right) =\left\{ 0\right\} \text{ for }p\geq 3 & \Omega
_{q}\left( Y\right) =\left\{ 0\right\} \text{ for }q\geq 3.%
\end{array}%
\end{equation*}%
Using Remark \ref{RemBasis}, we can obtain explicitly the basis in all $%
\Omega _{r}\left( Z\right) $. For example, 
\begin{eqnarray*}
\dim \Omega _{1}\left( Z\right) &=&19 \\
\dim \Omega _{2}\left( Z\right) &=&25 \\
\dim \Omega _{3}\left( Z\right) &=&19 \\
\Omega _{4}\left( Z\right) &=&\limfunc{span}\left\{
e_{01346}-e_{01356},e_{02346}-e_{02356},e_{12346}-e_{12356},e_{01234},e_{01235},e_{01246},e_{01256}\right\}
\\
\Omega _{5}\left( Z\right) &=&\limfunc{span}\left\{
e_{012346}-e_{012356}\right\} \\
\dim \Omega _{r}\left( Z\right) &=&0\ \text{for }r\geq 6.
\end{eqnarray*}%
Since by Proposition \ref{Pcycle} all homologies $\widetilde{H}_{p}\left(
X\right) $ and $\widetilde{H}_{q}\left( Y\right) $ are trivial, and so are $%
\widetilde{H}_{r}\left( Z\right) .$
\end{example}

\begin{example}
\RM Consider a slight modification of the previous example -- the digraph $%
Z=X\ast Y$ as on Fig. \ref{pic38}\FRAME{ftbphFU}{4.1635in}{2.0332in}{0pt}{%
\Qcb{A digraph $Z=X\ast Y$}}{\Qlb{pic38}}{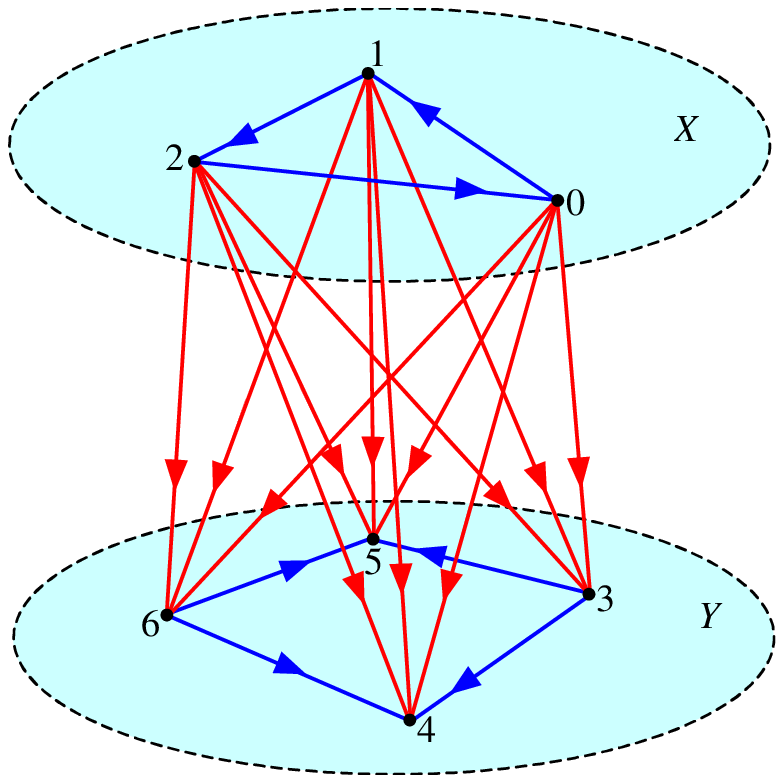}{\special{language
"Scientific Word";type "GRAPHIC";maintain-aspect-ratio TRUE;display
"USEDEF";valid_file "F";width 4.1635in;height 2.0332in;depth
0pt;original-width 5.9088in;original-height 2.8644in;cropleft "0";croptop
"1";cropright "1";cropbottom "0";filename 'pic38.eps';file-properties
"XNPEU";}}

In this case we have by Proposition \ref{Pcycle} that all homologies $%
\widetilde{H}_{p}\left( X\right) $ and $\widetilde{H}_{q}\left( Y\right) $
are trivial except for 
\begin{eqnarray*}
H_{1}\left( X\right) &=&\limfunc{span}\left\{ e_{01}+e_{12}+e_{20}\right\} ,
\\
H_{1}\left( Y\right) &=&\limfunc{span}\left\{
e_{35}-e_{65}+e_{64}-e_{34}\right\} .
\end{eqnarray*}%
Therefore, all $\widetilde{H}_{r}\left( Z\right) $ are trivial except for $%
H_{3}\left( Z\right) $ that is generated by a single element%
\begin{equation*}
e_{0135}-e_{0165}+e_{0164}-e_{0134}+e_{1235}-e_{1265}+e_{1264}-e_{1234}+e_{2035}-e_{2065}+e_{2064}-e_{2034}.
\end{equation*}
\end{example}

\subsection{Cones and simplexes}

\label{here copy(3)}\label{SecCone}

\begin{definition}
\RM A \emph{cone} over a digraph $X$ is a digraph $\limfunc{Cone}X$ that is
obtained from $X$ by adding one more vertex $a$ and all the edges of the
form $b\rightarrow a$ for all $b\in X$. The vertex $a$ is called the cone
vertex.
\end{definition}

Clearly, we have $\limfunc{Cone}X=X\ast Y$ where $Y$ consists of a single
vertex $a$.

\begin{proposition}
\label{Pcone}For any digraph $X$, we have for any $r\geq 0$%
\begin{equation}
\Omega _{r}\left( \limfunc{Cone}X\right) \cong \Omega _{r-1}\left( X\right) ,
\label{rr-1}
\end{equation}%
where the isomorphism is given by the mapping $u\mapsto ue_{a}$ from $\Omega
_{r-1}\left( X\right) $ to $\Omega _{r}\left( \limfunc{Cone}X\right) $ where 
$a$ is the cone vertex. Furthermore, all the reduced homologies of $\limfunc{%
Cone}X$ are trivial.
\end{proposition}

\begin{proof}
Since $\limfunc{Cone}X=X\ast Y$ with $Y=\left\{ a\right\} $, (\ref{rr-1})
follows from (\ref{Omrpq-1}) and $\Omega _{0}\left( Y\right) =\limfunc{span}%
\left\{ e_{a}\right\} $. Since all the homologies $\widetilde{H}_{q}\left(
Y\right) \ $are trivial, it follows from (\ref{HZ}) that all homologies $%
\widetilde{H}_{r}\left( Z\right) $ are also trivial. The latter follows from
Theorem \ref{Tstar} since $\limfunc{Cone}X$ is inverse star-shaped.
\end{proof}

\begin{example}
\RM Clearly, a simplex-digraph $\func{Sm}_{n}$ can be regarded as a cone
over $\func{Sm}_{n-1}$ (cf. Section \ref{Exsimplex}). Since $\Omega
_{0}\left( \func{Sm}_{0}\right) $ is spanned by a $0$-path $e_{0}$, we
obtain by induction from (\ref{rr-1}) that $\Omega _{n}\left( \func{Sm}%
_{n}\right) $ is spanned by a path $e_{01...n}$.
\end{example}

\begin{example}
\RM Let $G$ be a square digraph%
\begin{equation*}
\begin{array}{ccc}
_{1}\bullet & \longrightarrow & \bullet _{3} \\ 
\ \uparrow &  & \uparrow \  \\ 
_{0}\bullet & \longrightarrow & \bullet _{2}%
\end{array}%
\end{equation*}
Then $\func{Cone}G$ is a \emph{pyramid} shown on Fig. \ref{pic20}.\FRAME{%
ftbphFU}{5.3757in}{1.6596in}{0pt}{\Qcb{A pyramid graph}}{\Qlb{pic20}}{%
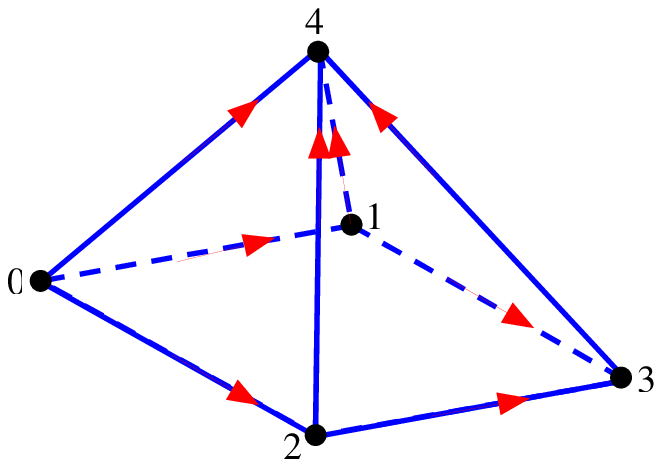}{\special{language "Scientific Word";type
"GRAPHIC";maintain-aspect-ratio TRUE;display "USEDEF";valid_file "F";width
5.3757in;height 1.6596in;depth 0pt;original-width 6.3027in;original-height
1.9268in;cropleft "0";croptop "1";cropright "1";cropbottom "0";filename
'pic20.eps';file-properties "XNPEU";}}Since $\Omega _{2}\left( G\right) $ is
spanned by a $2$-path $e_{013}-e_{023}$, we obtain that $\Omega _{3}\left( 
\func{Cone}G\right) $ is spanned by a $3$-path $e_{0134}-e_{0234}$.
\end{example}

\subsection{Suspension and spheres}

\label{SecSus}

\begin{definition}
\RM A \emph{suspension} over a digraph $X$ is a digraph $\func{Sus}X$ that
is obtained from $X$ by adding two vertices $a,b$ and all the edges $%
c\rightarrow a$ and $c\rightarrow b$ for all $c\in X$. The vertices $a,b$
are called the suspension vertices.
\end{definition}

Clearly, we have $\func{Sus}X=X\ast Y$ where $Y$ is a digraph that consists
of two vertices $a,b$ and no edges. An example of a suspension digraph is
shown on Fig. \ref{pic15}.

\FRAME{ftbphFU}{6.9613in}{1.7334in}{0pt}{\Qcb{A suspension digraph}}{\Qlb{%
pic15}}{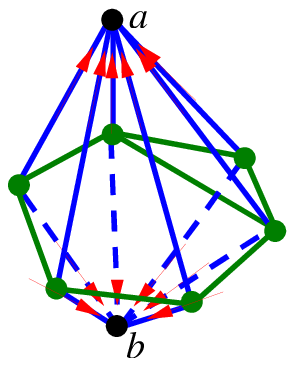}{\special{language "Scientific Word";type
"GRAPHIC";maintain-aspect-ratio TRUE;display "USEDEF";valid_file "F";width
6.9613in;height 1.7334in;depth 0pt;original-width 6.3027in;original-height
1.5939in;cropleft "0";croptop "1";cropright "1";cropbottom "0";filename
'pic15.eps';file-properties "XNPEU";}}

\begin{proposition}
\label{Tsus}For any digraph $X$ we have, for any $r\geq 0$,%
\begin{equation}
\Omega _{r}\left( \func{Sus}X\right) \cong \Omega _{r-1}\left( X\right)
\otimes \limfunc{span}\left\{ e_{a},e_{b}\right\} ,  \label{rab}
\end{equation}%
where $a,b$ are the suspension vertices and the isomorphism is given by the
mappings $u\otimes e_{a}\mapsto ue_{a}$ and $u\otimes e_{b}\mapsto ue_{b}.$
Furthermore, we have 
\begin{equation}
\widetilde{H}_{r}\left( \func{Sus}X\right) \cong \widetilde{H}_{r-1}\left(
X\right) ,  \label{pp-1}
\end{equation}%
where the isomorphism is given by the mapping $u\mapsto u\left(
e_{a}-e_{b}\right) $.
\end{proposition}

\begin{proof}
Let $Y$ as above. Since $\Omega _{0}\left( Y\right) =\limfunc{span}\left\{
e_{a},e_{b}\right\} $ and all other $\Omega _{q}\left( Y\right) $ are
trivial, (\ref{rab}) follows from (\ref{Omrpq-1}). Since $\widetilde{H}%
_{q}\left( Y\right) =\left\{ 0\right\} $ for all $q\neq 0$ and $\widetilde{H}%
_{0}\left( Y\right) =\limfunc{span}\left\{ e_{a}-e_{b}\right\} $, (\ref{pp-1}%
) follows from (\ref{HZ}).
\end{proof}

\begin{corollary}
We have $\chi \left( \func{Sus}X\right) =2-\chi \left( X\right) .$
\end{corollary}

\begin{proof}
Denoting $Z=$ $\func{Sus}X$ we obtain%
\begin{eqnarray*}
\chi \left( Z\right) &=&1+\sum_{p\geq 1}\left( -1\right) ^{p}\dim
H_{p}\left( Z\right) \\
&=&1+\sum_{p\geq 1}\left( -1\right) ^{p}\dim \widetilde{H}_{p-1}\left(
X\right) \\
&=&1-\sum_{q\geq 0}\left( -1\right) ^{q}\dim \widetilde{H}_{q}\left( X\right)
\\
&=&2-\sum_{q\geq 0}\left( -1\right) ^{q}\dim H_{q}\left( X\right) =2-\chi
\left( X\right) .
\end{eqnarray*}
\end{proof}

In particular, having examples of digraphs $X$ with arbitrary negative
values of $\chi $ (cf. Example \ref{Exantisnake}), we obtain examples of
digraphs $\func{Sus}X$ with arbitrary positive values of $\chi $.

\begin{example}
\RM Let $S$ be any cycle-graph that is neither triangle nor square. We
regards $S$ as an analog of a circle. Define $S_{n}$ inductively by $S_{1}=S$
and $S_{n+1}=\func{Sus}S_{n}.$ Then $S_{n}$ can be regarded as $n$%
-dimensional sphere-graph. An example of $S_{2}$ is shown on Fig. \ref%
{pic15a}. \FRAME{ftbphFU}{6.3304in}{1.8299in}{0pt}{\Qcb{A graph $S_{2}$
based on a 3-vertex cycle $S$}}{\Qlb{pic15a}}{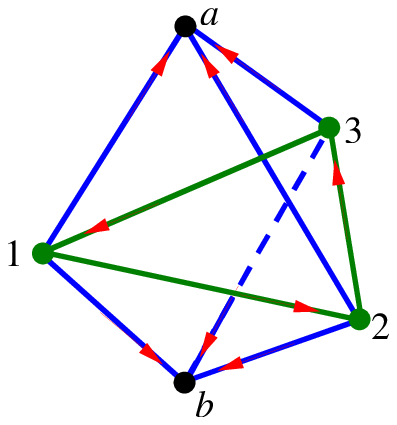}{\special{language
"Scientific Word";type "GRAPHIC";maintain-aspect-ratio TRUE;display
"USEDEF";valid_file "F";width 6.3304in;height 1.8299in;depth
0pt;original-width 6.3027in;original-height 1.8023in;cropleft "0";croptop
"1";cropright "1";cropbottom "0";filename 'pic15a.eps';file-properties
"XNPEU";}}

Since $\chi \left( S\right) =0$ by Proposition \ref{Propcycle}, it follows
that $\chi \left( S_{n}\right) =0\ \ $if $n$ is odd and $\chi \left(
S_{n}\right) =2$ if $n$ is even. Proposition \ref{Tsus} also implies that $%
\dim H_{n}\left( S_{n}\right) =\dim H_{1}\left( S\right) =1$, which gives an
example of a non-trivial $H_{n}$ with an arbitrary $n.$

Let $v$ be an $1$-path on $S$ that spans $H_{1}\left( S\right) $ (see
Section \ref{SecCycle}). If $S_{n+1}$ is a suspension of $S_{n}$ on the
vertices $a_{n},b_{n}$ then we obtain by induction that the spanning element
of $H_{n}\left( S_{n}\right) $ is 
\begin{equation*}
v\left( e_{a_{1}}-e_{b_{1}}\right) \left( e_{a_{2}}-e_{b_{2}}\right)
...\left( e_{a_{n-1}}-e_{b_{n-1}}\right) .
\end{equation*}%
For example, if $S$ is a cycle-graph on Fig. \ref{pic15a} with $V=\left\{
1,2,3\right\} $ and $E=\left\{ 12,23,31\right\} $, then $%
v=e_{12}+e_{23}+e_{31}$, whence the spanning element of $H_{2}\left(
S_{2}\right) $ is 
\begin{equation*}
v\left( e_{a}-e_{b}\right) =\left( e_{12a}+e_{23a}+e_{31a}\right) -\left(
e_{12b}+e_{23b}+e_{31b}\right) .
\end{equation*}
\end{example}

\begin{example}
\label{Expic13}\RM Another example of a $2$-dimensional sphere-graph $G$ is
shown on Fig. \ref{pic13}.\FRAME{ftbphFU}{5.2509in}{2.0791in}{0pt}{\Qcb{An
octahedron $G$}}{\Qlb{pic13}}{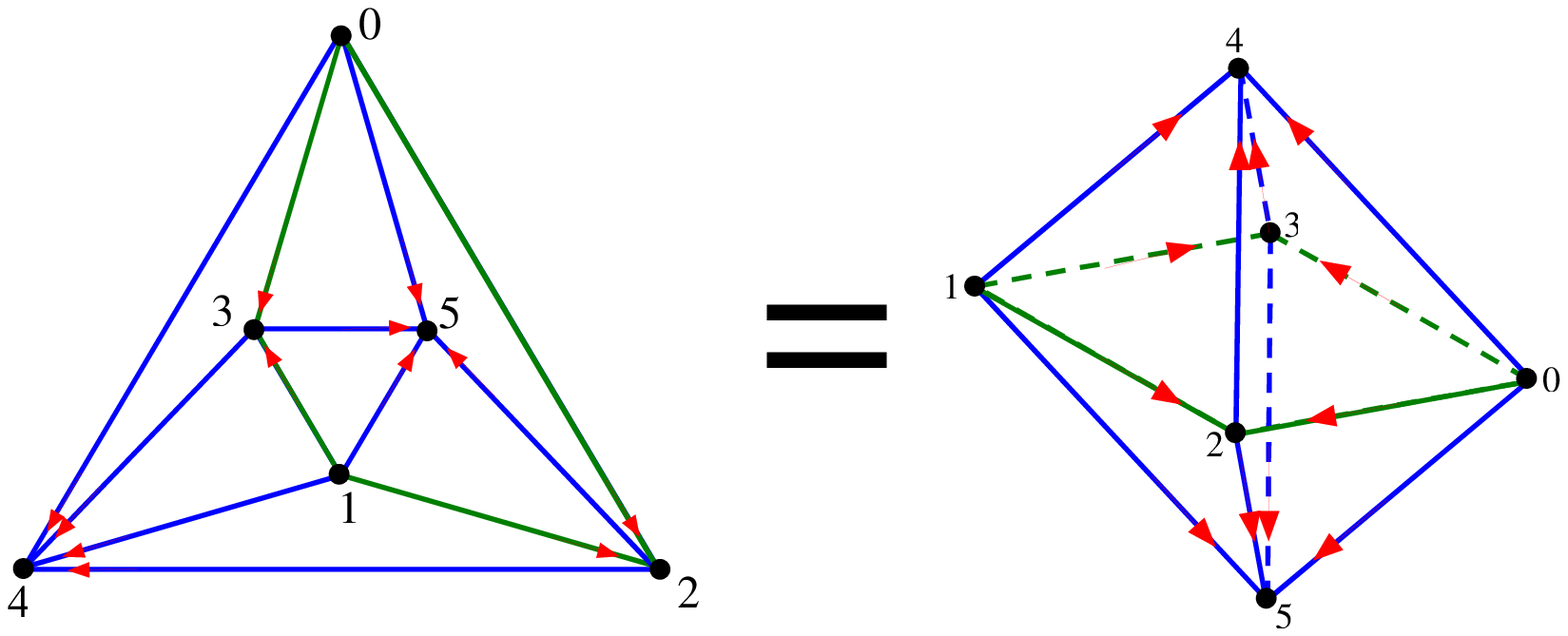}{\special{language "Scientific
Word";type "GRAPHIC";maintain-aspect-ratio TRUE;display "USEDEF";valid_file
"F";width 5.2509in;height 2.0791in;depth 0pt;original-width
7.0396in;original-height 2.7648in;cropleft "0";croptop "1";cropright
"1";cropbottom "0";filename 'pic13a.eps';file-properties "XNPEU";}}

Indeed, we have $G=\func{Sus}G^{\prime }$ where $G^{\prime }$ is the
subgraph with vertices $\left\{ 0,1,2,3\right\} $ that is a cycle-graph. By
Proposition \ref{Pcycle} we have 
\begin{equation*}
\dim H_{0}\left( G^{\prime }\right) =1,\ \ \dim H_{1}\left( G^{\prime
}\right) =1,\ \dim H_{p}\left( G^{\prime }\right) =0\ \text{for }p\geq 2.
\end{equation*}%
The same follows from the fact that $G^{\prime }=\func{Sus}G^{\prime \prime
} $ where $G^{\prime \prime }$ is a subgraph with vertices $\left\{
0,1\right\} $. By Proposition \ref{Tsus} we obtain 
\begin{equation}
\dim H_{0}\left( G\right) =1,\ \dim H_{1}\left( G\right) =0,\ \,\dim
H_{2}\left( G\right) =1,\ \dim H_{p}\left( G\right) =0\ \ \text{for }p\geq 3.
\label{G012}
\end{equation}%
Consequently, $\chi \left( G\right) =2.$

In the digraph $G$ we have%
\begin{equation*}
\dim \Omega _{0}=\left\vert V\right\vert =6\ \ \text{and\ \ }\dim \Omega
_{1}=\left\vert E\right\vert =12
\end{equation*}%
and%
\begin{equation*}
\mathcal{A}_{2}=\limfunc{span}\left\{
e_{024},e_{025},e_{034},e_{035},e_{124},e_{125},e_{134},e_{135}\right\} .
\end{equation*}%
The set of semi-edges is empty, whence by Proposition \ref{PropS} $\dim
\Omega _{2}=\dim \mathcal{A}_{2}=8$ and, hence, $\Omega _{2}=\mathcal{A}_{2}$%
. Alternatively, one can see that because all the $2$-paths spanning $%
\mathcal{A}_{2}$ are triangles and there are no squares. Also, there are no
allowed $3$-paths, so that $\mathcal{A}_{3}=\left\{ 0\right\} $ whence $\dim
\Omega _{p}=0$ for all $p\geq 3.$

Let us determine a spanning element of $H_{2}\left( G\right) .$ Clearly, $%
H_{0}\left( G^{\prime \prime }\right) =\limfunc{span}\left\{
e_{0},e_{1}\right\} $ whence by (\ref{H0v})%
\begin{equation*}
\widetilde{H}_{0}\left( G^{\prime \prime }\right) =\limfunc{span}\left\{
e_{0}-e_{1}\right\} .
\end{equation*}%
By the second claim of Proposition \ref{Tsus}, we have%
\begin{eqnarray*}
H_{1}\left( G^{\prime }\right) &=&\limfunc{span}\left\{ \tau \left(
e_{0}-e_{1}\right) \right\} =\limfunc{span}\left\{ \left( e_{0}-e_{1}\right)
\left( e_{2}-e_{3}\right) \right\} \\
&=&\limfunc{span}\left\{ e_{02}-e_{03}-e_{12}+e_{13}\right\} .
\end{eqnarray*}%
Alternatively, the same spanning element can be obtained by (\ref{vii+1}).

Applying Proposition \ref{Tsus} again, we obtain%
\begin{eqnarray*}
H_{2}\left( G\right) &=&\limfunc{span}\left\{ \tau \left(
e_{02}-e_{03}-e_{12}+e_{13}\right) \right\} \\
&=&\limfunc{span}\left\{ \left( e_{02}-e_{03}-e_{12}+e_{13}\right) \left(
e_{4}-e_{5}\right) \right\} \\
&=&\limfunc{span}\left\{
e_{024}-e_{025}-e_{034}+e_{035}-e_{124}+e_{125}+e_{134}-e_{135}\right\} .
\end{eqnarray*}%
Note that the spanning element of $H_{2}\left( G\right) $ is exactly a
surface path of a triangulation of $\mathbb{S}^{2}$ into an octahedron (cf.
Section \ref{Sectrian}). In this case the surface path is not exact so that
there is no solid path representing an octahedron.
\end{example}

\section{Cartesian product of path complexes}

\setcounter{equation}{0}\label{Sec6}\label{SecProduct}Let us fix some
notation to be used in this section. For a finite set $V$, denote by $%
R\left( V\right) $ the path complex on $V$ that consists of all regular
elementary paths on $V$. Then $R_{p}\left( V\right) $ denotes the set of all
regular elementary $p$-paths on $V$, for any non-negative integer $p$.

As before, the space $\mathcal{R}_{p}\left( V\right) $ of regular $p$-paths
is the set of all formal finite linear combinations of paths from $%
R_{p}\left( V\right) $ with the coefficients from the field $\mathbb{K}$.

In this Section all path complexes are regular and their chain complexes are
always truncated and regular. In particular, we set $\mathcal{R}%
_{-1}=\left\{ 0\right\} $. Notation $W_{\bullet }$\ means here $\left\{
W_{n}\right\} _{n\geq 0}.$

When considering more than one path complex we use the notation introduced
in Section \ref{SecJoin}.

\subsection{Cross product of regular paths}

\label{here copy(1)}Given two finite sets $X,Y$, consider their Cartesian
product 
\begin{equation*}
Z=X\times Y=\left\{ \left( x,y\right) :x\in X\text{ and }y\in Y\right\} .
\end{equation*}%
Let $z=z_{0}z_{1}...z_{r}$ be a regular elementary $r$-path on $Z$, where $%
z_{k}=\left( x_{k},y_{k}\right) $ with $x_{k}\in X$ and $y_{k}\in Y$. We say
that the path $z$ is \emph{step-like} if, for any $k=1,...,r$, either $%
x_{k-1}=x_{k}$ or $y_{k-1}=y_{k}$ is satisfied (in fact, exactly one of
these conditions holds as $z$ is regular); in other words, any couple $%
z_{k-1}z_{k}$ of consecutive vertices is either \emph{vertical} (when $%
x_{k-1}=x_{k}$) or \emph{horizontal} (when $y_{k-1}=y_{k}$):

\begin{equation*}
\text{vertical couple }%
\begin{array}{c}
\overset{z_{k}}{\bullet } \\ 
\uparrow \\ 
\overset{z_{k-1}}{\bullet }%
\end{array}%
\text{ and horizontal couple }%
\begin{array}{ccc}
\overset{z_{k-1}}{\bullet } & \longrightarrow & \overset{z_{k}}{\bullet }%
\end{array}%
\ \ \ .
\end{equation*}

Any step-like path $z$ on $Z$ determines regular elementary paths $x$ on $X$
and $y$ on $Y$ by projection. More precisely, $x$ is obtained from $z$ by
taking the sequence of all $X$-components of the vertices of $z$ and then by
collapsing in it any subsequence of repeated vertices to one vertex. The
same rule applies to $y$. By construction, the projections $x$ and $y$ are 
\emph{regular} elementary paths on $X$ and $Y$, respectively. Let the
projections of $z$ be $x=x_{0}...x_{p}$ and $y=y_{0}...y_{q}.$ Then $p+q=r$
where $r$ is the length of $z$, and every vertex $z_{k}$ of the path $z$ has
a form $\left( x_{i},y_{j}\right) $, whereas the previous vertex $z_{k-1}$
is either $\left( x_{i},y_{j-1}\right) $ or $\left( x_{i-1},y_{j}\right) $
as on the diagram:

\begin{equation*}
\begin{array}{c}
\overset{z_{k}=\left( x_{i},y_{j}\right) }{\bullet } \\ 
\uparrow \\ 
\overset{z_{k-1}=\left( x_{i},y_{j-1}\right) }{\bullet }%
\end{array}%
\ \ \text{or\ \ \ }%
\begin{array}{ccc}
\overset{z_{k-1}=\left( x_{i-1},y_{j}\right) }{\bullet } & \longrightarrow & 
\overset{z_{k}=\left( x_{i},y_{j}\right) }{\bullet }%
\end{array}%
.
\end{equation*}%
An example of a step-like path $z$ with its projections is shown on Fig. \ref%
{pic30}.

\FRAME{ftbpFU}{4.4394in}{2.363in}{0pt}{\Qcb{A step-like path $z$ and its
projections $x$ and $y$}}{\Qlb{pic30}}{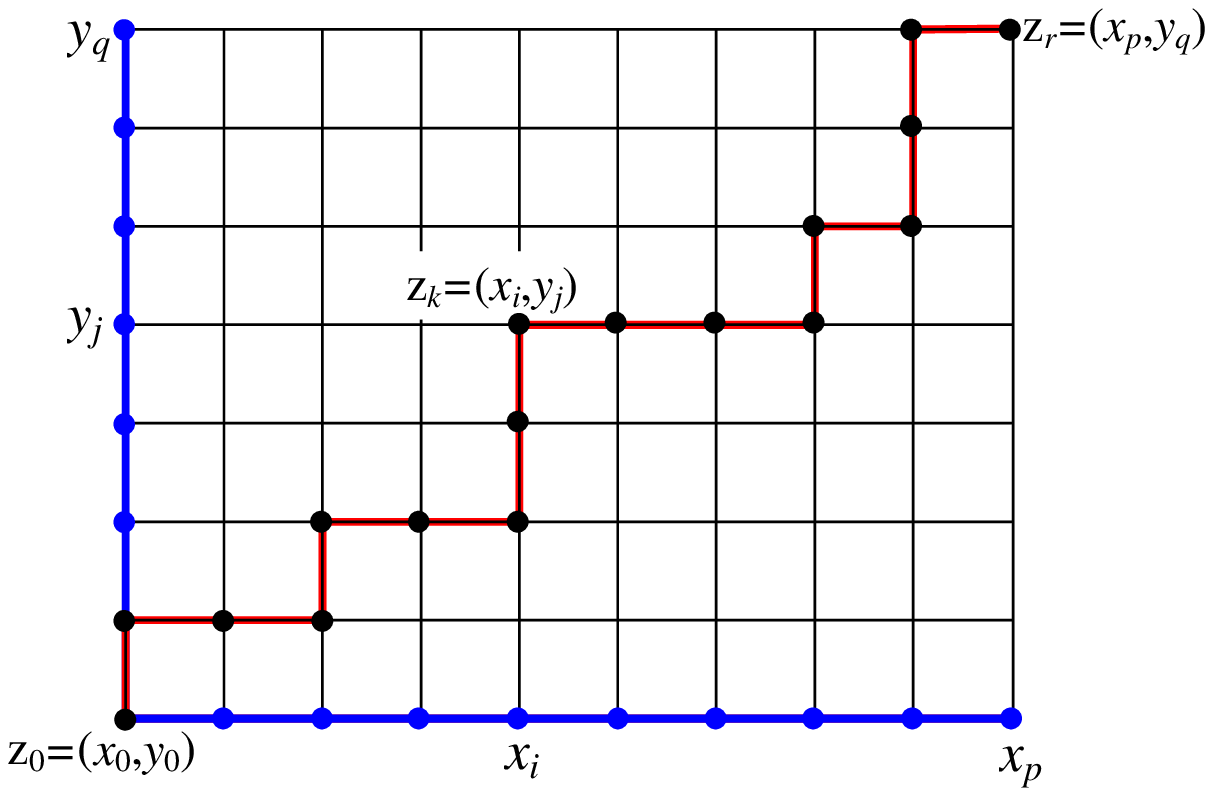}{\special{language
"Scientific Word";type "GRAPHIC";maintain-aspect-ratio TRUE;display
"USEDEF";valid_file "F";width 4.4394in;height 2.363in;depth
0pt;original-width 6.3027in;original-height 3.3356in;cropleft "0";croptop
"1";cropright "1";cropbottom "0";filename 'pic30.eps';file-properties
"XNPEU";}}

Every vertex $\left( x_{i},y_{j}\right) $ of a step-like path $z$ can be
represented as a point $\left( i,j\right) $ of $\mathbb{Z}^{2}$ so that the
whole path $z$ is represented by a \emph{staircase }$S\left( z\right) $ in $%
\mathbb{Z}^{2}$ connecting the points $\left( 0,0\right) $ and $\left(
p,q\right) $. Define the \emph{elevation }$L\left( z\right) $ of the path $z$
as the number of cells in $\mathbb{Z}_{+}^{2}$ below the staircase $S\left(
z\right) $ (the shaded area on Fig. \ref{pic22}).

\FRAME{ftbphFU}{4.5307in}{2.2615in}{0pt}{\Qcb{A staircase $S\left( z\right) $
and its elevation $L\left( z\right) .$ On this picture $L\left( z\right)
=30. $}}{\Qlb{pic22}}{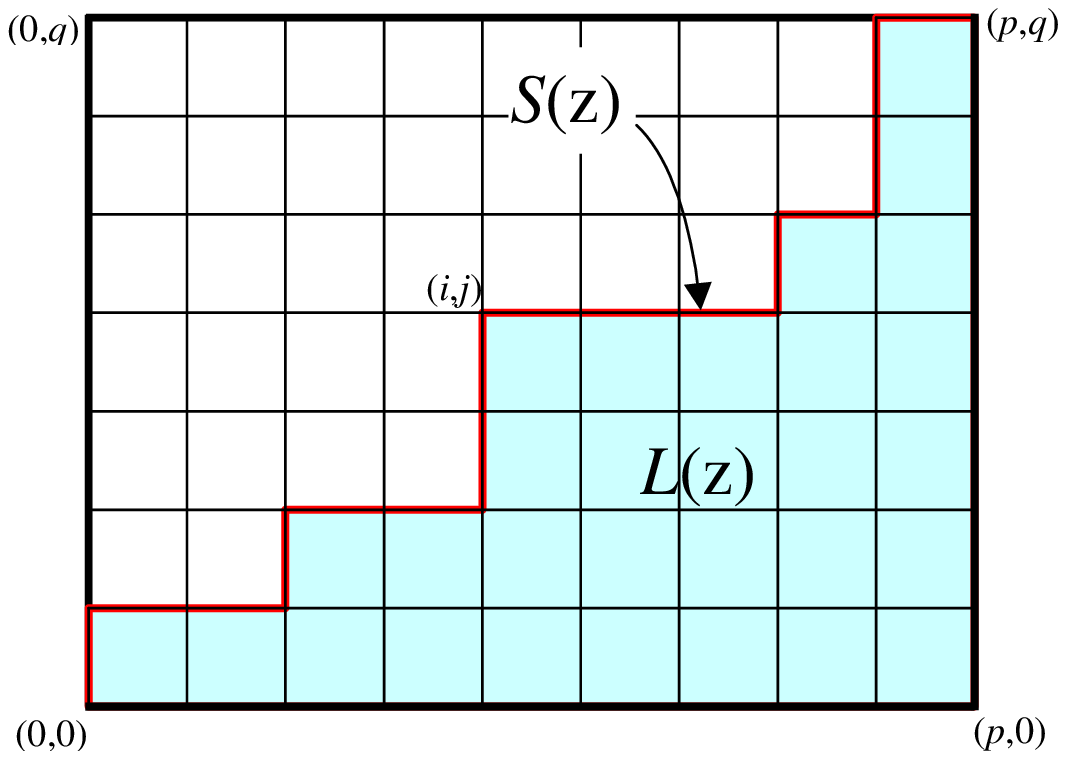}{\special{language "Scientific Word";type
"GRAPHIC";maintain-aspect-ratio TRUE;display "USEDEF";valid_file "F";width
4.5307in;height 2.2615in;depth 0pt;original-width 6.3027in;original-height
2.9879in;cropleft "0";croptop "1";cropright "1";cropbottom "0";filename
'pic22.eps';file-properties "XNPEU";}}

With any step-like path $z=z_{0}...z_{r}$ let us associate a sequence $%
\left\{ d_{k}\right\} _{k=1}^{r}$ where $d_{k}=1$ if the couple $%
z_{k-1}z_{k} $ is vertical, and $d_{k}=0$ if this couple is horizontal. If
the total number of horizontal couples is $p$ and of the vertical ones -- $q$%
, then the sequence $\left\{ d_{k}\right\} $ is a permutation of the
sequence 
\begin{equation}
\underset{p}{\{\underbrace{0,...,0}},\underset{q}{\underbrace{1,...,1}}\}.
\label{p0q1}
\end{equation}%
It is easy to see that $L\left( z\right) $ is equal to the minimal number of
transposition in $\left\{ d_{k}\right\} $ that brings it to the form (\ref%
{p0q1}). In the sequel we will need only the parity of the elevation $%
L\left( z\right) $ that is determined by the parity of the permutation $%
\left\{ d_{k}\right\} $.

\begin{definition}
\RM Given paths $u\in \mathcal{R}_{p}\left( X\right) $ and $v\in \mathcal{R}%
_{q}\left( Y\right) $ with some $p,q\geq 0$, define a path $u\times v$ on $Z$
by the following rule: for any step-like elementary $\left( p+q\right) $%
-path $z$ on $Z$, the component $\left( u\times v\right) ^{z}$ is defined by%
\begin{equation}
\left( u\times v\right) ^{z}=\left( -1\right) ^{L\left( z\right) }u^{x}v^{y},
\label{uvz}
\end{equation}%
where $x$ and $y$ are the projections of $z$ onto $X$ and $Y$, respectively,
and $u^{x}$ and $v^{y}$ are the corresponding components of $u$ and $v$. For
non-step-like paths $z$ set $\left( u\times v\right) ^{z}=0.$

The path $u\times v$ is called the (Cartesian) cross product of $u$ and $v$.
It follows that $u\times v\in \mathcal{R}_{p+q}\left( Z\right) .$
\end{definition}

Given a step-like $\left( p+q\right) $-path $z$ on $Z$, the projection of $z$
onto $X$ could be a \thinspace $p^{\prime }$-path $x$ with $p^{\prime }\neq
p $; in this case we set by definition $u^{x}\equiv 0.$ The same rule
applies to $v^{y}.$ In other words, $\left( u\times v\right) ^{z}$ may be
non-zero only when the projections of $z$ onto $X$ and $Y$ are $p$-path $x$
and $q$-path $y$ respectively, with non-zero $u^{x}$ and $v^{y}$.

For given paths $x\in R_{p}\left( X\right) $ and $y\in R_{q}\left( Y\right) $
with non-negative integers $p,q$, denote by $\Pi _{x,y}$ the set of all
step-like paths $z$ on $Z$ whose projections on $X$ and $Y$ are respectively 
$x$ and $y$. It follows from (\ref{uvz}) that 
\begin{equation}
e_{x}\times e_{y}=\sum_{z\in \Pi _{x,y}}\left( -1\right) ^{L\left( z\right)
}e_{z}.  \label{uvz1}
\end{equation}%
It is not difficult to see that the cross product is associative. \label%
{rem: prove that the cross product is associative}

\begin{example}
\RM Let us denote the vertices of $X$ by letters $a,b,c$ etc and the
vertices of $Y$ by integers $0,1,2,$ etc so that the vertices of $Z$ can be
denoted as the fields on the chessboard, for example, $a0,b1$ etc. Then we
have%
\begin{equation*}
e_{a}\times e_{01}=e_{a0a1}
\end{equation*}%
\begin{equation*}
e_{ab}\times e_{0}=e_{a0b0}
\end{equation*}%
\begin{equation*}
e_{ab}\times e_{01}=e_{a0b0b1}-e_{a0a1b1}
\end{equation*}%
\begin{equation*}
e_{abc}\times e_{01}=e_{a0b0c0c1}-e_{a0b0b1c1}+e_{a0a1b1c1}
\end{equation*}%
\begin{eqnarray*}
e_{abc}\times e_{012} &=&e_{a0b0c0c1c2}-e_{a0b0b1c1c2}+e_{a0b0b1b2c2} \\
&&+e_{a0a1b1c1c2}-e_{a0a1b1b2c2}+e_{a0a1a2b2c2}
\end{eqnarray*}%
etc (cf. Fig. \ref{pic24}). \FRAME{ftbphFU}{5.431in}{1.7781in}{0pt}{\Qcb{The
staircase $a0b0b1c1c2$ has elevation $1.$ Hence, $e_{a0b0b1c1c2}$ enters the
product $e_{abc}\times e_{012}$ with the negative sign.}}{\Qlb{pic24}}{%
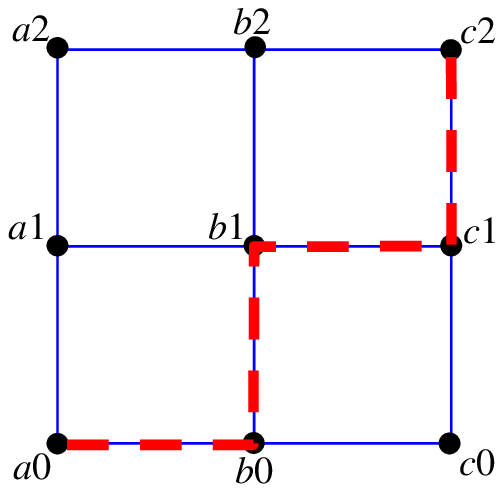}{\special{language "Scientific Word";type
"GRAPHIC";maintain-aspect-ratio TRUE;display "USEDEF";valid_file "F";width
5.431in;height 1.7781in;depth 0pt;original-width 6.0026in;original-height
1.9429in;cropleft "0";croptop "1";cropright "1";cropbottom "0";filename
'pic24.eps';file-properties "XNPEU";}}
\end{example}

\subsection{The product rule}

\label{Secduxv}In this and next sections we use the regular truncated
version of the boundary operator $\partial $.

\begin{proposition}
\label{Propdusqv}If$\ u\in \mathcal{R}_{p}\left( X\right) $ and $v\in 
\mathcal{R}_{q}\left( Y\right) $ where $p,q\geq 0$, then%
\begin{equation}
\partial \left( u\times v\right) =\left( \partial u\right) \times v+\left(
-1\right) ^{p}u\times \left( \partial v\right) .  \label{dusqv}
\end{equation}
\end{proposition}

\begin{proof}
It suffices to prove (\ref{dusqv}) for the case $u=e_{x}$ and $v=e_{y}$
where $x$ and $y$ are regular elementary \thinspace $p$-path on $X$ and $q$%
-path on $Y$, respectively. Set $r=p+q$ so that $e_{x}\times e_{y}\in 
\mathcal{R}_{r}\left( Z\right) $. We have by (\ref{uvz1}) and (\ref{dev}) 
\begin{equation}
\partial \left( e_{x}\times e_{y}\right) =\sum_{z\in \Pi _{x,y}}\left(
-1\right) ^{L\left( z\right) }\partial e_{z}=\sum_{z\in \Pi
_{x,y}}\sum_{k=0}^{r}\left( -1\right) ^{L\left( z\right) +k}e_{z_{\left(
k\right) }},  \label{dexey}
\end{equation}%
where we use a shortcut%
\begin{equation*}
z_{\left( k\right) }=z_{0}...\widehat{z_{k}}%
...z_{r}=z_{0}...z_{k-1}z_{k+1}...z_{r}.
\end{equation*}%
Switching the order of the sums, rewrite (\ref{dexey}) in the form%
\begin{equation}
\partial \left( e_{x}\times e_{y}\right) =\sum_{k=0}^{r}\sum_{z\in \Pi
_{x,y}}\left( -1\right) ^{L\left( z\right) +k}e_{z_{\left( k\right) }}.
\label{dek}
\end{equation}%
Given an index $k=0,...,r$ and a path $z\in \Pi _{x,y}$, consider the
following four logically possible cases how horizontal and vertical couples
combine around $z_{k}$:%
\begin{equation*}
\begin{array}{cccccc}
\left( H\right) : & \overset{z_{k-1}}{\bullet } & \longrightarrow & \overset{%
z_{k}}{\bullet } & \longrightarrow & \overset{z_{k+1}}{\bullet }%
\end{array}%
\ \ \ \ \ \ \ \ \ \ \ \ \ \ 
\begin{array}{ccc}
&  & \overset{z_{k+1}}{\bullet } \\ 
&  & \uparrow \\ 
\left( V\right) : &  & \overset{z_{k}}{\bullet } \\ 
&  & \uparrow \\ 
&  & \overset{z_{k-1}}{\bullet }%
\end{array}%
\end{equation*}%
\begin{equation*}
\begin{array}{cccccccccccc}
\left( R\right) : &  &  & \overset{z_{k+1}}{\bullet } &  & \  & \  & \  & 
\left( L\right) : & \overset{z_{k}}{\bullet } & \longrightarrow & \overset{%
z_{k+1}}{\bullet } \\ 
&  &  & \uparrow &  &  &  &  &  & \uparrow &  &  \\ 
& \overset{z_{k-1}}{\bullet } & \longrightarrow & \overset{z_{k}}{\bullet }
&  &  &  &  &  & \overset{z_{k-1}}{\bullet } &  & 
\end{array}%
\end{equation*}%
Here $\left( H\right) $ stands for a horizontal position, $\left( V\right) $
for vertical, $\left( R\right) $ for right and $\left( L\right) $ for left.
If $k=0$ or $k=r$ then $z_{k-1}$ or $z_{k+1}$ should be ignored, so that one
has only two distinct positions $\left( H\right) $ and $\left( V\right) $.

If $z\in \Pi _{x,y}$ and $z_{k}$ stands in $\left( R\right) $ or $\left(
L\right) $ then consider a path $z^{\prime }\in \Pi _{x,y}$ such that $%
z_{i}^{\prime }=z_{i}$ for all $i\neq k$, whereas $z_{k}^{\prime }$ stands
in the opposite position $\left( L\right) $ or $\left( R\right) $,
respectively, as on the diagrams:%
\begin{equation*}
\begin{array}{ccc}
\overset{z_{k}^{\prime }}{\bullet } & \longrightarrow & \overset{z_{k+1}}{%
\bullet } \\ 
\uparrow &  & \uparrow \\ 
\overset{z_{k-1}}{\bullet } & \longrightarrow & \overset{z_{k}}{\bullet }%
\end{array}%
\ \ \ \ \ \ \ \ 
\begin{array}{ccc}
\overset{z_{k}}{\bullet } & \longrightarrow & \overset{z_{k+1}}{\bullet } \\ 
\uparrow &  & \uparrow \\ 
\overset{z_{k-1}}{\bullet } & \longrightarrow & \overset{z_{k}^{\prime }}{%
\bullet }%
\end{array}%
\end{equation*}%
Clearly, we have $L\left( z^{\prime }\right) =L\left( z\right) \pm 1$ which
implies that the terms $e_{z_{\left( k\right) }}$ and $e_{z_{\left( k\right)
}^{\prime }}$ in (\ref{dek}) cancel out.

Denote by $\Pi _{x,y}^{k}$ the set of paths $z\in \Pi _{x.y}$ such that $%
z_{k}$ stands in position $\left( V\right) $ and by $\Pi _{x,y}^{\ \ k}$ the
set of paths $z\in \Pi _{x,y}$ such that $z_{k}$ stands in position $\left(
H\right) $. By the above observation, we can restrict the summation in (\ref%
{dek}) to those pairs $k,z$ where $z_{k}$ is either in vertical or
horizontal position, that is, 
\begin{equation}
\partial \left( e_{x}\times e_{y}\right) =\sum_{k=0}^{r}\sum_{z\in \Pi
_{x,y}^{k}\sqcup \Pi _{x,y}^{\ \ k}}\left( -1\right) ^{L\left( z\right)
+k}e_{z_{\left( k\right) }}.  \label{dekk}
\end{equation}%
Let us now compute the first term in the right hand side of (\ref{dusqv}):%
\begin{equation}
\left( \partial e_{x}\right) \times e_{y}=\sum_{l=0}^{p}\left( -1\right)
^{l}e_{x}\times e_{y}=\sum_{l=0}^{p}\sum_{w\in \Pi _{x_{\left( l\right)
}},y}\left( -1\right) ^{L\left( w\right) +l}e_{w}.  \label{dew}
\end{equation}%
Fix some $l=0,...,p$ and $w\in \Pi _{x_{\left( l\right) },y}$. Since the
projection of $w$ on $X$ is $x_{\left( l\right)
}=x_{0}...x_{l-1}x_{l+1}...x_{p}$, there exists a unique index $k$ such that 
$w_{k-1}$ projects onto $x_{l-1}$ and $w_{k}$ projects onto $x_{l+1}$. Then $%
w_{k-1}$ and $w_{k}$ have a common projection onto $Y$, say $y_{m}$. \FRAME{%
ftbphFU}{4.5044in}{2.3976in}{0pt}{\Qcb{Step-like paths $w$ and $z$. The
shaded area represents the difference $L\left( z\right) -L\left( w\right) $.}%
}{\Qlb{pic40}}{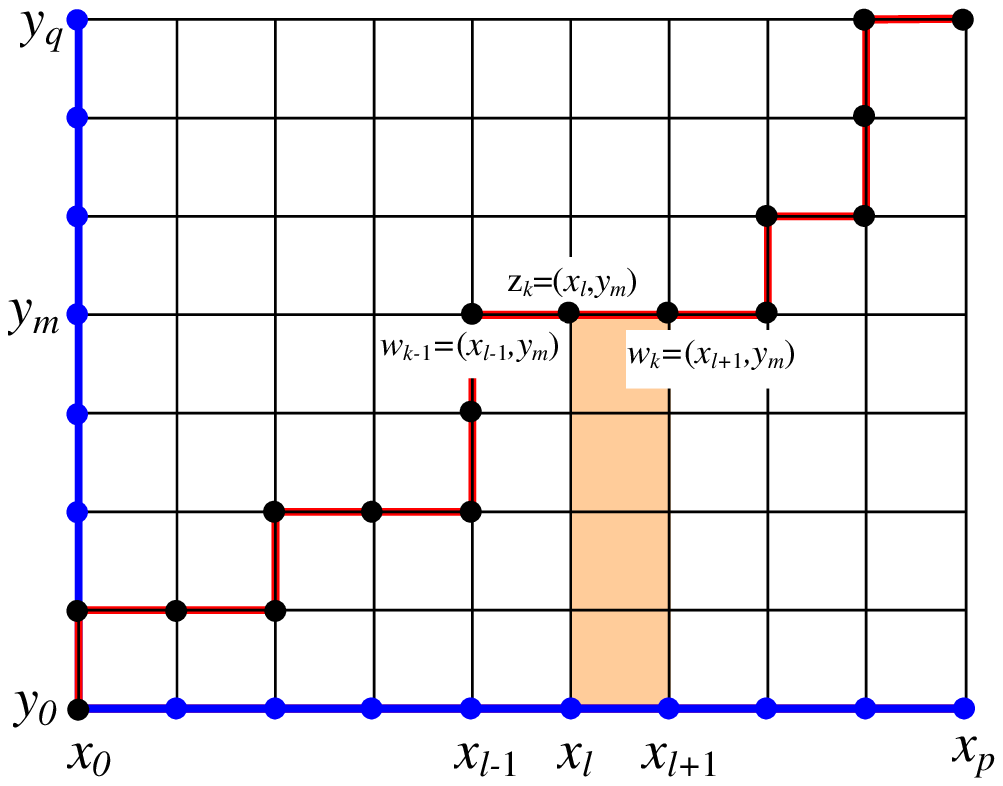}{\special{language "Scientific Word";type
"GRAPHIC";maintain-aspect-ratio TRUE;display "USEDEF";valid_file "F";width
4.5044in;height 2.3976in;depth 0pt;original-width 6.0026in;original-height
3.1767in;cropleft "0";croptop "1";cropright "1";cropbottom "0";filename
'pic40.eps';file-properties "XNPEU";}}

Define a path $z\in \Pi _{x,y}^{\ \ k}$ by setting%
\begin{equation}
\left. 
\begin{array}{ll}
z_{i}=w_{i} & \text{for }i\leq k-1, \\ 
z_{k}=\left( x_{l},y_{m}\right) & \text{for }i=k\text{,} \\ 
z_{i}=w_{i-1} & \text{for }i\geq k+1%
\end{array}%
\right.  \label{zw}
\end{equation}%
By construction we have $z_{\left( k\right) }=w.$ It also follows from the
construction that 
\begin{equation*}
L\left( z\right) =L\left( w\right) +m.
\end{equation*}%
Since $k=l+m$, we obtain that%
\begin{equation*}
L\left( z\right) +k=L\left( w\right) +l+2m.
\end{equation*}%
We see that each pair $l,w$ where $l=0,...,p$ and $w\in \Pi _{x_{\left(
l\right) },y}$ gives rise to a pair $k,z$ where $k=0,...,r$ , $z\in \Pi
_{x,y}^{\ \ k}$, and 
\begin{equation*}
\left( -1\right) ^{L\left( z\right) +k}e_{z_{\left( k\right) }}=\left(
-1\right) ^{L\left( w\right) +l}e_{w}.
\end{equation*}%
By reversing this argument, we obtain that each such pair $k,z$ gives back $%
l,w$ so that this correspondence between $k,z$ and $l,w$ is bijective.
Hence, we conclude that%
\begin{equation}
\left( \partial e_{x}\right) \times e_{y}=\sum_{l=0}^{p}\sum_{w\in \Pi
_{x_{\left( l\right) }},y}\left( -1\right) ^{L\left( w\right)
+l}e_{w}=\sum_{k=0}^{r}\sum_{z\in \Pi _{x,y}^{\ \ k}}\left( -1\right)
^{L\left( z\right) +k}e_{z_{\left( k\right) }}.  \label{dee}
\end{equation}%
The second term in the right hand side of (\ref{dusqv}) is computed
similarly:%
\begin{equation*}
\left( -1\right) ^{p}e_{x}\times \partial e_{y}=\sum_{m=0}^{q}\left(
-1\right) ^{m+p}e_{x}\times e_{y_{\left( m\right)
}}=\sum_{m=0}^{q}\sum_{w\in \Pi _{x,y_{\left( m\right) }}}\left( -1\right)
^{L\left( w\right) +m+p}e_{w}.
\end{equation*}%
Each pair $m,w$ here gives rise to a pair $k,z$ where $k=0,...,r$ and $z\in
\Pi _{x,y}^{k}$ in the following way: choose $k$ such that $w_{k-1}$
projects onto $y_{m-1}$ and $w_{k}$ projects onto $y_{m+1}$. Then $w_{k-1}$
and $w_{k}$ have a common projection onto $X$, say $x_{l}$. \FRAME{ftbphFU}{%
4.5044in}{2.3976in}{0pt}{\Qcb{Paths $w$ and $z$. The shaded area represents $%
L\left( z\right) -L\left( w\right) $.}}{\Qlb{pic41}}{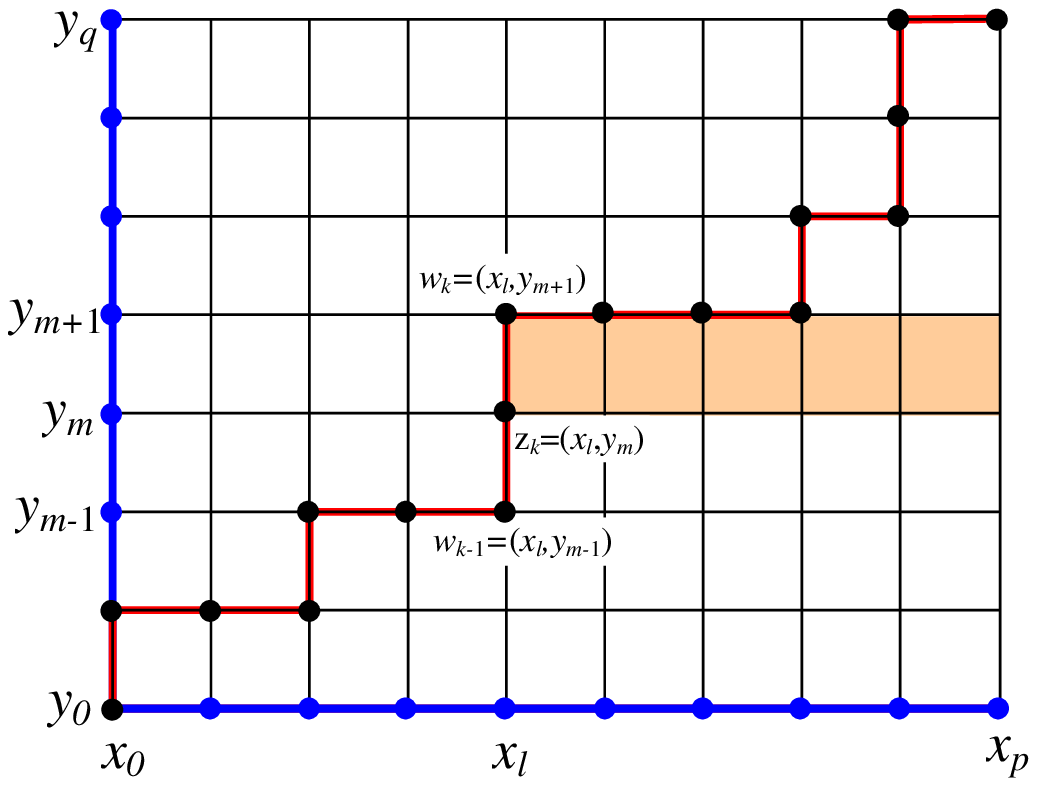}{\special%
{language "Scientific Word";type "GRAPHIC";maintain-aspect-ratio
TRUE;display "USEDEF";valid_file "F";width 4.5044in;height 2.3976in;depth
0pt;original-width 6.0026in;original-height 3.1767in;cropleft "0";croptop
"1";cropright "1";cropbottom "0";filename 'pic41.eps';file-properties
"XNPEU";}}

Define the path $z\in \Pi _{x,y}^{k}$ as in (\ref{zw}) (cf. Fig. \ref{pic41}%
). Then we have $w=z_{\left( k\right) }$ and 
\begin{equation*}
L\left( z\right) =L\left( w\right) +p-l.
\end{equation*}%
Since $k=l+m$, we obtain%
\begin{equation*}
L\left( z\right) +k=L\left( w\right) +p+m
\end{equation*}%
and%
\begin{equation*}
\left( -1\right) ^{p}e_{x}\times \partial e_{y}=\sum_{m=0}^{q}\sum_{w\in \Pi
_{x,y_{\left( m\right) }}}\left( -1\right) ^{L\left( w\right)
+m+p}e_{w}=\sum_{k=0}^{r}\sum_{z\in \Pi _{x,y}^{k}}\left( -1\right)
^{L\left( z\right) +k}e_{z_{\left( k\right) }}.
\end{equation*}%
Combining this with (\ref{dekk}) and (\ref{dee}), we obtain (\ref{dusqv}).
\end{proof}

\subsection{$\partial $-invariant paths on Cartesian product}

\label{SecXxY}

\begin{definition}
\RM Given two finite sets $X$ and $Y$ with path complexes $P\left( X\right) $
and $P\left( Y\right) $ respectively, define on the set $Z=X\times Y$ a path
complex $P\left( Z\right) $ as follows: the elements of $P\left( Z\right) $
are step-like paths on $Z$ whose projections on $X$ and $Y$ belong to $%
P\left( X\right) $ and $P\left( Y\right) $, respectively. The path complex $%
P\left( Z\right) $ is called the Cartesian product of the path complexes $%
P\left( X\right) $ and $P\left( Y\right) $ and is denoted by $P\left(
X\right) \boxplus P\left( Y\right) .$
\end{definition}

\label{rem: the Cartesian product is associative}In particular, if $x$ and $%
y $ are elementary allowed paths on $X$ and $Y$, respectively, then all the
paths $z\in \Pi _{x,y}$ are allowed on $Z$. It clearly follows from (\ref%
{uvz1}) that 
\begin{equation*}
u\in \mathcal{A}_{p}\left( X\right) \ \text{and\ }v\in \mathcal{A}_{q}\left(
Y\right) \ \ \Rightarrow \ \ u\times v\in \mathcal{A}_{p+q}\left( Z\right) .
\end{equation*}%
Furthermore, the following is true.

\begin{proposition}
\label{Puxv}If $u\in \Omega _{p}\left( X\right) $ and $v\in \Omega
_{q}\left( Y\right) $ then $u\times v\in \Omega _{p+q}\left( Z\right) .$
\end{proposition}

\begin{proof}
Indeed, $\partial u$ and $\partial v$ are allowed, whence also $\partial
u\times v$ and $u\times \partial v$ are allowed, whence $\partial \left(
u\times v\right) $ is allowed by the product rule (\ref{dusqv}). It follows
that $u\times v\in \Omega _{p+q}\left( Z\right) .$
\end{proof}

It follows easily from the product rule that the cross product of closed
paths is closed, and of exact and closed paths -- exact.

The next theorem gives a complete description of $\partial $-invariant paths
on $Z$.

\begin{theorem}
\label{TK}Let $P\left( X\right) $ and $P\left( Y\right) $ be two regular
path complexes. Then for their Cartesian product $P\left( Z\right) =P\left(
X\right) \boxplus P\left( Y\right) $ the following isomorphism of chain
complexes holds:%
\begin{equation}
\Omega _{\bullet }\left( Z\right) \cong \Omega _{\bullet }\left( X\right)
\otimes \Omega _{\bullet }\left( Y\right) ,  \label{Omrpq}
\end{equation}%
where the mapping $\Omega _{\bullet }\left( X\right) \otimes \Omega
_{\bullet }\left( Y\right) \rightarrow \Omega _{\bullet }\left( Z\right) $
is given by $u\otimes v\mapsto u\times v$.

Consequently we have%
\begin{equation}
H_{\bullet }\left( Z\right) \cong H_{\bullet }\left( X\right) \otimes
H_{\bullet }\left( Y\right)  \label{HXYZ}
\end{equation}%
(the K\"{u}nneth formula for Cartesian product).
\end{theorem}

The relation (\ref{HXYZ}) means that, for any $r\geq 0$, 
\begin{equation}
H_{r}\left( Z\right) \cong \tbigoplus_{\left\{ p,q\geq 0:p+q=r\right\}
}\left( H_{p}\left( X\right) \otimes H_{q}\left( Y\right) \right) .
\label{Hrpq}
\end{equation}%
For example, if the path complex $P\left( X\right) $ is connected and all
homologies $H_{p}\left( X\right) $, $p\geq 1$, are trivial then $H_{r}\left(
Z\right) \cong H_{r}\left( Y\right) .$

The proof of Theorem \ref{TK} will be given in Section \ref{SecK} after a
necessary preparation. Here we consider some examples of Cartesian products.

\RM Let $\left( X,E_{X}\right) $ and $\left( Y,E_{Y}\right) $ be two
digraphs. Their Cartesian product is the digraph $\left( Z,E_{Z}\right) $
where $Z=X\times Y$ and the set $E_{Z}$ of edges is defined as follows: $%
\left( x,y\right) \rightarrow \left( x^{\prime },y^{\prime }\right) $ if and
only if either $x\rightarrow x^{\prime }$ and $y=y^{\prime }$ (a \emph{%
horizontal} edge) and $y\rightarrow y^{\prime }$ and $x=x^{\prime }$ (a 
\emph{vertical} edge):%
\begin{equation*}
\begin{array}{cccccc}
y^{\prime }\bullet & \dots & \overset{\left( x,y^{\prime }\right) }{\bullet }
& \longrightarrow & \overset{\left( x^{\prime },y^{\prime }\right) }{\bullet 
} & \dots \\ 
\ \ \ \ \uparrow \  &  & \uparrow &  & \uparrow &  \\ 
y\bullet & \dots & \overset{\left( x,y\right) }{\bullet } & \longrightarrow
& \overset{\left( x^{\prime },y\right) }{\bullet } & \dots \\ 
\ \ \ \ \ \ \ \ \ \  &  &  &  &  &  \\ 
^{Y}\ \diagup \ _{X} & \dots & \underset{x}{\bullet } & \longrightarrow & 
\underset{x^{\prime }}{\bullet }\  & \dots%
\end{array}%
\ 
\end{equation*}%
Clearly, any allowed path on $\left( Z,E_{Z}\right) $ is step-like, and its
projections onto $X$ and $Y$ are also allowed. Hence, the path complex of
the digraph $\left( Z,E_{Z}\right) $ is the Cartesian product of the path
complexes of the digraphs $\left( X,E_{X}\right) $ and $\left(
Y,E_{Y}\right) $.

Let us give some explicit examples of product digraphs and $\partial $%
-invariant paths there. \label{rem: what is the purpose of all these
examples?}

\begin{example}
\RM Consider the Cartesian product \ $Z=X\boxplus Y$ of the digraphs 
\begin{equation*}
X=\ ^{a}\bullet \leftrightarrows \bullet ^{b}\ \ \text{and}\ \ Y=\
^{0}\bullet \leftrightarrows \bullet ^{1}
\end{equation*}%
that is shown on Fig. \ref{pic27}. \FRAME{ftbhFU}{6.3304in}{1.3015in}{0pt}{%
\Qcb{Cartesian product of $\ ^{a}\bullet \leftrightarrows \bullet ^{b}$ and $%
^{0}\bullet \leftrightarrows \bullet ^{1}$ }}{\Qlb{pic27}}{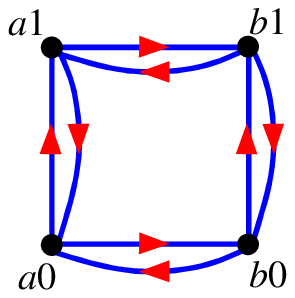}{%
\special{language "Scientific Word";type "GRAPHIC";maintain-aspect-ratio
TRUE;display "USEDEF";valid_file "F";width 6.3304in;height 1.3015in;depth
0pt;original-width 6.0026in;original-height 1.9553in;cropleft "0";croptop
"1";cropright "1";cropbottom "0";filename 'pic27.eps';file-properties
"XNPEU";}}The paths $e_{aba}$ and $e_{010}$ are $\partial $-invariant, so
that their cross product%
\begin{equation*}
e_{a0b0a0a1a0}-e_{a0b0b1a1a0}+e_{a0b0b1b0a0}+e_{a0a1b1a1a0}-e_{a0a1b1b0a0}+e_{a0a1a0b0a0}
\end{equation*}%
is $\partial $-invariant on $Z$. The paths $e_{ab}+e_{ba}$ and $%
e_{01}+e_{10} $ are exact, so that their cross product 
\begin{equation*}
e_{a0b0b1}-e_{a0a1b1}+e_{a1b1b0}-e_{a1a0b0}+e_{b0a0a1}-e_{b0b1a1}+e_{b1a1a0}-e_{b1b0a0}
\end{equation*}%
is an exact path on $Z$.
\end{example}

\begin{example}
\RM Let $Z=X\boxplus Y$ where%
\begin{equation*}
X=%
\begin{array}{c}
_{_{\nearrow }}\overset{b}{\bullet }_{_{\searrow }} \\ 
^{a}\bullet \ \rightarrow \ \bullet ^{c}%
\end{array}%
\ \ \text{and \ \ }Y=\ ^{0}\bullet \leftrightarrows \bullet ^{1}
\end{equation*}%
(see Fig. \ref{pic25}).\FRAME{ftbhFU}{6.3304in}{1.6769in}{0pt}{\Qcb{%
Cartesian product of a triangle and and $^{0}\bullet \leftrightarrows
\bullet ^{1}$}}{\Qlb{pic25}}{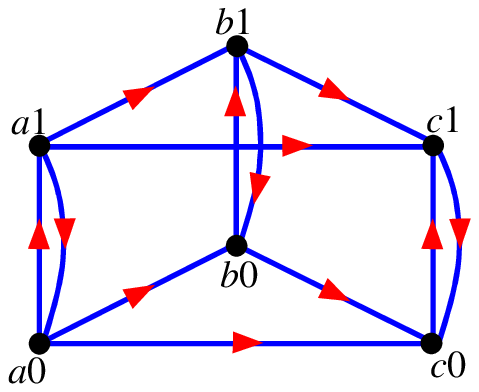}{\special{language "Scientific
Word";type "GRAPHIC";maintain-aspect-ratio TRUE;display "USEDEF";valid_file
"F";width 6.3304in;height 1.6769in;depth 0pt;original-width
6.0026in;original-height 2.7114in;cropleft "0";croptop "1";cropright
"1";cropbottom "0";filename 'pic25.eps';file-properties "XNPEU";}}The paths $%
e_{abc}$ is $\partial $-invariant on $X$ and $e_{01}+e_{10}$ is $\partial $%
-invariant on $Y$. Hence, their cross product%
\begin{equation*}
e_{a0b0c0c1}-e_{a0b0b1c1}+e_{a0a1b1c1}+e_{a1b1c1c0}-e_{a1b1b0c0}+e_{a1a0b0c0}
\end{equation*}%
is $\partial $-invariant on $Z$.
\end{example}

\begin{example}
\RM Let $Z=X\boxplus Y$ where%
\begin{equation*}
X=%
\begin{array}{c}
_{_{\nearrow }}\overset{b}{\bullet }_{_{\searrow }} \\ 
^{a}\bullet \ \rightarrow \ \bullet ^{c}%
\end{array}%
\ \ \text{and \ \ }Y=%
\begin{array}{ccc}
_{2}\bullet & \longrightarrow & \bullet _{3} \\ 
\ \uparrow &  & \uparrow \  \\ 
_{0}\bullet & \longrightarrow & \bullet _{1}%
\end{array}%
.
\end{equation*}%
(see Fig. \ref{pic42}).\FRAME{ftbphFU}{4.4391in}{2.2649in}{0pt}{\Qcb{%
Cartesian product of a triangle and a square}}{\Qlb{pic42}}{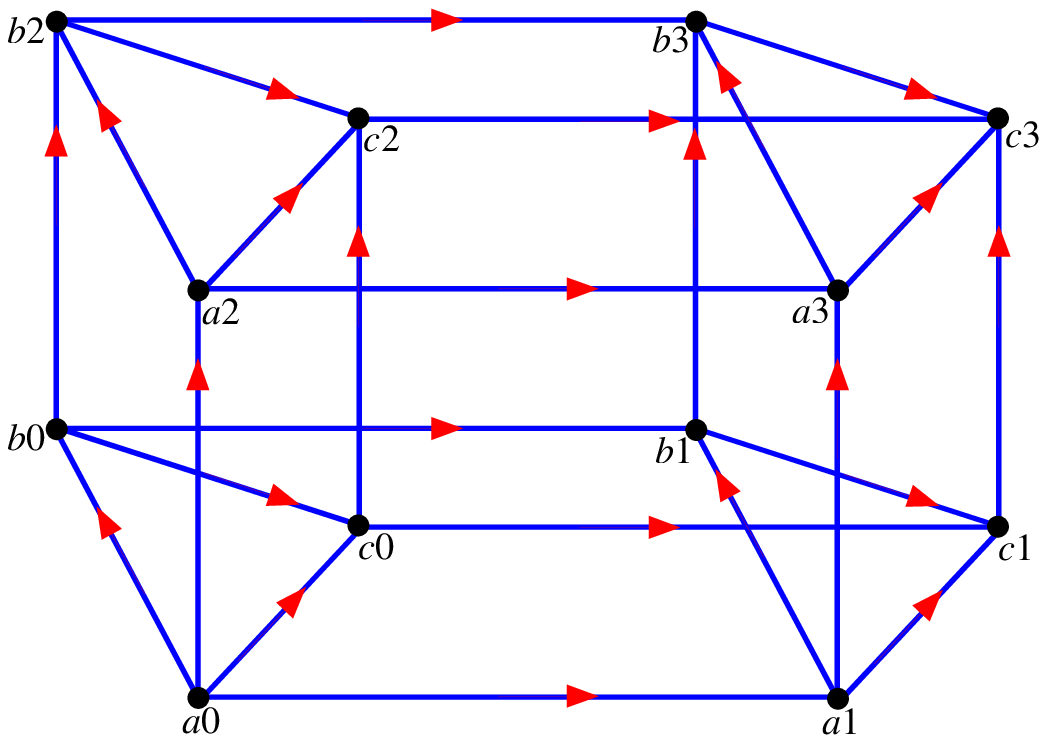}{%
\special{language "Scientific Word";type "GRAPHIC";maintain-aspect-ratio
TRUE;display "USEDEF";valid_file "F";width 4.4391in;height 2.2649in;depth
0pt;original-width 6.4013in;original-height 3.1955in;cropleft "0";croptop
"1";cropright "1";cropbottom "0";filename 'pic42.eps';file-properties
"XNPEU";}}Taking the cross product of $\partial $-invariant paths $e_{abc}$
and $e_{013}-e_{023}$, we obtain the following $\partial $-invariant path on 
$Z$:%
\begin{eqnarray*}
&&e_{a0b0c0c1c3}-e_{a0b0b1c1c3}+e_{a0b0b1b3c3} \\
&&+e_{a0a1b1c1c3}-e_{a0a1b1b3c3}+e_{a0a1a3b3c3} \\
&&-e_{a0b0c0c2c3}+e_{a0b0b2c2c3}-e_{a0b0b2b3c3} \\
&&-e_{a0a2b2c2c3}+e_{a0a2b2b3c3}-e_{a0a2a3b3c3}
\end{eqnarray*}
\end{example}

\subsection{Cylinders and hypercubes}

\label{SecCyl}For any digraph $X$, the cylinder over $X$ is the digraph%
\begin{equation*}
\limfunc{Cyl}X:=X\boxplus \ ^{0}\bullet \rightarrow \bullet ^{1}.
\end{equation*}%
Assuming that the vertices of $X$ are enumerated by $0,1,...,n-1$, we can
enumerate the vertices of $\limfunc{Cyl}X$ by $0,1,...,2n-1$ using the
following rule: $\left( x,0\right) $ is assigned the number $x$, while $%
\left( x,1\right) $ is assigned $x+n$.

Define the operation of \emph{lifting }paths from $X$ to $\limfunc{Cyl}X$ as
follows: for any regular path $v$ on $X$, the lifted path is denoted by $%
\widehat{v}$ and is defined by 
\begin{equation*}
\widehat{v}=v\times e_{01}.
\end{equation*}%
Since $e_{01}$ is $\partial $-invariant on $Y$, we obtain that if $v\in 
\mathcal{A}_{p}\left( X\right) $ then $\widehat{v}\in \mathcal{A}%
_{p+1}\left( \limfunc{Cyl}X\right) $, and if $v\in \Omega _{p}\left(
X\right) $ then $\widehat{v}\in \Omega _{p+1}\left( \limfunc{Cyl}X\right) $.

For example, if $v=e_{i_{0}...i_{p}}$ then 
\begin{equation}
\widehat{v}=e_{i_{0}...i_{p}}\times e_{01}=\sum_{k=0}^{p}\left( -1\right)
^{p-k}e_{i_{0}...i_{k}\left( i_{k}+n\right) ...\left( i_{p}+n\right) },
\label{elift}
\end{equation}%
since the path $i_{0}...i_{k}\left( i_{k}+n\right) ...\left( i_{p}+n\right) $
has the elevation $p-k$ as can be seen on the diagram: 
\begin{equation*}
\begin{array}{ccccccccc}
&  & \cdots & \overset{i_{k}+n}{\bullet } & \longrightarrow & \overset{%
i_{k+1}+n}{\bullet } & \longrightarrow & \cdots & \longrightarrow \overset{%
i_{p}+n}{\bullet } \\ 
&  &  & \uparrow &  & \uparrow &  &  &  \\ 
\overset{i_{0}}{\bullet }\longrightarrow & \cdots & \longrightarrow & 
\overset{i_{k}}{\bullet } & \longrightarrow & \overset{i_{k+1}}{\bullet } & 
\cdots &  & 
\end{array}%
,\ 
\end{equation*}

\begin{example}
\RM The cylinder over a triangle $X=%
\begin{array}{c}
_{_{\nearrow }}\overset{1}{\bullet }_{_{\searrow }} \\ 
^{0}\bullet \ \rightarrow \ \bullet ^{2}%
\end{array}%
$ is shown on Fig. \ref{pic8}.\FRAME{ftbphFU}{5.0704in}{1.7608in}{0pt}{\Qcb{%
A cylinder over a triangle}}{\Qlb{pic8}}{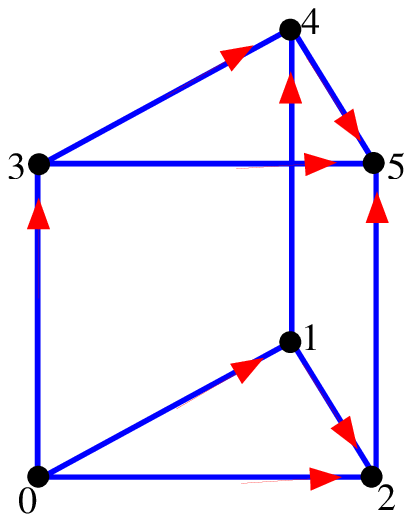}{\special{language
"Scientific Word";type "GRAPHIC";maintain-aspect-ratio TRUE;display
"USEDEF";valid_file "F";width 5.0704in;height 1.7608in;depth
0pt;original-width 6.3027in;original-height 2.2355in;cropleft "0";croptop
"1";cropright "1";cropbottom "0";filename 'pic8.eps';file-properties
"XNPEU";}}

Since $2$-path $e_{012}$ is $\partial $-invariant on $X$, lifting it to the
cylinder, we obtain the following $\partial $-invariant $3$-path on $%
\limfunc{Cyl}X$: 
\begin{equation*}
e_{0345}-e_{0145}+e_{0125}.
\end{equation*}
\end{example}

\begin{example}
\RM The cylinder over the graph $X=\ ^{0}\bullet \rightarrow \bullet ^{1}$
is a square: 
\begin{equation*}
\begin{array}{ccc}
^{2}\bullet & \longrightarrow & \bullet ^{3} \\ 
\ \uparrow &  & \uparrow \  \\ 
^{0}\bullet & \longrightarrow & \bullet ^{1}%
\end{array}%
\ 
\end{equation*}%
Lifting a $\partial $-invariant $1$-path $e_{01}$ on $X$ we obtain the
following $\partial $-invariant $2$-path on the square: $e_{013}-e_{023}.$

The cylinder over a square is a $3$-cube that is shown in Fig. \ref{pic9}.%
\FRAME{ftbhFU}{5.0704in}{1.9839in}{0pt}{\Qcb{A $3$-cube}}{\Qlb{pic9}}{%
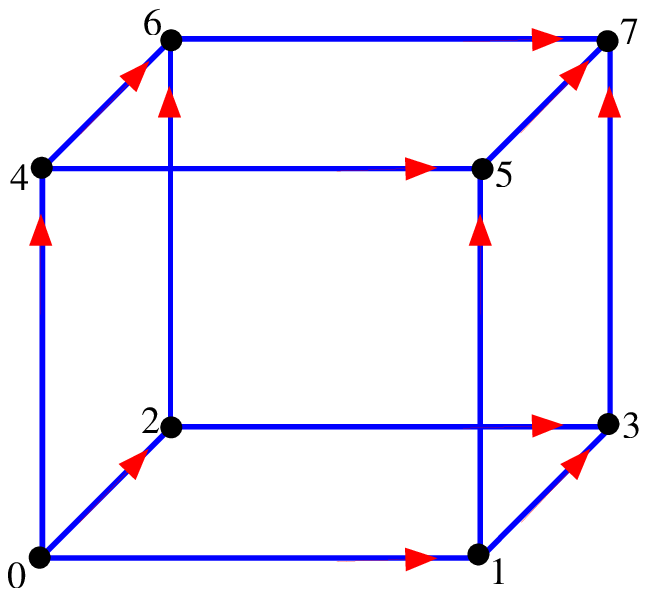}{\special{language "Scientific Word";type
"GRAPHIC";maintain-aspect-ratio TRUE;display "USEDEF";valid_file "F";width
5.0704in;height 1.9839in;depth 0pt;original-width 6.3027in;original-height
2.444in;cropleft "0";croptop "1";cropright "1";cropbottom "0";filename
'pic9.eps';file-properties "XNPEU";}}

Lifting the $2$-path $e_{013}-e_{023}$ we obtain the following $\partial $%
-invariant $3$-path on the $3$-cube:%
\begin{equation*}
e_{0457}-e_{0157}+e_{0137}-e_{0467}+e_{0267}-e_{0237}.
\end{equation*}%
Defining further $n$-cube for any positive integer $n$ by 
\begin{equation*}
\limfunc{Cube}\nolimits_{n}=\limfunc{Cyl}\limfunc{Cube}\nolimits_{n-1},
\end{equation*}%
we see that $\limfunc{Cube}\nolimits_{n}$ determines a $\partial $-invariant 
$n$-path that is a lifting of a $\partial $-invariant $\left( n-1\right) $%
-path from $\limfunc{Cube}\nolimits_{n-1}$ and that is an alternating sum of 
$n!$ elementary terms. It is easy to show that these terms correspond to
partitioning of a solid $n$-cube into $n!$ simplexes.

By (\ref{Hrpq}) all homology groups of $\limfunc{Cube}\nolimits_{n}$ are
trivial except for $H_{0}$.
\end{example}

\label{rem: partitioning cube into simplexes}

\subsection{Representation of $\partial $-invariant paths on product}

\label{Seccxy}As in Section \ref{SecXxY}, we work with two paths complexes $%
P\left( X\right) $, $P\left( Y\right) $ and their Cartesian product $P\left(
Z\right) =P\left( X\right) \boxplus P\left( Y\right) $ where $Z=X\times Y$.
The following statement is a partial converse of Proposition \ref{Puxv}.

\begin{proposition}
\label{Pexey}Any path $w\in \Omega _{\bullet }\left( Z\right) $ admits a
representation 
\begin{equation}
w=\sum_{x\in P\left( X\right) ,\ y\in P\left( Y\right) }c^{xy}\left(
e_{x}\times e_{y}\right)  \label{axy}
\end{equation}%
with some scalar coefficients $c^{xy}$ (only finitely many coefficients are
non-vanishing). Furthermore, the coefficients $c^{xy}$ are uniquely
determined by $w$.
\end{proposition}

\begin{proof}
Let us first show the uniqueness of $c^{xy}$, which is equivalent to the
linear independence of the family $\left\{ e_{x}\times e_{y}\right\} $
across all $x\in P\left( X\right) $ and $y\in P\left( Y\right) $. Indeed,
assume that, for some scalars $c^{xy}$,%
\begin{equation*}
\sum_{x\in P\left( X\right) ,y\in P\left( Y\right) }c^{xy}e_{x}\times
e_{y}=0,
\end{equation*}%
and prove that $c^{xy}=0$ for any couple $x,y$ as in the summation. Fix such
a couple $x,y$ and choose one $z\in \Pi _{x,y}$. Then by (\ref{uvz}) 
\begin{equation*}
\left( e_{x^{\prime }}\times e_{y^{\prime }}\right) ^{z}=\left\{ 
\begin{array}{ll}
\left( -1\right) ^{L\left( z\right) }, & x^{\prime }=x\text{ and }y^{\prime
}=y\text{,} \\ 
0, & \text{otherwise,}%
\end{array}%
\right.
\end{equation*}%
which implies that 
\begin{equation*}
\left( \sum_{x^{\prime }\in P\left( X\right) ,y^{\prime }\in P\left(
Y\right) }c^{x^{\prime }y^{\prime }}e_{x^{\prime }}\times e_{y^{\prime
}}\right) ^{z}=\left( -1\right) ^{L\left( z\right) }c^{xy}
\end{equation*}%
and, hence, $c^{xy}=0.$

Let us show existence of the representation (\ref{axy}) for any $w\in \Omega
_{r}\left( Z\right) $ and any $r\geq 0$. As before, for any elementary $r$%
-path $z$ on $Z$, $w^{z}$ denotes the $e_{z}$-coordinate of $w$. If $z$ is
an elementary $r^{\prime }$-path with $r^{\prime }\neq r$ then set $w^{z}=0$%
. For any $x\in P\left( X\right) $ and $y\in P\left( Y\right) $ choose some $%
z\in \Pi _{x,y}$ and set%
\begin{equation}
c^{xy}=\left( -1\right) ^{L\left( z\right) }w^{z}.  \label{cxyw}
\end{equation}%
Let us first show that the value of $c^{xy}$ in (\ref{cxyw}) is independent
of the choice of $z\in \Pi _{x,y}$. Set $z=i_{0}...i_{r}$. Let $k$ be an
index such that one of the couples $i_{k-1}i_{k}$, $i_{k}i_{k+1}$ is
vertical and the other is horizontal. If $i_{k-1}=\left( a,b\right) $ and $%
i_{k+1}=\left( a^{\prime },b^{\prime }\right) $ where $a,a^{\prime }\in X$
and $b,b^{\prime }\in Y$, then $i_{k}$ is either $\left( a^{\prime
},b\right) $ or $\left( a,b^{\prime }\right) $. Denote the other of these
two vertices by $i_{k}^{\prime }$, as, for example, on the diagram: 
\begin{equation*}
\begin{array}{cccccc}
\ \ \ \ \vdots &  &  &  & \vdots &  \\ 
\ ^{b^{\prime }}\bullet &  & \overset{i_{k}^{\prime }}{\bullet } & 
\longrightarrow & \overset{i_{k+1}}{\bullet } & \dots \\ 
\ \ \ \ \uparrow \  &  & \uparrow &  & \uparrow &  \\ 
\ \ ^{b}\bullet & \ \ \dots & \overset{i_{k-1}}{\bullet } & \longrightarrow
& \overset{i_{k}}{\bullet } &  \\ 
\ \ \ \ \ \ \ \vdots \ \ \ \  &  & \vdots &  &  &  \\ 
^{\ \ \overset{||}{y}} &  &  &  &  &  \\ 
\ \ \  & _{x=}\dots & \underset{a}{\bullet } & \longrightarrow & \underset{%
a^{\prime }}{\bullet }\  & \dots%
\end{array}%
\ 
\end{equation*}%
Replacing in the path $z=i_{0}...i_{r}$ the vertex $i_{k}$ by $i_{k}^{\prime
}$, we obtain the path $z^{\prime }=i_{0}...i_{k-1}i_{k}^{\prime
}i_{k+1}...i_{r}$ that clearly belongs to $\Pi _{x,y}$ and, hence, is
allowed. Since the $\left( r-1\right) $-path $i_{0}...i_{k-1}i_{k+1}...i_{r}$
is regular but non-allowed (as it is not step-like), while $\partial w$ is
allowed, we have%
\begin{equation}
\left( \partial w\right) ^{i_{0}...i_{k-1}i_{k+1}...i_{r}}=0.  \label{dw0}
\end{equation}%
On the other hand, we have by (\ref{dv})%
\begin{eqnarray}
\left( \partial w\right) ^{i_{0}...i_{k-1}i_{k+1}...i_{r}} &=&\sum_{j\in
Z}\left( \sum_{m=0}^{k-1}\left( -1\right)
^{m}w^{i_{0}...i_{m-1}ji_{m}...i_{k-1}i_{k+1}...i_{r}}\right.  \label{w1} \\
&&+\left( -1\right) ^{k}w^{i_{0}...i_{k-1}ji_{k+1}...i_{r}}  \label{w2} \\
&&+\left. \sum_{m=k+2}^{r+1}\left( -1\right)
^{m-1}w^{i_{0}...i_{k-1}i_{k+1}...i_{m-1}ji_{m}...i_{r}}\right) .  \label{w3}
\end{eqnarray}%
All the components of $w$ in the sums (\ref{w1}) and (\ref{w3}) vanish since
they correspond to non-allowed paths, while $w$ is allowed. The path $%
i_{0}...i_{k-1}ji_{k+1}...i_{r}$ in the term (\ref{w2}) is also non-allowed
the unless $j=i_{k}$ or $j=i_{k}^{\prime }$ (note that $i_{k}$ and $%
i_{k}^{\prime }$ are uniquely determined by $i_{k-1}$ and $i_{k+1}$). Hence,
the only non-zero terms in (\ref{w1})-(\ref{w3}) are $%
w^{i_{0}...i_{k-1}i_{k}i_{k+1}...i_{r}}=w^{z}$ and $%
w^{i_{0}...i_{k-1}i_{k}^{\prime }i_{k+1}...i_{r}}=w^{z^{\prime }}$.
Combining (\ref{dw0}) and (\ref{w1})-(\ref{w3}), we obtain%
\begin{equation*}
0=w^{z}+w^{z^{\prime }}.
\end{equation*}%
Since $L\left( z^{\prime }\right) =L\left( z\right) \pm 1$, it follows that 
\begin{equation}
\left( -1\right) ^{L\left( z^{\prime }\right) }w^{z^{\prime }}=\left(
-1\right) ^{L\left( z\right) }w^{z}.  \label{z'z}
\end{equation}

The transformation $z\mapsto z^{\prime }$ described above, allows us to
obtain from a given $z\in \Pi _{x,y}$ in a finite number of steps any other
path in $\Pi _{x,y}$. Since the quantity $\left( -1\right) ^{L\left(
z\right) }w^{z}$ does not change under this transformation, it follows that
it does not depend on a particular choice of $z\in \Pi _{x,y}$, which was
claimed. Hence, the coefficients $c^{xy}$ are well-defined by (\ref{cxyw}).

Finally, let us show that the identity (\ref{axy}) holds with the
coefficients $c^{xy}$ from (\ref{cxyw}). By (\ref{uvz1}) we have 
\begin{equation*}
e_{x}\times e_{y}=\sum_{z\in \Pi _{x,y}}\left( -1\right) ^{L\left( z\right)
}e_{z}.
\end{equation*}%
Using (\ref{cxyw}) we obtain 
\begin{eqnarray*}
\sum_{x\in P\left( X\right) ,\ y\in P\left( Y\right) }c^{xy}\left(
e_{x}\times e_{y}\right) &=&\sum_{x\in P\left( X\right) ,\ y\in P\left(
Y\right) }c^{xy}\sum_{z\in \Pi _{x,y}}\left( -1\right) ^{L\left( z\right)
}e_{z} \\
&=&\sum_{x\in P\left( X\right) ,\ y\in P\left( Y\right) }\sum_{z\in \Pi
_{x,y}}w^{z}e_{z} \\
&=&\sum_{z\in P\left( Z\right) }w^{z}e_{z}=w,
\end{eqnarray*}%
which finishes the proof.
\end{proof}

\begin{corollary}
\label{Corwuv}Any path $w\in \Omega _{\bullet }\left( Z\right) $ admits
representations 
\begin{equation}
w=\sum_{x\in P\left( X\right) }e_{x}\times u^{x}=\sum_{\ y\in P\left(
Y\right) }v^{y}\times e_{y}  \label{wuv}
\end{equation}%
where $u^{x}\in \Omega _{\bullet }\left( X\right) $ and $v^{y}\in \Omega
_{\bullet }\left( Y\right) $ are uniquely determined.
\end{corollary}

\begin{proof}
It follows from (\ref{axy}) that%
\begin{equation*}
w=\sum_{x\in P\left( X\right) }e_{x}\times u^{x}
\end{equation*}%
where%
\begin{equation*}
u^{x}=\sum_{y\in P\left( Y\right) }c^{xy}e_{y}\in \mathcal{A}_{\bullet
}\left( X\right) .
\end{equation*}%
It is obvious that $u^{x}$ are uniquely determined as so are the
coefficients $c^{xy}.$ Let us show that, in fact, $u^{x}\in \Omega _{\bullet
}\left( X\right) .$ Let us define the coefficients $\delta _{x^{\prime
}}^{x}\in \left\{ 0,1,-1\right\} $ by%
\begin{equation*}
\partial e_{x}=\sum_{x^{\prime }\in R\left( X\right) }\delta _{x}^{x^{\prime
}}e_{x^{\prime }}.
\end{equation*}%
We have by the product rule%
\begin{eqnarray*}
\partial w &=&\sum_{x\in P\left( X\right) }\partial e_{x}\times
u^{x}+\sum_{x\in P\left( X\right) }e_{x}\times \partial u^{x} \\
&=&\sum_{x\in P\left( X\right) }\sum_{x^{\prime }\in R\left( X\right)
}\delta _{x}^{x^{\prime }}e_{x^{\prime }}\times u^{x}+\sum_{x\in P\left(
X\right) }e_{x}\times \partial u^{x} \\
&=&\sum_{x\in R\left( X\right) }\sum_{x^{\prime }\in P\left( X\right)
}\delta _{x^{\prime }}^{x}e_{x}\times u^{x^{\prime }}+\sum_{x\in P\left(
X\right) }e_{x}\times \partial u^{x} \\
&=&\sum_{x\in P\left( X\right) }e_{x}\times \left( \sum_{x^{\prime }\in
P\left( X\right) }\delta _{x^{\prime }}^{x}u^{x^{\prime }}+\partial
u^{x}\right) \\
&&+\sum_{x\in R\left( X\right) \setminus P\left( X\right) }e_{x}\times
\left( \sum_{x^{\prime }\in P\left( X\right) }\delta _{x^{\prime
}}^{x}u^{x^{\prime }}\right) .
\end{eqnarray*}%
Since $\partial w$ is allowed on $Z$, it follows that \label{rem: more
details} 
\begin{equation*}
\sum_{x\in R\left( X\right) \setminus P\left( X\right) }e_{x}\times \left(
\sum_{x^{\prime }\in P\left( X\right) }\delta _{x^{\prime }}^{x}u^{x^{\prime
}}\right) =0.
\end{equation*}
\textbf{\ }On the other hand, since $\partial w\in \Omega _{\bullet }\left(
Z\right) $, we have a representation%
\begin{equation*}
\partial w=\sum_{x\in P\left( X\right) }e_{x}\times \widetilde{u}^{x}
\end{equation*}%
where $\widetilde{u}^{x}\in \mathcal{A}_{\bullet }\left( X\right) $.
Comparison with the previous computation of $\partial w$ yields%
\begin{equation*}
\widetilde{u}^{x}=\sum_{x^{\prime }\in P\left( X\right) }\delta _{x^{\prime
}}^{x}u^{x^{\prime }}+\partial u^{x}.
\end{equation*}%
Since $u^{x^{\prime }}\in \mathcal{A}_{\bullet }\left( X\right) $, it
follows that $\partial u^{x}\in \mathcal{A}_{\bullet }\left( X\right) $,
which proves that $u^{x}\in \Omega _{\bullet }\left( X\right) $. The second
identity in (\ref{wuv}) is proved similarly.
\end{proof}

Let us introduce in $\mathcal{A}_{p}\left( X\right) $ the $\mathbb{K}$%
-scalar product as follows: for all $u,v\in \mathcal{A}_{p}\left( X\right) $%
\begin{equation*}
\left[ u,v\right] :=\sum_{x\in P\left( X\right) }u^{x}v^{x}.
\end{equation*}%
If $\mathbb{K}=\mathbb{R}$ then this is a proper scalar product, but for a
general field $\mathbb{K}$ there is no positivity property (in fact, it can
happen that $\left[ u,u\right] =0$). Set also%
\begin{equation*}
\Omega _{p}^{\bot }\left( X\right) =\left\{ u\in \mathcal{A}_{p}\left(
X\right) :\left[ u,v\right] =0\text{ for all }v\in \Omega _{p}\left(
X\right) \right\} .
\end{equation*}%
If $\mathbb{K}=\mathbb{R}$ then $\Omega _{p}^{\bot }$ is a proper orthogonal
complement of $\Omega _{p}$ in $\mathcal{A}_{p}$ and $\mathcal{A}_{p}=\Omega
_{p}\oplus \Omega _{p}^{\bot }.$ For a general $\mathbb{K}$, this is not
true, as $\Omega _{p}$ and $\Omega _{p}^{\bot }$ may have a non-trivial
intersection, but for any field $\mathbb{K}$ it is still true that%
\begin{equation*}
\dim \Omega _{p}+\dim \Omega _{p}^{\bot }=\dim \mathcal{A}_{p}.
\end{equation*}

\begin{lemma}
\label{Lembot}If $u\in \Omega _{p}^{\bot }\left( X\right) $ and $v\in 
\mathcal{A}_{q}\left( Y\right) $ then $u\times v\in \Omega _{r}^{\bot
}\left( Z\right) $ where $r=p+q.$ Also, if $u\in \mathcal{A}_{p}\left(
X\right) $ and $v\in \Omega _{q}^{\bot }\left( Y\right) $ then $u\times v\in
\Omega _{r}^{\bot }\left( Z\right) .$
\end{lemma}

\begin{proof}
We need to prove that, for any $w\in \Omega _{r}\left( Z\right) $,%
\begin{equation}
\left[ u\times v,w\right] =0,  \label{uvw=0}
\end{equation}%
assuming that $u\in \Omega _{p}^{\bot }\left( X\right) $ (the second claim
is proved similarly). We have%
\begin{eqnarray*}
\left[ u\times v,w\right] &=&\sum_{z\in P_{r}\left( Z\right) }\left( u\times
v\right) ^{z}w^{z} \\
&=&\sum_{z\in P_{r}\left( Z\right) }\left( -1\right) ^{L\left( z\right)
}u^{x}v^{y}w^{z}\ \ \text{(}x,y\text{ are projections of }z\text{)} \\
&=&\sum_{x\in P_{p}\left( X\right) }\sum_{u\in P_{q}\left( Y\right)
}\sum_{z\in \Pi _{x,y}}\left( -1\right) ^{L\left( z\right) }u^{x}v^{y}w^{z}.
\end{eqnarray*}%
By Corollary \ref{Corwuv}, $w$ is a sum of the terms $\varphi \times \psi $
where $\varphi \in \Omega _{\bullet }\left( X\right) $ and $\psi \in 
\mathcal{A}_{\bullet }\left( Y\right) $, so that it suffices to prove (\ref%
{uvw=0}) for $w=\varphi \times \psi .$ Let $\varphi \in \Omega _{p}\left(
X\right) $ and, hence, $\psi \in \mathcal{A}_{q}\left( Y\right) $. Then we
have%
\begin{equation*}
w^{z}=\left( -1\right) ^{L\left( z\right) }\varphi ^{x}\psi ^{y}
\end{equation*}%
and%
\begin{equation*}
\left[ u\times v,w\right] =\sum_{x\in P_{p}\left( X\right) }\sum_{y\in
P_{q}\left( Y\right) }\sum_{z\in \Pi _{x,y}}u^{x}\varphi ^{x}v^{y}\psi ^{y}.
\end{equation*}%
Since 
\begin{equation*}
\sum_{x\in P_{p}\left( X\right) }u^{x}\varphi ^{x}=\left[ u,\varphi \right]
=0,
\end{equation*}%
we obtain (\ref{uvw=0}). If $\varphi \in \Omega _{p^{\prime }}$ with $%
p^{\prime }\neq p$, then $w^{z}=0$ for any $z\in \Pi _{x,y}$ with $x\in
P_{p}\left( X\right) $, and (\ref{uvw=0}) is trivially satisfied.
\end{proof}

\subsection{Proof of K\"{u}nneth formula}

\label{SecK}The main part of the proof of Theorem \ref{TK} is contained in
the following theorem.

\begin{theorem}
\label{Tuivi}Let $P\left( X\right) $ and $P\left( Y\right) $ be two regular
path complexes and let $P\left( Z\right) =P\left( X\right) \boxplus P\left(
Y\right) $ be their Cartesian product. Then any $\partial $-invariant path $%
w $ on $Z$ admits a representation in the form%
\begin{equation}
w=\sum_{i=1}^{k}u_{i}\times v_{i}  \label{uivi}
\end{equation}%
for some finite $k$, where $u_{i}$ and $v_{i}$ are $\partial $-invariant
paths on $X$ and $Y$, respectively.
\end{theorem}

\begin{proof}
The representation (\ref{uivi}) is simple in a special case when the path
complexes $P\left( X\right) $ and $P\left( Y\right) $ are perfect, that is,
when all allowed paths are $\partial $-invariant. Indeed, by Proposition \ref%
{Pexey}, any $w\in \Omega _{r}\left( Z\right) $ admits a representation in
the form (\ref{axy}), where $e_{x}$ and $e_{y}$ are allowed paths on $X$ and 
$Y$ respectively. By the assumption of the perfectness of $P\left( X\right) $
and $P\left( Y\right) $, the paths $e_{x}$ and $e_{y}$ are $\partial $%
-invariant, so that (\ref{axy}) implies (\ref{uivi}).

For arbitrary path complexes $P\left( X\right) $ and $P\left( Y\right) $,
the previous argument does not work since $e_{x}\times e_{y}$ does not have
to be $\partial $-invariant. Hence, we need a more elaborated strategy.
Given two subspaces $U\subset \mathcal{A}_{p}\left( X\right) $ and $V\subset 
\mathcal{A}_{q}\left( Y\right) $, denote by $U\times V$ the subspace of $%
\mathcal{A}_{r}\left( Z\right) $ that is spanned by all products $u\times v$
with $u\in U$ and $v\in V$. For any $r\geq 0$ set 
\begin{equation*}
\widetilde{\Omega }_{r}\left( Z\right) =\sum_{p+q=r}\Omega _{p}\left(
X\right) \times \Omega _{q}\left( Y\right) ,
\end{equation*}%
that is, $\widetilde{\Omega }_{r}\left( Z\right) $ is the space of paths on $%
Z$ that is spanned by all paths of the form $u\times v$ where $u\in \Omega
_{p}\left( X\right) $ and $v\in \Omega _{q}\left( Y\right) $ with some $%
p,q\geq 0$ such that $p+q=r$. By Proposition \ref{Puxv}, we have $u\times
v\in \Omega _{r}\left( Z\right) $ whence it follows that%
\begin{equation*}
\widetilde{\Omega }_{r}\left( Z\right) \subset \Omega _{r}\left( Z\right) .
\end{equation*}%
The existence of the representation (\ref{uivi}) is equivalent to the
opposite inclusion, that is, to the identity%
\begin{equation*}
\widetilde{\Omega }_{r}\left( Z\right) =\Omega _{r}\left( Z\right) .
\end{equation*}%
In fact, it suffices to show that 
\begin{equation}
\dim \Omega _{r}\left( Z\right) \leq \dim \widetilde{\Omega }_{r}\left(
Z\right) .  \label{rXYZ}
\end{equation}%
Consider also the space 
\begin{equation*}
\widetilde{\mathcal{A}}_{r}\left( Z\right) =\sum_{p+q=r}\mathcal{A}%
_{p}\left( X\right) \times \mathcal{A}_{q}\left( Y\right) .
\end{equation*}%
By definition of the cross product, all paths in $\widetilde{\mathcal{A}}%
_{r}\left( Z\right) $ are allowed, that is, 
\begin{equation*}
\widetilde{\mathcal{A}}_{r}\left( Z\right) \subset \mathcal{A}_{r}\left(
Z\right) .
\end{equation*}%
By Proposition \ref{Pexey}, any path from $\Omega _{r}\left( Z\right) $ is a
linear combination of paths $e_{x}\times e_{y}$ with allowed $x,y$, which
means that 
\begin{equation*}
\Omega _{r}\left( Z\right) \subset \widetilde{\mathcal{A}}_{r}\left(
Z\right) .
\end{equation*}%
In particular, we have also%
\begin{equation*}
\widetilde{\Omega }_{r}\left( Z\right) \subset \widetilde{\mathcal{A}}%
_{r}\left( Z\right) .
\end{equation*}%
Fix some triple $p,q,r$ with $p+q=r$ and consider the spaces (cf. Section %
\ref{Seccxy}):

\begin{itemize}
\item $\Omega _{p}^{\bot }\left( X\right) $ -- an orthogonal complement of $%
\Omega _{p}\left( X\right) $ in $\mathcal{A}_{p}\left( X\right) $;

\item $\Omega _{q}^{\bot }\left( Y\right) $ -- an orthogonal complement of $%
\Omega _{q}\left( Y\right) $ in $\mathcal{A}_{q}\left( Y\right) $;

\item $\Omega _{r}^{\bot }\left( Z\right) $ -- an orthogonal complement of $%
\Omega _{r}\left( Z\right) $ in $\widetilde{\mathcal{A}}_{r}\left( Z\right)
. $
\end{itemize}

Consider first the case when the field $\mathbb{K}$ is $\mathbb{R}$ or $%
\mathbb{Q}.$ In this case, a linear space with a $\mathbb{K}$-scalar product
is represented as a direct sum of a subspace with its orthogonal complement.
For each $u\in \mathcal{A}_{p}\left( X\right) $ consider a decomposition 
\begin{equation}
u=u_{\Omega }+u_{\bot }  \label{u'}
\end{equation}%
where $u_{\Omega }\in \Omega _{p}\left( X\right) $ and $u_{\bot }\in \Omega
_{p}^{\bot }\left( X\right) $, and a similar decomposition $v=v_{\Omega
}+v_{\bot }$ for $v\in \mathcal{A}_{q}\left( Y\right) .$ Then we have%
\begin{equation*}
u\times v=u_{\Omega }\times v_{\Omega }+u_{\Omega }\times v_{\bot
}+u_{\Omega }\times v_{\bot }+u_{\bot }\times v_{\bot }.
\end{equation*}%
Here we have $u_{\Omega }\times v_{\Omega }\in \widetilde{\Omega }_{r}\left(
Z\right) $, while by Lemma \ref{Lembot} all other terms in the right hand
side belong to $\Omega _{r}^{\bot }\left( Z\right) $, whence it follows that%
\begin{equation*}
u\times v\in \widetilde{\Omega }_{r}\left( Z\right) +\Omega _{r}^{\bot
}\left( Z\right) .
\end{equation*}%
Since $\widetilde{\mathcal{A}}_{r}\left( Z\right) $ is spanned by the
products $u\times v$ where $u,v$ are allowed, we obtain that 
\begin{equation*}
\widetilde{\mathcal{A}}_{r}\left( Z\right) =\widetilde{\Omega }_{r}\left(
Z\right) +\Omega _{r}^{\bot }\left( Z\right) .
\end{equation*}%
Comparing with the decomposition%
\begin{equation*}
\widetilde{\mathcal{A}}_{r}\left( Z\right) =\Omega _{r}\left( Z\right)
\oplus \Omega _{r}^{\bot }\left( Z\right) ,
\end{equation*}%
we obtain (\ref{rXYZ}).

Consider now the most general case of an arbitrary field $\mathbb{K}$. Let
us introduce the following notation:%
\begin{eqnarray*}
a_{p} &=&\dim \mathcal{A}_{p}\left( X\right) ,\ a_{q}=\dim \mathcal{A}%
_{q}\left( Y\right) ,\ a_{r}=\dim \widetilde{\mathcal{A}}_{r}\left( Z\right)
, \\
\omega _{p} &=&\dim \Omega _{p}\left( X\right) ,\ \omega _{q}=\dim \Omega
_{q}\left( Y\right) ,\ \omega _{r}=\dim \Omega _{r}\left( Z\right) ,
\end{eqnarray*}%
and observe that%
\begin{equation*}
\dim \Omega _{p}^{\bot }\left( X\right) =a_{p}-\omega _{p},\ \dim \Omega
_{q}^{\bot }\left( Y\right) =a_{q}-\omega _{q},\ \dim \Omega _{r}^{\bot
}\left( Z\right) =a_{r}-\omega _{r}.
\end{equation*}%
Let us prove that%
\begin{equation}
a_{r}=\sum_{p+q=r}a_{p}a_{q}.  \label{apq}
\end{equation}%
Indeed, $\mathcal{A}_{p}\left( X\right) $ is spanned by all elementary paths 
$e_{x}$ with $x\in P_{p}\left( X\right) $ and $\mathcal{A}_{q}\left(
Y\right) $ is spanned by all elementary paths $e_{y}$ with $y\in P_{q}\left(
Y\right) $. Therefore, $\widetilde{\mathcal{A}}_{r}\left( Z\right) $ is
spanned by all products $e_{x}\times e_{y}$ for $x,y$ as above over all
possible $p,q$ such that $p+q=r$. The number of such products $e_{x}\times
e_{y}$ is equal to the right hand side of (\ref{apq}), so that the identity (%
\ref{apq}) follows from the linear independence of the family $\left\{
e_{x}\times e_{y}\right\} $ (cf. Proposition \ref{Pexey}).

It follows from the above argument that 
\begin{equation}
\dim \left( \mathcal{A}_{p}\left( X\right) \times \mathcal{A}_{q}\left(
Y\right) \right) =a_{p}a_{q}  \label{apaq}
\end{equation}%
and that 
\begin{equation}
\widetilde{\mathcal{A}}_{r}\left( Z\right) =\tbigoplus_{p+q=r}\left( 
\mathcal{A}_{p}\left( X\right) \times \mathcal{A}_{q}\left( Y\right) \right)
.  \label{ArApAq}
\end{equation}
Let us show that, for any subspaces $U\subset \mathcal{A}_{p}\left( X\right) 
$ and $V\subset \mathcal{A}_{q}\left( Y\right) $, 
\begin{equation}
\dim \left( U\times V\right) =\dim U\dim V.  \label{dimUxV}
\end{equation}%
Indeed, let $u_{1},u_{2},..u_{k}$ be a basis in $U$ and $v_{1},...v_{l}$ be
a basis in $V$. Then $U\times V$ is spanned by all products $u_{i}\times
v_{j}$, so that%
\begin{equation}
\dim \left( U\times V\right) \leq kl.  \label{kl}
\end{equation}
Let us complement the basis $\left\{ u_{i}\right\} $ to a basis in $\mathcal{%
A}_{p}\left( X\right) $ by adding additional paths $u_{1}^{\prime
},...,u_{k^{\prime }}^{\prime }$, and the similarly complement $\left\{
v_{j}\right\} $ to a basis in $\mathcal{A}_{q}\left( Y\right) $ by adding $%
v_{1}^{\prime },...,v_{l^{\prime }}^{\prime }$. Set $U^{\prime }=\limfunc{%
span}\left\{ u_{i}^{\prime }\right\} $ and $V^{\prime }=\limfunc{span}%
\left\{ v_{j}^{\prime }\right\} .$ Then%
\begin{equation*}
\mathcal{A}_{p}\left( X\right) \times \mathcal{A}_{q}\left( Y\right) =\left(
U+U^{\prime }\right) \times \left( V+V^{\prime }\right) =U\times V+U\times
V^{\prime }+U^{\prime }\times V+U^{\prime }\times V^{\prime }
\end{equation*}%
whence by (\ref{apaq}) and (\ref{kl}) 
\begin{eqnarray}
a_{p}a_{q} &\leq &\dim \left( U\times V\right) +\dim \left( U\times
V^{\prime }\right) +\dim \left( U^{\prime }\times V\right) +\dim \left(
U^{\prime }\times V^{\prime }\right)  \label{U1} \\
&\leq &kl+kl^{\prime }+k^{\prime }l+k^{\prime }l^{\prime }.  \notag
\end{eqnarray}%
However, the right hand side here is equal to $\left( k+k^{\prime }\right)
\left( l+l^{\prime }\right) =a_{p}a_{q}$, which implies that we must have
the equality case in (\ref{U1}), in particular, $\dim \left( U\times
V\right) =kl,$ which proves (\ref{dimUxV}).

By Lemma \ref{Lembot}, we have%
\begin{equation*}
\Omega _{p}^{\bot }\left( X\right) \times \mathcal{A}_{q}\left( Y\right)
\subset \Omega _{r}^{\bot }\left( Z\right)
\end{equation*}%
and%
\begin{equation*}
\mathcal{A}_{p}\left( X\right) \times \Omega _{q}^{\bot }\left( Y\right)
\subset \Omega _{r}^{\bot }\left( Z\right)
\end{equation*}%
so that%
\begin{equation}
\sum_{p+q=r}\left( \left( \Omega _{p}^{\bot }\left( X\right) \times \mathcal{%
A}_{q}\left( Y\right) \right) +\left( \mathcal{A}_{p}\left( X\right) \times
\Omega _{q}^{\bot }\left( Y\right) \right) \right) \subset \Omega _{r}^{\bot
}\left( Z\right) .  \label{OmAOm}
\end{equation}%
It follows from (\ref{ArApAq}) and (\ref{OmAOm}) that%
\begin{equation*}
\sum_{p+q=r}\dim \left( \left( \Omega _{p}^{\bot }\left( X\right) \times 
\mathcal{A}_{q}\left( Y\right) \right) +\left( \mathcal{A}_{p}\left(
X\right) \times \Omega _{q}^{\bot }\left( Y\right) \right) \right) \leq \dim
\Omega _{r}^{\bot }\left( Z\right) .
\end{equation*}%
Note that the subspaces $\Omega _{p}^{\bot }\left( X\right) \times \mathcal{A%
}_{q}\left( Y\right) $ and $\mathcal{A}_{p}\left( X\right) \times \Omega
_{q}^{\bot }\left( Y\right) $ have intersection $\Omega _{p}^{\bot }\left(
X\right) \times \Omega _{q}^{\bot }\left( Y\right) $\label{rem: details},
which implies that%
\begin{eqnarray*}
&&\dim \left( \left( \Omega _{p}^{\bot }\left( X\right) \times \mathcal{A}%
_{q}\left( Y\right) \right) +\left( \mathcal{A}_{p}\left( X\right) \times
\Omega _{q}^{\bot }\left( Y\right) \right) \right) \\
&=&\dim \left( \Omega _{p}^{\bot }\left( X\right) \times \mathcal{A}%
_{q}\left( Y\right) \right) +\dim \left( \mathcal{A}_{p}\left( X\right)
\times \Omega _{q}^{\bot }\left( Y\right) \right) \\
&&-\dim \left( \Omega _{p}^{\bot }\left( X\right) \times \Omega _{q}^{\bot
}\left( Y\right) \right) \\
&=&\left( a_{p}-\omega _{p}\right) a_{q}+a_{p}\left( a_{q}-\omega
_{q}\right) -\left( a_{p}-\omega _{p}\right) \left( a_{q}-\omega _{q}\right)
\\
&=&a_{p}a_{q}-\omega _{p}\omega _{q}.
\end{eqnarray*}%
Hence,%
\begin{equation*}
\sum_{p+q=r}\left( a_{p}a_{q}-\omega _{p}\omega _{q}\right) \leq
a_{r}-\omega _{r}.
\end{equation*}%
Combining with (\ref{apq}) we obtain%
\begin{equation*}
\omega _{r}\leq \sum_{p+q=r}\omega _{p}\omega _{q}.
\end{equation*}%
Finally, we are left to observe that 
\begin{equation*}
\sum_{p+q=r}\omega _{p}\omega _{q}=\dim \widetilde{\Omega }_{r}\left(
Z\right) ,
\end{equation*}%
whence (\ref{rXYZ}) follows.
\end{proof}

\begin{proof}[Proof of Theorem \protect\ref{TK}]
The isomorphism (\ref{HXYZ}) follows from (\ref{Omrpq}) and the K\"{u}nneth
theorem (\ref{KAB}), so we only need to prove (\ref{Omrpq}). Define the
tensor product of graded linear spaces%
\begin{equation*}
\mathcal{A}_{\bullet }\left( X,Y\right) =\mathcal{A}_{\bullet }\left(
X\right) \otimes \mathcal{A}_{\bullet }\left( Y\right)
\end{equation*}%
and a linear mapping%
\begin{equation*}
\Phi :\mathcal{A}_{r}\left( X,Y\right) \rightarrow \mathcal{A}_{r}\left(
Z\right)
\end{equation*}%
that is defined by%
\begin{equation*}
\Phi \left( e_{x}\otimes e_{y}\right) =e_{x}\times e_{y}
\end{equation*}%
for all $x\in P_{p}\left( X\right) $ and $y\in P_{q}\left( Y\right) $ with $%
p+q=r$. In fact, we have%
\begin{equation*}
\Phi \left( \mathcal{A}_{r}\left( X,Y\right) \right) =\widetilde{\mathcal{A}}%
_{r}\left( Z\right)
\end{equation*}%
where $\widetilde{\mathcal{A}}_{r}\left( Z\right) $ is defined by (\ref%
{ArApAq}). It follows from the argument in the proof of Theorem \ref{Tuivi}
that the mapping $\Phi $ is injective.

Consider now the tensor product of the chain complexes%
\begin{equation*}
\Omega _{\bullet }\left( X,Y\right) =\Omega _{\bullet }\left( X\right)
\otimes \Omega _{\bullet }\left( Y\right)
\end{equation*}%
and observe that%
\begin{equation*}
\Phi \left( \Omega _{r}\left( X,Y\right) \right) =\widetilde{\Omega }%
_{r}\left( Z\right) .
\end{equation*}%
Since by Theorem \ref{Tuivi} 
\begin{equation}
\widetilde{\Omega }_{r}\left( Z\right) =\Omega _{r}\left( Z\right) ,
\label{rr}
\end{equation}%
we obtain that the mapping $\Phi $ provides a linear isomorphism of the
spaces $\Omega _{\bullet }\left( X,Y\right) $ and $\Omega _{\bullet }\left(
Z\right) $. Moreover, since $\Phi $ commutes with $\partial $ by (\ref{utsv}%
) and the product rule of Proposition \ref{Propdusqv}, $\Phi $ provides an
isomorphism of the chain complexes $\Omega _{\bullet }\left( X,Y\right) $
and $\Omega _{\bullet }\left( Z\right) $, which finishes the proof.
\end{proof}

\section{Minimal paths and hole detection}

\setcounter{equation}{0}\label{Sec8}The elements of $H_{p}\left( G\right) $
can be regarded as $p$-dimensional holes in the digraph $G$. To make this
notion more geometric, we can work with representatives of the homologies
classes, that are closed $p$-paths$.$ Recall that two closed $p$-paths $u$
and $v$ are homological, that is, represent the same homology class, if $u-v$
is exact. We write in this case $u\sim v$.

For any $p$-path $v$ define its length by%
\begin{equation*}
\ell \left( v\right) =\sum_{i_{0},...,i_{p}\in V}\left\vert
v^{i_{0}...i_{p}}\right\vert .
\end{equation*}%
Given a closed $p$-paths $v_{0}$, consider the minimization problem%
\begin{equation}
\ell \left( v\right) \mapsto \min \ \ \text{for }v\sim v_{0}.  \label{vmin}
\end{equation}%
This problem has always a solution, although not necessarily unique. Any
solution of (\ref{vmin}) is called a \emph{minimal} $p$-path. It is hoped
that minimal $p$-paths (in a given homology class) match our geometric
intuition of what holes in a graph should be. In this section we give some
examples of minimal paths to support this claim.

\begin{example}
\RM Consider the digraph $G=\left( V,E\right) $ as on Fig. \ref{pic19}.%
\FRAME{ftbphFU}{5.0704in}{2.2494in}{0pt}{\Qcb{A digraph with $14$ vertices
and $23$ edges}}{\Qlb{pic19}}{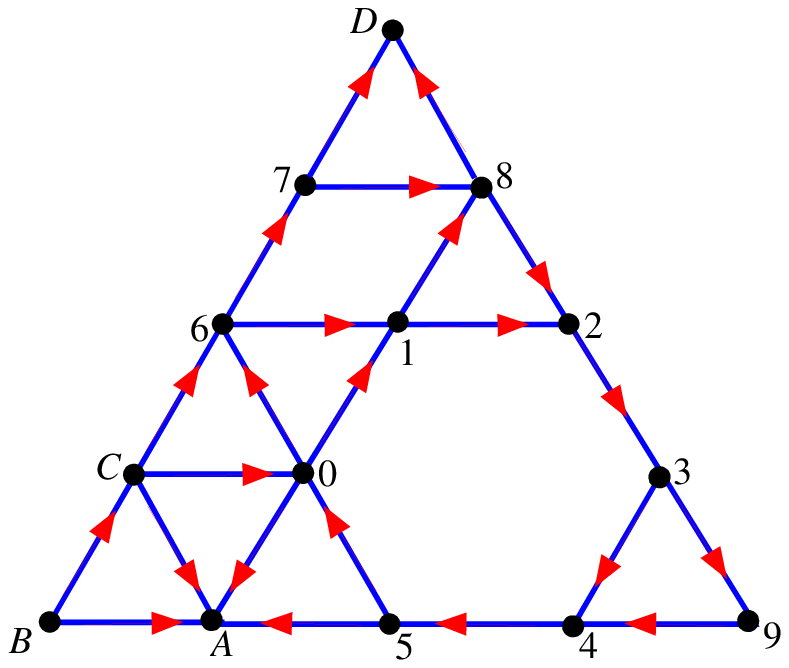}{\special{language "Scientific
Word";type "GRAPHIC";maintain-aspect-ratio TRUE;display "USEDEF";valid_file
"F";width 5.0704in;height 2.2494in;depth 0pt;original-width
6.3027in;original-height 2.0124in;cropleft "0";croptop "1";cropright
"1";cropbottom "0";filename 'pic19.eps';file-properties "XNPEU";}}

By Theorem \ref{Tabac}, we can remove successively the vertices $9,B,C,A,D$
(and their adjacent edges) without changing the homologies. Then by Theorem %
\ref{Tcab} we can remove the vertices $7,6,8$ equally without changing the
homologies. We are left with the graph $G^{\prime }=\left( V^{\prime
},E^{\prime }\right) $ where $V^{\prime }=\left\{ 0,1,2,3,4,5\right\} $ and $%
E^{\prime }=\left\{ 01,12,23,34,45,50\right\} $. By Proposition \ref%
{Propcycle} we obtain $\dim H_{1}\left( G\right) =1,$ while $H_{p}\left(
G\right) =\left\{ 0\right\} $ for all $p\geq 2$.

The following closed $1$-path is a minimal path in the non-trivial homology
class of $H_{1}\left( G\right) $: 
\begin{equation*}
v=e_{01}+e_{12}+e_{23}+e_{34}+e_{45}+e_{50},
\end{equation*}%
that is obviously associated with a hexagonal hole on Fig. \ref{pic19}.
\end{example}

\begin{example}
\RM Consider a digraph $G=\left( V,E\right) $ on Fig. \ref{pic6a}.\FRAME{%
ftbhFU}{5.0704in}{2.2952in}{0pt}{\Qcb{A digraph $G$ with $12$ vertices and $%
32$ edges.}}{\Qlb{pic6a}}{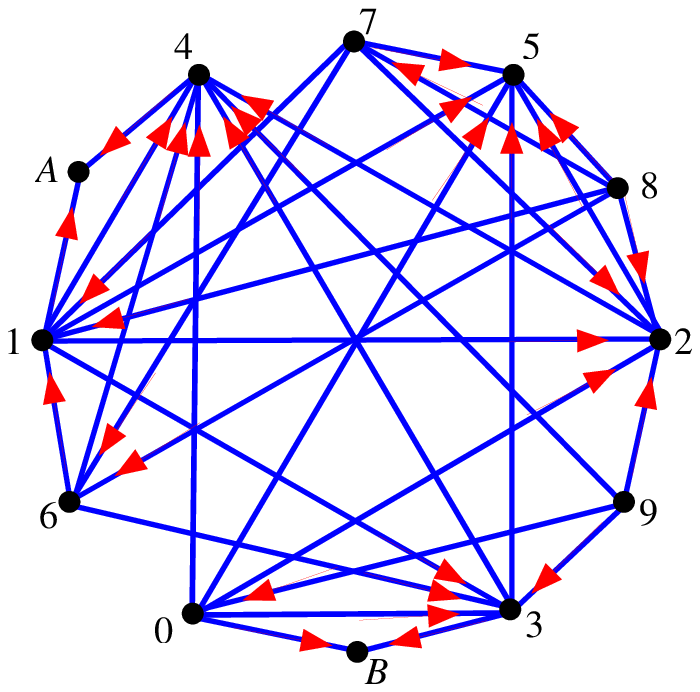}{\special{language "Scientific
Word";type "GRAPHIC";maintain-aspect-ratio TRUE;display "USEDEF";valid_file
"F";width 5.0704in;height 2.2952in;depth 0pt;original-width
6.3027in;original-height 2.6671in;cropleft "0";croptop "1";cropright
"1";cropbottom "0";filename 'pic6a.eps';file-properties "XNPEU";}}

Removing successively the vertices $A,B,8,9,6,7$ by Theorem \ref{Tabac}, we
obtain a digraph $G^{\prime }=\left( V^{\prime },E^{\prime }\right) $ with $%
V^{\prime }=\left\{ 0,1,2,3,4,5\right\} $ and $E^{\prime }=\left\{
02,03,04,05,12,13,14,15,24,25,34,35\right\} $ that has the same homologies
as $G$. The digraph $G^{\prime }$ is shown in two ways on Fig. \ref{pic6b}.
Clearly, the second representation of this graph is reminiscent of an
octahedron. \FRAME{ftbhFU}{4.8784in}{1.9138in}{0pt}{\Qcb{Two representations
of the digraph $G^{\prime }$}}{\Qlb{pic6b}}{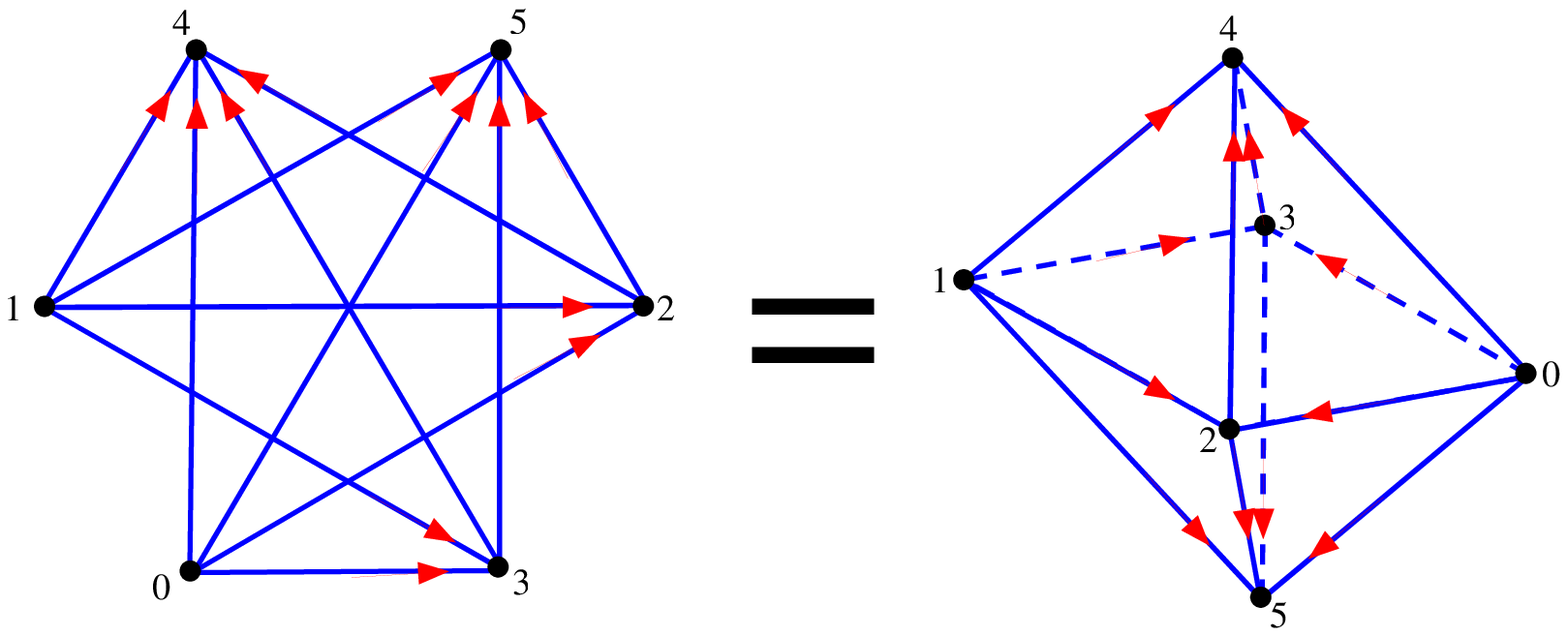}{\special{language
"Scientific Word";type "GRAPHIC";maintain-aspect-ratio TRUE;display
"USEDEF";valid_file "F";width 4.8784in;height 1.9138in;depth
0pt;original-width 6.3027in;original-height 2.5685in;cropleft "0";croptop
"1";cropright "1";cropbottom "0";filename 'pic6b.eps';file-properties
"XNPEU";}}

The digraph $G^{\prime }$ is the same as the $2$-dimensional sphere-graph of
Example \ref{Expic13} (cf. Fig. \ref{pic13}). Hence, we obtain by (\ref{G012}%
) that $\dim H_{2}\left( G\right) =1$ while $H_{p}\left( G\right) =\left\{
0\right\} $ for $p=1$ and $p>2$.

The following closed $2$-path is a minimal path in the non-trivial homology
class of $H_{2}\left( G\right) $: 
\begin{equation*}
v=e_{024}-e_{025}-e_{034}+e_{035}-e_{124}+e_{125}+e_{134}-e_{135},
\end{equation*}%
that is a $2$-path that determines a $2$-dimensional hole in $G$ given by
the octahedron. Note that on Fig. \ref{pic6a} this octahedron is hardy
visible, but it can be computed purely algebraically as shown above.
\end{example}

\section{Appendix: Elements of homological algebra}

\setcounter{equation}{0}\label{Sec3}

\subsection{Cochain complexes}

\label{Seccochain}A\ \emph{cochain complex} $X$ is a sequence 
\begin{equation}
\begin{array}{cccccccccccc}
0 & \rightarrow & X^{0} & \overset{d}{\rightarrow } & X^{1} & \overset{d}{%
\rightarrow } & \dots & \overset{d}{\rightarrow } & X^{p-1} & \overset{d}{%
\rightarrow } & X^{p} & \overset{d}{\rightarrow }\dots%
\end{array}
\label{Xp}
\end{equation}%
of vector spaces $\left\{ X^{p}\right\} _{p=0}^{\infty }$ over a field $%
\mathbb{K}$ and linear mappings $d:X^{p}\rightarrow X^{p+1}$ with the
property that $d^{2}=0$ at each level. To distinguish the operators $d$ on
different spaces, we will denote by $d|_{X^{p}}$ the operator $%
d:X^{p}\rightarrow X^{p+1}$. The condition $d^{2}=0$ means that 
\begin{equation*}
\func{Im}d|_{X^{p-1}}\subset \ker d|_{X^{p}}.
\end{equation*}
This allows to define the de Rham cohomologies of the complex $X$ by%
\begin{equation*}
H^{p}\left( X\right) =\ker d|_{X^{p}}\left/ \func{Im}d|_{X^{p-1}}\right.
\end{equation*}%
where $X^{-1}:=\left\{ 0\right\} $. The sequence (\ref{Xp}) is called exact
if $H^{p}\left( X\right) =\left\{ 0\right\} $ for all $p\geq 0$.

We always assume that the spaces $X^{p}$ are finitely dimensional.

\begin{lemma}
\label{LemdimHX}We have for any $p\geq 0$ 
\begin{eqnarray}
\dim H^{p}\left( X\right) &=&\dim X^{p}-\dim dX^{p}-\dim dX^{p-1}
\label{dimHX} \\
&=&\dim \ker d|_{X^{p}}+\dim \ker d|_{X^{p-1}}-\dim X^{p-1}.  \label{dimHX2}
\end{eqnarray}
\end{lemma}

\begin{proof}
By definition, we have%
\begin{equation}
\dim H^{p}\left( X\right) =\dim \ker d|_{X^{p}}-\dim \func{Im}d|_{X^{p-1}}.
\label{k-i}
\end{equation}%
Applying the nullity-rank theorem to the mapping $d:X^{p}\rightarrow X^{p+1}$%
, we obtain 
\begin{equation*}
\dim \ker d|_{X^{p}}=\dim X^{p}-\dim \func{Im}d|_{X^{p}}.
\end{equation*}%
Substituting into (\ref{k-i}), we obtain (\ref{dimHX}).

In the same way, substituting into (\ref{k-i}) the identity 
\begin{equation*}
\dim \func{Im}d|_{X^{p-1}}=\dim X^{p-1}-\dim \ker d|_{X^{p-1}},
\end{equation*}%
we obtain (\ref{dimHX2}).
\end{proof}

\begin{lemma}
\label{Lemsumsum}For a finite cochain complex%
\begin{equation}
\begin{array}{cccccccccccc}
0 & \rightarrow & X^{0} & \overset{d}{\rightarrow } & X^{1} & \overset{d}{%
\rightarrow } & \dots & \overset{d}{\rightarrow } & X^{n-1} & \overset{d}{%
\rightarrow } & X^{n} & \overset{d}{\rightarrow }0%
\end{array}%
,  \label{Xn0}
\end{equation}%
the following identity is satisfied%
\begin{equation}
\sum_{k=0}^{n}\left( -1\right) ^{k}\dim H^{k}\left( X\right)
=\sum_{k=0}^{n}\left( -1\right) ^{k}\dim X^{k}.  \label{sum=sum}
\end{equation}%
In particular, if the sequence \emph{(\ref{Xn0})} is exact, then%
\begin{equation}
\sum_{k=0}^{n}\left( -1\right) ^{k}\dim X^{k}=0.  \label{sum=0}
\end{equation}
\end{lemma}

\begin{proof}
We have by (\ref{dimHX})%
\begin{eqnarray*}
\sum_{k=0}^{n}\left( -1\right) ^{k}\dim H^{k}\left( X\right)
&=&\sum_{k=0}^{n}\left( -1\right) ^{k}\dim X^{k}-\sum_{k=0}^{n}\left(
-1\right) ^{k}\dim dX^{k}-\sum_{k=0}^{n}\left( -1\right) ^{k}\dim dX^{k-1} \\
&=&\sum_{k=0}^{n}\left( -1\right) ^{k}\dim X^{k}-\sum_{k=0}^{n-1}\left(
-1\right) ^{k}\dim dX^{k}-\sum_{j=0}^{n-1}\left( -1\right) ^{j+1}\dim dX^{j}
\\
&=&\sum_{k=0}^{n}\left( -1\right) ^{k}\dim X^{k},
\end{eqnarray*}%
whence (\ref{sum=sum}) follows. If in addition the sequence (\ref{Xn0}) is
exact then the left hand side of (\ref{sum=sum}) vanishes, whence (\ref%
{sum=0}) follows.
\end{proof}

For any finite cochain complex (\ref{Xn0}), define its Euler characteristic
by%
\begin{equation*}
\chi \left( X\right) =\sum_{p=0}^{n}\left( -1\right) ^{p}\dim X^{p}.
\end{equation*}%
Then (\ref{sum=sum}) implies 
\begin{equation*}
\chi \left( X\right) =\sum_{k=0}^{n}\left( -1\right) ^{k}\dim H^{k}\left(
X\right) .
\end{equation*}

\subsection{Chain complexes}

Given a cochain complex (\ref{Xp}) with finite-dimensional spaces $X^{p}$,
denote by $X_{p}$ the dual space to $X^{p}$ and by $\partial $ the dual
operator to $d$. Then we obtain a chain complex 
\begin{equation}
\begin{array}{cccccccccccc}
0 & \leftarrow & X_{0} & \overset{\partial }{\leftarrow } & X_{1} & \overset{%
\partial }{\leftarrow } & \dots & \overset{\partial }{\leftarrow } & X_{p-1}
& \overset{\partial }{\leftarrow } & X_{p} & \overset{\partial }{\leftarrow }%
\dots%
\end{array}
\label{Xpdual}
\end{equation}%
Denoting by $\left( \cdot ,\cdot \right) $ the natural pairing of dual
spaces.

For $\omega \in X^{p}$ and $v\in X_{p}$ we write $\omega \bot v$ if $\left(
\omega ,v\right) =0.$ If $S$ is a subset of $X^{p}$ then $S^{\bot }$ denotes
the annihilator in the dual space $X_{p}$, that is, 
\begin{equation*}
S^{\bot }=\left\{ v\in X_{p}:\omega \bot v\ \ \ \forall \omega \in S\right\}
.
\end{equation*}%
Clearly, $S^{\bot }$ is a linear subspace of $X_{p}$. In the same way one
defines the annihilator of subsets of $X_{p}$.

By definition we have%
\begin{equation*}
\left( d\omega ,v\right) =\left( \omega ,\partial v\right)
\end{equation*}%
for all $\omega \in X^{p}$ and $v\in X_{p+1}$. Since $d^{2}=0$, it follows
that also $\partial ^{2}=0.$ Hence, one can define the \emph{homologies} of
the chain complex (\ref{Xpdual}) by%
\begin{equation*}
H_{p}\left( X\right) =\ker \partial |_{X_{p}}\left/ \func{Im}\partial
|_{X_{p+1}}\right. .
\end{equation*}%
By duality we have%
\begin{equation}
\ker \partial |_{X_{p}}=\left( \func{Im}d|_{X^{p-1}}\right) ^{\bot },\ \ \ \
\ \ \ker d|_{X^{p}}=\left( \func{Im}\partial |_{X_{p+1}}\right) ^{\bot }\ .
\label{ker=Im}
\end{equation}

\begin{lemma}
\label{LemHdual}The spaces $H^{p}\left( X\right) $ and $H_{p}\left( X\right) 
$ are dual. In particular, $\dim H^{p}\left( X\right) =\dim H_{p}\left(
X\right) .$
\end{lemma}

\begin{proof}
Let $\omega \in X^{p}$ be a representative of an element of $H^{p}\left(
X\right) $ and $v\in X_{p}$ be a representative of an element of $H\left(
X_{p}\right) $. Let us show that the pairing $\left( \omega ,v\right) $ of $%
\omega $ and $v$ in the dual spaces $X^{p}$ and $X_{p}$ is also well-defined
for the elements of $H^{p}\left( X\right) $ and $H_{p}\left( X\right) $.
Indeed, $v$ is defined $\func{mod}\func{Im}\partial |_{X_{p+1}}$, that is, $%
v $ and $v+\partial u$ represent the same element of $H_{p}\left( X\right) $
for any $u\in X_{p+1}$. Since $\omega \in \ker d|_{X^{p}}$, we obtain%
\begin{equation*}
\left( \omega ,v\right) =\left( \omega ,v+\partial u\right) =\left( \omega
,v\right) +\left( d\omega ,u\right) =\left( \omega ,v\right) .
\end{equation*}%
In the same way, $\left( \omega ,v\right) $ does not change when adding $%
d\varphi $ to $\omega $.

If $\left( \omega ,v\right) =0$ for all $v\in \ker \partial |_{X_{p}}$ then $%
\omega \bot \ker \partial |_{X_{p}}$ whence by (\ref{ker=Im}) $\omega \in 
\func{Im}d|_{X^{p-1}}$, that is, $\omega $ represent the zero element of $%
H^{p}\left( X\right) $. In the same way, if $\left( \omega ,v\right) =0$ for
all $\omega \in \ker d|_{X^{p}}$ then $v$ represents the zero element of $%
H_{p}\left( X\right) $. Hence, $\left( \omega ,v\right) $ is a pairing
between $H^{p}\left( X\right) $ and $H_{p}\left( X\right) $, whence the
duality of these spaces follows.
\end{proof}

\begin{lemma}
\label{LemdimHX1}We have for any $p\geq 0$ 
\begin{eqnarray}
\dim H_{p}\left( X\right) &=&\dim X_{p}-\dim \partial X_{p}-\dim \partial
X_{p+1}  \label{dimHP2} \\
&=&\dim \ker \partial |_{X_{p}}+\dim \ker \partial |_{X_{p+1}}-\dim X_{p+1}.
\notag
\end{eqnarray}
\end{lemma}

The proof is similar to Lemma \ref{LemdimHX}\textbf{.}

\subsection{Tensor product of chain complexes}

\label{SecTensor}Let $\left\{ A_{n}\right\} $ be a sequence of finite
dimensional linear spaces over $\mathbb{K}$ enumerated by an integer
parameter $n$. We denote usually by $A_{\bullet }$ the whole sequence. Now
it will be convenient to denote by $A_{\bullet }$ the direct sum of all $%
A_{n}$, that is%
\begin{equation*}
A_{\bullet }=\tbigoplus_{n}A_{n}
\end{equation*}%
so that $A_{\bullet }$ is a graded linear space. If $\left\{ A_{n}\right\} $
is a chain complex with the boundary operator $\partial _{A}$ then $\partial
_{A}$ extends linearly to an operator in $A_{\bullet }$ that respect a
graded structure. In this case we will denote the chain complex also by $%
A_{\bullet }$. The homologies $H_{n}\left( A_{\bullet }\right) $ of the
chain complex $A_{\bullet }$ form also a graded linear space $H_{\bullet
}\left( A_{\bullet }\right) .$

Given two graded linear spaces $A_{\bullet }$ and $B_{\bullet }$ as above,
define their tensor product by 
\begin{equation*}
A_{\bullet }\otimes B_{\bullet }=\tbigoplus_{p,q}\left( A_{p}\otimes
B_{q}\right) ,
\end{equation*}%
where $A_{p}\otimes B_{q}$ is the tensor product over $\mathbb{K}$ of the
linear spaces $A_{p}$ and $B_{q}$. In other words, $A_{\bullet }\otimes
B_{\bullet }=C_{\bullet }$ where%
\begin{equation*}
C_{r}=\tbigoplus_{\left\{ p,q:p+q=r\right\} }\left( A_{p}\otimes
B_{q}\right) .
\end{equation*}

If $A_{\bullet }$ and $B_{\bullet }$ are chain complexes with the boundary
operators $\partial _{A}$ and $\partial _{B}$, respectively, then define the
boundary operator $\partial _{C}$ in $C_{\bullet }$ by by 
\begin{equation}
\partial _{C}\left( u\otimes v\right) =\left( \partial _{A}u\right) \otimes
v+\left( -1\right) ^{p}u\otimes \left( \partial _{B}v\right) ,  \label{utsv}
\end{equation}%
for all $u\in A_{p}$ and $v\in B_{q}$. It is well-known that $\partial
_{C}^{2}=0$ so that $C_{\bullet }$ with $\partial _{C}$ is a chain complex.
Furthermore, by a theorem of K\"{u}nneth, we have the following identity for
homologies:%
\begin{equation}
H_{\bullet }\left( C_{\bullet }\right) \cong H_{\bullet }\left( A_{\bullet
}\right) \otimes H_{\bullet }\left( B_{\bullet }\right)  \label{KAB}
\end{equation}%
that is,%
\begin{equation*}
H_{r}\left( C_{\bullet }\right) \cong \tbigoplus_{\left\{ p,q:p+q=r\right\}
}H_{p}\left( A_{\bullet }\right) \otimes H_{q}\left( B_{\bullet }\right)
\end{equation*}%
(see \cite{MacLane}\label{rem: citation for Kuenneth theorem}).

Let $A^{n}$ be a dual space to $A_{n}$. Then the graded linear space 
\begin{equation*}
A^{\bullet }=\tbigoplus_{n}A^{n}
\end{equation*}%
is dual to $A_{\bullet }$. Indeed, it follows from the following definition
of the pairing $\left( \omega ,v\right) $ between elements $\omega \in
A^{\bullet }$ and $v\in A_{\bullet }.$ If $\omega \in A^{n}$ and $v\in A_{m}$
then in the case $n=m$ then $\left( \omega ,v\right) $ coincides the pairing
between $A^{n}$ and $A_{n}$, whereas in the case $n\neq m$ set $\left(
\omega ,v\right) =0.$ Then extend this definition by bilinearity to all $%
\omega \in A^{\bullet }$ and $v\in A_{\bullet }$.

Finally, observe that if $A_{\bullet }$ and $B_{\bullet }$ are two graded
linear spaces then $A^{\bullet }\otimes B^{\bullet }$ is dual to $A_{\bullet
}\otimes B_{\bullet }$ using the following pairing:%
\begin{equation*}
\left( \varphi \otimes \psi ,u\otimes v\right) =\left( \varphi ,u\right)
\left( \psi ,v\right)
\end{equation*}%
for $\varphi \in A^{\bullet },\psi \in B^{\bullet },u\in A_{\bullet },v\in
B_{\bullet }$.\label{rem: more careful language since A is infity-
dimensional}

\subsection{Sub-complexes and quotient complexes}

Let $X$ be a cochain complex as in (\ref{Xp}), and assume that each $X^{p}$
has a subspace $J^{p}$ so that $d$ is invariant on $\left\{ J^{p}\right\} $,
that is, $dJ^{p}\subset J^{p+1}.$ Then we have a cochain \emph{sub-complex} $%
J$ as follows:%
\begin{equation}
\begin{array}{cccccccccccc}
0 & \rightarrow & J^{0} & \overset{d}{\rightarrow } & J^{1} & \overset{d}{%
\rightarrow } & \dots & \overset{d}{\rightarrow } & J^{p-1} & \overset{d}{%
\rightarrow } & J^{p} & \overset{d}{\rightarrow }\dots%
\end{array}
\label{Jp}
\end{equation}%
Since the operator $d$ is well defined also on the quotient spaces $%
X^{p}/J^{p},$ we obtain also a cochain \emph{quotient complex} $X/J$:%
\begin{equation}
\begin{array}{cccccccccccc}
0 & \rightarrow & X^{0}/J^{0} & \overset{d}{\rightarrow } & X^{1}/J^{1} & 
\overset{d}{\rightarrow } & \dots & \overset{d}{\rightarrow } & 
X^{p-1}/J^{p-1} & \overset{d}{\rightarrow } & X^{p}/J^{p} & \overset{d}{%
\rightarrow }\dots%
\end{array}
\label{X/Jp}
\end{equation}%
Consider the annihilator of $J^{p}$, that is the space%
\begin{equation*}
\left( J^{p}\right) ^{\bot }=\left\{ v\in X_{p}:v\bot J^{p}\right\} .
\end{equation*}

\begin{lemma}
\label{LemJbot}The dual operator $\partial $ of $d$ is invariant on $\left\{
\left( J^{p}\right) ^{\bot }\right\} $, and the chain sub-complex 
\begin{equation}
\begin{array}{cccccccccccc}
0 & \leftarrow & \left( J^{0}\right) ^{\bot } & \overset{\partial }{%
\leftarrow } & \left( J^{1}\right) ^{\bot } & \overset{\partial }{\leftarrow 
} & \dots & \overset{\partial }{\leftarrow } & \left( J^{p-1}\right) ^{\bot }
& \overset{\partial }{\leftarrow } & \left( J^{p}\right) ^{\bot } & \overset{%
\partial }{\leftarrow }\dots%
\end{array}
\label{Jback}
\end{equation}%
is dual to the cochain quotient complex \emph{(\ref{X/Jp})}.
\end{lemma}

\begin{proof}
If $v\in \left( J^{p}\right) ^{\bot }$ then, for any $\omega \in J^{p-1},$
we have $d\omega \in J^{p}$ and, hence,%
\begin{equation*}
\left( \omega ,\partial v\right) =\left( d\omega ,v\right) =0,
\end{equation*}%
which implies $\partial v\in \left( J^{p-1}\right) ^{\bot }$. Hence, $%
\partial $ maps $\left( J^{p}\right) ^{\bot }$ to $\left( J^{p-1}\right)
^{\bot }$, so that the complex (\ref{Jback}) is well-defined.

To prove the duality of (\ref{X/Jp}) and (\ref{Jback}), observe that $\left(
J^{p}\right) ^{\bot }$ is naturally isomorphic to the dual space $\left(
X^{p}/J^{p}\right) ^{\prime }.$ Indeed, each $v\in \left( J^{p}\right)
^{\bot }$ defines a linear functional on $X^{p}/J^{p}$ simply by $\omega
\mapsto \left( \omega ,v\right) $ where $\omega \in X^{p}$ is a
representative of an element of $X^{p}/J^{p}$. If $\omega _{1}\,=\omega
_{2}\ \func{mod}J^{p}$ then $\omega _{1}-\omega _{2}\in J^{p}$ whence $%
\left( \omega _{1}-\omega _{2},v\right) =0$ and $\left( \omega _{1},v\right)
=\left( \omega _{2},v\right) $. Clearly, the mapping $\left( J^{p}\right)
^{\bot }\rightarrow \left( X^{p}/J^{p}\right) ^{\prime }$ is injective and,
hence, surjective because of the identity of the dimensions of the two
spaces. Finally, the duality of the operators $d$ and $\partial $ on the
complexes (\ref{X/Jp}) and (\ref{Jback}) is a trivial consequence of their
duality on the complexes $X^{\bullet }$ and $X_{\bullet }$.
\end{proof}

Let us describe a specific method of constructing of $d$-invariant subspaces.

\begin{lemma}
\label{LemdE}Given any subspace $N^{p}$ of $X^{p}$, set%
\begin{equation}
J^{p}=N^{p}+dN^{p-1}.  \label{JS+S}
\end{equation}%
Then $d$ is invariant on $\left\{ J^{p}\right\} $. Besides, we have the
following identity%
\begin{equation}
\left( J^{p}\right) ^{\bot }=\left\{ v\in \left( N^{p}\right) ^{\bot
}:\partial v\in \left( N^{p-1}\right) ^{\bot }\right\} .  \label{Jbot=}
\end{equation}
\end{lemma}

\begin{proof}
The first claim follows from $d^{2}=0$ since%
\begin{equation*}
dJ^{p}\subset dN^{p}+d^{2}N^{p-1}=dN^{p}\subset J^{p+1}.
\end{equation*}%
The condition $v\in \left( J^{p}\right) ^{\bot }$ means that%
\begin{equation}
v\bot N^{p}\ \ \text{and\ \ }v\bot dN^{p-1}.  \label{vbS}
\end{equation}%
Clearly, the first condition here is equivalent to $v\in \left( N^{p}\right)
^{\bot }$, while the second condition is equivalent to%
\begin{equation*}
\left( d\omega ,v\right) =0\ \ \forall \omega \in N^{p-1}\Leftrightarrow
\left( \omega ,\partial v\right) =0\ \forall \omega \in
N^{p-1}\Leftrightarrow \partial v\bot N^{p-1}\Leftrightarrow \partial v\in
\left( N^{p-1}\right) ^{\bot },
\end{equation*}%
which proves (\ref{Jbot=}).
\end{proof}

\subsection{Zigzag Lemma}

Consider now three cochain complexes $X,Y,Z$ connected by vertical linear
mappings as on the diagram:

\begin{equation}
\begin{array}{cccccccc}
&  & 0 &  & 0 &  & 0 &  \\ 
&  & \downarrow &  & \downarrow &  & \downarrow &  \\ 
0 & \rightarrow & Y^{0} & \overset{d}{\rightarrow } & Y^{1} & \overset{d}{%
\rightarrow } & Y^{2} & \rightarrow \dots \\ 
&  & \downarrow ^{\alpha } &  & \downarrow ^{\alpha } &  & \downarrow
^{\alpha } &  \\ 
0 & \rightarrow & X^{0} & \overset{d}{\rightarrow } & X^{1} & \overset{d}{%
\rightarrow } & X^{2} & \rightarrow \dots \\ 
&  & \downarrow ^{\alpha } &  & \downarrow ^{\alpha } &  & \downarrow
^{\alpha } &  \\ 
0 & \rightarrow & Z^{0} & \overset{d}{\rightarrow } & Z^{1} & \overset{d}{%
\rightarrow } & Z^{2} & \rightarrow \dots \\ 
&  & \downarrow &  & \downarrow &  & \downarrow &  \\ 
&  & 0 &  & 0 &  & 0 & 
\end{array}
\label{ABCD}
\end{equation}%
Each horizontal mapping is denoted by ~$d$ and each vertical mapping is
denoted by $\alpha $. We assume that the diagram is commutative. Let us also
assume that each column in (\ref{ABCD}) is an exact sequence, that is, the
mapping $\alpha :Y^{p}\rightarrow X^{p}$ is an injection, and $\alpha
:X^{p}\rightarrow Z^{p}$ a surjection with the kernel $Y^{p}$. In this case
we can identify $Y^{p}$ with a subspace of $X^{p}$ and $Z^{p}$ with the
quotient $X^{p}/Y^{p}$.

\begin{proposition}
\label{Plong}\emph{(Zigzag Lemma)} Under the above conditions the sequence 
\begin{equation}
0\rightarrow H^{0}\left( Y\right) \rightarrow H^{0}\left( X\right)
\rightarrow H^{0}\left( Z\right) \rightarrow \dots \rightarrow H^{p}\left(
Y\right) \rightarrow H^{p}\left( X\right) \rightarrow H^{p}\left( Z\right)
\rightarrow H^{p+1}\left( Y\right) \rightarrow \dots  \label{long}
\end{equation}%
is exact.
\end{proposition}

The sequence (\ref{long}) is called a \emph{long exact sequence in
cohomology. }A similar result holds for homologies of chain complexes.

The meaning of the statement is that the mappings denoted in (\ref{long}) by
arrows, can be defined so that this sequence is exact. For example, the
mappings 
\begin{equation*}
H^{p}\left( Y\right) \rightarrow H^{p}\left( X\right) \rightarrow
H^{p}\left( Z\right)
\end{equation*}%
are obvious extensions of the mapping $\alpha $ in (\ref{ABCD}), whereas the
mapping $H^{p}\left( Z\right) \rightarrow H^{p+1}\left( Y\right) $ is
defined in a more tricky way. The details of the proof can be found in \cite%
{MacLane}.

One normally applies Proposition \ref{Plong} in the following form: if $X$
is a cochain complex (\ref{Xp}) and $J$ is its sub-complex (\ref{Jp}), then
the following long sequence is exact:%
\begin{equation}
0\rightarrow \dots \rightarrow H^{p}(J)\rightarrow H^{p}(X)\rightarrow
H^{p}(X/J)\rightarrow H^{p+1}(J)\rightarrow \dots  \label{longXJ}
\end{equation}%
Similarly, if $X$ is a chain complex (\ref{Xpdual}) and $J$ its sub-complex,
then the following long sequence is exact:%
\begin{equation}
0\leftarrow \dots \leftarrow H_{p}(X/J)\leftarrow H_{p}(X)\leftarrow
H_{p}(J)\leftarrow H_{p+1}(X/J)\leftarrow \dots  \label{longXJi}
\end{equation}

%TCIMACRO{%
%\TeXButton{Bibliography}{\addcontentsline{toc}{section}{References}
%{\footnotesize
%\begin{thebibliography}{9}
%\input bibstyle
%\input tcirefs
%\end{thebibliography}
%}}}%
%BeginExpansion
%
%EndExpansion

%TCIMACRO{\TeXButton{endofdocument}{\endofdocument}}%
%BeginExpansion
\endofdocument%
%EndExpansion

\section*{To-dos}

\label{SecTodo}

Directions of research:

\begin{enumerate}
\item Strong product with diagonal - K\"{u}nneth formula, cup product of
homologies, Lefschitz fix point theorem

\item Persistent homologies

\item $\varepsilon $-net - computing homologies of manifolds

\item Homotopy, $\pi _{1}$, loop spaces

\item Random graphs

\item Finitely generated groups and Cayley graphs

\item Optimal representatives of homology

\item \v{C}ech cohomology for graphs

\item Vector bundles and curvature
\end{enumerate}

\bigskip ======================================

Some problems:

\begin{enumerate}
\item $\times $ Notation $e$ for $\left( -1\right) $-path and $\left(
-1\right) $-form. Consistent use, simplifying some proofs.

\item simplicial complex and poset in homologies

\item Graphs as category with morphisms. Mappings between graphs and induced
mapping between homologies: make it clear. Also, restate equivalent of
homologies as induced by the mapping between graphs.

\item Mention operation bouquet for which $H^{k}\left( X\vee Y\right)
=H^{k}\left( X\right) \oplus H^{k}\left( Y\right) .$ It allows constructing
digraphs with prescribed values of $H^{1},H^{2},...,H^{n}.$

\item ? $\dim H_{0}=C$ via $H_{0}=\Omega _{0}/\partial \Omega _{1}$

\item Generators of $H_{0}$ and $\widetilde{H}_{0}$

\item ? If $\dim \Omega _{n}\leq 1$ then $\dim $ $\Omega _{n+1}=0$ via $%
\partial $

\item Relation to cup product on simplicial complexes.

\item Relation to direct product of simplicial complexes.

\item Cubical complexes as path complexes

\item $\times $ change $\mathcal{R}_{p}$ to $\widetilde{\mathcal{R}}_{p}$
and $R_{p}$ to $\mathcal{R}_{p}$

\item Homology of undirected graph. Define first edge direction
monotonically, using any order of vertices. Conjecture: the homologies do
not depend on the choice of the order of the vertices. Not true if there is
a square. For which graphs this may be true? Triangulations?

\item Proof of K\"{u}nneth formula for diagonal product. Use it to define
cup product in homologies as pullback from $X\rightarrow X\times X$. Use the
latter to prove Lefschitz fix point theorem.

\item Poincar\'{e} duality in the form $H_{p}\left( G\right) \cong
H_{q}\left( G^{\ast }\right) $ where the dual graph $G^{\ast }$ is yet to be
defined.

\item $\times $ Poincar\'{e} duality does not work for $G^{\ast }=G$ as
there is a graph with a prescribed values of $H^{k}$.

Let $\dim H_{N}=1$ and let $v$ be a generating element of $H_{N}.$ Fix $%
p+q=N $ and define for any $\varphi \in H^{p}$ a path $u\in H_{q}$ as
follows: $u$ acts on any $\psi \in H_{q}^{\ast }=H^{q}$ by%
\begin{equation*}
\langle u,\psi \rangle =\langle \varphi \psi ,v\rangle .
\end{equation*}%
Hence, we obtain mapping $H^{p}\rightarrow H_{q}.$ Questions:

\begin{enumerate}
\item When this mapping is isomorphism? It suffices to prove that it is
monomorphism since in the same way $H^{q}\rightarrow H_{p}$ would be
monomorphism and, hence, $H^{p}\rightarrow H_{q}$ epimorphism.

\item This mapping takes $d$ to $\partial $. If image of $d\varphi $ is $w$
then%
\begin{equation*}
\langle w,\psi \rangle =\left( \left( d\varphi \right) \psi ,v\right) .
\end{equation*}%
Why $\partial u$ satisfies this, that is,%
\begin{equation*}
\langle \partial u,\psi \rangle =\left( \left( d\varphi \right) \psi
,v\right) ?
\end{equation*}%
We have%
\begin{equation*}
\langle \partial u,\psi \rangle =\langle u,d\psi \rangle =\langle \varphi
d\psi ,v\rangle =\langle d\left( \varphi \psi \right) \pm \left( d\varphi
\right) \psi ,v\rangle =\pm \langle \left( d\varphi \right) \psi ,v\rangle .
\end{equation*}%
Here we have used that $\langle d\left( \varphi \psi \right) ,v\rangle
=\langle \varphi \psi ,\partial v\rangle =0$ because $\partial v=0.$
\end{enumerate}

\item Duality between $\times $ and $\ast $ products of digraphs

\item Disprove that $\varphi \psi =\left( -1\right) ^{pq}\psi \varphi $ for
cup product $H^{p}\cup H^{q}$

\item Modify the $\partial $ in order to include all polyhedra

\item \v{C}ech cohomology on graph using covering of graph by paths or
trees. Compare to digraph cohomology

\item Use that $H_{n}=\ker \partial |_{\mathcal{A}_{n}}/\left( \mathcal{A}%
_{n}\cap \partial \mathcal{A}_{n+1}\right) .$ We can modify $\partial $ to $%
\widetilde{\partial }$ so that $\widetilde{\partial }:\mathcal{A}%
_{n+1}\rightarrow \mathcal{A}_{n}$ and 
\begin{equation*}
\widetilde{\partial }\mathcal{A}_{n+1}=\mathcal{A}_{n}\cap \partial \mathcal{%
A}_{n+1}
\end{equation*}%
and, hence, $\partial \widetilde{\partial }=0$ and%
\begin{equation*}
H_{n}=\ker \partial |_{\mathcal{A}_{n}}/\widetilde{\partial }\mathcal{A}%
_{n+1}.
\end{equation*}

\item take a regular path complex, de-regularize it and show that the
regular homologies of the regular path complex are the same and non-regular
homologies of de-regularized path complex

\item Include Hochschild

\item Theorem \ref{Tcab}: can one remove the condition that characteristic
of $\mathbb{K}$ is $0$. If the characteristic is $c$ then consider
separately the case $\left\vert J\right\vert =0\func{mod}c$.
\end{enumerate}

\end{document}